\documentclass[leqno,10pt,oneside]{amsart}
\usepackage[english]{babel}
\usepackage[utf8]{inputenc}
\usepackage{amsmath,amssymb,amsthm,mathtools,stmaryrd,appendix}
\usepackage{microtype}
\usepackage{eqparbox}
\usepackage{graphicx}
\usepackage{float}
\usepackage{etoolbox}
\patchcmd{\thmhead}{(#3)}{#3}{}{} 

\usepackage{enumitem}
\setlist[enumerate]{label=\mbox{(\textit{\roman*}\hspace{.08em})},font=\rm,itemsep=0em}

\newcommand{\hair}{\ifmmode\mskip1.5mu\else\kern0.08em\fi}

\usepackage[colorlinks]{hyperref} 
\hypersetup{
  pdftitle   = {},
  pdfauthor  = {},
  pdfcreator = {}
}
\usepackage{tikz}
\usetikzlibrary{matrix,arrows.meta,decorations.markings,decorations.pathmorphing,shapes,positioning}
\tikzset{>=stealth}
\newcommand{\tikzto}{\mathrel{\tikz[baseline]\draw[ ->,line width=.45pt] (0ex,0.65ex) -- (3ex,0.65ex);}}
\newcommand{\tikzmapsto}{\mathrel{\tikz[baseline]\draw[|->,line width=.45pt] (0pt,0.65ex) -- (3ex,0.65ex);}}
\newcommand{\tikzshortrightarrow}{\mathrel{\tikz[baseline]\draw[->,line width=.45pt] (0ex,0.65ex) -- (1.75ex,0.65ex);}}
\newcommand{\tikztwoheadrightarrow}{\mathrel{\tikz[baseline]\draw[->>,line width=.45pt] (0ex,0.65ex) -- (3.5ex,0.65ex);}}
\newcommand{\tikztoarg}[1]{\mathrel{\tikz[baseline]\path[->,line width=.45pt] (0ex,0.65ex) edge node[above=-.4ex, overlay, font=\scriptsize] {$#1$} (3.5ex,.65ex);}}

\newcommand{\tikztosim}{\mathrel{\tikz[baseline] \path[->,line width=.45pt] (0ex,0.65ex) edge node[above=-.8ex, overlay, font=\normalsize, pos=.45] {$\sim$} (3.5ex,.65ex);}}

\newcommand{\hdot}{{\mkern.7mu\tikz[baseline] \draw [line width=1pt, line cap=round, fill=black] (0,0.5ex) circle(0.13ex);}}
\newcommand{\ldot}{{\mkern.1mu\tikz[baseline] \draw [line width=1pt, line cap=round, fill=black] (0,0.35ex) circle(0.13ex);\mkern3mu}}

\numberwithin{equation}{section}
\newtheorem{theorem}[equation]{Theorem}
\newtheorem*{theorem*}{Theorem}
\newtheorem{proposition}[equation]{Proposition}
\newtheorem{lemma}[equation]{Lemma}
\newtheorem*{lemma*}{Lemma}
\newtheorem{observation}[equation]{Observation}

\newtheorem{corollary}[equation]{Corollary}
\newtheorem*{corollary*}{Corollary}
\newtheorem{conjecture}[equation]{Conjecture}
\theoremstyle{definition}
\newtheorem{definition}[equation]{Definition}

\newtheorem{notation}[equation]{Notation}
\theoremstyle{remark}
\newtheorem{remark}[equation]{Remark}
\newtheoremstyle{none}
{} 
{} 
{} 
{} 
{} 
{} 
{0pt} 
{} 
\theoremstyle{none}
\newtheorem*{none*}{}

\renewcommand{\Bar}{\operatorname{Bar}}

\newcommand{\e}{\mathrm{e}}
\DeclareMathOperator{\Ext}{Ext}

\DeclareMathOperator{\HH}{HH}
\DeclareMathOperator{\Hom}{Hom}

\newcommand{\id}{\mathrm{id}}

\newcommand{\K}{{\mathrm{K}}}
\newcommand{\Q}{{\mathrm{Q}}}
\newcommand{\I}{{\mathrm{I}}}

\newcommand{\KK}{{\overline{\mathrm{K}}}{}}
\newcommand{\QQ}{{\overline{\mathrm{Q}}}{}}
\newcommand{\II}{{\overline{\mathrm{I}}}{}}
\newcommand{\RR}{{\overline{\mathrm{R}}}{}}
\renewcommand{\SS}{{\overline{\mathrm{S}}}{}}

\usepackage{accents}
\newcommand{\cupd}[1]{\underaccent{\smile}{#1}}
\newcommand{\capd}[1]{\accentset{\frown}{#1}}

\newcommand{\up}{{\mkern.5mu\begin{tikzpicture}[x=.27em,y=.45em,decoration={markings,mark=at position 0.99 with {\arrow[black]{Stealth[length=4.8pt]}}}]
\draw[color=white,postaction={decorate}] (0,-.1) -- (0,.1);
\end{tikzpicture}\mkern.5mu}}
\newcommand{\dn}{{\mkern.5mu\begin{tikzpicture}[x=.27em,y=.45em,decoration={markings,mark=at position 0.99 with {\arrow[black]{Stealth[length=4.8pt]}}}]
\draw[color=white,postaction={decorate}] (0,.1) -- (0,-.1);
\end{tikzpicture}\mkern.5mu}}

\newcommand{\UP}[1]{\draw[color=white,postaction={decorate}] (#1,-.4) -- (#1,.35);}
\newcommand{\DN}[1]{\draw[color=white,postaction={decorate}] (#1,.4) -- (#1,-.35);}
\newcommand{\CIRCLE}[1]{\draw[line width=.5pt,line cap=round] (#1,-.21) -- (#1,.21) arc[start angle=180, end angle=0, radius=.5] -- ++(0,-.42) arc[start angle=0, end angle=-180, radius=.5];}
\newcommand{\ROUNDCIRCLE}[2]{\draw[line width=.5pt,line cap=round] (#1,-.21) -- (#1,.21) arc[start angle=180, end angle=0, radius=#2] -- ++(0,-.42) arc[start angle=0, end angle=-180, radius=#2];}
\newcommand{\TINYCIRCLE}[1]{\draw[line width=.5pt,line cap=round] (#1,-.3) -- (#1,.3) arc[start angle=180, end angle=0, radius=1] -- ++(0,-.6) arc[start angle=0, end angle=-180, radius=1];}
\newcommand{\CIRCLES}[2]{\draw[line width=.5pt,line cap=round] (#1,-.21) -- (#1,.21) arc[start angle=180, end angle=0, x radius=#2, y radius =.75*#2] -- ++(0,-.42) arc[start angle=0, end angle=-180, x radius=#2, y radius=.75*#2];}

\newcommand{\SCIRCLES}[2]{\draw[line width=.5pt,line cap=round] (#1,-.21) -- (#1,.21) arc[start angle=180, end angle=0, x radius=#2, y radius =(.55-.015*#2)*#2] -- ++(0,-.42) arc[start angle=0, end angle=-180, x radius=#2, y radius=(.55-.015*#2)*#2];}
\newcommand{\SCCIRCLE}[1]{\draw[line width=.5pt,line cap=round] (#1,0) arc[start angle=180, end angle=-180, x radius=1.5, y radius=1.2];}

\newcommand{\CAP}[1]{\draw[line width=.5pt,line cap=round] (#1,.21) arc[start angle=180, end angle=0, radius=.5];}
\newcommand{\CUP}[1]{\draw[line width=.5pt,line cap=round] (#1,-.21) arc[start angle=-180, end angle=0, radius=.5];}
\newcommand{\RAY}[1]{\draw[line width=.5pt,line cap=round] (#1,-1) -- (#1,1);}
\newcommand{\RAYS}[2]{\draw[line width=.5pt,line cap=round] (#1,-#2) -- (#1,#2);}

\newcommand{\RAYUP}[1]{\draw[line width=.5pt,line cap=round] (#1,-.21) -- (#1,1);}
\newcommand{\RAYDN}[1]{\draw[line width=.5pt,line cap=round] (#1,.21) -- (#1,-1);}
\newcommand{\CONNECT}[1]{\draw[line width=.5pt,line cap=round] (#1,-.21) -- (#1,.21);}

\newcommand{\CCIRCLES}[1]{\draw[line width=.5pt,line cap=round] (#1,0) arc[start angle=180, end angle=-180, radius=1];}
\newcommand{\MRAY}[1]{\draw[line width=.5pt,line cap=round] (#1,1) -- ++(0,-2) ++(-.25,0) -- ++(.5,0) ++(0,2) -- ++(-.5,0);}

\newcommand{\SCCAP}[1]{\draw[line width=.5pt,line cap=round] (#1,0) arc[start angle=180, end angle=0, x radius=1.5, y radius=1.2];}

\newcommand{\SCCUP}[1]{\draw[line width=.5pt,line cap=round] (#1,0) arc[start angle=-180, end angle=0, x radius=1.5, y radius=1.2];}
\newcommand{\CCCAP}[2]{\draw[line width=.5pt,line cap=round] (#1,.21) arc[start angle=180, end angle=0,  x radius=#2, y radius=.75*#2];}
\newcommand{\CCCUP}[2]{\draw[line width=.5pt,line cap=round] (#1,-.21) arc[start angle=-180, end angle=0, radius=#2, y radius=.75*#2 ];}
\newcommand{\ROUNDCAP}[2]{\draw[line width=.5pt,line cap=round] (#1,.21) arc[start angle=180, end angle=0, radius=#2];}
\newcommand{\ROUNDCUP}[2]{\draw[line width=.5pt,line cap=round] (#1,-.21) arc[start angle=-180, end angle=0, radius=#2];}
\newcommand{\CONNECTUP}[1]{\draw[line width=.5pt,line cap=round] (#1,0) -- (#1,.21);}
\newcommand{\CONNECTDN}[1]{\draw[line width=.5pt,line cap=round] (#1,-.21) -- (#1,0);}
\newcommand{\RRAYUP}[1]{\draw[line width=.5pt,line cap=round] (#1,0) -- (#1,1);}
\newcommand{\RRAYDN}[1]{\draw[line width=.5pt,line cap=round] (#1,0) -- (#1,-1);}
\newcommand{\RARC}[1]{
\draw[line width=.45pt,line cap=round] (#1,-.21) -- ++(0,.42) arc[start angle=180, end angle=120, radius=1];
\draw[line width=.45pt,line cap=round] (#1,-.21) -- ++(0,.42) arc[start angle=180, end angle=120, radius=1];
\draw[line width=.45pt,line cap=round] (#1,-.21) arc[start angle=180, end angle=240, radius=1];
\draw[dash pattern=on 0pt off 1.3pt, line width=.45pt, line cap=round] (#1,-.21) ++(0,.42) ++(60:1) arc[start angle=120, end angle=80, radius=1.2];
\draw[dash pattern=on 0pt off 1.3pt, line width=.45pt, line cap=round] (#1,-.21) ++(-60:1) arc[start angle=-120, end angle=-80, radius=1.2];}
\newcommand{\RRARC}[1]{
\draw[line width=.45pt,line cap=round] (#1,-.21) -- ++(0,.42) arc[start angle=180, end angle=120, radius=1];
\draw[line width=.45pt,line cap=round] (#1,-.21) -- ++(0,.42) arc[start angle=180, end angle=120, radius=1];
\draw[line width=.45pt,line cap=round] (#1,-.21) arc[start angle=180, end angle=240, radius=1];
\draw[dash pattern=on 0pt off 1.3pt, line width=.45pt, line cap=round] (#1,-.21) ++(0,.42) ++(60:1) arc[start angle=120, end angle=90, radius=1.2];
\draw[dash pattern=on 0pt off 1.3pt, line width=.45pt, line cap=round] (#1,-.21) ++(-60:1) arc[start angle=-120, end angle=-90, radius=1.2];}
\newcommand{\LARC}[1]{
\draw[line width=.45pt,line cap=round] (#1,-.21) -- ++(0,.42) arc[start angle=0, end angle=60, radius=1];
\draw[line width=.45pt,line cap=round] (#1,-.21) arc[start angle=0, end angle=-60, radius=1];
\draw[dash pattern=on 0pt off 1.3pt, line width=.45pt, line cap=round] (#1,-.21) ++(0,.42) ++(120:1) arc[start angle=60, end angle=100, radius=1];
\draw[dash pattern=on 0pt off 1.3pt, line width=.45pt, line cap=round] (#1,-.21) ++(-120:1) arc[start angle=-60, end angle=-100, radius=1];}
\newcommand{\LLARC}[1]{
\draw[line width=.45pt,line cap=round] (#1,-.21) -- ++(0,.42) arc[start angle=0, end angle=60, radius=1];
\draw[line width=.45pt,line cap=round] (#1,-.21) arc[start angle=0, end angle=-60, radius=1];
\draw[dash pattern=on 0pt off 1.3pt, line width=.45pt, line cap=round] (#1,-.21) ++(0,.42) ++(120:1) arc[start angle=60, end angle=90, radius=1];
\draw[dash pattern=on 0pt off 1.3pt, line width=.45pt, line cap=round] (#1,-.21) ++(-120:1) arc[start angle=-60, end angle=-90, radius=1];}
\newcommand{\UPRARC}[1]{
\draw[line width=.45pt,line cap=round] (#1,-.21) -- ++(0,.42) arc[start angle=180, end angle=120, radius=1];
\draw[dash pattern=on 0pt off 1.3pt, line width=.45pt, line cap=round] (#1,-.21) ++(0,.42) ++(60:1) arc[start angle=120, end angle=80, radius=1.2];
}
\newcommand{\UPLARC}[1]{
\draw[line width=.45pt,line cap=round] (#1,-.21) -- ++(0,.42) arc[start angle=0, end angle=60, radius=1];
\draw[dash pattern=on 0pt off 1.3pt, line width=.45pt, line cap=round] (#1,-.21) ++(0,.42) ++(120:1) arc[start angle=60, end angle=100, radius=1];
}
\newcommand{\DNRARC}[1]{
\draw[line width=.45pt,line cap=round] (#1,-.21) arc[start angle=180, end angle=240, radius=1];
\draw[dash pattern=on 0pt off 1.3pt, line width=.45pt, line cap=round] (#1,-.21) ++(-60:1) arc[start angle=-120, end angle=-80, radius=1.2];}
\newcommand{\DNLARC}[1]{
\draw[line width=.45pt,line cap=round] (#1,-.21) arc[start angle=0, end angle=-60, radius=1];
\draw[dash pattern=on 0pt off 1.3pt, line width=.45pt, line cap=round] (#1,-.21) ++(-120:1) arc[start angle=-60, end angle=-100, radius=1];}

\newcommand{\DDOTS}[1]{\node[font=\scriptsize] at (#1,0) {.\hspace{.6pt}.\hspace{.6pt}.};}
\newcommand{\DDOT}[1]{\node[font=\scriptsize] at (#1,0) {.};}

\begin{document}

\title[A$_\infty$ deformations of extended Khovanov arc algebras]{A$_\infty$ deformations of extended Khovanov arc algebras and Stroppel's conjecture}

\author{Severin Barmeier}
\address{Universität zu Köln, Mathematisches Institut, Weyertal 86-90, 50931 Köln, Germany \\ \textup{and} Hausdorff Research Institute for Mathematics, Poppelsdorfer Allee 45, 53115 Bonn, Germany}
\email{s.barmeier@gmail.com}

\author{Zhengfang Wang}
\address{Nanjing University, School of Mathematics, Nanjing 210093, Jiangsu, China \\ \textup{and} Universität Stuttgart, Institut für Algebra und Zahlentheorie, Pfaffen\-wald\-ring 57, 70569 Stuttgart, Germany}
\email{zhengfangw@nju.edu.cn}



\vspace{-2em}
\begin{abstract}
Extended Khovanov arc algebras $\K_m^n$ are graded associative algebras which naturally appear in a variety of contexts, from knot and link homology, low-dimensional topology and topological quantum field theory to representation theory and symplectic geometry. C.~Stroppel conjectured in her ICM 2010 address that the bigraded Hochschild cohomology groups of $\K_m^n$ vanish in a certain range, implying that the algebras $\K_m^n$ admit no nontrivial A$_\infty$ deformations, in particular that the algebras are intrinsically formal.

Whereas Stroppel's conjecture is known to hold for the algebras $\K_m^1$ and $\K_1^n$ by work of Seidel and Thomas, we show that $\K_m^n$ does in fact admit nontrivial A$_\infty$ deformations with nonvanishing higher products for all $m, n \geq 2$.

We describe both the extended Khovanov arc algebras $\K_m^n$ and their Koszul duals concretely as path algebras of quivers with relations. This allows us to give an explicit algebraic construction of A$_\infty$ deformations of $\K_m^n$ by using the correspondence between A$_\infty$ deformations of a Koszul algebra and filtered associative deformations of its Koszul dual.
\end{abstract}

\maketitle

\vspace{-2em}
\tableofcontents

\section{Introduction}

Khovanov arc algebras were introduced by Khovanov in his seminal work \cite{khovanov} on the categorification of the Jones polynomial of links. The Khovanov arc algebras are ``diagrammatic'' algebras whose basis is given by ``arc diagrams'' with composition defined by stacking the diagrams vertically and applying certain simplification rules obtained from a $2$-dimensional TQFT (see \S\ref{section:arcalgebras} for more details). In \cite{stroppel1} Stroppel introduced the {\it extended Khovanov arc algebras} $\K_m^n$ for $m, n \geq 1$ as a generalisation, where the arc diagrams contain not only closed arcs (circles) but also open arcs (lines). The algebras $\K_m^n$ were also studied by Chen and Khovanov \cite{chenkhovanov} by viewing them as subquotients of the classical Khovanov arc algebras. Brundan and Stroppel showed that the algebras $\K_m^n$ are quasi-hereditary covers of the classical Khovanov arc algebras and that they have many nice properties, for example they are Koszul, cellular and homologically smooth \cite{brundanstroppel1,brundanstroppel2}. See also Stroppel's ICM 2022 address \cite{stroppel3} for a recent account and further applications to link and tangle invariants.

Besides their role in link homology, the algebras $\K_m^n$ play a central role in representation theory and naturally appear in symplectic geometry. In \cite{stroppel1} it is shown that the category $\mathrm{rep} (\K_m^n)$ of finite-dimensional representations of $\K_m^n$ is equivalent to the category of perverse sheaves on the Grassmannian $\mathrm{Gr} (m, m + n)$ and also equivalent to the principal block of the parabolic Bernstein--Gelfand--Gelfand category $\mathcal O$ associated to the parabolic subalgebra corresponding to $\mathfrak{gl}_m (\mathbb C) \oplus \mathfrak{gl}_n (\mathbb C) \subset \mathfrak{gl}_{m+n} (\mathbb C)$, see Brundan and Stroppel \cite{brundanstroppel1,brundanstroppel2,brundanstroppel3,brundanstroppel4}, also for further applications in representation theory.

A symplectic analogue of classical Khovanov homology and the Jones polynomial was given by Seidel and Smith \cite{seidelsmith} via the Lagrangian Floer homology of certain nilpotent slices, described by Manolescu \cite{manolescu} as Hilbert schemes of surfaces. Abouzaid and Smith \cite{abouzaidsmith1,abouzaidsmith2} proved this symplectic Khovanov homology to be isomorphic to Khovanov homology (after collapsing the bigrading of the latter). Similar results were recently obtained for the extended Khovanov arc algebras $\K_m^n$ by Mak and Smith \cite{maksmith} who showed that the DG category $\mathrm{perf}_{\mathrm{dg}} (\K_m^n)$ of perfect DG modules is quasi-equivalent to the Fukaya--Seidel A$_\infty$ category $\mathcal{FS} (X_m^n, \pi_m^n)$, where $X_m^n = \mathrm{Hilb}^m (A_{m+n-1}) \smallsetminus D$ is the affine complement of a certain divisor $D \subset \mathrm{Hilb}^m (A_{m+n-1})$ of the Hilbert scheme of $m$ points on the Milnor fibre (the generic deformation) of an $A_{m+n-1}$ surface singularity and $\pi_m^n \colon X_m^n \tikzto \mathbb C$ is a symplectic Lefschetz fibration. The key observation of their proof is that $\mathcal{FS} (X_m^n, \pi_m^n)$ admits a generator whose endomorphism A$_\infty$ algebra turns out to be formal and A$_\infty$-quasi-isomorphic to $\K_m^n$. Whereas there is a natural ``geometric'' generator of the Fukaya--Seidel category $\mathcal{FS} (X_m^n, \pi_m^n)$ given by the Lefschetz thimbles associated to vanishing cycles \cite[\S 18]{seidel2}, the endomorphism A$_\infty$ algebra of this generator is {\it not} formal whenever $m, n \geq 2$. This is shown in \cite[App.~A]{maksmith} using the nontrivial A$_\infty$ structure on the $\Ext$-algebra of parabolic Verma modules given by Klamt and Stroppel \cite{klamtstroppel}. The quasi-equivalence between $\mathcal{FS} (X_m^n, \pi_m^n)$ and the DG category of perfect DG modules over $\K_m^n$ should therefore be understood as a formality result for $\mathcal{FS} (X_m^n, \pi_m^n)$. All of these results can be viewed as part of the continually growing dictionary between knot homology, symplectic geometry and representation theory.

In her ICM 2010 address, Stroppel stated the following conjecture about the bigraded Hochschild cohomology groups of $\K_m^n$.

\begin{conjecture}[{\cite[Conj.~2.7]{stroppel2}}]
\label{conjecture:stroppel}
$\HH^2_{i-2} (\K_m^n, \K_m^n) = 0$ if $i \neq 0$.\footnote{What we denote by $\HH^2_{i-2} (A, A)$ in this article is denoted by $\mathbb H^{i} (A)_{2-i}$ in \cite{stroppel2}.}
\end{conjecture}

Whereas the classical Khovanov arc algebras are symmetric algebras with large Hochschild (co)homology groups which can be described in terms of the Khovanov homology of torus links (see Rozansky \cite[Thm.~6.9]{rozansky}), their quasi-hereditary covers $\K_m^n$ are homologically smooth and have much smaller Hochschild cohomology. Conjecture \ref{conjecture:stroppel} implies in particular that (contrary to the classical arc algebras) the extended Khovanov arc algebras $\K_m^n$ are {\it intrinsically formal}, i.e.\ that every A$_\infty$ structure on $\K_m^n$ is equivalent to the original associative structure. It was already known by earlier work of Seidel and Thomas \cite[Lem.~4.21]{seidelthomas} (also for other choices of grading) that Conjecture \ref{conjecture:stroppel} holds for $\K_m^n$ whenever $\min (m, n) = 1$. Moreover, Conjecture \ref{conjecture:stroppel} would yield the formality result of Mak and Smith \cite{maksmith} mentioned above, as well as the formality result of Abouzaid and Smith \cite{abouzaidsmith1} for the symplectic arc algebras.

In contrast, our main result is the following.

\begin{theorem*}[Theorem \ref{thm:maintheorem} and Corollary \ref{corollary:Ainfinitydeformation}]
For all $m, n \geq 2$ we have the following:
\begin{enumerate}
\item $\HH^2_{i-2} (\K_m^n, \K_m^n)$ is $1$-dimensional for $i = 2mn-4$. \label{main1}
\item $\K_m^n$ is not intrinsically formal. \label{main2}
\item $\K_m^n$ admits an explicit nontrivial A$_\infty$ deformation. \label{main3}
\end{enumerate}
\end{theorem*}

Although \ref{main1} settles Conjecture \ref{conjecture:stroppel} negatively for $m, n \geq 2$, \ref{main2} and \ref{main3} establish the existence of interesting higher structures on the extended Khovanov arc algebras $\K_m^n$ whenever $m, n \geq 2$, see Corollary \ref{corollary:Ainfinitydeformation} and also Fig.~\ref{figure:diagramdeformation} for a diagrammatic interpretation of the A$_\infty$ deformations of $\K_2^2$. Moreover, the A$_\infty$ deformations of $\K_m^n$ can be viewed as deformations of the Fukaya--Seidel category $\mathcal{FS} (X_m^n, \pi_m^n)$ (see Remark \ref{remark:fukayaseidel}). 

Our proof relies on a method of computing Hochschild cohomology and associative deformations developed in \cite{barmeierwang1} and can be divided into three steps as follows. In the first step, we realise $\K_m^n$ as a quiver with relations, i.e.\ we give an algebra isomorphism $\K_m^n \simeq \Bbbk \Q_m^n / \I_m^n$ for some finite quiver $\Q_m^n$ with admissible ideal $\I_m^n$ which may be of independent interest. We show that the quiver $\Q_m^n$ is the double quiver of a bipartite graph $\Gamma_m^n$ (Proposition \ref{proposition:bipartite}) and that the ideal $\I_m^n$ is generated by three types of quadratic relations (Proposition \ref{proposition:relations}). 

In the second step, noting that $\K_m^n$ is a Koszul algebra \cite[\S 5]{brundanstroppel2}, we describe its Koszul dual as $\KK_m^n = \Bbbk \QQ_m^n / \II_m^n$, where $\QQ_m^n$ is the opposite quiver of $\Q_m^n$ and $\II_m^n$ the ideal generated by the orthogonal quadratic relations (Proposition \ref{proposition:relationdual}). We show that $\KK_m^n$ admits a natural reduction system $\RR_m^n$ satisfying the diamond condition by showing that the number of the associated irreducible paths can be described by the Kazhdan--Lusztig polynomials (see \S\ref{subsection:irreduciblepaths}).

In the third step, we describe deformations of the reduction system $\RR_m^n$ to construct an explicit (nontrivial) first-order deformation of $\RR_m^n$ so that $\HH^{2}_{2mn-6}(\KK_m^n, \KK_m^n)$ is $1$-dimensional. By Keller's duality theorem (Theorem \ref{theorem:keller}) we obtain that $\HH^{2}_{2mn-6}(\K_m^n, \K_m^n)$ is also $1$-dimensional. The above construction in fact also gives a nontrivial (filtered) associative deformation of $\KK_m^n$. The minimal model of the derived endomorphism algebra of its simples then gives an explicit A$_\infty$ deformation of $\K_m^n$ (see Proposition \ref{prop:minimalmodel} and Corollary \ref{corollary:Ainfinitydeformation}).

The simplest example of a nontrivial A$_\infty$ deformation of $\K_m^n$ can be found in \S\ref{subsection:K22}, where the main results are proven for $\K_2^2$. (See Fig.~\ref{figure:22} for an illustration of the degree 0 and 1 arc diagrams in $\K_2^2$.) In fact, \S\ref{subsection:K22} can also be viewed as a blueprint for the general case which is technically more involved, but structurally similar. For degree reasons the algebra $\K_2^2$ in fact admits a {\it unique} A$_\infty$ deformation which is given in Corollary \ref{corollary:Ainfinitydeformation22}. A diagrammatic description of this A$_\infty$ deformation is given in Fig.~\ref{figure:diagramdeformation}.

\section{Extended Khovanov arc algebras}
\label{section:arcalgebras}

\subsection{Diagrammatics and grading}

We first recall the diagrammatics for the extended Khovanov arc algebras. For more details we refer to \cite{stroppel1,brundanstroppel1}.

Let $m, n \geq 1$ be fixed natural numbers. We fix a horizontal line in the plane and $m+n$ distinct points on it.

A {\it cup diagram} is a diagram consisting of $k \leq \lfloor \frac{m+n}{2} \rfloor$ nested cups and $m + n - 2k$ half-lines which are attached to the $m+n$ fixed points on the horizontal line. More precisely, the cups are nested lower semicircles lying below the horizontal line and joining pairs of points and the half-lines are attached to single points and extend downwards in such a way that cups and half-lines pairwise do not intersect (in particular, half-lines cannot be nested inside cups). For example, 
\[
\begin{tikzpicture}[x=.75em,y=.75em,decoration={markings,mark=at position 0.99 with {\arrow[black]{Stealth[length=4.8pt]}}}]
\RRAYDN{0} \CONNECTDN{1} \CCCUP{1}{1.5} \CONNECTDN{2} \CUP{2} \CONNECTDN{3} \CONNECTDN{4}
\end{tikzpicture}
\]
is a cup diagram for $m + n = 5$.

A {\it cap diagram} is the mirror image of a cup diagram obtained by reflecting all cups along the fixed horizontal line. For example,
\[
\begin{tikzpicture}[x=.75em,y=.75em,decoration={markings,mark=at position 0.99 with {\arrow[black]{Stealth[length=4.8pt]}}}]
\RRAYUP{0} \CONNECTUP{1} \CCCAP{1}{1.5} \CONNECTUP{2} \CAP{2} \CONNECTUP{3} \CONNECTUP{4}
\end{tikzpicture}
\]
is a cap diagram.

A {\it weight} of type ${}_m^n$ is a sequence of weight symbols $\dn$ and $\up$ placed at the $m+n$ points, where $m$ is the total number of $\dn$'s and $n$ is the total number of $\up$'s. For example,
\[
\begin{tikzpicture}[x=.75em,y=.75em,decoration={markings,mark=at position 0.99 with {\arrow[black]{Stealth[length=4.8pt]}}}]
\UP{0} \DN{1} \DN{2} \UP{3} \UP{4}
\end{tikzpicture}
\]
is a weight of type ${}_2^3$. Let $\Lambda_m^n$ denote the set of weights of type $_m^n$. For example $\Lambda_1^2 = \{ \up \, \up \, \dn, \up \, \dn \, \up, \dn \, \up \, \up \}$. Note that $\# \Lambda_m^n = \binom{m + n}{m}$. 

Given a weight $\nu$, let $\cupd \nu$ denote the unique cup diagram obtained by recursively connecting directly neighbouring $\dn \, \up$ pairs in $\nu$ with a cup, ignoring $\up$'s and $\dn$'s already connected in a previous step, drawing half-lines extending downwards for the remaining positions, and deleting the weight symbols $\dn$ and $\up$ in the end. For example,
\[
\begin{tikzpicture}[x=.75em,y=.75em,decoration={markings,mark=at position 0.99 with {\arrow[black]{Stealth[length=4.8pt]}}}]
\begin{scope}[shift={(0,0)}]
\node at (-1.8,0) {$\nu = {}$};
\UP{0} \DN{1} \DN{2} \UP{3} \UP{4}
\node at (6,-.18) {$\tikzshortrightarrow$};
\end{scope}
\begin{scope}[shift={(8,0)}]
\UP{0} \DN{1} \DN{2} \UP{3} \UP{4}
\CUP{2} \CONNECTDN{3}
\node at (6,-.18) {$\tikzshortrightarrow$};
\end{scope}
\begin{scope}[shift={(16,0)}]
\UP{0} \DN{1} \DN{2} \UP{3} \UP{4}
\CCCUP{1}{1.5} \CUP{2} \CONNECTDN{3} \CONNECTDN{4}
\node at (6,-.18) {$\tikzshortrightarrow$};
\end{scope}
\begin{scope}[shift={(24,0)}]
\UP{0} \DN{1} \DN{2} \UP{3} \UP{4}
\RAYDN{0} \CCCUP{1}{1.5} \CUP{2} \CONNECTDN{3} \CONNECTDN{4}
\node at (6,-.18) {$\tikzshortrightarrow$};
\end{scope}
\begin{scope}[shift={(32,.2)}]
\RRAYDN{0} \CONNECTDN{1} \CCCUP{1}{1.5} \CONNECTDN{2} \CUP{2} \CONNECTDN{3} \CONNECTDN{4}
\node at (5.8,-.38) {${} = \cupd \nu$};
\end{scope}
\end{tikzpicture}
\]
(Note that $\cupd \nu$ does not depend on the order of the neighbouring $\dn \, \up$ pairs chosen in the recursion.)

\begin{figure}
\begin{tikzpicture}[x=.75em,y=.75em]
\begin{scope}[decoration={markings,mark=at position 0.99 with {\arrow[black]{Stealth[length=4.8pt]}}}]
\DN{0} \UP{1} \DN{2} \UP{3} \DN{4} \UP{5} \DN{6} \UP{7} \DN{8} \UP{9} \DN{10} \UP{11} \UP{12} \DN{13} \DN{14}
\CIRCLE{3} \CIRCLE{8}
\draw[line width=.5pt,line cap=round] (6,-1.5) -- (6,.15) arc[start angle=180, end angle=0, radius=.5] -- ++(0,-.3) arc[start angle=-180, end angle=0, x radius=1.5, y radius=.75*1.5] -- ++(0,.3) arc[start angle=0, end angle=180, x radius=4.5, y radius=.55*4.5] -- ++(0,-.3) arc[start angle=0, end angle=-180, radius=.5] -- ++(0,.3) arc[start angle=180, end angle=0, x radius=5.5, y radius=.6*5.5] -- ++(0,-.3) arc[start angle=-180, end angle=0, x radius=1.5, y radius=.75*1.5] -- (14,3.5);
\draw[line width=.5pt,line cap=round] (2,-.15) -- ++(0,.3) arc[start angle=180, end angle=0, x radius=1.5, y radius=.75*1.5] -- ++(0,-.3) arc[start angle=0, end angle=-180, x radius=1.5, y radius=.75*1.5];
\draw[line width=.5pt,line cap=round] (12,3.5) -- (12,-.15) arc[start angle=-180, end angle=0, radius=.5] -- (13,3.5);
\end{scope}
\begin{scope}[shift={(18,.91)},decoration={markings,mark=at position 0.8 with {\arrow[black]{Stealth[length=2.8pt]}}}]
\DN{0} \DN{1} \DN{2} \DN{3} \UP{4} \UP{5} \DN{6} \UP{7} \DN{8} \UP{9} \UP{10} \UP{11} \UP{12} \DN{13} \DN{14}
\end{scope}
\begin{scope}[shift={(18,.5)}]
\draw[line width=.5pt,line cap=round] (0,.5) arc[start angle=180, end angle=0, x radius=5.5, y radius=.6*5.5];
\draw[line width=.5pt,line cap=round] (1,.5) arc[start angle=180, end angle=0, x radius=4.5, y radius=.55*4.5];
\draw[line width=.5pt,line cap=round] (2,.5) arc[start angle=180, end angle=0, x radius=1.5, y radius=.75*1.5];
\draw[line width=.5pt,line cap=round] (3,.5) arc[start angle=180, end angle=0, radius=.5];
\draw[line width=.5pt,line cap=round] (6,.5) arc[start angle=180, end angle=0, radius=.5];
\draw[line width=.5pt,line cap=round] (8,.5) arc[start angle=180, end angle=0, radius=.5];
\draw[line width=.5pt,line cap=round] (12,.5) -- (12,3.75);
\draw[line width=.5pt,line cap=round] (13,.5) -- (13,3.75);
\draw[line width=.5pt,line cap=round] (14,.5) -- (14,3.75);
\node[font=\footnotesize,right] at (15,1.2) {$\capd \beta$ \ cap diagram};
\node[font=\footnotesize,right] at (15,-.5) {$\lambda$ \ weight};
\node[font=\footnotesize,right] at (15,-2.2) {$\cupd \alpha$ \ cup diagram};
\end{scope}
\begin{scope}[shift={(18,-.91)},decoration={markings,mark=at position 0.8 with {\arrow[black]{Stealth[length=2.8pt]}}}]
\DN{0} \UP{1} \DN{2} \DN{3} \UP{4} \UP{5} \DN{6} \DN{7} \DN{8} \UP{9} \UP{10} \DN{11} \DN{12} \UP{13} \UP{14}
\end{scope}
\begin{scope}[shift={(18,-.5)}]
\draw[line width=.5pt,line cap=round] (0,-.5) arc[start angle=-180, end angle=0, radius=.5];
\draw[line width=.5pt,line cap=round] (2,-.5) arc[start angle=-180, end angle=0, x radius=1.5, y radius=.75*1.5];
\draw[line width=.5pt,line cap=round] (3,-.5) arc[start angle=-180, end angle=0, radius=.5];
\draw[line width=.5pt,line cap=round] (7,-.5) arc[start angle=-180, end angle=0, x radius=1.5, y radius=.75*1.5];
\draw[line width=.5pt,line cap=round] (8,-.5) arc[start angle=-180, end angle=0, radius=.5];
\draw[line width=.5pt,line cap=round] (11,-.5) arc[start angle=-180, end angle=0, x radius=1.5, y radius=.75*1.5];
\draw[line width=.5pt,line cap=round] (12,-.5) arc[start angle=-180, end angle=0, radius=.5];
\draw[line width=.5pt,line cap=round] (6,-.5) -- (6,-2);
\end{scope}
\begin{scope}[shift={(18,0)},decoration={markings,mark=at position 0.99 with {\arrow[black]{Stealth[length=4.8pt]}}}]
\DN{0} \UP{1} \DN{2} \UP{3} \DN{4} \UP{5} \DN{6} \UP{7} \DN{8} \UP{9} \DN{10} \UP{11} \UP{12} \DN{13} \DN{14}
\end{scope}
\end{tikzpicture}%
\caption{An arc diagram (left) of type $_8^7$ of degree $6$ and its decomposition as a gluing of a cup diagram and a cap diagram along a weight (right)}
\label{figure:87}
\end{figure}

Let $\alpha, \beta, \lambda$ be weights of type ${}_m^n$. The cup diagram $\cupd \alpha$ and the cap diagram $\capd \beta$ can be glued along $\lambda$ as follows. The weight symbols of $\lambda$ are placed at the $m + n$ fixed points and the cups, caps and half-lines in $\cupd \alpha$ and $\capd \beta$ are glued together at these $m + n$ points, giving rise to open and closed arcs in the plane passing through the $\dn$'s and $\up$'s of the weight $\lambda$. An {\it arc diagram} of type $_m^n$ is such a gluing, which we denote by $\cupd \alpha \lambda \capd \beta$, satisfying the following constraints:
\begin{enumerate}
\item \label{arc2} the $\up$'s and $\dn$'s of $\lambda$ through which any single arc passes induce a well-defined orientation on the arc
\item \label{arc3} for any two half-lines extending both to the top or both to the bottom, their orientation must {\it not} be of the form $\dn \dotsb \up$.
\end{enumerate}
See Fig.~\ref{figure:87} for an illustration of an arc diagram as a gluing of a cup and cap diagram along a weight and Fig.~\ref{figure:12} for an illustration of all arc diagrams for all weights of type $_1^2$. Note that the following diagrams are {\it not} arc diagrams of type $_1^2$
\[
\begin{tikzpicture}[x=.75em,y=.75em,decoration={markings,mark=at position 0.99 with {\arrow[black]{Stealth[length=4.8pt]}}}]
\begin{scope}[shift={(6,0)}]
\UP{0} \UP{1} \DN{2}
\RAYDN{0} \CAP{0} \RAYDN{1} \RAY{2}
\end{scope}
\begin{scope}[shift={(12,0)}]
\DN{0} \UP{1} \UP{2}
\RAYDN{0} \CAP{0} \RAYDN{1} \RAY{2}
\end{scope}
\begin{scope}[shift={(18,0)}]
\UP{0} \DN{1} \UP{2}
\RAY{0} \RAY{1} \RAY{2}
\end{scope}
\end{tikzpicture}
\]
because they violate conditions \ref{arc2}, \ref{arc3} and \ref{arc3}, respectively.

\begin{figure}
\centering
\begin{tikzpicture}[x=.75em,y=.75em,decoration={markings,mark=at position 0.99 with {\arrow[black]{Stealth[length=4.8pt]}}}]
\begin{scope}[shift={(0,2.5)}]
\node at (-7,0) {weight};
\node at (1,0) {degree $0$};
\node at (9,0) {degree $1$};
\node at (17,0) {degree $2$};
\end{scope}
\begin{scope}[shift={(-8,0)}]
\UP{0} \UP{1} \DN{2}
\end{scope}
\begin{scope}[shift={(-8,-3)}]
\UP{0} \DN{1} \UP{2}
\end{scope}
\begin{scope}[shift={(-8,-6)}]
\DN{0} \UP{1} \UP{2}
\end{scope}
\begin{scope}[shift={(0,0)}]
\UP{0} \UP{1} \DN{2}
\RAY{0} \RAY{1} \RAY{2}
\end{scope}
\begin{scope}[shift={(0,-3)}]
\UP{0} \DN{1} \UP{2}
\RAY{0} \CIRCLE{1}
\end{scope}
\begin{scope}[shift={(0,-6)}]
\DN{0} \UP{1} \UP{2}
\CIRCLE{0} \RAY{2}
\end{scope}
\begin{scope}[shift={(6,0)}]
\UP{0} \UP{1} \DN{2}
\RAY{0} \RAYUP{1} \CUP{1} \RAYUP{2}
\end{scope}
\begin{scope}[shift={(10,0)}]
\UP{0} \UP{1} \DN{2}
\RAY{0} \RAYDN{1} \CAP{1} \RAYDN{2}
\end{scope}
\begin{scope}[shift={(6,-3)}]
\UP{0} \DN{1} \UP{2}
\RAYDN{0} \CAP{0} \CONNECT{1} \CUP{1} \RAYUP{2}
\end{scope}
\begin{scope}[shift={(10,-3)}]
\UP{0} \DN{1} \UP{2}
\RAYUP{0} \CUP{0} \CONNECT{1} \CAP{1} \RAYDN{2}
\end{scope}
\begin{scope}[shift={(16,0)}]
\UP{0} \UP{1} \DN{2}
\RAY{0} \CIRCLE{1}
\end{scope}
\begin{scope}[shift={(16,-3)}]
\UP{0} \DN{1} \UP{2}
\CIRCLE{0} \RAY{2}
\end{scope}
\end{tikzpicture}
\caption{All arc diagrams for weights of type $_1^2$}
\label{figure:12}
\end{figure}

On the right of Fig.~\ref{figure:87} we have also drawn the weight symbols of $\alpha$ and $\beta$ for reference, but in the resulting arc diagram, the weights $\alpha$ and $\beta$ are only reflected in the cups and caps, but not in the orientation of the arcs which are determined by the weight $\lambda$.

\begin{definition}[{\cite[\S 2]{brundanstroppel1}}]
\label{definition:arcdiagramdegreezero}
The {\it degree} of an arc diagram is given by the number of clockwise caps and clockwise cups it contains. Note that for each weight $\lambda$, the arc diagram $e_\lambda := \cupd \lambda \lambda \capd \lambda$ is the unique arc diagram of degree $0$, since the cups in $\cupd \lambda$ and the caps in $\capd \lambda$ are all oriented counterclockwise in $e_\lambda$.

The number of (counterclockwise) circles in $e_\lambda$ is called the {\it defect} of $\lambda$ and is denoted by $\mathrm{def} (\lambda)$.
\end{definition}

For example, the arc diagram of Fig.~\ref{figure:87} is of degree $6$ as it contains four clockwise cups and two clockwise caps. 

\begin{definition}
\label{definition:Kmn}
Let $\Bbbk$ be a field of characteristic $0$. The {\it extended Khovanov arc algebra} $\K_m^n$ is the $\mathbb N$-graded $\Bbbk$-algebra with basis given by the arc diagrams for weights of type $_m^n$, graded by their degree, and multiplication given as follows.

For any two arc diagrams $a = \cupd \alpha \lambda \capd \beta$ and $b = \cupd \gamma \mu \capd \delta$, their product $ab$ is zero unless $\beta = \gamma$ in which case $ab$ is obtained by writing $a$ below $b$, connecting the open arcs (lines) in the same position, and repeatedly performing the surgery
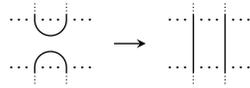
\[
\begin{tikzpicture}[x=.6em,y=.6em,decoration={markings,mark=at position 0.99 with {\arrow[black]{Stealth[length=3.8pt]}}}]
\begin{scope} 
\DDOTS{0} \DDOTS{2} \DDOTS{4}
\draw[line width=.5pt,line cap=round] (1,.3) -- (1,0) arc[start angle=-180, end angle=0, radius=1] -- ++(0,.3);
\draw[dash pattern=on 0pt off 1.3pt, line width=.5pt, line cap=round] (1,.35) -- ++(0,.7);
\draw[dash pattern=on 0pt off 1.3pt, line width=.5pt, line cap=round] (3,.35) -- ++(0,.7);
\end{scope}
\begin{scope}[shift={(0,-3)}] 
\DDOTS{0} \DDOTS{2} \DDOTS{4}
\draw[line width=.5pt,line cap=round] (1,-.3) -- (1,0) arc[start angle=180, end angle=0, radius=1] -- ++(0,-.3);
\draw[dash pattern=on 0pt off 1.3pt, line width=.5pt, line cap=round] (1,-.35) -- ++(0,-.7);
\draw[dash pattern=on 0pt off 1.3pt, line width=.5pt, line cap=round] (3,-.35) -- ++(0,-.7);
\end{scope}
\draw[->, line width=.5pt] (6,-1.5) -- ++(2,0);
\begin{scope}[shift={(10,0)}]
\begin{scope} 
\DDOTS{0} \DDOTS{2} \DDOTS{4}
\draw[line width=.5pt,line cap=round] (1,.3) -- (1,-3.3);
\draw[line width=.5pt,line cap=round] (3,.3) -- (3,-3.3);
\draw[dash pattern=on 0pt off 1.3pt, line width=.5pt, line cap=round] (1,.35) -- ++(0,.7);
\draw[dash pattern=on 0pt off 1.3pt, line width=.5pt, line cap=round] (3,.35) -- ++(0,.7);
\end{scope}
\begin{scope}[shift={(0,-3)}] 
\DDOTS{0} \DDOTS{2} \DDOTS{4}
\draw[dash pattern=on 0pt off 1.3pt, line width=.5pt, line cap=round] (1,-.35) -- ++(0,-.7);
\draw[dash pattern=on 0pt off 1.3pt, line width=.5pt, line cap=round] (3,-.35) -- ++(0,-.7);
\end{scope}
\end{scope}
\end{tikzpicture}
\]
on the underlying (unoriented) arcs for a cup--cap pair (i.e.\ cup and cap correspond to each other in the reflection identifying $\capd \beta$ and $\cupd \beta$) which can be connected without crossing other arcs. After each surgery (cutting the arcs open and reconnecting them as illustrated), the resulting arcs are reoriented by labelling the connected components containing the cup and the cap by one of $1$, $\epsilon$, or $\zeta$ according to whether the component is a closed arc oriented counterclockwise, a closed arc oriented clockwise, or an oriented open arc (line), respectively, and reorienting the resulting connected components according to the following rules:
\begin{enumerate}
\item If the surgery produces one connected component out of two, orient it according to
\begin{alignat*}{4}
1 \otimes 1 &\tikzmapsto 1 \qquad & 1 \otimes \epsilon &\tikzmapsto \epsilon \qquad & 1 \otimes \zeta &\tikzmapsto \zeta \qquad & \epsilon \otimes \zeta &\tikzmapsto 0 \\
\epsilon \otimes 1 &\tikzmapsto \epsilon & \epsilon \otimes \epsilon &\tikzmapsto 0 & \zeta \otimes 1 &\tikzmapsto \zeta & \zeta \otimes \epsilon &\tikzmapsto 0
\end{alignat*}
where $\tikzmapsto 0$ indicates that the arc diagram is deleted and does not contribute to the product.
\item If the surgery produces two connected components out of one, orient them according to
\[
1 \tikzmapsto 1 \otimes \epsilon + \epsilon \otimes 1 \qquad \epsilon \tikzmapsto \epsilon \otimes \epsilon \qquad \zeta \tikzmapsto \epsilon \otimes \zeta
\]
where the first rule means that the surgery produces a sum of two diagrams by duplicating --- in one diagram one of the two new closed arcs is oriented clockwise and the other counterclockwise and in the other diagram they are oriented in the opposite way (see the lower part of Fig.~\ref{figure:multiplication}). The two diagrams should separately be simplified further if necessary.
\item If the surgery produces two connected components out of two, then this only happens if both are open arcs (lines $\zeta \otimes \zeta$) and one sets $ab = 0$ unless both half-lines from one of the lines were oriented $\up$ and both half-lines from the other line were oriented $\dn$ and the orientation from the half-lines is carried over to the resulting lines (see the upper part of Fig.~\ref{figure:multiplication}). Symbolically we write
\[
\zeta \otimes \zeta \tikzmapsto \begin{cases} 0 \\ \zeta \otimes \zeta. \end{cases}
\]
\end{enumerate}
Finally, if all cup--cap pairs have been removed through surgery in each summand, then we can identify the weights and obtain a sum of arc diagrams for some weights of the same type. 
\end{definition}

For example, $\epsilon \tikzmapsto \epsilon \otimes \epsilon$ indicates that if the surgery produces two closed arcs out of one clockwise closed arc ($\epsilon$), then the resulting closed arcs are both oriented clockwise ($\epsilon \otimes \epsilon$). 
See the lower part of Fig.~\ref{figure:multiplication} for an illustration.

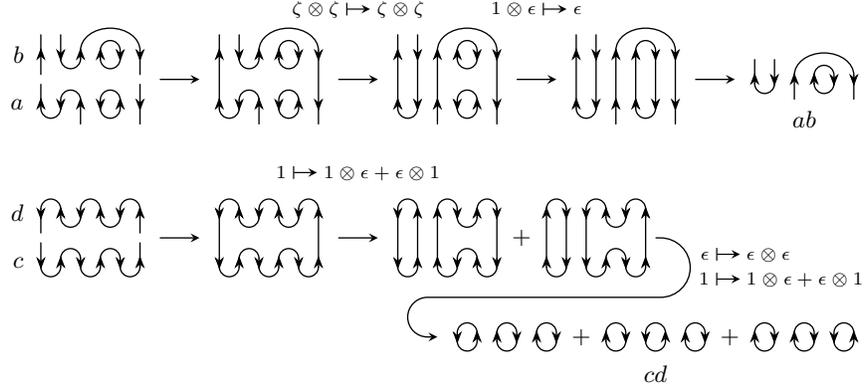
\begin{figure}
\begin{tikzpicture}[x=.75em,y=.75em,decoration={markings,mark=at position 0.99 with {\arrow[black]{Stealth[length=4.8pt]}}}]
\begin{scope}
\draw[->, line width=.5pt] (6,-1.25) -- ++(2,0);
\draw[->, line width=.5pt] (15,-1.25) -- ++(2,0);
\draw[->, line width=.5pt] (24,-1.25) -- ++(2,0);
\draw[->, line width=.5pt] (33,-1.25) -- ++(2,0);
\node[font=\scriptsize] at (16,2.25) {\strut$\zeta \otimes \zeta \tikzmapsto \zeta \otimes \zeta$};
\node[font=\scriptsize] at (25,2.25) {\strut$1 \otimes \epsilon \tikzmapsto \epsilon$};
\begin{scope} 
\node[font=\small,left] at (-.4,0) {$b$};
\UP{0} \DN{1} \UP{2} \UP{3} \DN{4} \DN{5}
\RAY{0} \RAYUP{1} \CUP{1} \CONNECT{2} \CCCAP{2}{1.5} \CIRCLE{3} \RAYDN{5}
\end{scope}
\begin{scope}[shift={(0,-2.5)}] 
\node[font=\small,left] at (-.4,0) {$a$};
\UP{0} \DN{1} \UP{2} \DN{3} \UP{4} \DN{5}
\RAYUP{0} \CUP{0} \CONNECT{1} \CAP{1} \RAYDN{2} \CIRCLE{3} \RAY{5}
\end{scope}
\begin{scope}[shift={(9,0)}]
\begin{scope} 
\UP{0} \DN{1} \UP{2} \UP{3} \DN{4} \DN{5}
\RAYUP{1} \CUP{1} \CONNECT{2} \CCCAP{2}{1.5} \CIRCLE{3}
\end{scope}
\begin{scope}[shift={(0,-2.5)}] 
\UP{0} \DN{1} \UP{2} \DN{3} \UP{4} \DN{5}
\CUP{0} \CONNECT{1} \CAP{1} \RAYDN{2} \CIRCLE{3}
\draw[line width=.5pt,line cap=round] (0,-.21) -- (0,3.5);
\draw[line width=.5pt,line cap=round] (5,-1) -- (5,2.71);
\end{scope}
\end{scope}
\begin{scope}[shift={(18,0)}]
\begin{scope} 
\UP{0} \DN{1} \UP{2} \UP{3} \DN{4} \DN{5}
\CCCAP{2}{1.5} \CIRCLE{3}
\end{scope}
\begin{scope}[shift={(0,-2.5)}] 
\UP{0} \DN{1} \UP{2} \DN{3} \UP{4} \DN{5}
\CUP{0} \CIRCLE{3}
\draw[line width=.5pt,line cap=round] (0,-.21) -- (0,3.5);
\draw[line width=.5pt,line cap=round] (1,-.21) -- (1,3.5);
\draw[line width=.5pt,line cap=round] (2,-1) -- (2,2.71);
\draw[line width=.5pt,line cap=round] (5,-1) -- (5,2.71);
\end{scope}
\end{scope}
\begin{scope}[shift={(27,0)}]
\begin{scope} 
\UP{0} \DN{1} \UP{2} \UP{3} \DN{4} \DN{5}
\CCCAP{2}{1.5} \CAP{3}
\end{scope}
\begin{scope}[shift={(0,-2.5)}] 
\UP{0} \DN{1} \UP{2} \UP{3} \DN{4} \DN{5}
\CUP{0} \CUP{3}
\draw[line width=.5pt,line cap=round] (0,-.21) -- (0,3.5);
\draw[line width=.5pt,line cap=round] (1,-.21) -- (1,3.5);
\draw[line width=.5pt,line cap=round] (2,-1) -- (2,2.71);
\draw[line width=.5pt,line cap=round] (3,-.21) -- (3,2.71);
\draw[line width=.5pt,line cap=round] (4,-.21) -- (4,2.71);
\draw[line width=.5pt,line cap=round] (5,-1) -- (5,2.71);
\end{scope}
\end{scope}
\begin{scope}[shift={(36,0)}]
\begin{scope}[shift={(0,-1.25)}] 
\UP{0} \DN{1} \UP{2} \UP{3} \DN{4} \DN{5}
\RAYUP{0} \CUP{0} \RAYUP{1} \RAYDN{2} \CCCAP{2}{1.5} \CIRCLE{3} \RAYDN{5}
\node[font=\small] at (2.5,-2) {$ab$};
\end{scope}%
\end{scope}%
\end{scope}
\begin{scope}[shift={(0,-8)}]
\draw[->, line width=.5pt] (6,-1.25) -- ++(2,0);
\draw[->, line width=.5pt] (15,-1.25) -- ++(2,0);
\draw[->, line width=.5pt] (31,-1.25) -- ++(.25,0) arc[start angle=90,end angle=-90,radius=1.5] -- ++(-11.75,0) arc[start angle=90,end angle=270,radius=1] -- ++(.5,0);
\node[font=\scriptsize] at (16,2) {\strut$1\tikzmapsto 1 \otimes \epsilon+ \epsilon\otimes 1$};
\node[right,font=\scriptsize] at (33.25,-2.125) {\strut$\mathllap{\epsilon} \tikzmapsto \epsilon\otimes \epsilon$};
\node[right,font=\scriptsize] at (33.25,-3.375) {\strut$\mathllap{1} \tikzmapsto 1 \otimes \epsilon+ \epsilon\otimes 1$};
\begin{scope} 
\node[font=\small,left] at (-.4,0) {$d$};
\DN{0} \UP{1} \DN{2} \UP{3} \DN{4} \UP{5}
\RAYDN{0} \CAP{0} \CUP{1}  \CONNECT{1}  \CAP{2}  \CONNECT{2} \CUP{3}  \CONNECT{3} \CAP{4} \CONNECT{4}  \RAYDN{5}\CONNECT{5} 
\end{scope}
\begin{scope}[shift={(0,-2.5)}] 
\node[font=\small,left] at (-.4,0) {$c$};
\DN{0} \UP{1} \DN{2} \UP{3} \DN{4} \UP{5} 
\RAYUP{0} \CUP{0} \CAP{1} \CONNECT{1}  \CUP{2}  \CONNECT{2} \CAP{3}  \CONNECT{3} \CUP{4}  \CONNECT{4} \RAYUP{5}\CONNECT{5} 
\end{scope}
\begin{scope}[shift={(9,0)}]
\begin{scope} 
\DN{0} \UP{1} \DN{2} \UP{3} \DN{4} \UP{5}
\CONNECT{0}\CAP{0} \CONNECT{1} \CUP{1} \CONNECT{2}  \CAP{2}  \CUP{3} \CONNECT{3} \CAP{4} \CONNECT{4} \CONNECT{5} 
\end{scope}
\begin{scope}[shift={(0,-2.5)}] 
\DN{0} \UP{1} \DN{2} \UP{3} \DN{4} \UP{5}
\CONNECT{0}\CUP{0} \CONNECT{1}\CAP{1} \CUP{2}  \CONNECT{2} \CAP{3} \CONNECT{3}\CUP{4} \CONNECT{4}\CONNECT{5} 
\draw[line width=.5pt,line cap=round] (0,-.21) -- (0,2.5);
\draw[line width=.5pt,line cap=round] (5,0) -- (5,2.71);
\end{scope}
\end{scope}
\begin{scope}[shift={(18,0)}]
\node[font=\small] at (6.25,-1.25) {$+$};
\begin{scope}
\begin{scope} 
\DN{0} \UP{1} \UP{2} \DN{3} \UP{4} \DN{5}
\CONNECT{0} \CAP{0} \CONNECT{1}\CAP{2} \CONNECT{2} \CUP{3} \CONNECT{3} \CAP{4} \CONNECT{4} \CONNECT{5} 
\end{scope}
\begin{scope}[shift={(0,-2.5)}] 
\DN{0} \UP{1} \UP{2} \DN{3}\UP{4} \DN{5}
\CUP{0}\CONNECT{0} \CUP{2}\CONNECT{1} \CAP{3} \CONNECT{2}\CUP{4} \CONNECT{3} \CONNECT{4}\CONNECT{5} 
\draw[line width=.5pt,line cap=round] (0,-.21) -- (0,2.5);
\draw[line width=.5pt,line cap=round] (1,-.21) -- (1,2.5);
\draw[line width=.5pt,line cap=round] (2,-.21) -- (2,2.5);
\draw[line width=.5pt,line cap=round] (5,0) -- (5,2.71);
\end{scope}
\end{scope}
\begin{scope}[shift={(7.5,0)}]
\begin{scope} 
\UP{0} \DN{1} \DN{2} \UP{3} \DN{4} \UP{5}
\CAP{0} \CONNECT{1}\CAP{2} \CONNECT{2} \CUP{3} \CONNECT{3} \CAP{4} \CONNECT{4} \CONNECT{5} 
\end{scope}
\begin{scope}[shift={(0,-2.5)}] 
\UP{0} \DN{1} \DN{2} \UP{3}\DN{4} \UP{5}
\CUP{0}\CONNECT{1} \CUP{2} \CONNECT{2}\CAP{3} \CONNECT{3}\CUP{4} \CONNECT{4}\CONNECT{5} 
\draw[line width=.5pt,line cap=round] (0,-.21) -- (0,2.5);
\draw[line width=.5pt,line cap=round] (1,-.21) -- (1,2.5);
\draw[line width=.5pt,line cap=round] (2,-.21) -- (2,2.5);
\draw[line width=.5pt,line cap=round] (5,0) -- (5,2.71);
\end{scope}
\end{scope}
\end{scope}
\begin{scope}[shift={(21,-5)}]
\node[font=\small] at (6.25,-1.25) {$+$};
\node[font=\small] at (13.75,-1.25) {$+$};
\begin{scope}[shift={(0,-1.25)}] 
\DN{0} \UP{1} \UP{2} \DN{3} \UP{4} \DN{5}
\CIRCLE{0} \CIRCLE{2} \CIRCLE{4} 
\end{scope}%
\begin{scope}[shift={(7.5,-1.25)}] 
\UP{0} \DN{1} \DN{2} \UP{3} \UP{4} \DN{5}
\CIRCLE{0} \CIRCLE{2} \CIRCLE{4} 
\node[font=\small,text height=1ex] at (2.5,-2) {$cd$};
\end{scope}%
\begin{scope}[shift={(15,-1.25)}] 
\UP{0} \DN{1} \UP{2} \DN{3} \DN{4} \UP{5}
\CIRCLE{0} \CIRCLE{2} \CIRCLE{4} 
\end{scope}
\end{scope}
\end{scope}
\end{tikzpicture}
\caption{Multiplications of arc diagrams in $\K_3^3$}
\label{figure:multiplication}
\end{figure}

\begin{remark}[Relation to TQFT]
The explicit description of the multiplication in $\K_m^n$ (Definition \ref{definition:Kmn}) can be obtained by viewing $\K_m^n$ as a quotient of Khovanov's original arc algebra $\mathrm{H}_{m+n}^{m+n}$ where all arc diagrams consist of (oriented) closed arcs. In $\mathrm{H}_k^k$, the multiplication only uses the rules involving $1, \epsilon$ and these rules are given by the multiplication and comultiplication in the Frobenius algebra $\Bbbk [\epsilon] / (\epsilon^2)$ or, equivalently, by the algebraic structure in the corresponding $2$-dimensional TQFT. This latter, topological viewpoint is useful for proving well-definedness and associativity of the multiplication as well as compatibility with the grading. In particular, the multiplication is independent of the order of the performed surgeries. 
\end{remark}

Note also that $\K_m^n$ is finite-dimensional, since there are finitely many arc diagrams for fixed $m, n \geq 1$. However, the dimension of $\K_m^n$ grows quite quickly in $m$ and $n$. For instance, $\dim_\Bbbk \K_l^1 = \dim_\Bbbk \K_1^l = 4l+1$ and $\dim_\Bbbk \K_l^2 = \dim_\Bbbk \K_2^l = 8l^2 + 14l - 13$, which may be computed using the Cartan matrix  factorisation \cite[(5.17)]{brundanstroppel2}.

\subsection{Generators and relations}
\label{subsection:genrel}

In order to calculate the Hochschild cohomology of $\K_m^n$, our first step is to translate the diagrammatics of $\K_m^n$ into generators and relations, i.e.\ we shall write $\K_m^n$ as the path algebra of a finite quiver $\Q_m^n$ modulo a two-sided ideal generated by three types of quadratic relations. This explicit algebraic description may be of independent interest.

Algorithmic descriptions of quivers with relations describing principal blocks of parabolic category $\mathcal O$ for any parabolic subalgebra $\mathfrak p \subset \mathfrak g$ in a semisimple complex Lie algebra were considered by Vybornov \cite{vybornov}, but they are generally difficult to make explicit. An explicit description for $\K_m^n$ in terms of generators and relations was first given from the point of view of perverse sheaves by Braden \cite{braden} (see also \cite[\S 5.6]{stroppel1} for a comparison with the diagrammatic description), but the relations in \cite{braden} are not homogeneous with respect to the natural grading of $\K_m^n$.

Our description is based on the diagrammatic description given by Brundan and Stroppel \cite{brundanstroppel1} which can be directly related to describing parabolic category $\mathcal O$ for parabolics associated to two-block partitions $\mathfrak{gl}_m (\mathbb C) \oplus \mathfrak{gl}_n (\mathbb C) \subset \mathfrak{gl}_{m+n} (\mathbb C)$ \cite{brundanstroppel3}.

It follows from \cite{beilinsonginzburgsoergel} and \cite[\S 5]{brundanstroppel2} that $\K_m^n$ is a Koszul algebra generated in degree $1$, so that we may take $\Q_m^n$ to be the quiver with vertices corresponding to arc diagrams of degree $0$ and arrows corresponding to arc diagrams of degree $1$. The vertices of $\Q_m^n$ are thus given by the degree $0$ arc diagrams $e_\lambda = \cupd \lambda \lambda \capd \lambda$ (see Definition \ref{definition:arcdiagramdegreezero}) and there is a natural surjective algebra map 
\begin{align}\label{align:rhomap}
\rho \colon \Bbbk \Q_m^n \tikztwoheadrightarrow \K_m^n
\end{align}
which sends the vertices $e_\lambda$ to the degree $0$ arc diagrams $e_\lambda$ and sends arrows to the corresponding degree $1$ arc diagrams. 

We first describe $\Q_m^n$ as the double quiver of a certain bipartite graph $\Gamma_m^n$.

\begin{definition}
Let $m, n \geq 1$ be fixed and let $\Gamma_m^n$ be the (simple and undirected) graph defined as follows:
\begin{itemize}
\item {\bf Vertices}\; The vertices $\mathrm V (\Gamma_m^n) = \{ e_\lambda \}_{\lambda \in \Lambda_m^n}$ are given by the degree $0$ arc diagrams.
\item {\bf Edges}\; There is an edge between two vertices $e_\lambda$ and $e_\mu$ if the weight of the one is obtained from the weight of the other by exchanging a $\dn \dotsb \up$ pair lying in the same counterclockwise circle.
\end{itemize}
\end{definition}

\begin{lemma}\label{lemma:bipartitegraph}
The quiver $\Q_m^n$ for $\K_m^n$ is the double quiver of $\Gamma_m^n$, i.e.\ obtained by replacing each edge in $\Gamma_m^n$ by two arrows in opposite directions. 
\end{lemma}

\begin{proof}
Clearly, $\Q_m^n$ has the same vertices as $\Gamma_m^n$. If there is an edge between $e_\lambda$ and $e_\mu$ in $\Gamma_m^n$ so that $e_\mu$ is obtained from $e_\lambda$ by exchanging a $\dn \dotsb \up$ pair lying in a counterclockwise circle of $e_\lambda$, then $\cupd \lambda \mu$ (resp.\ $ \mu \capd \lambda$) is a well-defined degree $1$ oriented cup diagram (resp.\ cap diagram). As a result, $\cupd \lambda \mu \capd \mu$ (resp.\ $\cupd \mu \mu \capd \lambda$) is a degree $1$ arc diagram from $\lambda$ to $\mu$ (resp.\ from $\mu$ to $\lambda$).

Conversely, if $\cupd \lambda \mu \capd \mu$ is a degree $1$ arc diagram from $\lambda$ to $\mu$ then $\cupd \lambda \mu$ is a degree $1$ oriented cup diagram. This means that $\cupd \lambda \mu$ is obtained from  $\cupd \lambda \lambda$ by exchanging the $\dn \dotsb \up$ pair lying in a counterclockwise cup of $\cupd \lambda \lambda$. In other words, $\mu$ can be obtained from $\lambda$ by exchanging the $\dn \dotsb \up$ pair lying in the counterclockwise cup. So there is an edge between $e_\lambda$ and $e_\mu$ in $\Gamma_m^n$. 
\end{proof}

In order to draw $\Gamma_m^n$, we shall order the vertices by the {\it Bruhat order} (see \cite[\S 2]{brundanstroppel1}), which is a partial order on $\Lambda_m^n \simeq \mathrm V (\Gamma_m^n)$, such that $\up \dotsb \up \, \dn \dotsb \dn$ is the highest weight and $\dn \dotsb \dn \, \up \dotsb \up$ the lowest weight. Concretely, we may define the {\it height} $\lvert \lambda \rvert$ of a weight $\lambda$ by summing the number of $\up$'s on the left of each $\dn$ in $\lambda$, for example
\[
\lvert \up \, \up \, \up \, \underset{_3}{\dn} \, \underset{_3}{\dn} \rvert = 3 + 3 = 6 \qquad \lvert \up \, \underset{_1}{\dn} \, \up \, \underset{_2}{\dn} \, \up \rvert = 1 + 2 = 3.
\]
Permuting $\up$'s to the left and $\dn$'s to the right increases the height and the Bruhat order. Fig.~\ref{figure:graphs} shows several graphs, where in each diagram the vertices aligned horizontally have the same height.

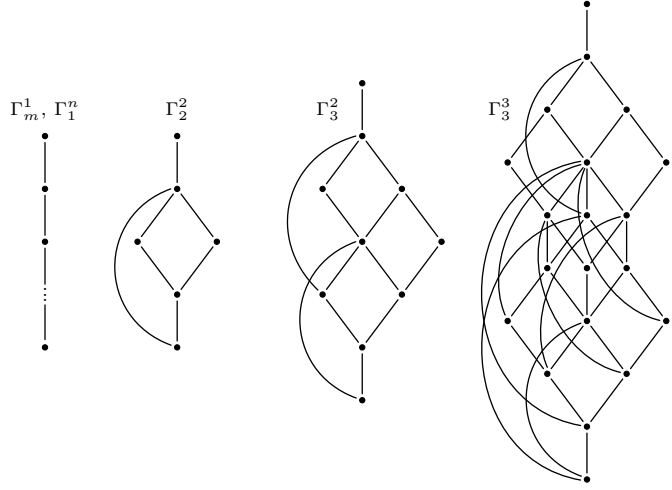
\begin{figure}
\begin{tikzpicture}[x=.5em,y=.5em]
\begin{scope}[shift={(0,0)}]
\node[font=\scriptsize] at (0,10) {$\Gamma_m^1$, $\Gamma_1^n$};
\node[shape=circle,scale=.5] (1-1) at (0,8) {};
\draw[fill=black] (0,8) circle(.1em);
\node[shape=circle,scale=.5] (1-2) at (0,4) {};
\draw[fill=black] (0,4) circle(.1em);
\node[shape=circle,scale=.5] (1-3) at (0,0) {};
\draw[fill=black] (0,0) circle(.1em);
\node[shape=circle,scale=.8] (1-4) at (0,-4) {};
\node[font=\scriptsize] at (0,-3.6) {.};
\node[font=\scriptsize] at (0,-4) {.};
\node[font=\scriptsize] at (0,-4.4) {.};
\node[shape=circle,scale=.6] (1-5) at (0,-8) {};
\draw[fill=black] (0,-8) circle(.1em);
\draw[line width=.5pt,line cap=round] (1-1) -- (1-2) -- (1-3) -- (1-4) -- (1-5);
\end{scope}
\begin{scope}[shift={(10,0)}]
\node[font=\scriptsize] at (0,10) {$\Gamma_2^2$};
\node[shape=circle,scale=.5] (2-1) at (0,8) {};
\draw[fill=black] (0,8) circle(.1em);
\node[shape=circle,scale=.5] (2-2) at (0,4) {};
\draw[fill=black] (0,4) circle(.1em);
\node[shape=circle,scale=.5] (2-3a) at (-3,0) {};
\draw[fill=black] (-3,0) circle(.1em);
\node[shape=circle,scale=.5] (2-3b) at (3,0) {};
\draw[fill=black] (3,0) circle(.1em);
\node[shape=circle,scale=.5] (2-4) at (0,-4) {};
\draw[fill=black] (0,-4) circle(.1em);
\node[shape=circle,scale=.5] (2-5) at (0,-8) {};
\draw[fill=black] (0,-8) circle(.1em);
\draw[line width=.5pt,line cap=round] (2-1) -- (2-2) -- (2-3a) -- (2-4) -- (2-5) (2-2) -- (2-3b) -- (2-4);
\path[line width=.5pt,line cap=round,out=-160,in=160,looseness=1.3] (2-2.180) edge (2-5);
\end{scope}
\begin{scope}[shift={(24,0)}]
\node[left=.5em,font=\scriptsize] at (0,10) {$\Gamma_3^2$};
\node[shape=circle,scale=.5] (3-1) at (0,12) {};
\draw[fill=black] (0,12) circle(.1em);
\node[shape=circle,scale=.5] (3-2) at (0,8) {};
\draw[fill=black] (0,8) circle(.1em);
\node[shape=circle,scale=.5] (3-3a) at (-3,4) {};
\draw[fill=black] (-3,4) circle(.1em);
\node[shape=circle,scale=.5] (3-3b) at (3,4) {};
\draw[fill=black] (3,4) circle(.1em);
\node[shape=circle,scale=.5] (3-4a) at (0,0) {};
\draw[fill=black] (0,0) circle(.1em);
\node[shape=circle,scale=.5] (3-4b) at (6,0) {};
\draw[fill=black] (6,0) circle(.1em);
\node[shape=circle,scale=.5] (3-5a) at (-3,-4) {};
\draw[fill=black] (-3,-4) circle(.1em);
\node[shape=circle,scale=.5] (3-5b) at (3,-4) {};
\draw[fill=black] (3,-4) circle(.1em);
\node[shape=circle,scale=.5] (3-6) at (0,-8) {};
\draw[fill=black] (0,-8) circle(.1em);
\node[shape=circle,scale=.5] (3-7) at (0,-12) {};
\draw[fill=black] (0,-12) circle(.1em);
\draw[line width=.5pt,line cap=round] (3-1) -- (3-2) -- (3-3a) -- (3-4a) -- (3-5a) -- (3-6) -- (3-7) (3-2) -- (3-3b) -- (3-4a) -- (3-5b) -- (3-6) (3-3b) -- (3-4b) -- (3-5b);
\path[line width=.5pt,line cap=round,out=-164,in=136,looseness=1.2] (3-2.180) edge (3-5a);
\path[line width=.5pt,line cap=round,out=-160,in=160,looseness=1.3] (3-4a.180) edge (3-7);
\end{scope}
\begin{scope}[shift={(41,0)}]
\node[left=2.5em,font=\scriptsize] at (0,10) {$\Gamma_3^3$};
\node[shape=circle,scale=.5] (4-1) at (0,18) {};
\draw[fill=black] (0,18) circle(.1em);
\node[shape=circle,scale=.5] (4-2) at (0,14) {};
\draw[fill=black] (0,14) circle(.1em);
\node[shape=circle,scale=.5] (4-3a) at (-3,10) {};
\draw[fill=black] (-3,10) circle(.1em);
\node[shape=circle,scale=.5] (4-3b) at (3,10) {};
\draw[fill=black] (3,10) circle(.1em);
\node[shape=circle,scale=.5] (4-4a) at (-6,6) {};
\draw[fill=black] (-6,6) circle(.1em);
\node[shape=circle,scale=.5] (4-4b) at (0,6) {};
\draw[fill=black] (0,6) circle(.1em);
\node[shape=circle,scale=.5] (4-4c) at (6,6) {};
\draw[fill=black] (6,6) circle(.1em);
\node[shape=circle,scale=.5] (4-5a) at (-3,2) {};
\draw[fill=black] (-3,2) circle(.1em);
\node[shape=circle,scale=.5] (4-5b) at (0,2) {};
\draw[fill=black] (0,2) circle(.1em);
\node[shape=circle,scale=.5] (4-5c) at (3,2) {};
\draw[fill=black] (3,2) circle(.1em);
\node[shape=circle,scale=.5] (4-6a) at (-3,-2) {};
\draw[fill=black] (-3,-2) circle(.1em);
\node[shape=circle,scale=.5] (4-6b) at (0,-2) {};
\draw[fill=black] (0,-2) circle(.1em);
\node[shape=circle,scale=.5] (4-6c) at (3,-2) {};
\draw[fill=black] (3,-2) circle(.1em);
\node[shape=circle,scale=.5] (4-7a) at (-6,-6) {};
\draw[fill=black] (-6,-6) circle(.1em);
\node[shape=circle,scale=.5] (4-7b) at (0,-6) {};
\draw[fill=black] (0,-6) circle(.1em);
\node[shape=circle,scale=.5] (4-7c) at (6,-6) {};
\draw[fill=black] (6,-6) circle(.1em);
\node[shape=circle,scale=.5] (4-8a) at (-3,-10) {};
\draw[fill=black] (-3,-10) circle(.1em);
\node[shape=circle,scale=.5] (4-8b) at (3,-10) {};
\draw[fill=black] (3,-10) circle(.1em);
\node[shape=circle,scale=.5] (4-9) at (0,-14) {};
\draw[fill=black] (0,-14) circle(.1em);
\node[shape=circle,scale=.5] (4-10) at (0,-18) {};
\draw[fill=black] (0,-18) circle(.1em);
\draw[line width=.5pt,line cap=round] (4-1) -- (4-2) -- (4-3a) -- (4-4a) -- (4-5a) -- (4-6a) (4-2) -- (4-3b) -- (4-4b) -- (4-5b) -- (4-6a) (4-3b) -- (4-4c) -- (4-5c) -- (4-6b) (4-5c) -- (4-6c) (4-3a) -- (4-4b) -- (4-5a) -- (4-6b) (4-4b) -- (4-5c) (4-5b) -- (4-6c) (4-6a) -- (4-7a) -- (4-8a) -- (4-9) -- (4-10) (4-6a) -- (4-7b) -- (4-8a) (4-6b) -- (4-7b) (4-8b) -- (4-7b) -- (4-6c) -- (4-7c) -- (4-8b) -- (4-9);
\path[line width=.5pt,line cap=round,out=-160,in=160,looseness=1.3] (4-2) edge (4-5b);
\path[line width=.5pt,line cap=round,out=-167,in=113,looseness=0.9] (4-4b.200) edge (4-7a);
\path[line width=.5pt,line cap=round,out=-113,in=167,looseness=0.9] (4-4b.250) edge (4-7c.175);
\path[line width=.5pt,line cap=round,out=-167,in=113,looseness=0.9] (4-5c) edge (4-8a.90);
\path[line width=.5pt,line cap=round,out=-113,in=167,looseness=0.9] (4-5a) edge (4-8b);
\path[line width=.5pt,line cap=round,out=-160,in=160,looseness=1.3] (4-7b) edge (4-10.150);
\path[line width=.5pt,line cap=round,out=-170,in=170,looseness=1.3] (4-4b.170) edge (4-9);
\path[line width=.5pt,line cap=round,out=-170,in=170,looseness=1.3] (4-5b) edge (4-10.185);
\end{scope}
\end{tikzpicture}
\caption{Examples of $\Gamma_m^n$ whose double quivers $\Q_m^n$ generate $\K_m^n$} 
\label{figure:graphs}
\end{figure}

The edges of $\Gamma_m^n$ only connect vertices corresponding to weights with height of different parity. This can be seen as follows. Assume that there is an edge between $\lambda$ and $\mu$ so that $\mu$ is obtained from $\lambda$ by exchanging a $\dn \dotsb \up$ pair lying in some circle $C$ of the arc diagram $e_\lambda$. Denote by $k$ the number of circles inside $C$ in $e_\lambda$. Then we have $k$ $\up$'s and $k$ $\dn$'s lying between the $\dn \dotsb \up$ pair being exchanged, whence $\lvert \mu \rvert = \lvert \lambda \rvert + 2k + 1$. This observation immediately yields the following result.

\begin{proposition}
\label{proposition:bipartite}
The graph $\Gamma_m^n$ is bipartite when decomposing $\mathrm V (\Gamma_m^n)$ into vertices of even and odd heights. In particular, paths in $\Gamma_m^n$ of even length can only connect weights whose heights are of the same parity and paths of odd length can only connect weights whose heights are of different parity.
\end{proposition}

Since $\K_m^n$ is Koszul and hence quadratic, we also need a good understanding of the length $2$ paths  in $\Gamma_m^n$ and we give a complete description in Proposition \ref{proposition:edgedescriptions}. The description is based on the following two lemmas.

\begin{lemma}
\label{lemma:paths1}
Let $m, n \geq 1$ and let $\lambda, \mu \in \Lambda_m^n$ be two distinct weights such that $\lvert \lambda \rvert \leq \lvert \mu \rvert$.
\begin{enumerate}
\item \label{paths1} There is no length $2$ (undirected) path in $\Gamma_m^n$ between $\lambda$ and $\mu$ passing through a vertex $\nu$ with $\lvert \nu \rvert$ equal to $\lvert \lambda \rvert$ or $\lvert \mu \rvert$.
\item \label{paths2} The set of length $2$ paths in $\Gamma_m^n$ passing through vertices $\nu$ satisfying $\lvert \nu \rvert < \lvert \lambda \rvert$ contains at most one element.
\item \label{paths3} The set of length $2$ paths in $\Gamma_m^n$ passing through vertices $\nu$ satisfying $\lvert \nu \rvert > \lvert \mu \rvert$ contains at most one element.
\end{enumerate}
\end{lemma}

\begin{proof}
\ref{paths1} follows directly from Proposition \ref{proposition:bipartite}, since $\nu$ is connected to $\lambda$ and $\mu$ by an edge whence  $\lvert \nu \rvert$ is different from $\lvert \lambda \rvert$ and $ \lvert \mu \rvert  \mod 2$.

To prove \ref{paths2}, let $\nu, \nu'$ be two distinct weights both connected to $\lambda$ and to $\mu$ by an edge with $\lvert \nu \rvert, \lvert \nu' \rvert < \lvert \lambda \rvert \leq \lvert \mu \rvert$. Since $\lambda$ and $\mu$ are distinct, they are obtained from $\nu$ (resp.\ $\nu'$) by exchanging two $\dn \dotsb \up$ pairs lying in distinct  circles of $e_\nu$ (resp.\ $e_{\nu'}$). These circles are either nested or not, so that $e_\nu, e_{\nu'}$ are of the following forms
\begin{equation}
\label{possibilities}
\begin{tikzpicture}[baseline=-.5ex,x=.6em,y=.6em,decoration={markings,mark=at position 0.99 with {\arrow[black]{Stealth[length=3.8pt]}}}]
\node at (4,0) {or};
\begin{scope}[shift={(-8,0)}]
\DDOTS{0} \DDOTS{2} \DDOTS{4} \DDOTS{6} \DDOTS{8}
\DN{1} \DN{3} \UP{5} \UP{7}
\draw[line width=.5pt,line cap=round] (1,-.15) -- ++(0,.3) arc[start angle=180, end angle=0, x radius=3, y radius=.75*3] -- ++(0,-.3) arc[start angle=0, end angle=-180, x radius=3, y radius=0.75*3];
\draw[line width=.5pt,line cap=round] (3,-.15) -- ++(0,.3) arc[start angle=180, end angle=0, radius=1] -- ++(0,-.3) arc[start angle=0, end angle=-180, radius=1];
\end{scope}
\begin{scope}[shift={(8,0)}]
\DDOTS{0} \DDOTS{2} \DDOTS{4} \DDOTS{6} \DDOTS{8}
\DN{1} \UP{3} \DN{5} \UP{7}
\draw[line width=.5pt,line cap=round] (1,-.15) -- ++(0,.3) arc[start angle=180, end angle=0, radius=1] -- ++(0,-.3) arc[start angle=0, end angle=-180, radius=1];
\draw[line width=.5pt,line cap=round] (5,-.15) -- ++(0,.3) arc[start angle=180, end angle=0, radius=1] -- ++(0,-.3) arc[start angle=0, end angle=-180, radius=1];
\end{scope}
\end{tikzpicture}
\end{equation}
and $\lambda, \mu, \nu, \nu'$ agree on all other positions represented by the dots. Note that  there is an edge between $\nu$ and $\nu'$ since the diagram on the right of \eqref{possibilities} may be obtained from the one on the left by exchanging the $\dn \dotsb \up$ pair in the inner circle, which implies that $\lvert \nu \rvert \not\equiv  \lvert \nu' \rvert \mod 2$. Since $\nu$ and $\nu'$ are connected to $\lambda$ by an edge we must also have $\lvert \nu \rvert \equiv \lvert \nu' \rvert\mod 2$ by Proposition \ref{proposition:bipartite}, giving a contradiction.

\ref{paths3} follows similarly, for if $\nu, \nu'$ are two distinct weights both connected to $\lambda$ and $\mu$ by an edge with $\lvert \lambda \rvert \leq \lvert \mu \rvert < \lvert \nu \rvert, \lvert \nu' \rvert$, then $\nu, \nu'$ are obtained from $\lambda$ by exchanging two $\dn \dotsb \up$ pairs in distinct circles of $e_\nu$, and $\nu, \nu'$ can also be obtained from $\mu$ in the same way. Thus $\lambda, \mu$ are both of the form \eqref{possibilities} agreeing on all positions represented by the dots. We get a contradiction since $\lvert \lambda \rvert \equiv \lvert \mu \rvert \mod 2$ and yet, the two arc diagrams in \eqref{possibilities} share an edge, whence their heights are different modulo $2$.
\end{proof}

\begin{lemma}\label{lemma:squarespaths}
Let $m, n \geq 2$. There are no other paths of length $2$ in $\Gamma_m^n$ (except the illustrated ones) between any two distinct vertices of the square in Fig.~\ref{figure:squares} (a) or in Fig.~\ref{figure:squares} (c). Moreover, any square in $\Gamma_m^n$ is of one of the forms appearing in Fig.~\ref{figure:squares}.
\end{lemma}

\begin{remark}
Note that the top square in Fig.~\ref{figure:squares} (c) is a special case of the square in Fig.~\ref{figure:squares} (b). Namely, if the middle
\begin{tikzpicture}[baseline=-.5ex]
\DDOTS{0}
\end{tikzpicture}
in the arc diagrams in Fig.~\ref{figure:squares} (b) is empty or consists of nested counterclockwise circles, then Fig.~\ref{figure:squares} (b) can be extended to Fig.~\ref{figure:squares} (c). So by a similar argument, the statement in Lemma \ref{lemma:squarespaths} holds for the square (b) if it cannot be extended to Fig.~\ref{figure:squares} (c). 

\begin{figure}
\begin{tikzpicture}[baseline=-.5ex,x=.6em,y=.6em]
\node at (4,8) {(a)};
\begin{scope}[decoration={markings,mark=at position 0.82 with {\arrow[black]{Stealth[length=2.4pt]}}}]
\DN{2} \UP{6}
\begin{scope}[shift={(-8,8)}]
\DN{2} \UP{6}
\end{scope}
\begin{scope}[shift={(8,8)}]
\DN{2} \UP{6}
\end{scope}
\begin{scope}[shift={(0,16)}]
\DN{2} \UP{6}
\end{scope}
\end{scope}
\begin{scope}[shift={(4,0)}]
\draw[line width=.5pt,line cap=round] (-2.6,2.6) -- (-5.4,5.4);
\draw[line width=.5pt,line cap=round] (2.6,2.6) -- (6,6);
\begin{scope}[shift={(0,16)}]
\draw[line width=.5pt,line cap=round] (-1.5,-1.5) -- (-5.4,-5.4);
\draw[line width=.5pt,line cap=round] (1.5,-1.5) -- (6,-6);
\end{scope}
\end{scope}
\begin{scope}[decoration={markings,mark=at position 1 with {\arrow[black]{Stealth[length=3.8pt]}}}]
\DDOTS{0} \DDOTS{2} \DDOTS{4} \DDOTS{6} \DDOTS{8}
\DN{1} \DN{3} \UP{5} \UP{7}
\CIRCLES{1}{3} \ROUNDCIRCLE{3}{1}
\draw[dash pattern=on 2.1pt off 1.53pt, line width=.35pt,line cap=round] (2,-.21) -- (2,.21) arc[start angle=180, end angle=0, x radius=2, y radius=.8*2] -- ++(0,-.42) arc[start angle=0, end angle=-180, x radius=2, y radius =.8*2];
\end{scope}
\begin{scope}[shift={(-8,8)},decoration={markings,mark=at position 0.99 with {\arrow[black]{Stealth[length=3.8pt]}}}]
\DDOTS{0} \DDOTS{2} \DDOTS{4} \DDOTS{6} \DDOTS{8}
\DN{1} \UP{3} \DN{5} \UP{7}
\CIRCLES{1}{3} \CIRCLE{2} \CIRCLE{5}
\end{scope}
\begin{scope}[shift={(8,8)},decoration={markings,mark=at position 0.99 with {\arrow[black]{Stealth[length=3.8pt]}}}]
\DDOTS{0} \DDOTS{2} \DDOTS{4} \DDOTS{6} \DDOTS{8}
\UP{1} \DN{3} \UP{5} \DN{7}
\LARC{1} \ROUNDCIRCLE{3}{1} \RARC{7}
\draw[dash pattern=on 2.1pt off 1.53pt, line width=.35pt,line cap=round] (2,-.21) -- (2,.21) arc[start angle=180, end angle=0, x radius=2, y radius=.8*2] -- ++(0,-.42) arc[start angle=0, end angle=-180, x radius=2, y radius=.8*2];
\end{scope}
\begin{scope}[shift={(0,16)},decoration={markings,mark=at position 0.99 with {\arrow[black]{Stealth[length=3.8pt]}}}]
\DDOTS{0} \DDOTS{2} \DDOTS{4} \DDOTS{6} \DDOTS{8}
\UP{1} \UP{3} \DN{5} \DN{7}
\LARC{1} \CIRCLE{2} \CIRCLE{5} \RARC{7}
\end{scope}
\begin{scope}[shift={(31,0)}] 
\node at (4,8) {(b)};
\begin{scope}[shift={(4,0)}]
\draw[line width=.5pt,line cap=round] (-1.9,1.9) -- (-6.1,6.1);
\draw[line width=.5pt,line cap=round] (1.9,1.9) -- (6.1,6.1);
\begin{scope}[shift={(0,16)}]
\draw[line width=.5pt,line cap=round] (-1.5,-1.5) -- (-6.1,-6.1);
\draw[line width=.5pt,line cap=round] (1.5,-1.5) -- (6.1,-6.1);
\end{scope}
\end{scope}
\begin{scope}[decoration={markings,mark=at position 0.99 with {\arrow[black]{Stealth[length=3.8pt]}}}]
\DDOTS{0} \DDOTS{2} \DDOTS{4} \DDOTS{6} \DDOTS{8}
\DN{1} \UP{3} \DN{5} \UP{7}
\ROUNDCIRCLE{1}{1} \ROUNDCIRCLE{5}{1}
\end{scope}
\begin{scope}[shift={(8,8)},decoration={markings,mark=at position 0.99 with {\arrow[black]{Stealth[length=3.8pt]}}}]
\DDOTS{0} \DDOTS{2} \DDOTS{4} \DDOTS{6} \DDOTS{8}
\UP{1} \DN{3} \DN{5} \UP{7}
\LARC{1} \RARC{3} \ROUNDCIRCLE{5}{1}
\end{scope}
\begin{scope}[shift={(-8,8)},decoration={markings,mark=at position 0.99 with {\arrow[black]{Stealth[length=3.8pt]}}}]
\DDOTS{0} \DDOTS{2} \DDOTS{4} \DDOTS{6} \DDOTS{8}
\DN{1} \UP{3} \UP{5} \DN{7}
\ROUNDCIRCLE{1}{1} \LARC{5} \RARC{7}
\end{scope}
\begin{scope}[shift={(0,16)},decoration={markings,mark=at position 0.99 with {\arrow[black]{Stealth[length=3.8pt]}}}]
\DDOTS{-.3} \DDOTS{1.7} \DDOTS{4} \DDOTS{6.3} \DDOTS{8.3}
\UP{.7} \DN{2.7} \UP{5.3} \DN{7.3}
\LARC{.7} \RRARC{2.7} \LLARC{5.3} \RARC{7.3}
\end{scope}
\end{scope} 
\begin{scope}[shift={(15.5,-19.5)}] 
\node at (4,20) {(c)};
\begin{scope}[shift={(4,0)}]
\draw[line width=.5pt,line cap=round] (-1.9,1.9) -- (-6.1,6.1);
\draw[line width=.5pt,line cap=round] (1.9,1.9) -- (6.1,6.1);
\draw[line width=.5pt,line cap=round] (0,-1.5) -- (0,-4);
\begin{scope}[shift={(0,16)}]
\draw[line width=.5pt,line cap=round] (-1.5,-1.5) -- (-6.1,-6.1);
\draw[line width=.5pt,line cap=round] (1.5,-1.5) -- (6.1,-6.1);
\draw[line width=.5pt,line cap=round] (-6,0) arc[start angle=105, end angle=255, radius=7.2em];
\end{scope}
\end{scope}
\begin{scope}[shift={(0,-7)}, decoration={markings,mark=at position 0.99 with {\arrow[black]{Stealth[length=3.8pt]}}}]
\DDOTS{0} \DDOTS{2} \DDOTS{4} \DDOTS{6} \DDOTS{8}
\DN{1} \DN{3} \UP{5} \UP{7}
\CIRCLES{1}{3} \ROUNDCIRCLE{3}{1}
\end{scope}
\begin{scope}[decoration={markings,mark=at position 0.99 with {\arrow[black]{Stealth[length=3.8pt]}}}]
\DDOTS{0} \DDOTS{2} \DDOTS{4} \DDOTS{6} \DDOTS{8}
\DN{1} \UP{3} \DN{5} \UP{7}
\ROUNDCIRCLE{1}{1} \ROUNDCIRCLE{5}{1}
\end{scope}
\begin{scope}[shift={(8,8)}, decoration={markings,mark=at position 0.99 with {\arrow[black]{Stealth[length=3.8pt]}}}]
\DDOTS{0} \DDOTS{2} \DDOTS{4} \DDOTS{6} \DDOTS{8}
\UP{1} \DN{3} \DN{5} \UP{7}
\LARC{1} \RARC{3} \ROUNDCIRCLE{5}{1}
\end{scope}
\begin{scope}[shift={(-8,8)}, decoration={markings,mark=at position 0.99 with {\arrow[black]{Stealth[length=3.8pt]}}}]
\DDOTS{0} \DDOTS{2} \DDOTS{4} \DDOTS{6} \DDOTS{8}
\DN{1} \UP{3} \UP{5} \DN{7}
\ROUNDCIRCLE{1}{1} \LARC{5} \RARC{7}
\end{scope}
\begin{scope}[shift={(0,16)}, decoration={markings,mark=at position 0.99 with {\arrow[black]{Stealth[length=3.8pt]}}}]
\DDOTS{0} \DDOTS{2} \DDOTS{4} \DDOTS{6} \DDOTS{8}
\UP{1} \DN{3} \UP{5} \DN{7}
\LARC{1} \ROUNDCIRCLE{3}{1} \RARC{7}
\end{scope}
\end{scope}
\end{tikzpicture}
\caption{Squares in $\Gamma_m^n$ giving rise to commutative squares in $\K_m^n$. In (a) the dashed circle indicates the innermost circle enclosing the smaller solid circle, which must exist and be different from the bigger solid circle for the square to exist (cf.\ Fig.~\ref{figure:single} (c) for when such a circle does not exist)}
\label{figure:squares}
\end{figure}
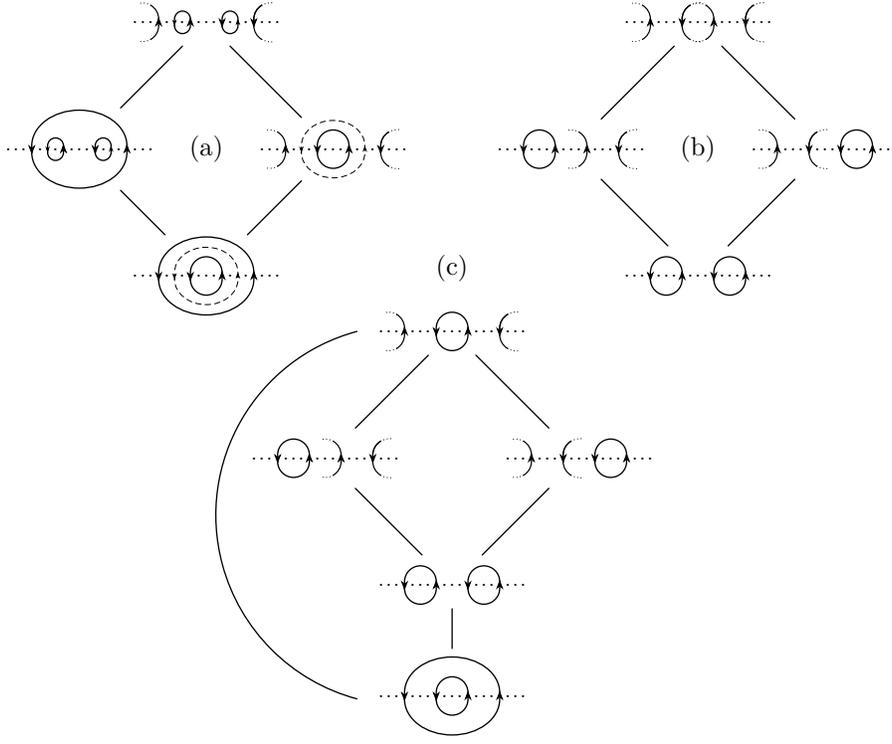

\begin{figure}
\begin{tikzpicture}[baseline=-.5ex,x=.6em,y=.6em,decoration={markings,mark=at position 0.99 with {\arrow[black]{Stealth[length=3.8pt]}}}]
\begin{scope}[shift={(-10,0)}] 
\begin{scope} 
\node at (3,3.5) {(a)};
\draw[line width=.5pt,line cap=round] (3,-2.5) -- (3,-5.5);
\DDOTS{0} \DDOTS{2} \DDOTS{4} \DDOTS{6}
\UP{1} \UP{3} \DN{5}
\RAYS{1}{1.3} \RAYS{3}{1.3} \RARC{5}
\end{scope}
\begin{scope}[shift={(0,-8)}] 
\draw[line width=.5pt,line cap=round] (3,-2.5) -- (3,-5.5);
\DDOTS{0} \DDOTS{2} \DDOTS{4} \DDOTS{6}
\UP{1} \DN{3} \UP{5}
\RAYS{1}{1.3} \ROUNDCIRCLE{3}{1}
\end{scope}
\begin{scope}[shift={(0,-16)}] 
\DDOTS{0} \DDOTS{2} \DDOTS{4} \DDOTS{6}
\DN{1} \UP{3} \UP{5}
\ROUNDCIRCLE{1}{1} \RAYS{5}{1.3}
\end{scope}
\end{scope}
\begin{scope}[shift={(3,0)}] 
\begin{scope} 
\node at (3,3.5) {(b)};
\draw[line width=.5pt,line cap=round] (3,-2.5) -- (3,-5.5);
\DDOTS{0} \DDOTS{2} \DDOTS{4} \DDOTS{6}
\UP{1} \DN{3} \DN{5}
\LARC{1} \RAYS{3}{1.3} \RAYS{5}{1.3}
\end{scope}
\begin{scope}[shift={(0,-8)}] 
\draw[line width=.5pt,line cap=round] (3,-2.5) -- (3,-5.5);
\DDOTS{0} \DDOTS{2} \DDOTS{4} \DDOTS{6}
\DN{1} \UP{3} \DN{5}
\ROUNDCIRCLE{1}{1} \RAYS{5}{1.3}
\end{scope}
\begin{scope}[shift={(0,-16)}] 
\DDOTS{0} \DDOTS{2} \DDOTS{4} \DDOTS{6}
\DN{1} \DN{3} \UP{5}
\RAYS{1}{1.3} \ROUNDCIRCLE{3}{1}
\end{scope}
\end{scope}
\begin{scope}[shift={(16,0)}] 
\begin{scope} 
\node at (4,3.5) {(c)};
\draw[line width=.5pt,line cap=round] (4,-2.5) -- (4,-5.5);
\DDOTS{0} \DDOTS{2} \DDOTS{4} \DDOTS{6} \DDOTS{8}
\UP{1} \UP{3} \DN{5} \DN{7}
\LARC{1} \LARC{3} \RARC{5} \RARC{7}
\end{scope}
\begin{scope}[shift={(0,-8)}] 
\draw[line width=.5pt,line cap=round] (4,-2.5) -- (4,-5);
\DDOTS{0} \DDOTS{2} \DDOTS{4} \DDOTS{6} \DDOTS{8}
\UP{1} \DN{3} \UP{5} \DN{7}
\LARC{1} \ROUNDCIRCLE{3}{1} \RARC{7}
\end{scope}
\begin{scope}[shift={(0,-16)}] 
\DDOTS{0} \DDOTS{2} \DDOTS{4} \DDOTS{6} \DDOTS{8}
\DN{1} \DN{3} \UP{5} \UP{7}
\CIRCLES{1}{3} \ROUNDCIRCLE{3}{1}
\end{scope}
\end{scope}
\begin{scope}[shift={(0,-38)}] 
\begin{scope}[shift={(6,0)}]
\begin{scope}[shift={(0,16)}]
\node at (0,-5.5) {(d)};
\draw[line width=.5pt,line cap=round] (-1.9,-1.9) -- (-7.5,-7.5);
\draw[line width=.5pt,line cap=round] (1.9,-1.9) -- (7.5,-7.5);
\end{scope}
\end{scope}
\begin{scope}[shift={(10,6)}]
\DDOTS{0} \DDOTS{2} \DDOTS{4} \DDOTS{6} \DDOTS{8} \DDOTS{10} \DDOTS{12}
\UP{1} \DN{3} \DN{5} \DN{7} \UP{9} \UP{11}
\LARC{1} \RARC{3} \CIRCLES{5}{3} \ROUNDCIRCLE{7}{1}
\end{scope}
\begin{scope}[shift={(-10,6)}]
\DDOTS{0} \DDOTS{2} \DDOTS{4} \DDOTS{6} \DDOTS{8} \DDOTS{10} \DDOTS{12}
\DN{1} \DN{3} \UP{5} \UP{7} \UP{9} \DN{11}
\CIRCLES{1}{3} \ROUNDCIRCLE{3}{1} \LARC{9} \RARC{11}
\end{scope}
\begin{scope}[shift={(0,16)}]
\DDOTS{0} \DDOTS{2} \DDOTS{4} \DDOTS{6} \DDOTS{8} \DDOTS{10} \DDOTS{12}
\UP{1} \DN{3} \UP{5} \DN{7} \UP{9} \DN{11}
\LARC{1} \ROUNDCIRCLE{3}{1} \ROUNDCIRCLE{7}{1} \RARC{11}
\end{scope}
\end{scope}
\end{tikzpicture}
\caption{Paths in $\Gamma_m^n$ of length $2$ without parallel length $2$ paths. In (c) there is no circle between the two nested circles enclosing the smaller one (cf.\ Fig.~\ref{figure:squares} (a) for the case when such a circle does exist)}
\label{figure:single}
\end{figure}
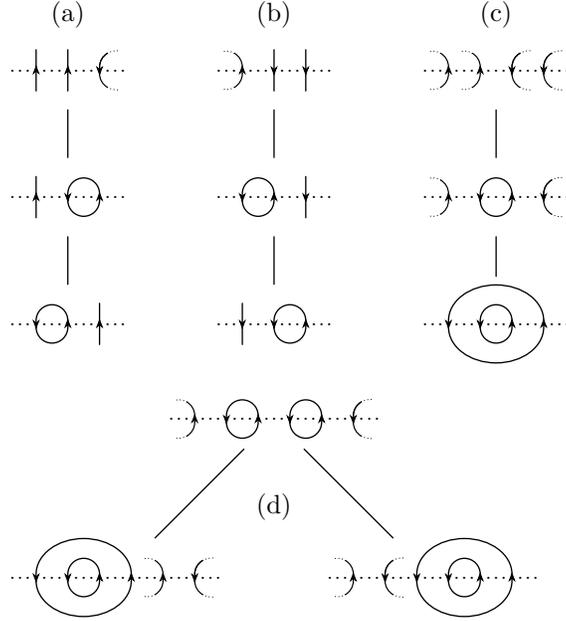

Note that in Fig.~\ref{figure:squares} (c) there are three squares, the obvious square at the top, and two further squares obtained by deleting the left or right vertex (see Remark \ref{remark:commutativity} for illustrations).
\end{remark}

\begin{proof}[Proof of Lemma \ref{lemma:squarespaths}]
Let us first consider square (a). Denote the bottom, top, left and right vertices by $\lambda, \mu, \nu$ and $\nu'$,  respectively, and note that by Lemma \ref{lemma:paths1} \ref{paths1} $\lvert \lambda \rvert < \lvert \nu \rvert < \lvert \nu' \rvert < \lvert \mu \rvert$. If there is a path of length $2$ between $\nu$ and $\nu'$ passing through a vertex $\xi$ which is different from $\lambda$ and $\mu$, then by Lemma \ref{lemma:paths1} \ref{paths2} and \ref{paths3} we have $\lvert \nu \rvert  < \lvert \xi \rvert < \lvert \nu' \rvert$. Note that $\nu$ and $\nu'$ are distinct in four positions and agree on all other positions. Since $\xi$ is connected to $\nu$ by an edge, we have that $\xi$ and $\nu$ are distinct in two positions and thus $\xi$ must be obtained from $\nu$ by exchanging a $\dn \dotsb \up$ pair which lies in two of the four positions and also in a circle of $e_{\nu}$ (i.e.\ the outer one). This implies $\xi = \mu$, giving a contradiction.  

If there is a path of length $2$ between $\lambda$ and $\mu$ passing through a vertex $\xi$ which is different from $\nu$ and $\nu'$, then we have the following cases.

\noindent {\it Case $\lvert \xi \rvert > \lvert \mu \rvert$.}\; Since $\lambda$ and $\mu$ are distinct in four positions and agree on all other positions represented by the dots, we have that $\xi$ is obtained from $\lambda$ by exchanging a $\dn \dotsb \up$ pair which lies in two of the four positions and also in a circle of $e_{\lambda}$ (resp. $e_\mu$). Then $\xi$ has to equal $\nu$ or $\nu'$ which contradicts the assumption.

\noindent {\it Case $ \lvert \xi \rvert < \lvert \lambda \rvert$.}\; We have that $\lambda$ is obtained from $\xi$ by exchanging a $\dn \dotsb \up$ pair which lies in two of the four positions. But this cannot happen since in the four positions $\lambda$ is of the form $\dotsb \dn \dotsb \dn \dotsb \up \dotsb \up \dotsb$ and thus is already as low as possible.

\noindent {\it Case $\lvert \lambda \rvert  < \lvert \xi \rvert < \lvert \mu \rvert$.}\; By a similar reason we have that $\xi$ is obtained from $\lambda$ by exchanging a $\dn \dotsb \up$ pair which lies in two of the four positions and also lies in a circle of $e_\lambda$. So $\xi$ has to equal $\nu$ or $\nu'$.

Now let us consider square (c). Denote the vertices from bottom to top by $\kappa, \lambda, \nu, \nu'$ and $\mu$. Note that $\lambda$ (resp.\ $\nu$) and $\mu$ (resp.\ $\nu'$) are distinct in four positions and agree on all other positions.  Then by similar  arguments as before, we have that there are no other paths of length $2$ between $\lambda$ (resp.\ $\nu$) and $\mu$ (resp.\ $\nu'$). It thus suffices to prove the assertion for $\kappa$ and $\nu$ (resp.\ $\nu'$). If there is a path of length $2$ between $\kappa $ and $\nu$ passing through a vertex $\xi$ which is different from $\lambda$ and $\mu$, then by Lemma \ref{lemma:paths1} \ref{paths2} we have $\lvert \xi \rvert < \lvert \nu \rvert$.

\noindent {\it Case $\lvert \kappa \rvert <  \lvert \xi \rvert  < \lvert \nu \rvert$.}\; Note that $\kappa$ and $\nu$ are distinct in two positions and agree on all other positions.  Then $\xi$ is obtained from $\kappa$ by exchanging a $\dn \dotsb \up$ pair lying in a circle $C$ of $e_{\kappa}$ such that either the $\dn$ or $\up$ lies in one of the two positions (otherwise, $\xi$ and $
\nu$ are distinct in four positions whence there are no edges between them). Thus, $C$ has to be one of the nested circles in $e_{\kappa}$ so that $\xi = \lambda$.

\noindent {\it Case $\lvert \xi \rvert < \lvert \kappa \rvert$.}\; We have that $\kappa, \nu$ are obtained from $\xi$ by exchanging two $\dn \dotsb \up$ pairs in distinct circles $C, C'$ in $e_\xi$, so $\kappa$ and $\nu$ are distinct in four positions which contradicts the fact that $\kappa$ and $\nu$ are only distinct in two positions. 

To obtain a square in Fig.~\ref{figure:squares}, we may start with a vertex $e_\lambda$ and choose two circles $C_1$ and $C_2$ in $e_\lambda$ which correspond to two edges starting from $e_\lambda$. If $C_1$ and $C_2$ are nested such that there is a circle between them enclosing the smaller one then the two edges may be completed to the square (a), if they are nested such that there is no circle between them then we obtain the square (c), and if they are not nested then we obtain the square (b). Note that any square may be obtained in this way so that it is one of the forms in Fig.~\ref{figure:squares}.
\end{proof}

We have the following complete description of the length $2$ paths in $\Gamma_m^n$.
\begin{proposition}
\label{proposition:edgedescriptions}
Let $\lambda, \mu \in \Lambda^n_m$ be two distinct weights and assume that $\lvert \lambda \rvert \leq \lvert \mu \rvert$. Then there are at most three paths of length $2$ between $\lambda$ and $\mu$. Concretely, if there exists a length $2$ path between $\lambda$ and $\mu$, then we have the following cases
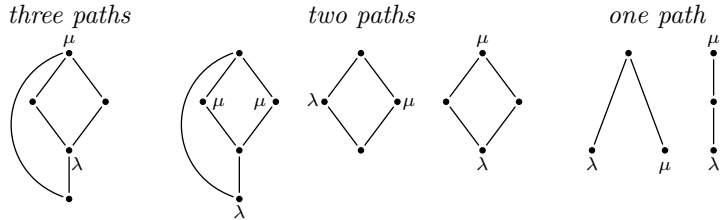
\[
\begin{tikzpicture}[x=.45em,y=.45em]
\node at (-2.7,5) {one path};
\begin{scope}[shift={(2,0)}]
\node[font=\scriptsize] at (0,3) {$\mu$};
\node[font=\scriptsize] at (0,-7.25) {\strut$\lambda$};
\node[shape=circle,scale=.5] (1-1) at (0,2) {};
\draw[fill=black] (0,2) circle(.1em);
\node[shape=circle,scale=.5] (1-2) at (0,-2) {};
\draw[fill=black] (0,-2) circle(.1em);
\node[shape=circle,scale=.5] (1-3) at (0,-6) {};
\draw[fill=black] (0,-6) circle(.1em);
\draw[line width=.5pt,line cap=round] (1-1) -- (1-2) -- (1-3);
\end{scope}
\begin{scope}[shift={(-5,0)}]
\node[font=\scriptsize] at (-3,-7.25) {\strut$\lambda$};
\node[font=\scriptsize] at (3,-7.25) {\strut$\mu$};
\node[shape=circle,scale=.5] (2-1) at (0,2) {};
\draw[fill=black] (0,2) circle(.1em);
\node[shape=circle,scale=.5] (2-2) at (-3,-6) {};
\draw[fill=black] (-3,-6) circle(.1em);
\node[shape=circle,scale=.5] (2-3) at (3,-6) {};
\draw[fill=black] (3,-6) circle(.1em);
\draw[line width=.5pt,line cap=round] (2-2) -- (2-1) -- (2-3);
\end{scope}
\begin{scope}[shift={(-17,0)}]
\node[font=\scriptsize] at (0,3) {$\mu$};
\node[font=\scriptsize] at (0,-7.25) {\strut$\lambda$};
\node[shape=circle,scale=.5] (3-1) at (0,2) {};
\draw[fill=black] (0,2) circle(.1em);
\node[shape=circle,scale=.5] (3-2) at (-3,-2) {};
\draw[fill=black] (-3,-2) circle(.1em);
\node[shape=circle,scale=.5] (3-3) at (3,-2) {};
\draw[fill=black] (3,-2) circle(.1em);
\node[shape=circle,scale=.5] (3-4) at (0,-6) {};
\draw[fill=black] (0,-6) circle(.1em);
\draw[line width=.5pt,line cap=round] (3-1) -- (3-2) -- (3-4) -- (3-3) -- (3-1);
\end{scope}
\begin{scope}[shift={(-27,0)}]
\node at (0,5) {two paths};
\node[font=\scriptsize] at (-3.95,-2) {\strut$\lambda$};
\node[font=\scriptsize] at (3.95,-2) {\strut$\mu$};
\node[shape=circle,scale=.5] (3-1) at (0,2) {};
\draw[fill=black] (0,2) circle(.1em);
\node[shape=circle,scale=.5] (3-2) at (-3,-2) {};
\draw[fill=black] (-3,-2) circle(.1em);
\node[shape=circle,scale=.5] (3-3) at (3,-2) {};
\draw[fill=black] (3,-2) circle(.1em);
\node[shape=circle,scale=.5] (3-4) at (0,-6) {};
\draw[fill=black] (0,-6) circle(.1em);
\draw[line width=.5pt,line cap=round] (3-1) -- (3-2) -- (3-4) -- (3-3) -- (3-1);
\end{scope}

\begin{scope}[shift={(-37,0)}]
\node[font=\scriptsize] at (1.7,-2) {\strut$\mu$};
\node[font=\scriptsize] at (-1.7,-2) {\strut$\mu$};
\node[font=\scriptsize] at (0,-11.25) {\strut$\lambda$};
\node[shape=circle,scale=.5] (3-4a) at (0,2) {};
\draw[fill=black] (0,2) circle(.1em);
\node[shape=circle,scale=.5] (3-5a) at (-3,-2) {};
\draw[fill=black] (-3,-2) circle(.1em);
\node[shape=circle,scale=.5] (3-5b) at (3,-2) {};
\draw[fill=black] (3,-2) circle(.1em);
\node[shape=circle,scale=.5] (3-6) at (0,-6) {};
\draw[fill=black] (0,-6) circle(.1em);
\node[shape=circle,scale=.5] (3-7) at (0,-10) {};
\draw[fill=black] (0,-10) circle(.1em);
\draw[line width=.5pt,line cap=round] (3-4a) -- (3-5a) -- (3-6) -- (3-7) (3-4a) -- (3-5b) -- (3-6);
\path[line width=.5pt,line cap=round,out=-160,in=160,looseness=1.3] (3-4a.180) edge (3-7);
\end{scope}

\begin{scope}[shift={(-51,0)}]
\node at (0,5) {three paths};
\node[font=\scriptsize] at (0,3) {$\mu$};
\node[font=\scriptsize] at (.66,-7.25) {\strut$\lambda$};
\node[shape=circle,scale=.5] (3-4a) at (0,2) {};
\draw[fill=black] (0,2) circle(.1em);
\node[shape=circle,scale=.5] (3-5a) at (-3,-2) {};
\draw[fill=black] (-3,-2) circle(.1em);
\node[shape=circle,scale=.5] (3-5b) at (3,-2) {};
\draw[fill=black] (3,-2) circle(.1em);
\node[shape=circle,scale=.5] (3-6) at (0,-6) {};
\draw[fill=black] (0,-6) circle(.1em);
\node[shape=circle,scale=.5] (3-7) at (0,-10) {};
\draw[fill=black] (0,-10) circle(.1em);
\draw[line width=.5pt,line cap=round] (3-4a) -- (3-5a) -- (3-6) -- (3-7) (3-4a) -- (3-5b) -- (3-6);
\path[line width=.5pt,line cap=round,out=-160,in=160,looseness=1.3] (3-4a.180) edge (3-7);
\end{scope}
\end{tikzpicture}
\]
where
\begin{enumerate}
\item \label{length3} there are exactly three paths if $\lambda$ and $\mu$ are the two indicated vertices in Fig.~\ref{figure:squares} (c)
\item \label{length2} there are exactly two paths if $\lambda$ and $\mu$ are the two indicated vertices of Fig.~\ref{figure:squares} (c) or the ones of Fig.~\ref{figure:squares} (a) or (b) (excluding the case appearing in \ref{length3})
\item \label{length1} there is exactly one path if the vertices and edges are as in Fig.~\ref{figure:single}.
\end{enumerate}
\end{proposition}

\begin{proof}
\ref{length3} and \ref{length2} follow from Lemma \ref{lemma:squarespaths} and \ref{length1} follows by noting that the length $2$ paths in Fig.~\ref{figure:single} are precisely all the possible paths of length $2$ which do not appear in any square. For instance, to see this for Fig.~\ref{figure:single} (d), assume that there is a length $2$ path between $\lambda$ and $\mu$ passing through a vertex $\nu$ such that $\lvert \nu \rvert > \lvert \mu \rvert \geq \lvert \lambda \rvert$. Then $\nu$ is obtained from $\lambda$ (resp.\ $\mu$) by exchanging a $\dn \dotsb \up$ pair lying in a circle of $e_{\lambda}$ (resp.\ $e_{\mu}$). Note that $\lambda, \mu$ can differ in at most four positions. Since $\lambda \neq \mu$ and $\lvert \lambda \rvert \leq \lvert \mu \rvert$ we have the following three cases.

In the first case, $\lambda$ and $\mu$ are distinct in two positions and agree on all other positions. We claim this can only happen when $\lambda, \mu, \nu$ are vertices in Fig.~\ref{figure:squares} (c) with $\lambda$ the bottom vertex, $\nu$ the top vertex and $\mu$ either the left or the right vertex. Firstly, note that the two positions in which $\lambda, \mu$ differ cannot lie in a single circle of $e_\lambda$, as otherwise $\mu$ would be obtained by changing the $\dn \dotsb \up$ pair lying in this circle, contradicting the fact that $\lambda$ and $\mu$ have the same height modulo $2$. Since the two distinct positions should agree after exchanging a single $\dn \dotsb \up$ pair in both $\lambda$ and in $\mu$, one finds that $\lambda, \mu, \nu$ must be the claimed vertices in Fig.~\ref{figure:squares} (c).

In the second case, $\lambda$ and $\mu$ are distinct in four positions and $\nu$ is of the form $\dotsb \up \dotsb \dn \dotsb \up \dotsb \dn \dotsb$ in these four positions. Then $\nu$ has to be obtained from $\lambda$ or $\mu$ by exchanging a $\dn \dotsb \up$ pair lying in the first and second position or lying in the third and fourth position in $e_\lambda$ or $e_\mu$, whence $\lambda, \mu, \nu$ may be completed into a square as in Fig.~\ref{figure:squares} (b). 

In the third case, $\lambda$ and $\mu$ are distinct in four positions and $\nu$ is of the form $\dotsb \up \dotsb \up \dotsb \dn \dotsb \dn \dotsb$.  If $\nu$ is obtained from $\lambda$ or $\mu$ by exchanging the $\dn \dotsb \up$ lying in the first and fourth position or lying in the second and third position, then $\lambda, \mu, \nu$ can be completed into a square as in Fig.~\ref{figure:squares} (a). If $\nu$ is obtained from $\lambda$ or $\mu$ by exchanging the $\dn \dotsb \up$ lying in the first and third position or lying in the second and fourth position, then we obtain Fig.~\ref{figure:single} (d) which by Lemma \ref{lemma:squarespaths} cannot appear in any square.
\end{proof}

\begin{definition}
Let $\Q_m^n$ be the quiver obtained from $\Gamma_m^n$ by replacing each edge by two arrows in opposite directions.
\begin{figure}
\begin{tikzpicture}[x=.5em,y=.5em]
\begin{scope}[shift={(0,0)}]
\node[font=\small] at (0,10) {$\Gamma_m^n$};
\node[shape=circle,scale=.5] (T) at (0,6) {};
\draw[fill=black] (0,6) circle(.1em);
\node[shape=circle,scale=.5] (B) at (0,0) {};
\draw[fill=black] (0,0) circle(.1em);
\node[font=\scriptsize] at (0,-1) {$\lambda$};
\node[font=\scriptsize] at (0,7) {$\mu$};
\path[-, line width=.5pt] (B) edge node[font=\scriptsize, left=-.4ex] {\strut $\gamma_\lambda^\mu$} (T);
\end{scope}
\begin{scope}[shift={(8,0)}]is emptyis empty
\node[font=\small] at (0,10) {$\Q_m^n$};
\node[shape=circle,scale=.5] (T) at (0,6) {};
\draw[fill=black] (0,6) circle(.1em);
\node[shape=circle,scale=.5] (B) at (0,0) {};
\draw[fill=black] (0,0) circle(.1em);
\path[->, line width=.5pt] (B) edge[transform canvas={xshift=-.4ex}] node[font=\scriptsize, left=-.4ex] {\strut $x_\lambda^\mu$} (T);
\path[->, line width=.5pt] (T) edge[transform canvas={xshift=.4ex}]  node[font=\scriptsize, right=-.4ex] {\strut $y_\lambda^\mu$} (B);
\end{scope}
\begin{scope}[shift={(16,0)}]
\node[font=\small] at (0,10) {$\QQ_m^n$};
\node[shape=circle,scale=.5] (T) at (0,6) {};
\draw[fill=black] (0,6) circle(.1em);
\node[shape=circle,scale=.5] (B) at (0,0) {};
\draw[fill=black] (0,0) circle(.1em);
\path[<-, line width=.5pt] (B) edge[transform canvas={xshift=-.4ex}] node[font=\scriptsize, left=-.4ex] {\strut $\bar x_\lambda^\mu$} (T);
\path[<-, line width=.5pt] (T) edge[transform canvas={xshift=.4ex}]  node[font=\scriptsize, right=-.4ex] {\strut $\bar y_\lambda^\mu$} (B);
\end{scope}
\end{tikzpicture}
\caption{An edge in the graph $\Gamma_m^n$ between two weights $\lambda, \mu \in \Lambda_m^n$ where $\mu$ is obtained from $\lambda$ so that $\lvert \lambda \rvert < \lvert \mu \rvert$ and the corresponding arrows in the double quiver $\Q_m^n$ for $\K_m^n$ and in the opposite quiver $\QQ_m^n$ for the Koszul dual $\KK_m^n$}
\label{figure:double}
\end{figure}
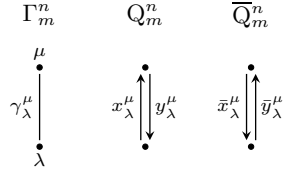
Edges only connect vertices of different heights. For an edge $\gamma_\lambda^\mu$ in $\Gamma_m^n$ connecting two vertices $\lambda, \mu \in \Lambda_m^n$ with $\lvert \lambda \rvert < \lvert \mu \rvert$, we denote the ascending arrow in $\Q_m^n$ (from $\lambda$ to $\mu$) by $x_\lambda^\mu$ and the descending arrow (from $\mu$ to $\lambda$) by $y_\lambda^\mu$ (see Fig.~\ref{figure:double} and Fig.~\ref{figure:22} for an illustration of $\Q_2^2$). 
\end{definition}

\begin{remark}\label{remark:defect}
Note that the arrow $x_\lambda^\mu$ (resp.\ $y_\lambda^\mu$) corresponds to the degree $1$ arc diagram $\cupd \lambda \mu \capd \mu$ (resp.\ $\cupd \mu \mu \capd \lambda$) in $\Q_m^n$. The total number of ascending arrows starting from $\lambda$ equals the number of circles in $e_\lambda$, i.e.\ the  defect $\mathrm{def}(\lambda)$ (see Definition \ref{definition:arcdiagramdegreezero}).
\end{remark}

The following definition will be used to describe the quadratic relations of $\K_m^n$ and $\KK_m^n$.

\begin{definition}\label{definition:number}
Fix a weight $\lambda \in \Lambda_m^n$. Let $\mu$ be a weight obtained from $\lambda$ by  exchanging a $\dn \dotsb \up$ pair lying in  a circle $C$ of $e_\lambda$. Let $\kappa$ be a weight such that $\lambda$ is obtained from $\kappa$ by exchanging a $\dn \dotsb \up$ pair lying in a circle $D$ of $e_{\kappa}$. Then define an integer $c_\kappa^\mu(\lambda)$ as follows
\[
c_{\kappa}^{\mu}(\lambda) = \begin{cases}
1 & \textup{if the circle $C$ does not appear in $e_{\kappa}$}\\
0 & \textup{if $C$ appears in $e_{\kappa}$ but does not enclose the circle $D$}\\
(-1)^{j-1}2 & \textup{if $C$ in $e_{\kappa}$ is the $j$th circle from inner to outer enclosing $D$}. 
\end{cases}
\]
\end{definition}

Recall from \eqref{align:rhomap} the surjective algebra map $\rho \colon \Bbbk \Q_m^n \tikzto \K_m^n$.
 
\begin{proposition}
\label{proposition:relations}
The map $\rho$ sends the following quadratic relations to $0$. 
\begin{itemize}
\item {\bf Monomial relations.} The paths between the top and bottom vertices in Fig.~\ref{figure:single} (a) and (b) are zero respectively. 
\item {\bf Commutativity relations across all squares.} All parallel paths of length $2$ excluding $2$-cycles in $\Q_m^n$ are equal.
\item {\bf Relations at vertices.} Fix $\lambda \in \Lambda_m^n$. Let $x_\lambda^{\mu_i}$  be the ascending arrows in $\Q_m^n$ starting at $e_\lambda$ for $1 \leq i \leq \mathrm{def}(\lambda)$. Let $x_{\kappa}^\lambda$ be any ascending arrow ending at $e_\lambda$. Then we have
\begin{align}\label{align:relationsvertex}
y^\lambda_{\kappa} x_{\kappa}^\lambda = \sum_{i = 1}^{\mathrm{def}(\lambda)} c_{\kappa}^{\mu_i}(\lambda) x_\lambda^{\mu_i}y_\lambda^{\mu_i} 
\end{align}
where $c_{\kappa}^{\mu_i}(\lambda)$ is given in Definition \ref{definition:number}. In particular, if $\lambda$ is the highest weight then $y^\lambda_{\kappa} x_{\kappa}^\lambda = 0$ and if $\lambda$ is the lowest weight then \eqref{align:relationsvertex} is empty. 
\end{itemize}
\end{proposition}

\begin{remark}
\label{remark:commutativity}
Note that by Proposition \ref{proposition:edgedescriptions} the parallel length $2$ paths appearing in the commutativity relations across squares are of the following forms
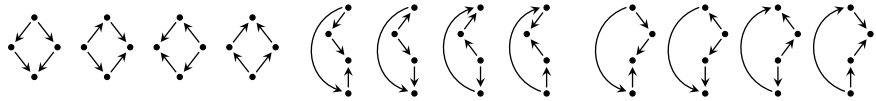
\[
\begin{tikzpicture}[x=.25em,y=.25em]
\begin{scope}[shift={(-2.5,-2)}]
\begin{scope}
\node[shape=circle,scale=.5] (T) at (0,2.5) {};
\draw[fill=black] (0,2.5) circle(.1em);
\node[shape=circle,scale=.5] (L) at (-3.5,-2) {};
\draw[fill=black] (-3.5,-2) circle(.1em);
\node[shape=circle,scale=.5] (R) at (3.5,-2) {};
\draw[fill=black] (3.5,-2) circle(.1em);
\node[shape=circle,scale=.5] (M) at (0,-6.5) {};
\draw[fill=black] (0,-6.5) circle(.1em);
\draw[->,line width=.5pt,line cap=round] (T) -- (L);
\draw[->,line width=.5pt,line cap=round] (T) -- (R);
\draw[->,line width=.5pt,line cap=round] (L) -- (M);
\draw[->,line width=.5pt,line cap=round] (R) -- (M);
\end{scope}
\begin{scope}[shift={(11,0)}]
\node[shape=circle,scale=.5] (T) at (0,2.5) {};
\draw[fill=black] (0,2.5) circle(.1em);
\node[shape=circle,scale=.5] (L) at (-3.5,-2) {};
\draw[fill=black] (-3.5,-2) circle(.1em);
\node[shape=circle,scale=.5] (R) at (3.5,-2) {};
\draw[fill=black] (3.5,-2) circle(.1em);
\node[shape=circle,scale=.5] (M) at (0,-6.5) {};
\draw[fill=black] (0,-6.5) circle(.1em);
\draw[<-,line width=.5pt,line cap=round] (T) -- (L);
\draw[->,line width=.5pt,line cap=round] (T) -- (R);
\draw[->,line width=.5pt,line cap=round] (L) -- (M);
\draw[<-,line width=.5pt,line cap=round] (R) -- (M);
\end{scope}
\begin{scope}[shift={(22,0)}]
\node[shape=circle,scale=.5] (T) at (0,2.5) {};
\draw[fill=black] (0,2.5) circle(.1em);
\node[shape=circle,scale=.5] (L) at (-3.5,-2) {};
\draw[fill=black] (-3.5,-2) circle(.1em);
\node[shape=circle,scale=.5] (R) at (3.5,-2) {};
\draw[fill=black] (3.5,-2) circle(.1em);
\node[shape=circle,scale=.5] (M) at (0,-6.5) {};
\draw[fill=black] (0,-6.5) circle(.1em);
\draw[->,line width=.5pt,line cap=round] (T) -- (L);
\draw[<-,line width=.5pt,line cap=round] (T) -- (R);
\draw[<-,line width=.5pt,line cap=round] (L) -- (M);
\draw[->,line width=.5pt,line cap=round] (R) -- (M);
\end{scope}
\begin{scope}[shift={(33,0)}]
\node[shape=circle,scale=.5] (T) at (0,2.5) {};
\draw[fill=black] (0,2.5) circle(.1em);
\node[shape=circle,scale=.5] (L) at (-3.5,-2) {};
\draw[fill=black] (-3.5,-2) circle(.1em);
\node[shape=circle,scale=.5] (R) at (3.5,-2) {};
\draw[fill=black] (3.5,-2) circle(.1em);
\node[shape=circle,scale=.5] (M) at (0,-6.5) {};
\draw[fill=black] (0,-6.5) circle(.1em);
\draw[<-,line width=.5pt,line cap=round] (T) -- (L);
\draw[<-,line width=.5pt,line cap=round] (T) -- (R);
\draw[<-,line width=.5pt,line cap=round] (L) -- (M);
\draw[<-,line width=.5pt,line cap=round] (R) -- (M);
\end{scope}
\end{scope}
\begin{scope}[shift={(45,0)}]
\begin{scope}
\node[shape=circle,scale=.5] (T) at (0,2) {};
\draw[fill=black] (0,2) circle(.1em);
\node[shape=circle,scale=.5] (L) at (-3,-2) {};
\draw[fill=black] (-3,-2) circle(.1em);
\node[shape=circle,scale=.5] (M) at (0,-6) {};
\draw[fill=black] (0,-6) circle(.1em);
\node[shape=circle,scale=.5] (B) at (0,-11) {};
\draw[fill=black] (0,-11) circle(.1em);
\draw[->,line width=.5pt,line cap=round] (T) -- (L);
\draw[->,line width=.5pt,line cap=round] (L) -- (M);
\draw[<-,line width=.5pt,line cap=round] (M) -- (B);
\path[->,line width=.5pt,line cap=round,out=-160,in=160,looseness=1.3] (T.180) edge (B);
\end{scope}
\begin{scope}[shift={(10,0)}]
\node[shape=circle,scale=.5] (T) at (0,2) {};
\draw[fill=black] (0,2) circle(.1em);
\node[shape=circle,scale=.5] (L) at (-3,-2) {};
\draw[fill=black] (-3,-2) circle(.1em);
\node[shape=circle,scale=.5] (M) at (0,-6) {};
\draw[fill=black] (0,-6) circle(.1em);
\node[shape=circle,scale=.5] (B) at (0,-11) {};
\draw[fill=black] (0,-11) circle(.1em);
\draw[<-,line width=.5pt,line cap=round] (T) -- (L);
\draw[->,line width=.5pt,line cap=round] (L) -- (M);
\draw[->,line width=.5pt,line cap=round] (M) -- (B);
\path[->,line width=.5pt,line cap=round,out=-160,in=160,looseness=1.3] (T.180) edge (B);
\end{scope}
\begin{scope}[shift={(20,0)}]
\node[shape=circle,scale=.5] (T) at (0,2) {};
\draw[fill=black] (0,2) circle(.1em);
\node[shape=circle,scale=.5] (L) at (-3,-2) {};
\draw[fill=black] (-3,-2) circle(.1em);
\node[shape=circle,scale=.5] (M) at (0,-6) {};
\draw[fill=black] (0,-6) circle(.1em);
\node[shape=circle,scale=.5] (B) at (0,-11) {};
\draw[fill=black] (0,-11) circle(.1em);
\draw[<-,line width=.5pt,line cap=round] (T) -- (L);
\draw[<-,line width=.5pt,line cap=round] (L) -- (M);
\draw[->,line width=.5pt,line cap=round] (M) -- (B);
\path[<-,line width=.5pt,line cap=round,out=-160,in=160,looseness=1.3] (T.180) edge (B);
\end{scope}
\begin{scope}[shift={(30,0)}]
\node[shape=circle,scale=.5] (T) at (0,2) {};
\draw[fill=black] (0,2) circle(.1em);
\node[shape=circle,scale=.5] (L) at (-3,-2) {};
\draw[fill=black] (-3,-2) circle(.1em);
\node[shape=circle,scale=.5] (M) at (0,-6) {};
\draw[fill=black] (0,-6) circle(.1em);
\node[shape=circle,scale=.5] (B) at (0,-11) {};
\draw[fill=black] (0,-11) circle(.1em);
\draw[->,line width=.5pt,line cap=round] (T) -- (L);
\draw[<-,line width=.5pt,line cap=round] (L) -- (M);
\draw[<-,line width=.5pt,line cap=round] (M) -- (B);
\path[<-,line width=.5pt,line cap=round,out=-160,in=160,looseness=1.3] (T.180) edge (B);
\end{scope}
\end{scope}
\begin{scope}[shift={(88,0)}]
\begin{scope}
\node[shape=circle,scale=.5] (T) at (0,2) {};
\draw[fill=black] (0,2) circle(.1em);
\node[shape=circle,scale=.5] (R) at (3,-2) {};
\draw[fill=black] (3,-2) circle(.1em);
\node[shape=circle,scale=.5] (M) at (0,-6) {};
\draw[fill=black] (0,-6) circle(.1em);
\node[shape=circle,scale=.5] (B) at (0,-11) {};
\draw[fill=black] (0,-11) circle(.1em);
\draw[->,line width=.5pt,line cap=round] (T) -- (R);
\draw[->,line width=.5pt,line cap=round] (R) -- (M);
\draw[<-,line width=.5pt,line cap=round] (M) -- (B);
\path[->,line width=.5pt,line cap=round,out=-160,in=160,looseness=1.3] (T.180) edge (B);
\end{scope}
\begin{scope}[shift={(11,0)}]
\node[shape=circle,scale=.5] (T) at (0,2) {};
\draw[fill=black] (0,2) circle(.1em);
\node[shape=circle,scale=.5] (R) at (3,-2) {};
\draw[fill=black] (3,-2) circle(.1em);
\node[shape=circle,scale=.5] (M) at (0,-6) {};
\draw[fill=black] (0,-6) circle(.1em);
\node[shape=circle,scale=.5] (B) at (0,-11) {};
\draw[fill=black] (0,-11) circle(.1em);
\draw[<-,line width=.5pt,line cap=round] (T) -- (R);
\draw[->,line width=.5pt,line cap=round] (R) -- (M);
\draw[->,line width=.5pt,line cap=round] (M) -- (B);
\path[->,line width=.5pt,line cap=round,out=-160,in=160,looseness=1.3] (T.180) edge (B);
\end{scope}
\begin{scope}[shift={(22,0)}]
\node[shape=circle,scale=.5] (T) at (0,2) {};
\draw[fill=black] (0,2) circle(.1em);
\node[shape=circle,scale=.5] (R) at (3,-2) {};
\draw[fill=black] (3,-2) circle(.1em);
\node[shape=circle,scale=.5] (M) at (0,-6) {};
\draw[fill=black] (0,-6) circle(.1em);
\node[shape=circle,scale=.5] (B) at (0,-11) {};
\draw[fill=black] (0,-11) circle(.1em);
\draw[<-,line width=.5pt,line cap=round] (T) -- (R);
\draw[<-,line width=.5pt,line cap=round] (R) -- (M);
\draw[->,line width=.5pt,line cap=round] (M) -- (B);
\path[<-,line width=.5pt,line cap=round,out=-160,in=160,looseness=1.3] (T.180) edge (B);
\end{scope}
\begin{scope}[shift={(33,0)}]
\node[shape=circle,scale=.5] (T) at (0,2) {};
\draw[fill=black] (0,2) circle(.1em);
\node[shape=circle,scale=.5] (R) at (3,-2) {};
\draw[fill=black] (3,-2) circle(.1em);
\node[shape=circle,scale=.5] (M) at (0,-6) {};
\draw[fill=black] (0,-6) circle(.1em);
\node[shape=circle,scale=.5] (B) at (0,-11) {};
\draw[fill=black] (0,-11) circle(.1em);
\draw[->,line width=.5pt,line cap=round] (T) -- (R);
\draw[<-,line width=.5pt,line cap=round] (R) -- (M);
\draw[<-,line width=.5pt,line cap=round] (M) -- (B);
\path[<-,line width=.5pt,line cap=round,out=-160,in=160,looseness=1.3] (T.180) edge (B);
\end{scope}
\end{scope}
\end{tikzpicture}
\]
where the first four diagrams appear in Fig.~\ref{figure:squares} (a), (b) and (c) and the last eight in (c) only.
\end{remark}

\begin{proof}[Proof of Proposition \ref{proposition:relations}.]
The monomial relations stem from the vanishing of the following multiplications (using the rule $\zeta \otimes \zeta \tikzmapsto 0$ in Definition \ref{definition:Kmn}):
\[
\begin{tikzpicture}[baseline=-.5ex,x=.6em,y=.6em,decoration={markings,mark=at position .99 with {\arrow[black]{Stealth[length=3.8pt]}}}]
\begin{scope}
\begin{scope}[shift={(-13,0)}]
\begin{scope}[shift={(4,0)}]
\DDOTS{0} \DDOTS{2} \DDOTS{4} \DDOTS{6} 
\UP{1} \UP{3} \DN{5} 
\RAY{1} \RAYUP{3} \ROUNDCUP{3}{1} \UPRARC{5}
\end{scope}
\begin{scope}[shift={(4,-3.5)}]
\DDOTS{0} \DDOTS{2} \DDOTS{4} \DDOTS{6} 
\UP{1} \DN{3} \UP{5} 
\RAYUP{1} \ROUNDCUP{1}{1} \CONNECT{3} \ROUNDCAP{3}{1} \RAYDN{5}
\end{scope}
\node at (13, -1.75) {$={} 0$};
\end{scope}

\begin{scope}[shift={(0,0)}]
\begin{scope}[shift={(4,0)}]
\DDOTS{0} \DDOTS{2} \DDOTS{4} \DDOTS{6} 
\UP{1} \DN{3} \DN{5} 
\UPLARC{1} \RAYUP{3} \RAY{5} \ROUNDCUP{1}{1} 
\end{scope}
\begin{scope}[shift={(4,-3.5)}]
\DDOTS{0} \DDOTS{2} \DDOTS{4} \DDOTS{6} 
\DN{1} \UP{3} \DN{5} 
\RAYDN{1} \ROUNDCUP{3}{1} \CONNECT{3} \ROUNDCAP{1}{1} \RAYUP{5}
\end{scope}
\node at (13, -1.75) {$={} 0$};
\end{scope}

\begin{scope}[shift={(13,0)}]
\begin{scope}[shift={(4,0)}]
\DDOTS{0} \DDOTS{2} \DDOTS{4} \DDOTS{6} 
\UP{1} \DN{3} \UP{5} 
\RAYDN{1} \ROUNDCAP{1}{1} \CONNECT{3} \ROUNDCUP{3}{1} \RAYUP{5}
\end{scope}
\begin{scope}[shift={(4,-3.5)}]
\DDOTS{0} \DDOTS{2} \DDOTS{4} \DDOTS{6} 
\UP{1} \UP{3} \DN{5} 
\RAY{1} \RAYDN{3} \ROUNDCAP{3}{1} \DNRARC{5} \CONNECT{5}
\end{scope}
\node at (13, -1.75) {$={} 0$};
\end{scope}

\begin{scope}[shift={(26,0)}]
\begin{scope}[shift={(4,0)}]
\DDOTS{0} \DDOTS{2} \DDOTS{4} \DDOTS{6} 
\DN{1} \UP{3} \DN{5} 
\RAYUP{1}  \ROUNDCUP{1}{1} \CONNECT{3} \ROUNDCAP{3}{1} \RAYDN{5}
\end{scope}
\begin{scope}[shift={(4,-3.5)}]
\DDOTS{0} \DDOTS{2} \DDOTS{4} \DDOTS{6} 
\UP{1} \DN{3} \DN{5} 
\DNLARC{1} \CONNECT{1} \ROUNDCAP{1}{1} \RAYDN{3} \RAY{5}
\end{scope}
\node at (13, -1.75) {$={} 0$.};
\end{scope}
\end{scope}
\end{tikzpicture}
\]

The commutativity relations across squares follow from the following {\it claim}: Let $p$ be a length $2$ path lying in a square in Fig.~\ref{figure:squares} from $e_\nu$ to $e_{\nu'}$ with $\nu \neq \nu'$. Then $\rho(p)=\cupd \nu \mu \capd \nu'$, where $\mu$ is the top vertex in the square.  

Let us prove the claim for the square (a) in Fig.~\ref{figure:squares}.  Denoting the bottom, top, left and right vertices in Fig.~\ref{figure:squares} (a) by $\lambda$, $\mu$, $\nu$ and $\nu'$ respectively,  we have 
\[
\begin{tikzpicture}[baseline=-.5ex,x=.6em,y=.6em,decoration={markings,mark=at position .99 with {\arrow[black]{Stealth[length=3.8pt]}}}]
\begin{scope}
\node at (-7, -4) {$\rho(x_\lambda^\nu x_\nu^\mu) = \rho(x_\lambda^\nu) \rho( x_\nu^\mu)  =$};
\begin{scope}[shift={(4,-1)}]
\DDOTS{0} \DDOTS{2} \DDOTS{4} \DDOTS{6} \DDOTS{8} \DDOTS{10} \DDOTS{12}
\UP{1} \DN{3} \UP{5} \DN{7} \UP{9} \DN{11}
\LARC{1}  \ROUNDCAP{3}{1} \ROUNDCAP{7}{1} \ROUNDCUP{5}{1} \CCCUP{3}{3} \RARC{11}
\CONNECT{3} \CONNECT{5} \CONNECT{7} \CONNECT{9}
\end{scope}
\begin{scope}[shift={(4,-7)}]
\DDOTS{0} \DDOTS{2} \DDOTS{4} \DDOTS{6} \DDOTS{8} \DDOTS{10} \DDOTS{12}
\UP{1} \DN{3} \DN{5} \UP{7} \UP{9} \DN{11}
\UPLARC{1} \CCCUP{1}{5} \CIRCLES{3}{3} \ROUNDCIRCLE{5}{1} \UPRARC{11}
\end{scope}
\node at (18, -4) {$=$};
\begin{scope}[shift={(20,-4)}]
\DDOTS{0} \DDOTS{2} \DDOTS{4} \DDOTS{6} \DDOTS{8} \DDOTS{10} \DDOTS{12}
\UP{1} \DN{3} \UP{5} \DN{7} \UP{9} \DN{11}
\UPLARC{1}  \CCCUP{1}{5} \CCCUP{3}{3} \ROUNDCUP{5}{1} \ROUNDCAP{3}{1} \ROUNDCAP{7}{1} \UPRARC{11}
\CONNECT{3} \CONNECT{5} \CONNECT{7} \CONNECT{9}
\end{scope}
\end{scope}
\end{tikzpicture}
\]
using the rule $1 \otimes 1 \tikzmapsto 1$ twice. Similarly we have
\[
\begin{tikzpicture}[baseline=-.5ex,x=.6em,y=.6em,decoration={markings,mark=at position .99 with {\arrow[black]{Stealth[length=3.8pt]}}}]
\begin{scope}
\node at (-7, -6) {$\rho(x_\lambda^{\nu'} x_{\nu'}^\mu) = \rho(x_\lambda^{\nu'}) \rho( x_{\nu'}^\mu)  =$};
\begin{scope}[shift={(4,-1.5)}]
\DDOTS{0} \DDOTS{2} \DDOTS{4} \DDOTS{6} \DDOTS{8} \DDOTS{10} \DDOTS{12}
\UP{1} \DN{3} \UP{5} \DN{7} \UP{9} \DN{11}
\UPLARC{1}  \CCCUP{1}{5} \ROUNDCUP{3}{1} \ROUNDCUP{7}{1} \ROUNDCAP{3}{1} \ROUNDCAP{7}{1} \UPRARC{11}
\CONNECT{3} \CONNECT{5} \CONNECT{7} \CONNECT{9}
\end{scope}
\begin{scope}[shift={(4,-10.5)}]
\DDOTS{0} \DDOTS{2} \DDOTS{4} \DDOTS{6} \DDOTS{8} \DDOTS{10} \DDOTS{12}
\DN{1} \DN{3} \UP{5} \DN{7} \UP{9} \UP{11}
\CIRCLES{1}{5} \CCCUP{3}{3} \ROUNDCUP{5}{1} \ROUNDCAP{3}{1} \ROUNDCAP{7}{1}
\CONNECT{3} \CONNECT{5} \CONNECT{7} \CONNECT{9}
\end{scope}
\node at (18, -6) {$=$};
\begin{scope}[shift={(20,-6)}]
\DDOTS{0} \DDOTS{2} \DDOTS{4} \DDOTS{6} \DDOTS{8} \DDOTS{10} \DDOTS{12}
\UP{1} \DN{3} \UP{5} \DN{7} \UP{9} \DN{11}
\UPLARC{1}  \CCCUP{1}{5} \CCCUP{3}{3} \ROUNDCUP{5}{1} \ROUNDCAP{3}{1} \ROUNDCAP{7}{1} \UPRARC{11}
\CONNECT{3} \CONNECT{5} \CONNECT{7} \CONNECT{9}
\end{scope}
\end{scope}
\end{tikzpicture}
\]
using $1 \otimes \ast \tikzmapsto \ast$ once and $1 \otimes 1 \tikzmapsto 1$ twice. This verifies that $\rho(x_\lambda^\nu x_\nu^\mu) =\cupd \lambda \mu \capd \mu= \rho(x_\lambda^{\nu'} x_{\nu'}^\mu) $. Using the involution in Remark \ref{remark:involution} we  have $\rho( y_\nu^\mu y_\lambda^\nu) = \cupd \mu \mu \capd \lambda= \rho(y_{\nu'}^\mu y_\lambda^{\nu'} ).$ By a similar computation we have $$\rho(x_\nu^\mu y_{\nu'}^\mu)=  \cupd \nu \mu \capd \nu' = \rho(y_{\lambda}^\nu x_{\lambda}^{\nu'}), \quad \rho( x_{\nu'}^\mu y_\nu^\mu)=  \cupd \nu' \mu \capd \nu = \rho( y_{\lambda}^{\nu'} x_{\lambda}^\nu).$$ 
This proves the claim for Fig.~\ref{figure:squares} (a). The other cases may be proved in a similar way. 

It remains to verify the relations at vertices. Note that $\lambda$ is obtained from $\kappa$ by exchanging a $\dn \dotsb \up$ pair lying in a circle $D$ of $e_{\kappa}$. We have the following two cases. In the first case, there are circles in $e_\lambda$ which appear in $e_{\kappa}$ enclosing $D$. Then there are two circles (denoted by $C_1$ and $C_2$ from left to right) in $e_\lambda$ which do not appear in $e_{\kappa}$. Note that the innermost circle in $e_\kappa$ enclosing $D$ does not appear in $e_\lambda$.    Denote by $C_i$ the $(i-1)$th circle in $e_\kappa$ enclosing $D$ for $3 \leq i \leq k$ and denote by $\mu_i$ the weights obtained from $\lambda$ by exchanging the $\dn \dotsb \up$ pair lying in $C_i$. Then $e_\kappa, e_\lambda$ and $e_{\mu_i}$ have the following forms
 
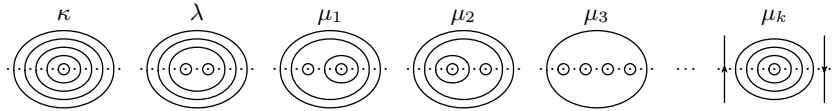
\[
\begin{tikzpicture}[scale = 0.35, baseline=-.5ex,x=.6em,y=.6em,decoration={markings,mark=at position 1 with {\arrow[black]{Stealth[length=3.2pt]}}}]
\begin{scope}
\begin{scope}[shift={(-96,0)}]
\node[font=\small] at (10,10) {\strut$\kappa$};
\DDOT{0} \DDOT{2} \DDOT{4} \DDOT{6} \DDOT{8} \DDOT{10} \DDOT{12} \DDOT{14} \DDOT{16} \DDOT{18} \DDOT{20}
\CIRCLES{1}{9} \CIRCLES{3}{7} \CIRCLES{5}{5} \CIRCLES{7}{3} \CCIRCLES{9}
\end{scope}
\begin{scope}[shift={(-72,0)}]
\node[font=\small] at (10,10) {\strut$\lambda$};
\DDOT{0} \DDOT{2} \DDOT{4} \DDOT{6} \DDOT{8} \DDOT{10} \DDOT{12} \DDOT{14} \DDOT{16} \DDOT{18} \DDOT{20}
\CIRCLES{1}{9} \CIRCLES{3}{7} \CIRCLES{5}{5} \CCIRCLES{7} \CCIRCLES{11}
\end{scope}
\begin{scope}[shift={(-48,0)}]
\node[font=\small] at (10,10) {\strut$\mu_1$};
\DDOT{0} \DDOT{2} \DDOT{4} \DDOT{6} \DDOT{8} \DDOT{10} \DDOT{12} \DDOT{14} \DDOT{16} \DDOT{18} \DDOT{20}
\CIRCLES{1}{9} \CIRCLES{3}{7} \CCIRCLES{5} \CIRCLES{9}{3} \CCIRCLES{11}
\end{scope}
\begin{scope}[shift={(-24,0)}]
\node[font=\small] at (10,10) {\strut$\mu_2$};
\DDOT{0} \DDOT{2} \DDOT{4} \DDOT{6} \DDOT{8} \DDOT{10} \DDOT{12} \DDOT{14} \DDOT{16} \DDOT{18} \DDOT{20}
\CIRCLES{1}{9} \CIRCLES{3}{7} \CIRCLES{5}{3} \CCIRCLES{7} \CCIRCLES{13}
\end{scope}
\begin{scope}[shift={(0,0)}]
\node[font=\small] at (10,10) {\strut$\mu_3$};
\DDOT{0} \DDOT{2} \DDOT{4} \DDOT{6} \DDOT{8} \DDOT{10} \DDOT{12} \DDOT{14} \DDOT{16} \DDOT{18} \DDOT{20}
\CIRCLES{1}{9} \CCIRCLES{3} \CCIRCLES{7} \CCIRCLES{11} \CCIRCLES{15} 
\end{scope}
\begin{scope}[shift={(26,0)}]
\DDOTS{0}
\end{scope}
\begin{scope}[shift={(32,0)}]
\node[font=\small] at (10,10) {\strut$\mu_k$};
\DDOT{0} \DDOT{2} \DDOT{4} \DDOT{6} \DDOT{8} \DDOT{10} \DDOT{12} \DDOT{14} \DDOT{16} \DDOT{18} \DDOT{20}
\begin{scope}[yshift=3]
\UP{1}
\end{scope}
\begin{scope}[yshift=-3]
\DN{19}
\end{scope}
\RAYS{1}{6} \RAYS{19}{6} \CIRCLES{3}{7} \CIRCLES{5}{5} \CIRCLES{7}{3} \CCIRCLES{9}
\end{scope}
\end{scope}
\end{tikzpicture}
\]
Using the rules $1 \otimes 1 \tikzmapsto 1$  and $1 \tikzmapsto 1 \otimes \epsilon + \epsilon \otimes 1$, we have
\[
\begin{tikzpicture}[scale = 0.35, baseline=-.5ex,x=.6em,y=.6em,decoration={markings,mark=at position .99 with {\arrow[black]{Stealth[length=3.2pt]}}}]
\begin{scope}
\node[left] at (0,0) {\strut$\rho(y_{\kappa}^{\lambda})\rho( x_{\kappa}^\lambda) ={}$};
\DDOT{0} \DDOT{2} \DDOT{4} \DDOT{6} \DDOT{8} \DDOT{10} \DDOT{12} \DDOT{14} \DDOT{16} \DDOT{18} \DDOT{20}
\begin{scope}[yshift=3]
\UP{7}
\end{scope}
\begin{scope}[yshift=-3]
\DN{9}
\end{scope}
\CIRCLES{1}{9} \CIRCLES{3}{7} \CIRCLES{5}{5} \TINYCIRCLE{7} \CCIRCLES{11}
\begin{scope}[shift={(29,0)}]
\DDOT{0} \DDOT{2} \DDOT{4} \DDOT{6} \DDOT{8} \DDOT{10} \DDOT{12} \DDOT{14} \DDOT{16} \DDOT{18} \DDOT{20}
\begin{scope}[yshift=3]
\UP{11}
\end{scope}
\begin{scope}[yshift=-3]
\DN{13}
\end{scope}
\CIRCLES{1}{9} \CIRCLES{3}{7} \CIRCLES{5}{5} \CCIRCLES{7} \TINYCIRCLE{11}
\end{scope}
\node at (24.5, 0) {$+$};
\node at (66.5,0) {$ = \cupd \lambda \mu_1 \capd \lambda + \cupd \lambda \mu_2 \capd \lambda$};
\end{scope}

\begin{scope}[shift={(0, -16)}]
\begin{scope}
\node[left] at (0,0) {\strut$\rho(x_{\lambda}^{\mu_1})\rho( y_{\lambda}^{\mu_1}) ={}$};
\DDOT{0} \DDOT{2} \DDOT{4} \DDOT{6} \DDOT{8} \DDOT{10} \DDOT{12} \DDOT{14} \DDOT{16} \DDOT{18} \DDOT{20}
\begin{scope}[yshift=3]
\UP{7}
\end{scope}
\begin{scope}[yshift=-3]
\DN{9}
\end{scope}
\CIRCLES{1}{9} \CIRCLES{3}{7} \CIRCLES{5}{5} \TINYCIRCLE{7} \CCIRCLES{11}
\begin{scope}[shift={(29,0)}]
\DDOT{0} \DDOT{2} \DDOT{4} \DDOT{6} \DDOT{8} \DDOT{10} \DDOT{12} \DDOT{14} \DDOT{16} \DDOT{18} \DDOT{20}
\begin{scope}[yshift=3]
\UP{5}
\end{scope}
\begin{scope}[yshift=-3]
\DN{15}
\end{scope}
\CIRCLES{1}{9} \CIRCLES{3}{7} \CIRCLES{5}{5} \CCIRCLES{7} \CCIRCLES{11}
\end{scope} 
\node at (24.5, 0) {$+$};
\node at (66.5,0) {$ = \cupd \lambda \mu_1 \capd \lambda + \cupd \lambda \mu_3 \capd \lambda.$};
\end{scope}
\end{scope}
\end{tikzpicture}
\]
Similarly for $2 \leq i < k$ we have
$
\rho(x_{\lambda}^{\mu_i})\rho( y_{\lambda}^{\mu_i}) = \cupd \lambda \mu_i \capd \lambda + \cupd \lambda \mu_{i+1} \capd \lambda
$ 
and $ \rho(x_{\lambda}^{\mu_k})\rho( y_{\lambda}^{\mu_k}) =   \cupd \lambda \mu_k \capd \lambda.$ 
Combining the above equalities we obtain 
\[
\rho(y_{\kappa}^{\lambda})\rho( x_{\kappa}^\lambda)  = \rho(x_{\lambda}^{\mu_1} y_{\lambda}^{\mu_1}) + \rho(x_{\lambda}^{\mu_2} y_{\lambda}^{\mu_2}) - 2 \rho(x_{\lambda}^{\mu_3} y_{\lambda}^{\mu_3}) + \dotsb + (-1)^{k}2\rho(x_{\lambda}^{\mu_k} y_{\lambda}^{\mu_k}).
\]
This yields \eqref{align:relationsvertex} since $c_{\kappa}^{\mu_1}(\lambda) =1=  c_{\kappa}^{\mu_2}(\lambda)$ and $c_{\kappa}^{\mu_i}(\lambda) = (-1)^{i} 2$ for $3 \leq i \leq k$. 

In the second case, there is no circle in $e_\lambda$ which appears in $e_{\kappa}$ enclosing $D$. Then $\kappa, \lambda$ have the following subcases
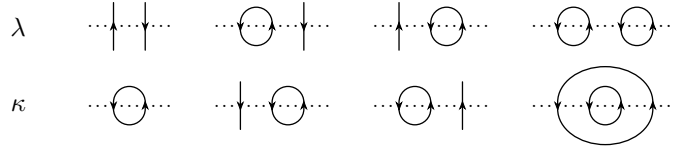
\[
\begin{tikzpicture}[ baseline=-.5ex,x=.6em,y=.6em,decoration={markings,mark=at position .99 with {\arrow[black]{Stealth[length=3.8pt]}}}]
\begin{scope}[shift={(0, 5)}] 
\node at (-33, 0) {$ \lambda$};
\begin{scope}[shift={(0,0)}]
\DDOTS{0} \DDOTS{2} \DDOTS{4} \DDOTS{6}  \DDOTS{8} 
\DN{1} \UP{3} \DN{5} \UP{7}
\ROUNDCIRCLE{1}{1} \ROUNDCIRCLE{5}{1}
\end{scope}
\begin{scope}[shift={(-10, 0)}]
\DDOTS{0} \DDOTS{2} \DDOTS{4}  \DDOTS{6} 
\UP{1} \DN{3} \UP{5}
\ROUNDCIRCLE{3}{1} \RAYS{1}{1.5}
\end{scope}
\begin{scope}[shift={(-20, 0)}]
\DDOTS{0} \DDOTS{2} \DDOTS{4} \DDOTS{6}  
\DN{1} \UP{3} \DN{5} 
\RAYS{5}{1.5} \ROUNDCIRCLE{1}{1}
\end{scope}
\begin{scope}[shift={(-28, 0)}]
 \DDOTS{0} \DDOTS{2} \DDOTS{4}  
\UP{1} \DN{3} 
\RAYS{1}{1.5} \RAYS{3}{1.5}
\end{scope}
\end{scope}

\begin{scope}
\node at (-33, 0) {$ \kappa$};
\begin{scope}[shift={(0,0)}]
\DDOTS{0} \DDOTS{2} \DDOTS{4} \DDOTS{6}  \DDOTS{8} 
\DN{1} \DN{3} \UP{5} \UP{7}
\CIRCLES{1}{3} \ROUNDCIRCLE{3}{1}
\end{scope}

\begin{scope}[shift={(-10, 0)}]
\DDOTS{0} \DDOTS{2} \DDOTS{4}  \DDOTS{6} 
\DN{1} \UP{3} \UP{5}
\ROUNDCIRCLE{1}{1} \RAYS{5}{1.5}
\end{scope}
\begin{scope}[shift={(-20, 0)}]
\DDOTS{0} \DDOTS{2} \DDOTS{4} \DDOTS{6}  
\DN{1} \DN{3} \UP{5} 
\RAYS{1}{1.5} \ROUNDCIRCLE{3}{1}
\end{scope}
\begin{scope}[shift={(-28, 0)}]
\DDOTS{0} \DDOTS{2} \DDOTS{4}  
\DN{1} \UP{3} 
\ROUNDCIRCLE{1}{1}
\end{scope}
\end{scope}
\end{tikzpicture}
\]
Note that in the first subcase we have that $c_\kappa^{\mu'}(\lambda) = 0$ where $\mu'$ is any weight obtained from $\lambda$ by exchanging the $\dn \dotsb \up$ pair in a circle of $e_\lambda$ and that 
\[
\begin{tikzpicture}[baseline=-.5ex,x=.6em,y=.6em,decoration={markings,mark=at position .99 with {\arrow[black]{Stealth[length=3.8pt]}}}]
\begin{scope}
\node at (-1, -1.75) {$\rho(y_{\kappa}^{\lambda} x_{\kappa}^\lambda) ={} $};
\begin{scope}[shift={(4,-.2)}]
\DDOTS{0} \DDOTS{2} \DDOTS{4} 
\UP{1} \DN{3} 
\RAYUP{1}  \ROUNDCUP{1}{1} \RAYUP{3}
\end{scope}
\begin{scope}[shift={(4,-3.3)}]
\DDOTS{0} \DDOTS{2} \DDOTS{4} 
\UP{1} \DN{3} 
\RAYDN{1} \ROUNDCAP{1}{1} \RAYDN{3}
\end{scope}
\node at (11, -1.75) {$= 0$.};
\end{scope}
\end{tikzpicture}
\]
For the second and third subcases we have 
$\rho(y_{\kappa}^{\lambda} x_{\kappa}^\lambda)  = \cupd \lambda \mu \capd \lambda= \rho(x_{\lambda}^{\mu} y_\lambda^\mu)$ 
where $\mu$ is obtained from $\lambda$ by exchanging the $\dn \dotsb \up$ pair lying in the (illustrated) circle of $e_{\lambda}$. Note that $c_\kappa^\mu (\lambda) =1$ and $c_\kappa^{\mu'}(\lambda) = 0$ for any $\mu'\neq \mu$. For the fourth subcase we have 
 \[
\begin{tikzpicture}[baseline=-.5ex,x=.6em,y=.6em,decoration={markings,mark=at position .99 with {\arrow[black]{Stealth[length=3.8pt]}}}]
\begin{scope}
\node at (-5, 0) {$\rho(y_{\kappa}^{\lambda} x_{\kappa}^\lambda) ={} $};
\begin{scope}[shift={(0,0)}]
\DDOTS{0} \DDOTS{2} \DDOTS{4} \DDOTS{6}  \DDOTS{8} 
\UP{1} \DN{3} \DN{5} \UP{7}
\ROUNDCIRCLE{1}{1} \ROUNDCIRCLE{5}{1}
\end{scope}
\node at (9.5, 0) {$+$};
\begin{scope}[shift={(11,0)}]
\DDOTS{0} \DDOTS{2} \DDOTS{4} \DDOTS{6}  \DDOTS{8} 
\DN{1} \UP{3} \UP{5} \DN{7}
\ROUNDCIRCLE{1}{1} \ROUNDCIRCLE{5}{1}
\end{scope}
\node at (30, 0) {$= \rho(x_\lambda^{\mu_1} y_\lambda^{\mu_1}) + \rho(x_\lambda^{\mu_2} y_\lambda^{\mu_2})$};
\end{scope}
\end{tikzpicture}
\]
where $\mu_1, \mu_2$ are the weights obtained from $\lambda$ by exchanging the $\dn \dotsb \up$ pairs lying in the (illustrated) circles of $e_{\lambda}$. Note that $c_\kappa^{\mu_1} (\lambda) = 1 = c_\kappa^{\mu_2} (\lambda)$.
\end{proof}

Denote by $\I_m^n$ the two-sided ideal of $\Bbbk \Q_m^n$ generated by the above quadratic relations.  Proposition \ref{proposition:relations} yields an algebra map (still denoted by $\rho$)
\begin{align}\label{align:rho}
\rho \colon \Bbbk \Q_m^n / \I_m^n \tikztwoheadrightarrow \K_m^n.
\end{align}

\begin{remark}\label{remark:involution}
Note that $\Bbbk \Q_m^n / \I_m^n$ admits a natural involution map 
\begin{align}
(-)^*\colon \Bbbk \Q_m^n / \I_m^n \tikzto \Bbbk \Q_m^n / \I_m^n
\end{align}
determined by $(e_\lambda)^*=e_\lambda$ and $(x)^* = y$ and $(y)^*=x$ for a pair of opposite arrows $x, y$ (cf.\ Fig.~\ref{figure:double}).  
We also recall the involution map of $\K_m^n$ defined in \cite[(4.14)]{brundanstroppel1}
$$
(-)^*\colon \K_m^n \tikzto \K_m^n, \quad \cupd \alpha \lambda \capd \beta \tikzmapsto \cupd \beta \lambda \capd \alpha.
$$
Clearly $\rho$ intertwines these two maps, i.e.\ we have $\rho(a^*) = \rho(a)^*$ for any $a \in  \Bbbk \Q_m^n / \I_m^n $. 
\end{remark}

In the rest of this section we will show that the map $\rho$ in \eqref{align:rho} is a bijection. As a result, the algebra $\K_m^n$ can be written as the path algebra $\Bbbk \Q_m^n$ of the quiver $\Q_m^n$ with relations $\I_m^n$. Note that $\rho$ is bijective in degrees $0$ and $1$, and the algebra $\K_m^n$ is Koszul and in particular quadratic. Therefore, it only remains to show that $\rho$ is bijective in degree $2$. 

\begin{lemma}\label{lemma:dimension}
$\dim_\Bbbk e_\lambda (\Bbbk \Q_m^n / \I_m^n)_2 e_\mu \leq \begin{cases} 1 & \text{if } \lambda \neq \mu \\ \mathrm{def} (\lambda) & \text{if } \lambda = \mu. \end{cases}$
\end{lemma}

\begin{proof}
Let $\lambda \neq \mu$. If $\dim_\Bbbk e_\lambda (\Bbbk \Q_m^n / \I_m^n)_2 e_\mu >1$ then there exist two distinct paths $p, q$ of length $2$ from $e_\lambda$ to $e_\mu$ in $\Q_m^n$ such that $\rho(p) \neq \rho(q)$, which contradicts the commutativity relations across squares. Thus, $\dim_\Bbbk e_\lambda (\Bbbk \Q_m^n / \I_m^n)_2 e_\mu \leq 1$. 

If $\lambda = \mu$, from the relations at each vertex $\lambda$ we note that $e_\lambda (\Bbbk \Q_m^n / \I_m^n)_2 e_\lambda$ is spanned by the set of $2$-cycles at $\lambda$ of the form $xy$. This set corresponds bijectively to the counterclockwise circles in $e_\lambda$, whence $\dim e_\lambda (\Bbbk \Q_m^n / \I_m^n)_2 e_\lambda\leq \mathrm{def}(\lambda)$.
\end{proof}

\begin{lemma}\label{lemma:dimension1}
If $\lambda \neq \mu$ then $\dim_\Bbbk e_\lambda (\Bbbk \Q_m^n / \I_m^n)_2 e_\mu = \dim_\Bbbk e_\lambda (\K_m^n)_2 e_\mu.$
\end{lemma}
\begin{proof}
Since $\rho$ is surjective by Lemma \ref{lemma:dimension} we have 
\begin{align*}
1\geq \dim_\Bbbk e_\lambda (\Bbbk \Q_m^n / \I_m^n)_2 e_\mu  \geq \dim_\Bbbk e_\lambda (\K_m^n)_2 e_\mu.
\end{align*}
We claim that the second inequality is an equality. Indeed, if $\dim_\Bbbk e_\lambda (\Bbbk \Q_m^n / \I_m^n)_2 e_\mu = 0$ the claim holds since both sides have to be zero. If $\dim_\Bbbk e_\lambda (\Bbbk Q_m^n / \I_m^n)_2 e_\mu = 1$ then there exists a length $2$ path $p$ in $\Q_m^n$  from $e_\lambda$ to $e_\mu$ such that $[p] \neq 0$ in $e_\lambda (\Bbbk Q_m^n / \I_m^n)_2 e_\mu$. This can only happen in the following  two cases. 

In the first case, the path $p$ appears in a square. Then it follows from the claim in the proof of Proposition \ref{proposition:relations} that $\rho(p)= \cupd \lambda \xi \capd \mu$. So $e_\lambda (\K_m^n)_2 e_\mu \neq 0$. 

In the second case, the path $p$ does not lie in any square. Then $p$ is the ascending path or the descending path of length $2$ in Fig.~\ref{figure:single} (c) or the length $2$ path from left to right or from right to left in Fig.~\ref{figure:single} (d). If $p$ is the ascending path in Fig.~\ref{figure:single} (c) then
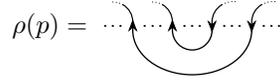
\[
\begin{tikzpicture}[x=.75em,y=.75em,decoration={markings,mark=at position 0.99 with {\arrow[black]{Stealth[length=4.8pt]}}}]
\begin{scope}
\node at (-3, 0) {$\rho(p) ={}$};
\DDOTS{0} \DDOTS{2} \DDOTS{4} \DDOTS{6} \DDOTS{8}
\UP{1} \UP{3} \DN{5} \DN{7}
\UPLARC{1}  \UPLARC{3} \UPRARC{5} \UPRARC{7}
 \ROUNDCUP{3}{1} \CCCUP{1}{3}
 \end{scope}
\end{tikzpicture}
\]
which is nonzero in $\K_m^n$. Using the involution in Remark \ref{remark:involution} we obtain an analogous statement for the descending path. If $p$ is the path from left to right in Fig.~\ref{figure:single} (d) then 
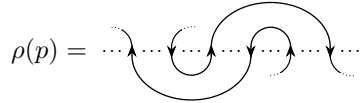
\[
\begin{tikzpicture}[x=.75em,y=.75em,decoration={markings,mark=at position 0.99 with {\arrow[black]{Stealth[length=4.8pt]}}}]
\begin{scope}
\node at (-3, 0) {$\rho(p) ={}$};
\DDOTS{0} \DDOTS{2} \DDOTS{4} \DDOTS{6} \DDOTS{8}  \DDOTS{10} \DDOTS{12}
\UP{1} \DN{3} \UP{5} \DN{7} \UP{9} \DN{11}
 \UPLARC{1}  \UPRARC{3} \DNLARC{9} \DNRARC{11}
 \ROUNDCUP{3}{1} \CCCUP{1}{3} \CCCAP{5}{3} \ROUNDCAP{7}{1}
\CONNECT{5} \CONNECT{7} \CONNECT{9} \CONNECT{11}
\end{scope}
\end{tikzpicture}
\]
which is nonzero in $\K_m^n$. Again, using the involution the same holds for the path from right to left.
\end{proof}

\begin{proposition}
The algebra map $\rho\colon \Bbbk \Q_m^n / \I_m^n \tikzto \K_m^n$ is an isomorphism. 
\end{proposition}
\begin{proof}
As mentioned above, we need to show that $\rho$ is bijective in degree $2$. For this, it suffices to show that for any weights $\lambda, \mu$ we have 
\begin{align}\label{equality}
\dim_\Bbbk e_\lambda (\Bbbk \Q_m^n / \I_m^n)_2 e_\mu  = \dim_\Bbbk e_\lambda (\K_m^n)_2 e_\mu.
\end{align}
The equality \eqref{equality} holds for $\lambda = \mu$ since 
\[
\mathrm{def}(\lambda) \geq \dim_\Bbbk e_\lambda (\Bbbk \Q_m^n / \I_m^n)_2 e_\lambda  \geq \dim_\Bbbk e_\lambda (\K_m^n)_2 e_\lambda=\mathrm{def}(\lambda)
\]
where the first inequality follows from Lemma \ref{lemma:dimension} and the second one follows from $\rho$ being surjective, and the third equality follows from Remark \ref{remark:defect}. 
If $\lambda \neq \mu$ then the equality \eqref{equality} follows from Lemma \ref{lemma:dimension1}. 
\end{proof}

\section{The Koszul-dual picture}

Let $A = \Bbbk Q / I$, where $Q$ is a finite quiver with vertex set $Q_0$ and $I \subset \Bbbk Q_{\geq 2}$ is a two-sided ideal of relations, where $Q_{\geq 2}$ denotes the set of paths of length $\geq 2$. We say that $A$ is {\it Koszul} in the sense of \cite{priddy,beilinsonginzburgsoergel} if it is graded by the path length such that the $A$-module $\Bbbk Q_0$ admits a graded projective resolution 
\[
\dotsb \tikzto P_{-i} \tikzto \dotsb \tikzto P_{-1} \tikzto P_0 \tikztwoheadrightarrow \Bbbk Q_0
\]
with $P_{-i}$ being generated in degree $i$ as $A$-module. Then the ideal $I$ is necessarily generated by quadratic relations $I_2\subset \Bbbk Q_2$ and the Koszul dual $A^! = \Ext^\hdot_A (\Bbbk Q_0, \Bbbk Q_0)$ (disregarding the grading of $A$) is also a Koszul algebra (with elements in $\Ext^d_A (\Bbbk Q_0, \Bbbk Q_0)$ being of degree $d$) such that $(A^!)^! \simeq A$.

The Koszul dual $A^!$ is isomorphic to the linear dual algebra $\Bbbk \overline Q / (I_2^{\perp})$ which we briefly recall from \cite[\S 2]{beilinsonginzburgsoergel}. Let $\overline Q$ denote the opposite quiver of $Q$, i.e.\ $\overline Q$ has the same vertices as $Q$ and the opposite arrow $\bar a$ for each arrow $a$ in $Q$, so that $\overline{\overline Q}$ can be naturally identified with $Q$.  We have a natural nondegenerate $\Bbbk$-bilinear pairing $\langle - {,} - \rangle \colon \Bbbk Q \times \Bbbk \overline Q \tikzto \Bbbk$ which is uniquely determined by 
\[
\langle a_1a_2\dotsb a_n, \bar b_n\dotsb \bar b_2\bar b_1\rangle = \delta_{a_1, b_1} \dotsb \delta_{a_n, b_n} \quad \text{for paths $a_1\dotsb a_n, b_1\dotsb b_n$ in $Q$.}
\]
Then the quadratic relations for $A^!$ are given by $I_2^{\perp} = \{ p \in \Bbbk \overline Q_2 \mid \langle I_2, p\rangle = 0\}$.

\subsection{Generators and relations of the Koszul dual}

We now give a quiver description for the Koszul dual of the extended Khovanov arc algebras $\K_m^n$.

\begin{notation}
In order to simplify notation, let us write
\[
\KK_m^n := (\K_m^n)^! \simeq \Bbbk \QQ_m^n / \II_m^n
\]
for the Koszul dual of $\K_m^n \simeq \Bbbk \Q_m^n / \I_m^n$, where $\QQ_m^n$ is the opposite quiver of $\Q_m^n$ and $\II_m^n = (\I_m^n)_2^\perp$. We write $\bar x_\lambda^\mu, \bar y_\lambda^\mu$ for the Koszul duals of the generators $x_\lambda^\mu, y_\lambda^\mu$ of $\K_m^n$. Therefore, $\bar x_\lambda^\mu$ is a descending arrow and $\bar y_\lambda^\mu$ an ascending arrow (see Fig.~\ref{figure:double}).
\end{notation}

\begin{remark}\label{remark:involutionmapkosuzldaul}
Similar to Remark \ref{remark:involution}, $\KK_m^n$ also admits an involution 
\[
(-)^*\colon \KK_m^n \tikzto \KK_m^n
\]
satisfying $(\bar x)^* = \bar y$ and $(\bar y)^* = \bar x$ for a pair of opposite arrows $\bar x, \bar y$.
\end{remark}

\begin{proposition}\label{proposition:relationdual}
The ideal $\II_m^n$ is generated by the following quadratic relations:
\begin{itemize}
\item {\bf Monomial relations.} The paths of length $2$ between the top and bottom vertices in Fig.~\ref{figure:single} (c) and (d) are zero. 
\item {\bf Anticommutativity and Plücker-type relations.} If there is more than one parallel path of length $2$ (excluding $2$-cycles), then the sum of all such parallel paths is $0$. 
\item {\bf Relations at vertices.} Fix $\lambda \in \Lambda_m^n$. Let $\bar y^\lambda_{\kappa_j} $ be the ascending arrows in $\QQ_m^n$ ending at $e_\lambda$. Let $\bar y_\lambda^\mu$ be any ascending arrow starting at $e_\lambda$. Then  
\begin{align}
\bar y^\mu_\lambda \bar x_\lambda^\mu =- \sum_{j} c^{ \mu}_{\kappa_j}(\lambda) \bar x^\lambda_{\kappa_j} \bar y^\lambda_{\kappa_j} 
\end{align}
where the coefficients $c^{\mu}_{\kappa_j} (\lambda)$ are given in Definition \ref{definition:number}.
\end{itemize}
\end{proposition}

\begin{proof}
We need to show that the above relations are orthogonal to the quadratic relations described in Proposition \ref{proposition:relations}.  The monomial relations are clear since by Proposition \ref{proposition:relations} the monomials do not appear in the quadratic relations $(\I_m^n)_2$. The anticommutativity and Plücker-type relations involving length $2$ paths (excluding $2$-cycles) are clearly orthogonal  to the commutativity relations across squares in $\I_m^n$. 

Let us check the relations at vertices. Fix the vertex $e_\lambda$. Let $\mu_j$ be the weights obtained from $\lambda$ by exchanging the $\dn \dotsb \up$ pair lying in  circles $C_j$ of $e_\lambda$ as in Proposition \ref{proposition:relations}. Note that for any $\kappa_k$ and $\mu_l$ we have  
\begin{align*}
\bigl\langle y_{\kappa_k}^\lambda x_{\kappa_k}^\lambda - \textstyle\sum\limits_i c_{ \kappa_k}^{\mu_i} x_\lambda^{\mu_i}y_\lambda^{\mu_i} , \; \bar y_\lambda^{\mu_l} \bar x_\lambda^{\mu_l} + \sum\limits_j c_{\kappa_j}^{\mu_l} \bar x_{\kappa_j}^\lambda \bar y_{\kappa_j}^{\lambda} \bigr\rangle &= \sum_j \delta_{\kappa_k, \kappa_j} c_{\kappa_j}^{\mu_l} - \sum_i \delta_{\mu_i, \mu_l} c_{\kappa_k}^{\mu_i} = 0
\end{align*}
where we use $\langle y_{\kappa_k}^\lambda x_{\kappa_k}^\lambda, \bar x_{\kappa_j}^\lambda \bar y_{\kappa_j}^{\lambda} \rangle= \delta_{\kappa_k, \kappa_j}$ and $\langle x_\lambda^{\mu_i}y_\lambda^{\mu_i} , \bar y_\lambda^{\mu_l} \bar x_\lambda^{\mu_l}\rangle = \delta_{\mu_i, \mu_l}$, and we denote $c_{\kappa}^\mu(\lambda) = c_\kappa^\mu$. This verifies the relations at vertex $e_\lambda$.
\end{proof}

The following definition will be useful later.

\begin{definition}\label{definitionhi}
Let $p=\bar y_{\lambda_1}^{\lambda_2} \bar y_{\lambda_2}^{\lambda_3} \dotsc \bar y_{\lambda_k}^{\lambda_{k+1}}$ be an ascending path of length $k \geq 2$ in $\QQ_m^n$.  Assume that the weight $\lambda_{i+1}$ is obtained from $\lambda_{i}$ by exchanging the $\dn \dotsb \up$ pair lying in a  circle $C_{i}$ of $e_{\lambda_{i}}$ for $1\leq  i \leq k$. 

We denote by $h_i$  the integer such that $C_i$ appears in $e_{\lambda_{i-j}}$ for each $0 \leq j \leq h_i$ but $C_i$ does not appear in $e_{\lambda_{i-h_i-1}}$. For instance, if  $h_i = 0$ then $C_i$ appears in $e_{\lambda_i}$ but does not appear in $e_{\lambda_{i-1}}$. If $h_k = k-1$ then $C_k$ appears in $e_{\lambda_j}$ for all $1\leq j \leq k$.
\end{definition}

\begin{lemma}\label{lemma:cubicrelations}
Let $\bar y_{\lambda_1}^{\lambda_2} \bar y_{\lambda_2}^{\lambda_3} \dotsc \bar y_{\lambda_k}^{\lambda_{k+1}}$ be a path of length $k \geq 3$ such that $\lambda_{i+1}$ is obtained from $\lambda_{i}$ by exchanging the $\dn \dotsb \up$ pair lying in a circle $C_{i}$ of $e_{\lambda_{i}}$ for $1\leq  i \leq k$. Assume that the following two conditions hold:
\begin{enumerate}
\item For each $1\leq i < k$ we have $h_i = 0$ and $h_k = k-1$ (see Definition \ref{definitionhi}).
\item $C_k$ is enclosed in $C_1$ in $e_{\lambda_1}$ and lies on the left of $C_i$ in $e_{\lambda_i}$ for $2 \leq i < k$.
\end{enumerate}
Denote by $\mu_{i+1}$ the weight obtained from $\lambda_{i}$ by exchanging the $\dn \dotsb \up$ pair lying in $C_{k}$ of $e_{\lambda_i}$ for $1\leq i\leq k-1$. Then we have the following two statements.

\begin{figure}
\begin{tikzpicture}[baseline=-.5ex,x=.4em,y=.4em,decoration={markings,mark=at position .85 with {\arrow[black]{Stealth[length=3.6pt]}}}]
\begin{scope}
\begin{scope}[shift={(0,0)}]
\node[left,font=\small] at (-1,0) {$\lambda_4$};
\DDOT{0} \DDOT{2} \DDOT{4} \DDOT{6} \DDOT{8} \DDOT{10} \DDOT{12} \DDOT{14} \DDOT{16}
\begin{scope}[yshift=.1em]
\UP{3} \UP{7} \UP{11} \UP{13}
\end{scope}
\begin{scope}[yshift=-.1em]
\DN{1} \DN{5} \DN{9} \DN{15}
\end{scope}
\ROUNDCIRCLE{1}{1} \ROUNDCIRCLE{5}{1} \ROUNDCIRCLE{9}{1} \LARC{13} \RARC{15} 
\end{scope}

\begin{scope}[shift={(-11,-11)}]
\node[left,font=\small] at (-1,0) {$\lambda_3$};
\DDOT{0} \DDOT{2} \DDOT{4} \DDOT{6} \DDOT{8} \DDOT{10} \DDOT{12} \DDOT{14} \DDOT{16}
\begin{scope}[yshift=.1em]
\UP{3} \UP{9} \UP{11} \UP{13}
\end{scope}
\begin{scope}[yshift=-.1em]
\DN{1} \DN{5} \DN{7} \DN{15}
\end{scope}
\ROUNDCIRCLE{1}{1} \CIRCLES{5}{3} \ROUNDCIRCLE{7}{1} \LARC{13} \RARC{15}
\draw[->,line width=.5pt, line cap=round] (8,0) ++(45:3.5) ++(135:.5ex) -- ++(45:9);
\draw[<-,line width=.5pt, line cap=round] (8,0) ++(45:3.5) ++(-45:.5ex) -- ++(45:9);
\end{scope}

\begin{scope}[shift={(11,-11)}]
\node[right,font=\small] at (17,0) {$\mu_3$};
\DDOT{0} \DDOT{2} \DDOT{4} \DDOT{6} \DDOT{8} \DDOT{10} \DDOT{12} \DDOT{14} \DDOT{16}
\begin{scope}[yshift=.1em]
\UP{3} \UP{7} \UP{11} \UP{15}
\end{scope}
\begin{scope}[yshift=-.1em]
\DN{1} \DN{5} \DN{9} \DN{13}
\end{scope}
\ROUNDCIRCLE{1}{1} \ROUNDCIRCLE{5}{1} \ROUNDCIRCLE{9}{1} \ROUNDCIRCLE{13}{1}
\draw[->,line width=.5pt, line cap=round] (8,0) ++(135:3.25) ++(225:.5ex) -- ++(135:9);
\draw[<-,line width=.5pt, line cap=round] (8,0) ++(135:3.25) ++(45:.5ex) -- ++(135:9);
\end{scope}

\begin{scope}[shift={(5,-22)}]
\node[right,font=\small] at (17,0) {$\lambda_2$};
\DDOT{0} \DDOT{2} \DDOT{4} \DDOT{6} \DDOT{8} \DDOT{10} \DDOT{12} \DDOT{14} \DDOT{16}
\begin{scope}[yshift=.1em]
\UP{3} \UP{9} \UP{11} \UP{15}
\end{scope}
\begin{scope}[yshift=-.1em]
\DN{1} \DN{5} \DN{7} \DN{13}
\end{scope}
\ROUNDCIRCLE{1}{1} \CIRCLES{5}{3} \ROUNDCIRCLE{7}{1} \ROUNDCIRCLE{13}{1}
\draw[->,line width=.5pt, line cap=round] (8,0) ++(70:3.5) ++(160:.5ex) -- ++(70:6.25);
\draw[<-,line width=.5pt, line cap=round] (8,0) ++(70:3.5) ++(-20:.5ex) -- ++(70:6.25);
\draw[->,line width=.5pt, line cap=round] (8,0) ++(140:4) ++(240:.5ex) -- ++(150:11.5);
\draw[<-,line width=.5pt, line cap=round] (8,0) ++(140:4) ++(60:.5ex) -- ++(150:11.5);
\end{scope}

\begin{scope}[shift={(-11,-38)}] 
\node[left,font=\small] at (-1,0) {$\mu_2$};
\DDOT{0} \DDOT{2} \DDOT{4} \DDOT{6} \DDOT{8} \DDOT{10} \DDOT{12} \DDOT{14} \DDOT{16}
\begin{scope}[yshift=.1em]
\UP{7} \UP{11} \UP{13} \UP{15}
\end{scope}
\begin{scope}[yshift=-.1em]
\DN{1} \DN{3} \DN{5} \DN{9}
\end{scope}
\CIRCLES{1}{7} \CIRCLES{3}{5} \ROUNDCIRCLE{5}{1} \ROUNDCIRCLE{9}{1}
\draw[->,line width=.5pt, line cap=round] (8,0) ++(90:6.5) ++(160:.5ex) arc[start angle=170, end angle=122, radius=30];
\draw[<-,line width=.5pt, line cap=round] (8,0) ++(90:6.5) ++(-20:.5ex) arc[start angle=170, end angle=122, radius=29];
\end{scope}

\begin{scope}[shift={(0,-50)}]
\node[left,font=\small] at (-1,0) {$\lambda_1$};
\DDOT{0} \DDOT{2} \DDOT{4} \DDOT{6} \DDOT{8} \DDOT{10} \DDOT{12} \DDOT{14} \DDOT{16}
\begin{scope}[yshift=.1em]
\UP{9} \UP{11} \UP{13} \UP{15}
\end{scope}
\begin{scope}[yshift=-.1em]
\DN{1} \DN{3} \DN{5} \DN{7}
\end{scope}
\CIRCLES{1}{7} \CIRCLES{3}{5} \CIRCLES{5}{3} \ROUNDCIRCLE{7}{1}
\draw[->,line width=.5pt, line cap=round] (8,0) ++(135:7) ++(215:.5ex) -- ++(125:2.5);
\draw[<-,line width=.5pt, line cap=round] (8,0) ++(135:7) ++(35:.5ex) -- ++(125:2.5);
\draw[->,line width=.5pt, line cap=round] (8,0) ++(85:6.25) ++(200:.5ex) arc[start angle=200, end angle=135, radius=18];
\draw[<-,line width=.5pt, line cap=round] (8,0) ++(85:6.25) ++(20:.5ex) arc[start angle=200, end angle=135, radius=17];
\end{scope}
\end{scope}
\begin{scope}[shift={(48,0)}] 
\begin{scope}[shift={(0,0)}]
\node[right,font=\small] at (13,0) {$\lambda_4$};
\draw[<-,line width=.5pt, line cap=round] (-1.5,0) ++(105:.5) arc[start angle=105, end angle=270, radius=14.2];
\draw[->,line width=.5pt, line cap=round] (-1.5,0) ++(285:.5) arc[start angle=105, end angle=270, radius=13.2];
\DDOT{0} \DDOT{2} \DDOT{4} \DDOT{6} \DDOT{8} \DDOT{10} \DDOT{12}
\begin{scope}[yshift=.1em]
\UP{3} \UP{5} \UP{9}
\end{scope}
\begin{scope}[yshift=-.1em]
\DN{1} \DN{7} \DN{11}
\end{scope}
\ROUNDCIRCLE{1}{1} \LARC{5} \ROUNDCIRCLE{7}{1} \RARC{11} 
\end{scope}
\begin{scope}[shift={(-9,-9)}]
\node[right,font=\small] at (13,0) {$\lambda_3$};
\DDOT{0} \DDOT{2} \DDOT{4} \DDOT{6} \DDOT{8} \DDOT{10} \DDOT{12}
\begin{scope}[yshift=.1em]
\UP{3} \UP{7} \UP{9}
\end{scope}
\begin{scope}[yshift=-.1em]
\DN{1} \DN{5} \DN{11}
\end{scope}
\ROUNDCIRCLE{1}{1} \ROUNDCIRCLE{5}{1} \LARC{9} \RARC{11}
\draw[->,line width=.5pt, line cap=round] (6,0) ++(45:2.95) ++(135:.5ex) -- ++(45:6.5);
\draw[<-,line width=.5pt, line cap=round] (6,0) ++(45:2.95) ++(-45:.5ex) -- ++(45:6.5);
\end{scope}

\begin{scope}[shift={(0,-18)}]
\node[right,font=\small] at (13,0) {$\lambda_2$};
\DDOT{0} \DDOT{2} \DDOT{4} \DDOT{6} \DDOT{8} \DDOT{10} \DDOT{12}
\begin{scope}[yshift=.1em]
\UP{3} \UP{7} \UP{11}
\end{scope}
\begin{scope}[yshift=-.1em]
\DN{1} \DN{5} \DN{9}
\end{scope}
\ROUNDCIRCLE{1}{1} \ROUNDCIRCLE{5}{1} \ROUNDCIRCLE{9}{1}
\draw[<-,line width=.5pt, line cap=round] (-1.5,0) ++(115:.5) arc[start angle=105, end angle=240, radius=17.5];
\draw[->,line width=.5pt, line cap=round] (-1.5,0) ++(295:.5) arc[start angle=105, end angle=240, radius=16.5];
\draw[->,line width=.5pt, line cap=round] (6,0) ++(135:2.75) ++(225:.5ex) -- ++(135:6.5);
\draw[<-,line width=.5pt, line cap=round] (6,0) ++(135:2.75) ++(45:.5ex) -- ++(135:6.5);
\end{scope}

\begin{scope}[shift={(4.5,-27)}] 
\node[right,font=\small] at (13,0) {$\mu_3'$};
\DDOT{0} \DDOT{2} \DDOT{4} \DDOT{6} \DDOT{8} \DDOT{10} \DDOT{12}
\begin{scope}[yshift=.1em]
\UP{3} \UP{9} \UP{11}
\end{scope}
\begin{scope}[yshift=-.1em]
\DN{1} \DN{5} \DN{7}
\end{scope}
\ROUNDCIRCLE{1}{1} \CIRCLES{5}{3} \ROUNDCIRCLE{7}{1}
\draw[->,line width=.5pt, line cap=round] (6,0) ++(110:2.75) ++(200:.5ex) -- ++(110:4.5);
\draw[<-,line width=.5pt, line cap=round] (6,0) ++(110:2.75) ++(20:.5ex) -- ++(110:4.5);
\end{scope}

\begin{scope}[shift={(-4.5,-37)}] 
\node[right,font=\small] at (13,0) {$\mu_2$};
\DDOT{0} \DDOT{2} \DDOT{4} \DDOT{6} \DDOT{8} \DDOT{10} \DDOT{12}
\begin{scope}[yshift=.1em]
\UP{5} \UP{9} \UP{11}
\end{scope}
\begin{scope}[yshift=-.1em]
\DN{1} \DN{3} \DN{7}
\end{scope}
\CIRCLES{1}{5} \ROUNDCIRCLE{3}{1} \ROUNDCIRCLE{7}{1}
\draw[<-,line width=.5pt, line cap=round] (6,0) ++(45:5.5) ++(-45:.5ex) -- ++(45:5);
\draw[->,line width=.5pt, line cap=round] (6,0) ++(45:5.5) ++(135:.5ex) -- ++(45:5);
\end{scope}

\begin{scope}[shift={(-4.5,-50)}]
\node[right,font=\small] at (13,0) {$\lambda_1$};
\DDOT{0} \DDOT{2} \DDOT{4} \DDOT{6} \DDOT{8} \DDOT{10} \DDOT{12}
\begin{scope}[yshift=.1em]
\UP{7} \UP{9} \UP{11}
\end{scope}
\begin{scope}[yshift=-.1em]
\DN{1} \DN{3} \DN{5}
\end{scope}
\CIRCLES{1}{5} \CIRCLES{3}{3} \ROUNDCIRCLE{5}{1}
\draw[<-,line width=.5pt, line cap=round] (6,4.75) ++(.5ex,0) -- ++(0,3.5);
\draw[->,line width=.5pt, line cap=round] (6,4.75) ++(-.5ex,0) -- ++(0,3.5);
\end{scope}
\end{scope}
\end{tikzpicture}
\caption{Paths of length $3$ in $\QQ_m^n$ which are equal in $\KK_m^n$}
\label{figure:length3}
\end{figure}

If in $e_{\lambda_1}$ there exists a circle between $C_1$ and $C_k$ which encloses $C_k$ (see the left part of Fig.~\ref{figure:length3} for $k=3$) then 
\begin{align*}
\bar y_{\lambda_1}^{\lambda_2} \bar y_{\lambda_2}^{\lambda_3} \dotsc \bar y_{\lambda_k}^{\lambda_{k+1}} - (-1)^{k-1} \bar y_{\lambda_1}^{\mu_2} \bar y_{\mu_2}^{\mu_3} \dotsc \bar y_{\mu_{k-1}}^{\mu_{k}}  \bar y_{\mu_k}^{\lambda_{k+1}} &\in \II_m^n \\
 \bar x_{\lambda_k}^{\lambda_{k+1}} \dotsc \bar x_{\lambda_2}^{\lambda_3}  \bar x_{\lambda_1}^{\lambda_{2}} -  (-1)^{k-1} \bar x_{\mu_k}^{\lambda_{k+1}}  \bar x_{\mu_{k-1}}^{\mu_{k}}\dotsc \bar x_{\mu_2}^{\mu_3}  \bar x_{\lambda_1}^{\mu_{2}} & \in \II_m^n.
\end{align*}
Otherwise, we have (see the right part of Fig.~\ref{figure:length3} for $k = 3$)
\begin{align*}
\bar y_{\lambda_1}^{\lambda_2} \bar y_{\lambda_2}^{\lambda_3} \dotsc \bar y_{\lambda_k}^{\lambda_{k+1}} - (-1)^{k-1} \bar y_{\lambda_1}^{\mu_2} \bar y_{\mu_2}^{\mu_3'} \bar y_{\mu_3'}^{\mu_4} \dotsb \bar y_{\mu_{k-1}}^{\mu_{k}}  \bar y_{\mu_k}^{\lambda_{k+1}}  & \in \II_m^n \\
\bar x_{\lambda_k}^{\lambda_{k+1}} \dotsc \bar x_{\lambda_2}^{\lambda_3}  \bar x_{\lambda_1}^{\lambda_{2}} - (-1)^{k-1} \bar x_{\mu_k}^{\lambda_{k+1}} \bar x_{\mu_{k-1}}^{\mu_{k}} \dotsb \bar x_{\mu_3'}^{\mu_4} \bar x_{\mu_2}^{\mu_3'} \bar x_{\lambda_1}^{\mu_2} & \in \II_m^n
\end{align*}
where we note that there are two circles in $e_{\mu_2}$ which do not appear in $e_{\lambda_1}$ and $\mu_3'$ is obtained from $\mu_2$ by exchanging the $\dn \dotsb \up$ pair lying in the left one of the two circles.
\end{lemma}

\begin{proof}
We only verify the statements for the ascending paths as the statements for descending paths follow by using the involution in Remark \ref{remark:involutionmapkosuzldaul}.

For the first case, note that the weight $\mu_i$ may be also obtained from $\mu_{i-1}$ by exchanging the $\dn \dotsb \up$ pair lying in $C_{i-1}$ so that the vertices $\lambda_{i-1}, \lambda_i, \mu_i, \mu_{i+1}$ form a square in Fig.~\ref{figure:squares} (b) for $2 < i \leq k$ and a square in Fig.~\ref{figure:squares} (a) for $i = 2$. Here, we write $\mu_{k+1} = \lambda_{k+1}$. So we have the anticommutativity relations for $2 \leq i \leq k$
\begin{align} \label{align:commutativitysquareyk}
\bar y_{\lambda_{i-1}}^{\lambda_i} \bar y_{\lambda_i}^{\mu_{i+1}} + \bar y_{\lambda_{i-1}}^{\mu_i} \bar y_{\mu_i}^{\mu_{i+1}} = 0.
\end{align}
Applying these recursively to $\bar y_{\lambda_1}^{\lambda_2} \bar y_{\lambda_2}^{\lambda_3} \dotsc \bar y_{\lambda_k}^{\lambda_{k+1}}$ we obtain the desired equality.

For the second case, the anticommutativity relations \eqref{align:commutativitysquareyk} still hold for $i> 3$. But for $i = 3$ it is replaced by the Plücker-type relation
\begin{align}\label{align:plucker}
\bar y_{\lambda_2}^{\lambda_3} \bar y_{\lambda_3}^{\mu_4} + \bar y_{\lambda_2}^{\mu_3} \bar y_{\mu_3}^{\mu_4} + \bar x_{\mu_3'}^{\lambda_2} \bar y_{\mu_3'}^{\mu_4}= 0.
\end{align}
By the monomial relation  $\bar y_{\lambda_1}^{\lambda_2}  \bar y_{\lambda_2}^{\mu_3}  = 0$ from Fig.~\ref{figure:single} (c) in Proposition \ref{proposition:relationdual}, we have
\begin{align*}
0 &= -\bar y_{\lambda_1}^{\lambda_2} \bar y_{\lambda_2}^{\mu_3} \bar y_{\mu_3}^{\mu_4} \bar y_{\mu_4}^{\mu_5} \dotsb \bar y_{\mu_{k-1}}^{\mu_{k}}   \bar y_{\mu_k}^{\lambda_{k+1}}\\
&=\bar y_{\lambda_1}^{\lambda_2} \bar y_{\lambda_2}^{\lambda_3} \bar y_{\lambda_3}^{\mu_4} \bar y_{\mu_4}^{\mu_5} \dotsb \bar y_{\mu_{k-1}}^{\mu_{k}}  \bar y_{\mu_k} ^{\lambda_{k+1}}  + \bar y_{\lambda_1}^{\lambda_2} \bar x_{\mu_3'}^{\lambda_2} \bar y_{\mu_3'}^{\mu_4} \bar y_{\mu_4}^{\mu_5} \dotsb \bar y_{\mu_{k-1}}^{\mu_{k}}   \bar y_{\mu_k}^{\lambda_{k+1}}\\
&=(-1)^{k-1} \bar y_{\lambda_1}^{\lambda_2} \bar y_{\lambda_2}^{\lambda_3} \bar y_{\lambda_3}^{\lambda_4} \dotsb \bar y_{\lambda_{k-1}}^{\lambda_{k}}  \bar y_{\lambda_k}^{\lambda_{k+1}} - \bar y_{\lambda_1}^{\mu_2} \bar y_{\mu_2}^{\mu_3'} \bar y_{\mu_3'}^{\mu_4} \bar y_{\mu_4}^{\mu_5} \dotsb \bar y_{\mu_{k-1}}^{\mu_{k}}   \bar y_{\mu_k}^{\lambda_{k+1}}
\end{align*}
where the second equality follows from \eqref{align:plucker} and the third one from \eqref{align:commutativitysquareyk} and the relation $ \bar y_{\lambda_1}^{\lambda_2}  \bar x_{\mu_3'}^{\lambda_2} + \bar y_{\lambda_1}^{\mu_2} \bar y_{\mu_2}^{\mu_3'} = 0$.
\end{proof}

\subsection{Reduction systems and Bergman's Diamond Lemma}
\label{subsection:reduction}

In this section we recall the notion of a {\it reduction system} and the closely related Diamond Lemma. Whereas $\K_m^n$ has a diagrammatic $\Bbbk$-basis given by arc diagrams, the Koszul dual $\KK_m^n$ is a priori only defined algebraically. Encoding the relations of $\KK_m^n$ into a reduction system gives a natural algebraic $\Bbbk$-basis of $\KK_m^n$ given by ``irreducible'' paths.

\begin{definition}\label{definition:irreducible}
Let $Q$ be a finite quiver and let $S \subset Q_{\geq 2}$ be a subset of paths of length $\geq 2$ such that for any $s, s' \in S$, $s$ is not a subpath of $s'$. We call a path in $Q$ {\it irreducible} (with respect to $S$) if it does not contain elements in $S$ as subpaths and we denote the set of all irreducible paths by $\mathrm{Irr}_S$.

A {\it reduction system} is given by a set of pairs
\[
R = \{ (s, \varphi_s) \mid s \in S \text{ and } \varphi_s \in \Bbbk \mathrm{Irr}_S \}
\]
such that each path appearing in the linear combination $\varphi_s$ is parallel to $s$.
\end{definition}

Given a reduction system $R = \{ (s, \varphi_s) \}_{s \in S}$, a {\it basic reduction} $\mathbf{r}_{q,s,r}$ is a map $\Bbbk Q \tikzto \Bbbk Q$ determined by
\[
\mathbf{r}_{q,s,r} (p) = 
\begin{cases}
q \varphi_s r & \text{if } p = qsr \\
p & \text{otherwise}
\end{cases}
\]
for any paths $p, q, r \in Q_\ldot$ and any $s \in S$. In other words, $\mathbf{r}_{q,s,r}$ replaces $s$ by $\varphi_s$ in the path $qsr$ and leaves all other paths invariant. A {\it reduction} is any (finite) composition of basic reductions.

\begin{definition}
A reduction system $R = \{ (s, \varphi_s) \}_{s \in S}$ is called
\begin{itemize}
\item {\it reduction-finite} if for any path $p \in Q_\ldot$ and any infinite sequence $(\mathbf r_1, \mathbf r_2, \dotsc)$ of reductions, there exists $n_0 \in \mathbb N$ such that for all $n \geq n_0$ we have $\mathbf r_n \circ \dotsb \circ \mathbf r_1 (p) = \mathbf r_{n_0} \circ \dotsb \circ \mathbf r_1 (p)$
\item {\it reduction-unique} if it is reduction-finite and moreover for any path $p$ and any reductions $\mathbf r, \mathbf r'$ such that $\mathbf r (p), \mathbf r' (p) \in \Bbbk \mathrm{Irr}_S$, we have that $\mathbf r (p) = \mathbf r' (p)$.
\end{itemize}
We say that $R$ satisfies the {\it diamond condition} for a two-sided ideal $I \subset \Bbbk Q$ if $R$ is reduction-unique and $I$ is equal to the ideal generated by the set $\{ s - \varphi_s \}_{s \in S}$.
\end{definition}

\begin{definition}
Let $R = \{ (s, \varphi_s) \}_{s \in S}$ be a reduction system which is reduction-finite. An {\it overlap ambiguity} of $S$ is any path of the form $pqr \in Q_{\geq 3}$ with $p, q, r \in Q_{\geq 1}$ and $pq, qr \in S$.

An overlap ambiguity of $S$ is said to be {\it resolvable} with respect to $R$ if $\mathbf r (p \varphi_{qr}) = \mathbf r' (\varphi_{pq} r)$ for some reductions $\mathbf r, \mathbf r'$.
\end{definition}

The central result about reduction systems is the following Diamond Lemma.

\begin{theorem}[{\cite[Thm.~1.2]{bergman}}]
\label{theorem:diamond}
Let $R = \{ (s, \varphi_s) \}_{s \in S}$ be a reduction system for $\Bbbk Q$ and let $I = \langle s - \varphi_s \rangle_{s \in S}$. If $R$ is reduction-finite, then the following are equivalent:
\begin{enumerate}
\item $R$ is reduction-unique, i.e.\ $R$ satisfies the diamond condition for $I$.
\item All overlap ambiguities of $S$ are resolvable with respect to $R$.
\item The image of the irreducible paths $\mathrm{Irr}_S$ under the projection $\Bbbk Q \tikzto \Bbbk Q / I$ forms a $\Bbbk$-basis of $A = \Bbbk Q / I$. 
\end{enumerate}
\end{theorem}

\subsection{Irreducible paths and Kazhdan--Lusztig polynomials}
\label{subsection:irreduciblepaths}

In this subsection we will define a reduction system $\RR_m^n$ for $\KK_m^n$ and show that it satisfies the diamond condition, i.e.\ it is reduction-unique (cf.\ \S\ref{subsection:reductionsystem}). In \S\ref{section:hochschild} we will see that deformations of $\RR_m^n$ describe A$_\infty$ deformations of $\K_m^n$ and can also be used to compute the Hochschild cohomology of $\K_m^n$.

\begin{definition}[Reduction system for $\KK_m^n$] \label{definition:reductionsystem}
Let $\SS_m^n$ be the set of all paths in $\QQ_m^n$ of the following four types:
\begin{itemize}
\item[(I)] length $2$ paths of the form $\bar y_\lambda^\nu \bar x_\mu^\nu$ for some weights $\lambda, \mu, \nu \in \Lambda_m^n$
\item[(II)] length $2$ paths of the form $\bar y_\lambda^\nu \bar y^\mu_\nu$ or $\bar x^{\mu}_\nu \bar x_{\lambda}^\nu$ in Fig.~\ref{figure:single} (c)
\item[(III)] length $2$ paths of the form $\bar y_\lambda^\nu \bar y_\nu^\mu$ or $\bar x_\nu^\mu \bar x_\lambda^\nu$ in Fig.~\ref{figure:squares}, where $\nu$ is the right vertex in each figure
\item[(IV)] length $\geq 3$ paths of the form $\bar y_{\lambda_1}^{\lambda_2} \bar y_{\lambda_2}^{\lambda_3} \dotsb \bar y_{\lambda_k}^{\lambda_{k+1}}$ or $\bar x_{\lambda_k}^{\lambda_{k+1}} \dotsb \bar x_{\lambda_2}^{\lambda_3} \bar x_{\lambda_1}^{\lambda_{2}}$ as in Lemma \ref{lemma:cubicrelations}
\end{itemize}
where in (I) we allow $\lambda = \mu$.

For each $s \in \SS_m^n$ let $\varphi_s \in \Bbbk \mathrm{Irr}_{\SS_m^n}$ denote the linear combination of the irreducible paths so that $s - \varphi_s \in \II_m^n$. Then we define $\RR_m^n = \{ (s, \varphi_s) \mid s \in \SS_m^n\}$ (see \S\ref{subsection:K22} for a concrete description of $\RR_2^2$). Note that for any path $s \in \SS_m^n$ of type (II) we have $\varphi_s = 0$ by the monomial relations in Proposition \ref{proposition:relationdual}.
\end{definition}

\begin{lemma}
The reduction system $\RR_m^n$ is reduction-finite.
\end{lemma}

\begin{proof}
It suffices to show that reduction-finiteness holds when acting on an arbitrary single path $p$. Let us prove this by induction on the path length of $p$. Clearly this holds if $p$ is a vertex or an arrow. 

Assume that it holds for all paths of length $< m$. Let us prove it for any path of length $m$. Let $p = a_1 \dotsb a_m$ be a length $m$ path, where $a_1, \dotsc, a_m$ are arrows in $\QQ_m^n$, and denote by $e_{\lambda_i}$ the starting vertex of $a_i$ for $1\leq i \leq m$. Let us call a vertex $e_{\lambda_i}$ a {\it peak} along $p$ if $\lvert \lambda_{i-1} \rvert < \lvert \lambda_i \rvert > \lvert \lambda_{i+1} \rvert$, i.e.\ $a_{i-1} = \bar y_{\lambda_{i-1}}^{\lambda_i}$ is an ascending arrow but $a_i = \bar x_{\lambda_{i+1}}^{\lambda_i}$ is descending. Denote by $V$ the set of peaks along $p$. Note that peaks are only changed by reductions of type (I) and reductions of types (II)--(IV) leave the peaks along $p$ invariant.

We have the following two cases. In the first case $V$ is not empty. Let $e_{\lambda_{i_0}}\in V$ be a peak whose height is not smaller than the height of any other peak. By induction, both $a_1 \dotsb a_{i_0-1}$ and $a_{i_0} \dotsb a_m$ are reduction-finite so at some point we have to use the reduction of type (I) on $a_{i_0-1}' a_{i_0}'$ passing through $e_{\lambda_{i_0}}$, where $a_{i_0 - 1}'$ (resp.\ $a_{i_0}'$) is the last (resp.\ first) arrow obtained by performing some reductions on $a_1 \dotsb a_{i_0-1}$ (resp.\ $a_{i_0} \dotsb a_m$). (Note that $a_{i_0-1}'$ and $a_{i_0}'$ may or may not coincide with $a_{i_0 - 1}$ and $a_{i_0}$ but $a_{i_0-1}'$ is necessarily ascending and $a_{i_0}'$ descending.) This reduction causes the heights of the peaks to strictly decrease. Since the heights of vertices are bounded below by $0$, it follows that $p$ must be reduction-finite.

In the second case $V$ is empty. Then we have the following three subcases. In the first subcase, the path $p$ is of the form $\bar x_1 \dotsb \bar x_i \bar y_{i+1} \dotsb \bar y_{m}$ for some $1\leq i \leq m-1$. By induction both $\bar x_1 \dotsb \bar x_i$ and $\bar y_{i+1} \dotsb \bar y_m$ are reduction-finite, hence so is $p$ since $\bar x_i \bar y_{i+1}$ is irreducible. In the second subcase, the path $p=\bar y_1 \dotsb \bar y_m$ is an ascending path. For this, we may give an order on ascending paths as follows. Set $\bar y_\lambda^\nu \prec \bar y_\lambda^{\nu'}$ whenever $\nu, \nu'$ are obtained from $\lambda$ by exchanging two different $\dn \dotsb \up$ pairs lying in circles $C, C'$ of $e_\lambda$, respectively, and either $C, C'$ are nested and $C$ is contained in $C'$ or $C, C'$ are disjoint and $C$ lies to the right of $C'$. Extend $\prec$ by the degree--lexicographic order for longer ascending paths. Then reductions of types (II)--(IV) respect this order in the sense that reductions strictly decrease the order, so that $p$ is reduction-finite since the set of length $m$ paths is finite. Note that a descending arrow may appear in a summand $p'$ when performing the reduction of type (III) from the Plücker-type relations on the ascending path $p$. Then there is a peak along $p'$ so that $p'$ is reduction-finite by the first case. In the third subcase, $p$ is a descending path. Then $p$ is reduction-finite by the second subcase and the fact that the reduction system is invariant under involution.
\end{proof}

\begin{remark}\label{remark:reductionsystem}
If $m = 1$ or $n = 1$ then $\SS_m^n$ only consists of paths $\bar y_\lambda^\nu \bar x_\lambda^\nu$ of type (I). The paths of type (IV) only exist when $m \geq 2$ and $n > 2$. In other words, $\SS_m^2$ consists only of paths of types (I)--(III) so that $\RR_m^2$ is quadratic.
\end{remark}

The following lemma completely describes the ascending irreducible paths in $\KK_m^n$.

\begin{lemma}\label{lemma:irreduciblepathdescription}
Let $p=\bar y_{\lambda_1}^{\lambda_2} \bar y_{\lambda_2}^{\lambda_3} \dotsb \bar y_{\lambda_k}^{\lambda_{k+1}}$ be a path of length $k \geq 2$ such that $\lambda_{i+1}$ is obtained from $\lambda_{i}$ by exchanging the $\dn \dotsb \up$ pair lying in a circle $C_{i}$ of $e_{\lambda_{i}}$. 

Then $p$ is irreducible if and only if for any $i$ with  $h_i \neq 0$ (Definition \ref{definitionhi}), the circle $C_i$ either lies on the left of $C_{i-j}$ or encloses $C_{i-j}$ in $e_{\lambda_{i-j}}$ for each $0< j \leq  h_i$. In particular, if $h_i = 0$ for all $1\leq i \leq k$ then $p$ is irreducible. 
\end{lemma}

\begin{proof}
Let us prove the ``if'' part.  If $p$ is not irreducible then by definition it contains elements in $\SS_m^n$ as subpaths. We have the following two cases. 

In the first case, there is a subpath $\bar y_{\lambda_{i-1}}^{\lambda_{i}} \bar y_{\lambda_{i}}^{\lambda_{i+1}}$ of length $2$ belonging to $\SS_m^n$. Then $\bar y_{\lambda_{i-1}}^{\lambda_{i}} \bar y_{\lambda_{i}}^{\lambda_{i+1}}$ is in Fig.~\ref{figure:squares}, so that $\lambda_{i}$ is the right vertex in each figure. Note that $C_{i}$ appears in $e_{\lambda_{i-1}}$ and it lies on the right of $C_{i-1}$ for the squares (b) and (c) and is enclosed in $C_{i-1}$ for the square (a), giving a contradiction. 

In the second case, there is a subpath $\bar y_{\lambda_j}^{\lambda_{j+1}} \bar y_{\lambda_{j+1}}^{\lambda_{j+2}} \dotsb \bar y_{\lambda_i}^{\lambda_{i+1}}$ of length $i-j+1>2$ belonging to $\SS_m^n$. Then by the assumption in Lemma \ref{lemma:cubicrelations} we have that $C_i$ appears in $e_{\lambda_j}$ and is enclosed in $C_j$ in $e_{\lambda_j}$, giving a contradiction.

Let us prove the \lq\lq only if'' part. Namely, we need to prove that if there exists $1\leq i_0 \leq k$ with $h_{i_0} \neq 0$ so that one of the following conditions holds
\begin{enumerate}
\item \label{conditionirreduciblepath1} $C_{i_0}$ is on the right of $C_{i_0-j}$ for some $0 < j \leq h_{i_0}$
\item \label{conditionirreduciblepath2}  $C_{i_0}$ is enclosed in $C_{i_0-j}$ for some $0 < j \leq h_{i_0}$
\end{enumerate}
 then $p$ is not irreducible. We will prove this by induction on $k$. Clearly it holds for $k = 2$ by the definition of $\SS_m^n$. Assume $k > 2$. If $i_0 < k$ then consider the subpath $p'=\bar y_{\lambda_1}^{\lambda_2} \bar y_{\lambda_2}^{\lambda_3} \dotsb \bar y_{\lambda_{k-1}}^{\lambda_{k}}$ of length $k-1$. Note that the above conditions \ref{conditionirreduciblepath1} and \ref{conditionirreduciblepath2}  still hold for $p'$.  It follows by induction that $p'$ is not irreducible, so neither is $p$. If $i_0 = k$ but $h_k < k-1$ then consider $p''=\bar y_{\lambda_2}^{\lambda_3} \bar y_{\lambda_2}^{\lambda_3} \dotsb \bar y_{\lambda_{k}}^{\lambda_{k+1}}$. Similarly, the conditions  \ref{conditionirreduciblepath1} and \ref{conditionirreduciblepath2} still hold for $p''$. By  induction  $p''$ is not irreducible, so neither is $p$. 
 
 If $i_0 =k$ and $h_k = k-1$, by the induction hypothesis (on the subpath $ \bar y_{\lambda_2}^{\lambda_3} \dotsb \bar y_{\lambda_{k}}^{\lambda_{k+1}}$) we may assume that $C_k$ either lies on the left of $C_i$ or encloses $C_i$  in  $e_{\lambda_i}$ for each $2\leq  i< k$. Consider the first condition  \ref{conditionirreduciblepath1}, namely $C_k$ is on the right of $C_1$ in $e_{\lambda_1}$. Since by assumption  $C_2$ is either on the right of $C_k$ or enclosed in $C_k$, it follows that $C_2$ has to appear in $e_{\lambda_1}$ and lie on the right of $C_1$, whence $\bar y_{\lambda_1}^{\lambda_2} \bar y_{\lambda_2}^{\lambda_3}$ is not irreducible so that neither is $p$. 
 
Consider  \ref{conditionirreduciblepath2}, namely $C_k$ is enclosed in $C_1$. Since by assumption $C_2$ is either on the right of $C_k$ or enclosed in $C_k$ in $e_{\lambda_2}$,  we may then further assume that $C_2$ does not appear in $e_{\lambda_1}$, i.e.\ $h_2 = 0$ (otherwise $C_2$ must be either enclosed in $C_1$ or on the right of $C_1$, whence $\bar y_{\lambda_1}^{\lambda_2} \bar y_{\lambda_2}^{\lambda_3}$ is not irreducible and neither is $p$). Similarly, we may also assume that $h_3 = 0$.  Otherwise $C_3$ must be either enclosed in $C_2$ or on the right of $C_2$ in $e_{\lambda_2}$, whence $\bar y_{\lambda_2}^{\lambda_3} \bar y_{\lambda_3}^{\lambda_4}$ is not irreducible and neither is $p$. By a similar argument we may assume that $h_i = 0$ for all $2\leq i < k$. Then the path $p$ satisfies the assumption in Lemma \ref{lemma:cubicrelations} so that $p \in \SS_m^n$ which is not irreducible. 
\end{proof}

The diamond condition for $\RR_m^n$ will follow from properties of the {\it Kazhdan--Lusztig polynomials} which were first defined geometrically by Kazhdan and Lusztig \cite{kazhdanlusztig} with a closed formula for Grassmannians given by Lascoux and Schützenberger \cite{lascouxschuetzenberger}, see also \cite[\S 5]{brundanstroppel2}. The Kazhdan--Lusztig polynomials for Grassmannians can be defined in terms of the quiver $\QQ_m^n$ for $\KK_m^n$ as follows. 

\begin{definition}
Let $\lambda, \mu\in \Lambda_m^n$ be two weights. For each $k \geq 0$, denote by $a_k$ the number of {\it ascending} irreducible paths of length $k$ from $\lambda$ to $\mu$.  Define a polynomial 
$$
P_{\lambda, \mu}(q) = \sum_{k \geq 0} a_k q^k
$$
where we set $a_k = 0$ if there are no ascending irreducible paths of length $k$.
\end{definition}

Clearly, $P_{\lambda, \mu} (1)$ is the total number of ascending irreducible paths from $\lambda$ to $\mu$. Note that $a_k$ also equals the number of descending irreducible paths of length $k$ from $\mu$ to $\lambda$. 

\begin{proposition}\label{prop:KL}
The following properties uniquely determine $P_{\lambda, \mu}(q)$.
\begin{enumerate}
\item If $\lambda = \mu$ then $P_{\lambda, \mu}(q) = 1$ and if $\lvert \lambda \rvert \geq \lvert \mu \rvert$ for $\lambda \neq \mu$ then $P_{\lambda, \mu}(q) = 0$. 
\item If $\lvert \lambda \rvert < \lvert \mu \rvert$ denote by $C$ the rightmost circle in $e_\lambda$ which does not enclose any other circles. Let $\lambda'$ and $\mu'$ be the weights in $\Lambda_{m-1}^{n-1}$ obtained from $\lambda$ and $\mu$, respectively, by deleting the vertices in the position of $C$ and let $\lambda''$ be the weight in $\Lambda_m^n$ obtained from $\lambda$ by exchanging the $\dn \, \up$ pair in $C$. Then 
\begin{align}
P_{\lambda, \mu}(q) = \begin{cases}
P_{\lambda', \mu'}(q) + q P_{\lambda'', \mu}(q) & \text{if $e_\mu$ contains the circle $C$}\\
q P_{\lambda'', \mu}(q)  & \text{otherwise.}
\end{cases}
\end{align}
\end{enumerate}
As a result, $P_{\lambda, \mu}(q)$ coincides with the Kazhdan--Lusztig polynomials in \cite[\S5]{brundanstroppel2}. 
\end{proposition}

\begin{proof}
The first assertion is clear since if $\lvert \lambda \rvert \geq \lvert \mu \rvert$ then there are no ascending paths from $\lambda$ to $\mu$, except for the ``lazy path'' of length $0$ at $\lambda = \mu$.

Let us prove the second assertion. Assume that $e_\mu$ contains the circle $C$. Let $p=\bar y_{\lambda_1}^{\lambda_2} \bar y_{\lambda_2}^{\lambda_3} \dotsb \bar y_{\lambda_{k}}^{\lambda_{k+1}}$ be an ascending irreducible path from $\lambda$ to $\mu$ such that $\lambda_{i+1}$ be obtained from $\lambda_{i}$ by exchanging the $\dn \dotsb \up$ pair lying in a  circle $C_{i}$ of $e_{\lambda_{i}}$ (denote $\lambda_1 = \lambda$ and $\lambda_{k+1} = \mu$). Then the path is of one of the following two forms:

In the first one, the vertex $e_{\lambda_2}$ contains the circle $C$. Then we claim that $e_{\lambda_j}$ must also contain the circle $C$ for all $1 \leq j \leq k+1$. Indeed, if $j_0>1$ is the smallest integer so that $e_{\lambda_{j_0+1}}$ does not contain $C$ then $\lambda_{j_0+1}$ is obtained from $\lambda_{j_0}$ by exchanging the $\dn \, \up$ pair in $C$, whence $C_{j_0} = C$ and $h_{j_0} = j_0-1\neq 0$. Since $C$ is the rightmost circle in $e_{\lambda_1}$ and in particular  $C_{j_0}$  is on the right of  $C_1$, it follows by Lemma \ref{lemma:irreduciblepathdescription} that the path $p$ is not irreducible. This proves the claim. Thus, the ascending irreducible paths of this form bijectively correspond to the ascending irreducible paths from $\lambda'$ to $\mu'$, since removing $C$ from all $\lambda_i$ the irreducible path $p$ induces an irreducible path in $\KK_{m-1}^{n-1}$ by Lemma \ref{lemma:irreduciblepathdescription} again. 

In the second one, $e_{\lambda_2}$ does not contain $C$ (i.e.\ $\lambda_2$ is obtained from $\lambda_1$ by exchanging the $\dn \, \up$ pair in $C$) then $\lambda_2 = \lambda''$. Since $C$ is the rightmost circle, by Lemma \ref{lemma:irreduciblepathdescription}  the irreducible paths of this form bijectively correspond to the ascending irreducible paths from $\lambda''$ to $\mu$ by composing with $\bar y_{\lambda_1}^{\lambda_2}$.

So putting these two forms together we have  $P_{\lambda, \mu} (q) = P_{\lambda', \mu'}(q) + q P_{\lambda'', \mu}(q)$. 

Assume that $e_{\mu}$ does not contain $C$. Then we claim that $\lambda_2 = \lambda''$. Indeed, if $\lambda_2 \neq \lambda''$ then $C$ appears in $e_{\lambda_2}$. Since $e_{\mu}$ does not contain $C$ there exists $j_0\leq k$ such that $C$ appears in $e_{\lambda_j}$ for each $j \leq j_0$ but does not appear in $e_{\lambda_{j_0+1}}$.  That is, $\lambda_{j_0+1}$ is obtained from $\lambda_{j_0}$ by exchanging the $\dn \dotsb \up$ pair lying in $C$, so $C_{j_0} = C$ and $h_{j_0}=j_0-1$.  Since $C$ is the rightmost circle in $e_{\lambda_1}$ it follows that $C_{j_0}$ is on the right of $C_1$ in $e_{\lambda_1}$ whence $p$ is not irreducible by Lemma \ref{lemma:irreduciblepathdescription}. This proves the claim.  Note that ascending irreducible paths from $\lambda''$ to $\mu$ bijectively induce ascending irreducible paths from $\lambda$ to $\mu$ by composing with $\bar y_\lambda^{\lambda_1}$.  Therefore, we have $P_{\lambda, \mu}(q) = q P_{\lambda'', \mu}(q)$. 

Note that the recursive formulae for $P_{\lambda, \mu}(q)$ proved above coincide with the ones for the (combinatorial) Kazhdan--Lusztig polynomials (see \cite[Lem.~5.2]{brundanstroppel2}).
\end{proof}

\begin{lemma}\label{lemma:irreduciblepath}
The number of irreducible paths from $\lambda$ to $\mu$ is equal to 
$$
\sum_{\kappa\in \Lambda_m^n} P_{\kappa, \lambda}(1) P_{\kappa, \mu}(1).$$
\end{lemma}

\begin{proof}
Note that any irreducible path can be written as $\bar x_1 \bar x_2 \dotsb \bar x_k \bar y_l \dotsb \bar y_2 \bar y_1$ where $\bar x_1 \bar x_2 \dotsb \bar x_k$ (resp.\ $\bar y_l \dotsb \bar y_2 \bar y_1$) is a descending (resp.\ ascending)  irreducible path from $\lambda$ to $\kappa$ (resp.\ from $\kappa$ to $\mu$) for some weight $\kappa$.
\end{proof}

\begin{theorem}\label{theorem:irreduciblediamond}
The irreducible paths form a basis of $\KK_m^n$. As a result, the reduction system $\RR_m^n$ is reduction-unique and therefore satisfies the diamond condition. 
\end{theorem}

\begin{proof}
By \cite[Cor.~5.9]{brundanstroppel1} and Proposition \ref{prop:KL} we have 
$$
\dim_\Bbbk e_\lambda \KK_m^n e_\mu = \sum_{\kappa \in \Lambda_m^n} P_{\kappa, \lambda}(1) P_{\kappa, \mu}(1).
$$
By Lemma \ref{lemma:irreduciblepath} the right-hand side is equal to the number of irreducible paths from $\lambda$ to $\mu$, whence the irreducible paths form a basis of $\KK_m^n$. By the Diamond Lemma (Theorem \ref{theorem:diamond}) it follows that $\RR_m^n$ is reduction-unique. 
\end{proof}

\begin{remark}\label{remark:irreduciblegrading}
Denote by $(\KK_m^n)_i$ the image of the length $i$ paths under the projection $\Bbbk \QQ_m^n \tikzto \KK_m^n$ so that $\KK_m^n = \bigoplus_i (\KK_m^n)_i$. Let $\lambda, \mu \in \Lambda_m^n$ be any two weights. Then $e_\lambda (\KK_m^n)_i e_\mu $ is one dimensional if $i = \lvert \lambda \rvert + \lvert \mu \rvert$ and is zero if $i > \lvert \lambda \rvert + \lvert \mu \rvert$. In particular,  $(\KK_m^n)_{i} = 0 $ if either $i< 0$ or $i> 2mn$ and $(\KK_m^n)_{2mn}$ is of dimension $1$ spanned by the irreducible path $\bar x_0 \bar x_1 \dotsb \bar x_{mn-1} \bar y_{mn-1} \dotsb \bar y_1 \bar y_0$ in Fig.~\ref{figure:counterexample} which is parallel to the vertex corresponding to the highest weight.  

The following observation is useful. Let $p$ be any path of length $\lvert \lambda \rvert +\lvert \mu \rvert$ from $e_\lambda$ to $e_\mu$. Then $\bar y p = 0$ in $\KK_m^n$ for any ascending arrow $\bar y$, since  $\bar y p \in e_\kappa (\KK_m^n)_{\lvert \lambda \rvert +1}  e_{\mu} = 0$, where $e_\kappa$ is the starting vertex of $\bar y$ and note that $\lvert \kappa \rvert < \lvert \lambda \rvert$. Similarly, we have $p \bar x = 0$ in $\KK_m^n$ for any descending arrow $\bar x$. 
\end{remark}

\begin{remark}
\label{remark:koszuldiagrammatic}
It would be very interesting to provide a geometric or diagrammatic description of the Koszul dual $\KK_m^n$, analogous to the diagrammatic description of $\K_m^n$. We refer to Webster \cite[Thm.~3.7]{webster} for an explicit construction of diagrammatic algebras which are Morita equivalent to $\KK_m^n$.
\end{remark}

\subsection{Some useful relations in the Koszul dual}

Since the multiplication in $\KK_m^n$ has no straightforward diagrammatic description, computing the multiplication of arbitrary elements in $\KK_m^n$ is rather complicated. In this section we record some formulae for ``long'' irreducible paths in $\QQ_m^n$ which are used in the proof of our main theorem (Theorem \ref{thm:maintheorem}).

By \cite[Proof of Cor.~1.5]{maksmith} we have $\K_m^n \simeq (\K_n^m)^{\mathrm{op}}$ as graded algebras and thus $\KK_m^n \simeq (\KK_n^m)^{\mathrm{op}}$. In the remainder of this section let us therefore assume that $m \geq n$. 

\begin{notation}
Denote by $\bar x_0 \bar x_1 \dotsb \bar x_{mn-1} \bar y_{mn-1} \dotsb \bar y_1 \bar y_0$ the longest irreducible path in $\KK_m^n$ (see Fig.~\ref{figure:counterexample}). For simplicity, we shall often use the short-hand notation 
\[
\bar x_{0 \ldots mn-1} = \bar x_0 \bar x_1 \dotsb \bar x_{mn-1} \quad \text{ and } \quad \bar y_{mn-1 \ldots 0} = \bar y_{mn-1} \dotsb \bar y_1 \bar y_0
\]
and similar for other sequences of descending or ascending arrows in Fig.~\ref{figure:counterexample}.
\end{notation}

We have the following relations at vertices (see Proposition \ref{proposition:relationdual})
\begin{align}\label{align:relationsatvertex}
\begin{aligned}
\bar y_{ni+j} \bar x_{ni+j} + \bar x_{ni+j+1} \bar y_{ni+j+1} & = 0\\
 \bar y_{ni+n-1} \bar x_{ni+n-1} &= 0 \\
 \bar y_{ni} \bar x_{ni} + \bar x_{ni+1} \bar y_{ni+1} + \bar  x_{ni+1}^{1'} \bar y_{ni+1}^{1'} & = 0\\
 \bar y_{n(m-1)} \bar x_{n(m-1)} +\bar x_{n(m-1)+1} \bar y_{n(m-1)+1} & = 0
\end{aligned}
\end{align}
where $0 \leq i \leq m-1, \ 1 \leq j \leq n-2$ and  $ \bar  x_{ni+1}^{1'},  \bar y_{ni+1}^{1'} $ are depicted in Fig.~\ref{figure:counterexample1}.

\begin{lemma} \label{lemma:higherrelations1}
\begin{enumerate}
\item \label{lemma:higherrelations30} For each $0 \leq i \leq m-1$ and $1 \leq j \leq k \leq n-1$ we have
\begin{align*}
&\bar y_{ni+j} \bar x_{ni+j \ldots ni+k} = \begin{cases}
0 & \text{if $k = n-1$}\\
(-1)^{k-j+1}\bar x_{ni+j+1 \ldots ni+k+1} \bar y_{ni+k+1} & \text{otherwise.}
\end{cases}
\end{align*}
\item \label{lemma:higherrelations31} For each $ 1\leq k \leq n$ we have 
\begin{align*}
\bar y_{n(m-1)} \bar x_{n(m-1) \ldots nm-k} ={} &  \begin{cases} 0 & \text{if $k = 1$}\\
 (-1)^{n-k+1} \bar x_{n(m-1)+1 \ldots nm-k+1} \bar y_{nm-k+1} & \text{if $k > 1$}.
 \end{cases}
\end{align*}
\end{enumerate}
\end{lemma}

\begin{proof}
Let us prove the first assertion. 
By \eqref{align:relationsatvertex} the desired equation holds if $j = k$. The general case follows by induction on $j$ since for $j < k$ we have
\begin{flalign*}
&& \bar y_{ni+j} \bar x_{ni+j \ldots ni+k} &= - \bar x_{ni+j+1} \bar y_{ni+j+1} \bar x_{ni+j+1 \ldots ni+k}. &&
\end{flalign*}
For the second assertion, by \eqref{align:relationsatvertex} we have 
\[
\bar y_{n(m-1)} \bar x_{n(m-1) \ldots nm-k}= -\bar x_{n(m-1)+1} \bar y_{n(m-1)+1} \bar x_{n(m-1)+1 \ldots nm-k}.
\]
Then the desired identity follows from \ref{lemma:higherrelations30}.
\end{proof}

\begin{lemma}\label{lemma:higherrelations3}
For each $0 \leq i \leq m-2$ we have 
\begin{align*}
\bar y_{ni} \bar x_{ni \ldots nm-3} &= -(-1)^{n+m-i} \bar x_{ni+1 \ldots nm-2} \bar y_{nm-2} \\
&\quad\ - \bar x_{ni+3 \ldots n(i+1)+1}^2 \bar x_{n(i+1)+2 \ldots nm-1} \bar y_{nm-1} \bar y_{nm-2}.
\end{align*}
As a result, we have that $\bar y_{ni} \bar x_{ni \ldots nm-1}  = 0$ and 
\begin{align*}
\bar y_{ni} \bar x_{ni \ldots nm-2} =(-1)^{n+m-i} \bar x_{ni+1 \ldots nm-1} \bar y_{nm-1}
\end{align*}
\end{lemma}

\begin{proof}
By the third equality in \eqref{align:relationsatvertex} we obtain
\begin{align}\label{align:lemmahigherrelations3}
\begin{aligned}
\bar y_{ni} \bar x_{ni \ldots nm-3} &= -\bar x_{ni+1} \bar y_{ni+1} \bar x_{ni+1 \ldots nm-3} - \bar  x_{ni+1}^{1'} \bar y_{ni+1}^{1'} \bar x_{ni+1 \ldots nm-3} \\
&= -\bar x_{ni+1}^{1'} \bar y_{ni+1}^{1'} \bar x_{ni+1 \ldots nm-3} \\
&= (-1)^{n} \bar x_{ni+1}^{1'} \bar x_{ni+2 \ldots n(i+1)}^{1} \bar y_{n(i+1)} \bar x_{n(i+1) \ldots nm-3}
\end{aligned}
\end{align}
where the second equality follows from Lemma \ref{lemma:higherrelations1} \ref{lemma:higherrelations30} and the third one uses the anticommutativity relations $\bar y_{ni+j}^{1'} \bar x_{ni+j} = -\bar x_{ni+j+1}^1 \bar y_{ni+j+1}^{1'}$ for $1 \leq j \leq n-1$ (see Fig.~\ref{figure:counterexample1}).  Here we denote  $\bar y_{n(i+1)}^{1'}=\bar y_{n(i+1)}$ when $j = n-1$. 
Since $\bar x_{ni+1}^{1'} \bar x_{ni+2}^{1}$ lies in a square as in Fig.~\ref{figure:squares} (c), it satisfies the Plücker-type relation
$\bar x_{ni+1}^{1'} \bar x_{ni+2}^1 = - \bar x_{ni+1} \bar x_{ni+2}^{1'} - \bar x_{ni+3}^2 \bar y_{ni+3}^{2'}$. Substituting this into \eqref{align:lemmahigherrelations3} we obtain 
\begin{align}\label{align:lemmahigherrelations4}
\begin{aligned}
\bar y_{ni} \bar x_{ni \ldots nm-3} &= (-1)^{n+1} \bar x_{ni+1} \bar x_{ni+2}^{1'}\bar x_{ni+3 \ldots n(i+1)}^{1}\bar y_{n(i+1)} \bar x_{n(i+1) \ldots nm-3} \\
&\quad\ + (-1)^{n+1} \bar x_{ni+3}^2 \bar y_{ni+3}^{2'} \bar x_{ni+3 \ldots n(i+1)}^{1} \bar y_{n(i+1)} \bar x_{n(i+1) \ldots nm-3} \\
&= -\bar x_{ni+1 \ldots n(i+1)} \bar y_{n(i+1)} \bar x_{n(i+1) \ldots nm-3} \\
&\quad\ - \bar x_{ni+3 \ldots n(i+1)+1}^2 \bar y_{n(i+1)+1} \bar y_{n(i+1)} \bar x_{n(i+1) \ldots nm-3}
\end{aligned}
\end{align}
where the second equality uses the anticommutativity relations (see Fig.~\ref{figure:counterexample1})
$$
\bar x_{ni+j}^{1'} \bar x_{ni+j+1}^1=-\bar x_{ni+j} \bar x_{ni+j+1}^{1'} \quad \text{ and } \quad \bar y_{ni+j+1}^{2'} \bar x_{ni+j+1}^1=-\bar x_{ni+j+2}^{2} \bar y_{ni+j+2}^{2'}
$$
where $2\leq j \leq n-1$ and denote  $\bar x_{n(i+1)}^{1'} = \bar x_{n(i+1)}$ and $\bar y_{n(i+1)+1}^{2'} = \bar y_{n(i+1)+1}$. 

Take $i = m-2$ in \eqref{align:lemmahigherrelations4} and apply the identity in Lemma \ref{lemma:higherrelations1} \ref{lemma:higherrelations31} to obtain
\begin{multline*}
\bar y_{n(m-2)} \bar x_{n(m-2) \ldots nm-3} \\
\begin{aligned}
&= (-1)^{n+1} \bar x_{n(m-2)+1 \ldots nm-2} \bar y_{nm-2} \\
&\quad\ + (-1)^{n+1} \bar x_{n(m-2)+3 \ldots n(m-1)+1}^2 \bar y_{n(m-1)+1} \bar x_{n(m-1)+1 \ldots nm-2} \bar y_{nm-2}.
\end{aligned}
\end{multline*}
Then applying Lemma \ref{lemma:higherrelations1} \ref{lemma:higherrelations30} to the term $\bar y_{n(m-1)+1} \bar x_{n(m-1)+1 \ldots nm-2}$ we obtain the desired identity. 
  
For $i < m-2$ from \eqref{align:lemmahigherrelations4} again we have 
\begin{multline*}
\bar y_{ni} \bar x_{ni \ldots nm-3} \\
\begin{aligned}
&= -(-1)^{n+m-i} \bar x_{ni+1 \ldots nm-2} \bar y_{nm-2} \\
&\quad\ + \bar x_{ni+1 \ldots n(i+1)} \bar x_{n(i+1)+3 \ldots n(i+2)+1}^2 \bar x_{n(i+2)+2 \ldots nm-1} \bar y_{nm-1} \bar y_{nm-2} \\
&\quad\ - (-1)^{n+m-i} \bar x_{ni+3 \ldots n(i+1)+1}^2 \bar y_{n(i+1)+1} \bar x_{n(i+1)+1 \ldots nm-2} \bar y_{nm-2} \\
&\quad\ + \bar x_{ni+3 \ldots n(i+1)+1}^2 \bar y_{n(i+1)+1} \bar x_{n(i+1)+3 \ldots n(i+2)+1}^2 \bar x_{n(i+2)+2 \ldots nm-1} \bar y_{nm-1} \bar y_{nm-2}\\
&= -(-1)^{n+m-i} \bar x_{ni+1 \ldots nm-2} \bar y_{nm-2} \\
&\quad\ - \bar x_{ni+3 \ldots n(i+1)+1}^2 \bar x_{n(i+1)+2 \ldots nm-1} \bar y_{nm-1} \bar y_{nm-2}
\end{aligned}
\end{multline*} 
where the first equality uses the induction (we assume it holds for $i+1$ and then prove it for $i$). Let us explain the second equality. Using  $\bar x_{n(i+1)} \bar x_{n(i+1)+3}^2 =0$ and  Lemma \ref{lemma:higherrelations1} \ref{lemma:higherrelations30} we see that the second and third summands vanish. For the  fourth summand, we are abusing notation since $\bar x^2_{n(i+1)+3}$ denotes the arrow which has the same starting vertex as $\bar x_{n(i+1)+1}$ (just similar to the arrows $\bar x^2_{ni+3}$ and $\bar x_{ni+1}$).  For this summand, we use the following identity (see Fig.~\ref{figure:counterexample1} where we shift $i$ by $i+1$)
\begin{align}\label{align:equationinlemma:higherrelations3}
\begin{aligned}
 \bar y_{n(i+1)+1} \bar x_{n(i+1)+3 \ldots n(i+2)+1}^2 \bar x_{n(i+2)+2 \ldots nm-1} = -  \bar x_{n(i+1)+2 \ldots nm-1}
 \end{aligned}
 \end{align}
 which may be obtained by using the anticommutativity relations iteratively ($2n-4$ times) and the Plücker-type relation $\bar x_{n(i+1)+3}^{2'} \bar x_{n(i+1)+4}^2 = \bar x_{n(i+1)+3}^1\bar x_{n(i+1)+4}^{2'} + \bar x'\bar y''$. Note that the summand corresponding to the term $\bar x' \bar y''$ vanishes since $\bar y'' \bar x^2_{n(i+1)+5 \ldots n(i+2)+1} \bar x_{n(i+2)+2 \ldots nm-1} = 0$ by the observation in Remark \ref{remark:irreduciblegrading}.
\end{proof}

\begin{lemma} \label{lemma:higherrelations4}
For each $0 \leq i \leq m-2$ we have (see Fig.~\ref{figure:counterexample})
\begin{align*}
 \bar y_{ni} \bar x_{ni \ldots n(m-1)-2} \bar x_{n(m-1)-1 \ldots nm-3}^0  
= \bar x_{ni+1 \ldots nm-2} \bar y_{nm-2}'.
\end{align*}
As a result we have 
\begin{multline*}
\bar y_{ni} \bar x_{ni \ldots n(m-1)-2} \bar x_{n(m-1)-1 \ldots nm-3}^0 \bar y_{nm-3 \ldots n(m-1)-1}^0\bar y_{n(m-1)-2 \ldots n(m-2)}  \\
= - (-1)^{n} \bar x_{ni+1 \ldots nm-2} \bar y_{nm-2 \ldots n(m-2)}.
\end{multline*}
\end{lemma}

\begin{proof}
Let us first consider the case of $i = m-2$. We have 
\begin{multline*}
\bar y_{n(m-2)} \bar x_{n(m-2) \ldots n(m-1)-2} \bar x_{n(m-1)-1 \ldots nm-3}^0 \\
\begin{aligned}
&= - \bar x_{n(m-2)+1} \bar y_{n(m-2)+1} \bar x_{n(m-2)+1 \ldots n(m-1)-2} \bar x_{n(m-1)-1 \ldots nm-3}^0 \\
&\quad\ - \bar x_{n(m-2)+1}^{1'}\bar y_{n(m-2)+1}^{1'}  \bar x_{n(m-2)+1 \ldots n(m-1)-2} \bar x_{n(m-1)-1 \ldots nm-3}^0 \\
&= (-1)^{n-1} \bar x_{n(m-2)+1 \ldots n(m-1)-1} \bar y_{n(m-1)-1} \bar x_{n(m-1)-1 \ldots nm-3}^0 \\
&\quad\ + (-1)^{n-1} \bar x_{n(m-2)+1}^{1'} \bar x_{n(m-2)+2 \ldots n(m-1)-1}^{1} \bar y_{n(m-1)-1}^0 \bar x_{n(m-1)-1 \ldots nm-3}^0 \\ 
&= \bar x_{n(m-2)+1} \bar x_{n(m-2)+2 \ldots nm-2} \bar y_{nm-2}'
\end{aligned}
\end{multline*}
where the first equality follows from \eqref{align:relationsatvertex} and the second one from Lemma \ref{lemma:higherrelations1} \ref{lemma:higherrelations30} and $\bar y_{n(m-2)+j}^{1'} \bar x_{n(m-2)+j}=-\bar x_{n(m-2)+j+1}^1 \bar y_{n(m-2)+j+1}^{1'}$ for $1\leq j \leq n-2$ (denote $\bar y_{n(m-1)-1}^{1'}= \bar y_{n(m-1)-1}^0$ for $j = n-2$), and the third one  from the anticommutativity relations involving $\bar y_{n(m-1)-j}\bar x^0_{n(m-1)-j}$ ($n-3$ times) and the following identity
\begin{align*}
\bar y_{n(m-1)-1}^0  \bar x_{n(m-1)-1 \ldots nm-3}^0  = (-1)^{n-2}\bar x_{n(m-1) \ldots nm-3}^0 \bar y_{nm-3}^0 \bar x_{nm-3}^0 = 0
\end{align*}
which uses $\bar y_{n(m-1)+j-1}^0 \bar x_{n(m-1)+j-1}^0 = - \bar x_{n(m-1)+j}^0 \bar y_{n(m-1)+j}^0$ for $0 \leq j \leq  n-3$ and $\bar y_{nm-3}^0 \bar x_{nm-3}^0  = 0$. 

For $i < m-2$ we may proceed by induction since similar to \eqref{align:lemmahigherrelations4} we have
\begin{multline*}
\bar y_{ni} \bar x_{ni \ldots n(m-1)-2} \bar x_{n(m-1)-1 \ldots nm-3}^0 \\
\begin{aligned}
&= - \bar x_{ni+1 \ldots n(i+1)} \bar y_{n(i+1)} \bar x_{n(i+1) \ldots n(m-1)-2} \bar x_{n(m-1)-1 \ldots nm-3}^0\\
&\quad\ -\bar x_{ni+3 \ldots n(i+1)+1}^2 \bar y_{n(i+1)+1} \bar y_{n(i+1)} \bar x_{n(i+1) \ldots n(m-1)-2} \bar x_{n(m-1)-1 \ldots nm-3}^0
\end{aligned}
\end{multline*}
where we also need to use Lemma \ref{lemma:higherrelations1} \ref{lemma:higherrelations30} to show that the second summand vanishes. This proves the first desired identity. 

To prove the second desired identity, it suffices to prove the following one 
\begin{align*}
\bar y_{nm-2 \ldots n(m-1)-1}^0 \bar y_{n(m-1)-2 \ldots n(m-2)} = (-1)^{n-1}  \bar y_{nm-2 \ldots n(m-2)} .
\end{align*}
For this, by the Plücker-type relation $\bar y_{nm-2}'  \bar y_{nm-3}^0 =- \bar y_{nm-2}  \bar y'_{nm-3}- \bar x_{nm-1} \bar y'''$, where $\bar y'''$ is the missing arrow connecting the bottom vertex with the ending vertex of $\bar y_{nm-3}^0$ in Fig.~\ref{figure:counterexample}, it follows that 
\begin{multline*}
\bar y_{nm-2}' \bar y_{nm-3 \ldots n(m-1)-1}^0 \bar y_{n(m-1)-2 \ldots n(m-2)} \\
\begin{aligned}
&= -\bar y_{nm-2} \bar y'_{nm-3} \bar y_{nm-4 \ldots n(m-1)-1}^0 \bar y_{n(m-1)-2 \ldots n(m-2)} \\
&\quad\ - \bar x_{nm-1} \bar y''' \bar y_{nm-4 \ldots n(m-1)-1}^0 \bar y_{n(m-1)-2 \ldots n(m-2)}\\
&= (-1)^{n-1} \bar y_{nm-2 \ldots n(m-2)}
\end{aligned}
\end{multline*}
where the second equality repeatedly uses the anticommutativity relations involving $\bar y'_{nm-i}\bar y_{nm-i-1}^0$ for $3 \leq i \leq n$. Note that $\bar y'''  \bar y_{nm-4 \ldots n(m-1)-1}^0 \bar y_{n(m-1)-2} $ is in $\SS_m^n$ and using Lemma \ref{lemma:cubicrelations} iteratively ($n-2$ times for $k = n, n-1, \dotsc, 3$) we obtain 
\begin{multline*}
\bar x_{nm-1} \bar y''' \bar y_{nm-4 \ldots n(m-1)-1}^0 \bar y_{n(m-1)-2 \ldots n(m-2)} \\
= (-1)^{\frac{(n-2)(n+1)}{2}} \bar x_{nm-1} \bar y_{nm-1} \bar y_{nm-2}' \dotsb \bar y_{n(m-2)+3}^2 \bar y_{n(m-2)} = 0
\end{multline*}
where we do not need the explicit arrows appearing in $\dotsb$ since the second equality already follows from $\bar y_{n(m-2)+3}^2 \bar y_{n(m-2)} = 0$.
\end{proof}

\begin{figure} 
\begin{tikzpicture}[x=4em,y=4em,decoration={markings,mark=at position 0.99 with {\arrow[black]{Stealth[length=3pt]}}}]
\begin{scope}[scale=1, shift={(0,0)}]
\begin{scope}[shift={(0,0)}]
\node[shape=ellipse,minimum height=2.2em,minimum width=2.8em] (a0) at (0,0) {};
\node[shape=ellipse,minimum height=2.5em,minimum width=3.2em] (a1) at (0,-1) {};
\node[shape=ellipse,minimum height=2.8em,minimum width=3.4em] (a2) at (0,-2) {};
\node[shape=ellipse,minimum height=3.1em,minimum width=3.7em] (a3) at (0,-3) {};
\path[-, line width=.5pt, line cap=round] (a0) edge node[right=-.4ex,font=\scriptsize] {$\bar y_{ni}$} (a1);
\path[-, line width=.5pt, line cap=round] (a1) edge node[right=-.4ex,font=\scriptsize] {$\bar y_{ni+1}$} (a2);
\path[-, line width=.5pt, line cap=round] (a2) edge node[right=-.4ex,font=\scriptsize] {$\bar y_{ni+2}$} (a3);
\node[shape=ellipse,minimum height=2.2em,minimum width=3.4em] (b1) at (2,-2) {};
\node[shape=ellipse,minimum height=2.2em,minimum width=3.5em] (b2) at (2,-3) {};
\node[shape=ellipse,minimum height=2.8em,minimum width=3.7em] (b3) at (2,-4) {};
\path[-, line width=.5pt, line cap=round] (b1) edge node[below=-.4ex,font=\scriptsize,pos=.4] {$\bar y_{ni+1}^{1'}$} (a1);
\path[-, line width=.5pt, line cap=round] (b2) edge node[below=-.4ex,font=\scriptsize,pos=.4] {$\bar y_{ni+2}^{1'}$} (a2);
\path[-, line width=.5pt, line cap=round] (b3) edge node[below=-.4ex,font=\scriptsize,pos=.4] {$\bar y_{ni+3}^{1'}$} (a3);
\node[font=\scriptsize] at (3.85,-2.85) {$\bar y_{ni+3}^2$};
\path[-, line width=.5pt, line cap=round] (b1) edge node[right=-.4ex,font=\scriptsize] {$\bar y_{ni+2}^1$} (b2);
\path[-, line width=.5pt, line cap=round] (b2) edge node[right=-.35ex,font=\scriptsize] {$\bar y_{ni+3}^1$} (b3);
\node[shape=ellipse,minimum height=2.4em,minimum width=3.3em] (c2) at (4,-4) {};
\node[shape=ellipse,minimum height=3em,minimum width=3.8em] (c3) at (4,-5) {};
\path[-, line width=.5pt, line cap=round] (c2) edge node[below=-.4ex,font=\scriptsize] {$\bar y_{ni+3}^{2'}$} (b2);
\path[-, line width=.5pt, line cap=round] (c3) edge node[below=-.4ex,font=\scriptsize] {$\bar y_{ni+4}^{2'}$} (b3);
\path[-, line width=.5pt, line cap=round] (c2) edge node[right=-.4ex,font=\scriptsize] {$\bar y_{ni+4}^2$} (c3);
\path[-,line width=.5pt, line cap=round,in=-10,out=110,looseness=1.1] (c2) edge (a1.-20);
\begin{scope}[shift={(-.945em,0)},x=.27em,y=.45em]
\DN{0} \DN{1} \DDOT{2} \DN{3} \UP{4} \DDOT{5} \UP{6}  \DN{7}
\MRAY{0} \SCIRCLES{1}{2.5} \CIRCLE{3}  \MRAY{7}
\node[font=\scriptsize] at (9,0) {$=$};
\DN{11}  \DN{14} \DN{15} \DDOT{16} \DN{17} \UP{18} \DDOT{19}  \UP{20} \DN{21}  \DN{24}
\node[font=\scriptsize] at (12.5,0) {.\hspace{-.3pt}.\hspace{-.3pt}.};
\node[font=\scriptsize] at (22.5,0) {.\hspace{-.3pt}.\hspace{-.3pt}.};
\RAY{11} \RAY{14} \SCIRCLES{15}{2.5} \CIRCLE{17} \RAY{21}\RAY{24}
\node[font=\small] at (22.5,1.8) {\rotatebox{-90}{$\{$}};
\node[font=\scriptsize] at (22.5,3.1) {\strut$_{m-i}$};
\end{scope}
\begin{scope}[shift={(-1.215em,-1)},x=.27em,y=.45em]
\DN{0} \DN{1} \DDOT{2} \DN{3} \UP{4} \DDOT{5} \UP{6}  \DN{7} \UP{8} \DN{9}
\MRAY{0} \SCIRCLES{1}{2.5} \CIRCLE{3} \CIRCLE{7} \MRAY{9}
\begin{scope}[xshift=3]
\node[font=\scriptsize] at (10,0) {$=$};
\DN{12}  \DN{15} \DN{16} \DDOT{17} \DN{18} \UP{19} \DDOT{20}  \UP{21} \DN{22}  \UP{23} \DN{24} \DN{27}
\node[font=\scriptsize] at (13.5,0) {.\hspace{-.3pt}.\hspace{-.3pt}.};
\node[font=\scriptsize] at (25.5,0) {.\hspace{-.3pt}.\hspace{-.3pt}.};
\RAY{12} \RAY{15} \SCIRCLES{16}{2.5} \CIRCLE{18} \CIRCLE{22} \RAY{24}\RAY{27} 
\node[font=\small] at (25.5,1.8) {\rotatebox{-90}{$\{$}};
\node[font=\scriptsize] at (25.5,3.1) {\strut$_{m-i-1}$};
\end{scope}
\end{scope}
\begin{scope}[shift={(-1.485em,-2)},x=.27em,y=.45em]
\DN{0} \DN{1} \DN{2} \DDOT{3}\DN{4} \UP{5} \DDOT{6} \UP{7} \DN{8} \UP{9} \UP{10} \DN{11} 
\MRAY{0} \SCIRCLES{1}{4.5} \SCIRCLES{2}{2.5} \CIRCLE{4} \CIRCLE{8} \MRAY{11} 
\end{scope}
\begin{scope}[shift={(-1.775em, -3)},x=.27em,y=.45em]
\DN{0} \DN{1} \DN{2} \DN{3} \DDOT{4} \DN{5} \UP{6} \DDOT{7} \UP{8} \DN{9} \UP{10} \UP{11} \UP{12}\DN{13}
\MRAY{0} \SCIRCLES{1}{5.5} \SCIRCLES{2}{4.5} \SCIRCLES{3}{2.5} \CIRCLE{5} \CIRCLE{9} \MRAY{13}
\end{scope}
\begin{scope}[shift={(6.65em,-2)},x=.27em,y=.45em]
\DN{0} \DN{1} \DDOT{2} \DN{3} \UP{4} \DDOT{5} \UP{6}  \DN{7} \DN{8} \UP{9} \DN{10}
\MRAY{0} \SCIRCLES{1}{2.5} \CIRCLE{3}\RAY{7} \CIRCLE{8} \MRAY{10}
\end{scope}
\begin{scope}[shift={(6.515em,-3)},x=.27em,y=.45em]
\DN{0} \DN{1} \DDOT{2} \DN{3} \UP{4} \DDOT{5} \UP{6}  \DN{7} \UP{8} \DN{9} \UP{10} \DN{11}
\MRAY{0} \SCIRCLES{1}{2.5} \CIRCLE{3}\CIRCLE{7} \CIRCLE{9} \MRAY{11}
\end{scope}
\begin{scope}[shift={(6.245em,-4)},x=.27em,y=.45em]
\DN{0} \DN{1} \DN{2} \DDOT{3} \DN{4} \UP{5} \DDOT{6} \UP{7} \DN{8} \UP{9} \UP{10} \DN{11} \UP{12} \DN{13}
\MRAY{0} \SCIRCLES{1}{4.5} \SCIRCLES{2}{2.5} \CIRCLE{4} \CIRCLE{8} \CIRCLE{11} \MRAY{13}
\end{scope}
\begin{scope}[shift={(14.515em,-4)},x=.27em,y=.45em]
\DN{0} \DN{1} \DDOT{2} \DN{3} \UP{4} \DDOT{5} \UP{6} \DN{7} \DN{8} \UP{9} \UP{10} \DN{11}
\MRAY{0} \SCIRCLES{1}{2.5} \CIRCLE{3} \SCCIRCLE{7} \CIRCLE{8} \MRAY{11}
\end{scope}
\begin{scope}[shift={(14.245em,-5)},x=.27em,y=.45em]
\DN{0} \DN{1} \DN{2}  \DDOT{3} \DN{4} \UP{5} \DDOT{6} \UP{7} \DN{8} \UP{9} \DN{10} \UP{11} \UP{12} \DN{13}
\MRAY{0} \SCIRCLES{1}{5.5} \SCIRCLES{2}{2.5} \CIRCLE{4} \CIRCLE{8}\CIRCLE{10} \MRAY{13}
\end{scope}
\end{scope}

\begin{scope}[shift={(0,-4.33)}]
\node[shape=ellipse,minimum height=2.6em,minimum width=3.4em] (e0) at (0,0) {};
\node[shape=ellipse,minimum height=2.2em,minimum width=2.6em] (e1) at (0,-1) {};
\node[shape=ellipse,minimum height=2.2em,minimum width=3.1em] (e2) at (0,-2) {};
\node[shape=ellipse,minimum height=2.6em,minimum width=3.8em] (e3) at (0,-3) {};
\path[-, line width=.5pt, line cap=round] (e0) edge node[right=-.4ex,font=\scriptsize,pos=.6] {$\bar y_{n(i+1)-1}$} (e1);
\path[-, line width=.5pt, line cap=round] (e1) edge node[right=-.4ex,font=\scriptsize] {$\bar y_{n(i+1)}$} (e2);
\path[-, line width=.5pt, line cap=round] (e2) edge node[right=-.4ex,font=\scriptsize] {$\bar y_{n(i+1)+1}$} (e3);
\node[shape=ellipse,minimum height=2.9em,minimum width=3.7em] (f0) at (2, -1) {};
\node[shape=ellipse,minimum height=3.2em,minimum width=4.8em] (g0) at (4, -2) {};
\path[-, line width=.5pt, line cap=round] (f0) edge (e0);
\path[-, line width=.5pt, line cap=round] (f0) edge (g0);
\path[-, line width=.5pt, line cap=round] (g0) edge node[below=.4ex,font=\scriptsize] {$\bar y_{n(i+1)+1}^2$} (e3);
\path[-, line width=.5pt, line cap=round] (f0) edge node[below=.5ex,font=\scriptsize,pos=.3] {$\bar y_{n(i+1)}^1$} (e2);
\draw[dotted, line width=.6pt] (a3) -- (e0);
\draw[dotted, line width=.6pt] (f0) -- (b3);
\draw[dotted, line width=.6pt] (g0) -- (c3);
\begin{scope}[shift={(-1.485em,0)},x=.27em,y=.45em]
\DN{0} \DN{1} \DDOT{2} \DN{3} \DN{4} \UP{5} \DN{6} \UP{7} \UP{8} \DDOT{9} \UP{10} \DN{11}
\MRAY{0} \SCIRCLES{1}{4.5} \SCIRCLES{3}{2.5} \CIRCLE{4} \CIRCLE{6} \MRAY{11}
\end{scope}
\begin{scope}[shift={(-.945em,-1)},x=.27em,y=.45em]
\DN{0} \DN{1} \DDOT{2}  \DN{3} \UP{4} \DDOT{5} \UP{6} \DN{7}
\MRAY{0} \SCIRCLES{1}{2.5} \CIRCLE{3} \MRAY{7}
\end{scope}
\begin{scope}[shift={(-1.215em,-2)},x=.27em,y=.45em]
\DN{0} \DN{1} \DDOT{2} \DN{3}\UP{4} \DDOT{5} \UP{6} \DN{7} \UP{8} \DN{9}
\MRAY{0} \SCIRCLES{1}{2.5} \CIRCLE{3} \CIRCLE{7} \MRAY{9}
\end{scope}
\begin{scope}[shift={(-1.485em,-3)},x=.27em,y=.45em]
\DN{0} \DN{1} \DN{2} \DDOT{3} \DN{4}\UP{5} \DDOT{6} \UP{7} \DN{8} \UP{9} \UP{10} \DN{11}
\MRAY{0} \SCIRCLES{1}{4.5} \SCIRCLES{2}{2.5} \CIRCLE{4} \CIRCLE{8} \MRAY{11}
\end{scope}
\begin{scope}[shift={(6.245em,-1)},x=.27em,y=.45em]
\DN{0} \DN{1} \DDOT{2} \DN{3} \DN{4} \UP{5} \DN{6} \UP{7} \UP{8} \DDOT{9} \UP{10}  \DN{11}\UP{12} \DN{13}
\MRAY{0} \SCIRCLES{1}{4.5}\SCIRCLES{3}{2.5} \CIRCLE{4} \CIRCLE{6} \CIRCLE{11} \MRAY{13}
\end{scope}
\begin{scope}[shift={(13.975em,-2)},x=.27em,y=.45em]
\DN{0} \DN{1} \DN{2} \DDOT{3} \DN{4} \DN{5}\UP{6}\DN{7} \UP{8} \UP{9} \DDOT{10} \UP{11} \DN{12}\UP{13} \UP{14}\DN{15}
\MRAY{0} \SCIRCLES{1}{6.5} \SCIRCLES{2}{4.5} \SCIRCLES{4}{2.5} \CIRCLE{12} \CIRCLE{5}\CIRCLE{7}\MRAY{15}
\end{scope}
\end{scope}
\end{scope}
\end{tikzpicture}
\caption{$n-1\leq i \leq m-2$}
\label{figure:counterexample1}
\end{figure}
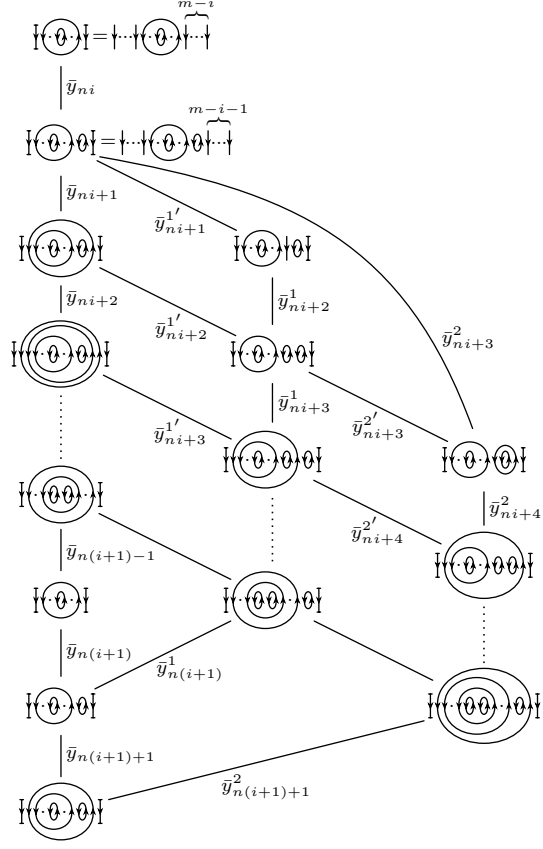

\section{\texorpdfstring{%
Hochschild cohomology and deformations for Koszul-dual algebras}{%
Hochschild cohomology and deformations for Koszul-dual algebras}}
\label{section:hochschild}

Let $A$ be a graded algebra. The Hochschild cochain complex $\mathrm C^\hdot(A, A)$ is the product total complex of the double complex 
\begin{align}
A \tikztoarg{\delta^0} \Hom(A, A) \tikzto \dotsb \tikzto \Hom(A^{\otimes p}, A) \tikztoarg{\delta^p} \Hom(A^{\otimes p+1}, A)  \tikzto\dotsb
\end{align}
where the horizontal differential $\delta^p$ is the Hochschild differential and each term $\Hom (A^{\otimes i}, A) = \bigoplus_j \Hom (A^{\otimes i}, A)^j$ is viewed as a graded vector space with trivial (vertical) differential, where $\Hom (A^{\otimes i}, A)^j$ is the set of graded $\Bbbk$-linear maps of degree $j$ from $A^{\otimes i}$ to $A$ so that we have $\mathrm C^k (A, A) = \prod_{i} \Hom(A^{\otimes i}, A)^{k-i}$.

\subsection{Bigraded Hochschild cohomology and Koszul duality} 

Let $A = \Bbbk Q/ I$ be a Koszul algebra and $A^! = \Bbbk \overline Q / (I_2^{\perp})$ be its Koszul dual. In this subsection we recall from Keller \cite{keller3} the isomorphism between the (bigraded) Hochschild cohomology of $A$ and $A^!$.   

\subsubsection{Bigraded vector spaces} 

Consider bigraded vector spaces $
V = \bigoplus_{p, q\in \mathbb Z} V^p_q $
where the superscript $^p$ denotes the (cohomological) {\it differential grading} and the subscript $_q$ denotes the {\it Adams grading} which is viewed as an additional grading. Let $(V, d)$ be a {\it differential bigraded vector space}, where $d \colon V \tikzto V$ is a differential of bidegree $(1, 0)$ giving for each fixed $q \in \mathbb Z$ a cochain complex $V^\hdot_q$
\[
\dotsb \tikzto V_q^p \tikztoarg{d} V_q^{p+1} \tikztoarg{d} V_q^{p+2}\tikzto \dotsb
\]

For any two differential bigraded vector spaces $(U, d_U), (V, d_V)$, consider the differential bigraded Hom-space $\Hom (U, V) = \bigoplus_{p, q \in \mathbb Z} \Hom (U, V)^p_q$ where
$$
\Hom (U, V)^p_q := \prod_{i, j\in \mathbb Z} \Hom(U^i_j, V^{p+i}_{q+j})
$$
with differential 
$\delta(f) = d_V\circ f - (-1)^p f \circ d_U$ for any $f \in  \Hom(U, V)^p_q$. 
We may also consider the bigraded tensor product $U \otimes V= \bigoplus_{p, q \in \mathbb Z} (U \otimes V)^p_q$ where 
$$
(U \otimes V)^p_q: = \bigoplus_{i, j \in \mathbb Z} U^i_j \otimes V^{p-i}_{q-j}
$$
with differential $d(u \otimes v) = d_U(u) \otimes v + (-1)^{k} u \otimes d_V(v)$ for any $u\otimes v \in U^k_l \otimes V^m_n$.

Let $A = \bigoplus_{k \in \mathbb Z} A^k$ be a graded algebra with additional Adams grading so that $A^k = \bigoplus_{l \in \mathbb Z} A^k_l$. We denote the Hochschild cochain complex with this additional grading by $\mathrm C^\hdot_\ldot (A, A)$ where $\mathrm C^p_q (A, A) = \prod_{i \geq 0} \Hom(A^{\otimes i}, A)^{p-i}_q$ with the induced differential. For each $q$ let $\HH_q^\hdot(A, A)$ denote the cohomology of the complex $\mathrm C^\hdot_q (A, A)$. The B$_\infty$ structure on $\mathrm C^\hdot (A, A)$ can be restricted to the subspace
$$
\bigoplus_{q \in \mathbb Z} \mathrm C^\hdot_q (A, A) \subset \mathrm C^\hdot(A, A)
$$
but the inclusion is not necessarily a (quasi-)isomorphism of complexes, even if $A$ is finite-dimensional (see Remark \ref{remark:example} below). 

\subsubsection{Hochschild cohomology under Koszul duality}

Let $A = \Bbbk Q/I$ be a Koszul algebra. We view $A$ as bigraded by assigning any arrow $a \in Q$ the bidegree $(0,1)$. Recall that the Koszul dual $A^! $ is isomorphic to $\Bbbk \overline Q / (I_2^{\perp})$. It also admits a bigrading by assigning any arrow $\bar a \in \overline Q$ the bidegree $(1,-1)$.

\begin{theorem}[{\cite[\S3.5]{keller3}}]\label{theorem:keller}
There is an isomorphism in the homotopy category of (bigraded) B$_\infty$ algebras
\[
\mathrm C^\hdot_\ldot(A, A) \tikzto \mathrm C^\hdot_\ldot (A^!, A^!).
\]
In particular, we have $\HH^p_q (A^!, A^!) \simeq \HH^p_q (A, A)$ for all $p, q$.
\end{theorem}

\begin{remark}\label{remark:example}
The additional Adams grading is crucial in Theorem \ref{theorem:keller}, as otherwise there may be some issues with completions as discussed in \cite{keller4,positselski}. For instance, let $A = \Bbbk[x]$ be the polynomial algebra. Its Koszul dual $A^! = \Bbbk[y]/(y^2)$  is the $2$-dimensional algebra of dual numbers. Take the differential gradings $\lvert x \rvert = 0$ and $\lvert y \rvert = 1$. Then $\HH^0 (A, A) \simeq \Bbbk[z]$ whereas $\HH^0(A^!, A^!) \simeq \Bbbk \llbracket z\rrbracket$ which are not isomorphic (see \cite[\S1]{keller4}). However, the Adams grading gives isomorphisms $\HH^0_q(A, A) \simeq \Bbbk z^q \simeq \HH^0_q(A^!, A^!)$ for each $q \geq 0$. In fact, we have
$$
\HH^0(A, A) = \bigoplus_{q \in \mathbb Z} \HH^0_q (A, A) \quad\text{ and } \quad \HH^0 (A^!, A^!) \simeq \prod_{q \in \mathbb Z} \HH^0_q (A^!, A^!).
$$
\end{remark}

\begin{remark}
As mentioned in \cite{keller3} the isomorphism between Hochschild cohomology of $A$ and $A^!$ as graded algebras was announced by Buchweitz \cite{buchweitz}. See \cite{chenyangzhou} for the isomorphism preserving the Batalin--Vilkovisky structures for a Koszul Calabi--Yau algebra $A$.
\end{remark}

\subsection{\texorpdfstring{
Associative and A$_\infty$ deformations under Koszul duality}{%
Associative and A-infinity deformations under Koszul duality}}
\label{subsection:duality}

For a graded algebra we may consider the following two different types of deformations. 

\begin{definition}
\label{definition:deformation}
Let $A = (\bigoplus_{k \geq 0} A^k, \mu)$ be a (non-negatively) graded algebra. 
\begin{enumerate}
\item A {\it filtered deformation} of $A$ is an associative algebra $(A, \widetilde \mu)$ such that 
\begin{alignat*}{2}
\widetilde \mu (a \otimes b) - \mu (a \otimes b) &= 0                      &&\text{if $a \in A^0$ or $b \in A^0$} \\
\widetilde \mu (a \otimes b) - \mu (a \otimes b) &\in A^{>(\lvert a \rvert + \lvert b \rvert)} \quad &&\text{for homogeneous $a, b \in A^{> 0}$}.
\end{alignat*}
\item An {\it A$_\infty$ deformation} of $A$ is an A$_\infty$ algebra $(A, m_1, m_2, m_3, \dotsc)$ such that $m_1 = 0$ and $m_2=\mu$.
\end{enumerate}
\end{definition}

\begin{remark}\label{remark:firstorderdeformationhh}
The two types of deformations in Definition \ref{definition:deformation} are inherently very different: a filtered deformation is strictly associative, whereas an A$_\infty$ deformation is generally only associative up to higher homotopies. However, from the point of view of bigraded algebras, these two types of deformations naturally appear as two sides of the same coin, the two sides being related by Koszul duality. For a graded algebra $A=(\bigoplus_{k \geq 0} A^k, \mu)$, we may equip $A$ with two different natural bigradings. The first bigrading assigns each element $a \in A^k$ the bidegree $(0, k)$. Then the cocycles of $\HH_q^2(A, A)$ lie in
$$
\mathrm C^2_q (A, A) = \prod_{i \geq 0} \Hom(A^{\otimes i}, A)^{2-i}_q = \Hom(A^{\otimes 2}, A)^0_{q}
$$ 
where $\Hom(A^{\otimes i}, A)^{2-i}_q = 0$ for $i\neq 2$ (since $A$ is trivially graded with respect to the differential grading whence the same is true for $\Hom(A^{\otimes i}, A)$). Therefore, if $A^0$ is semisimple then $\HH_q^2(A, A)$ corresponds to (equivalence classes of) first-order filtered deformations $(A [t] / (t^2), \widetilde \mu)$ of $A$ such that $\widetilde \mu(a \otimes b) - \mu(a \otimes b) \in A^{\lvert a \rvert + \lvert b \rvert + q} t$ for any homogeneous elements $a, b \in A^{>0}$.

The second bigrading assigns each element $a \in A^k$ the bidegree $(k, -k)$. The cocycles of $\HH_q^2 (A, A)$ then lie in 
$$
\mathrm C^2_q (A, A) = \prod_{i \geq 0} \Hom(A^{\otimes i}, A)^{2-i}_q =  \Hom(A^{\otimes q+2}, A)^{-q}_q
$$
where $\Hom(A^{\otimes i}, A)^{2-i}_q = 0$ if $i \neq q+2$ (since the total degree of $A$ is zero whence the same is true for $\Hom(A^{\otimes i}, A)$). Therefore, $\HH_q^2 (A, A)$ corresponds to first-order A$_\infty$ deformations of $A$ with the only nontrivial higher product being $m_{q+2}$. Moreover, it is known from \cite[Cor.~4]{kadeishvili2} and \cite[Thm.~4.7]{seidelthomas} that if $\HH^2_q(A, A) =0$\footnote{this is denoted by $\HH^{2+q, -q}(A, A)$ in \cite{kadeishvili2} and \cite{seidelthomas}.} for each $q > 0$ then $A$ is {\it intrinsically formal}, i.e.\ any A$_\infty$ deformation of $A$ is A$_\infty$-quasi-isomorphic to the underlying graded algebra $A$. 
\end{remark}

Let $A = \Bbbk Q / I$ be an algebra such that $I \subset \Bbbk Q_{\geq 2}$. By \cite{kadeishvili} the graded space $\Ext^\hdot_A (\Bbbk Q_0, \Bbbk Q_0)$ admits a natural minimal A$_\infty$ algebra structure, which is the {\it minimal model} of the derived endomorphism algebra of $\Bbbk Q_0$. Moreover, the minimal model is formal (i.e.\ quasi-isomorphic to the underlying graded associative algebra which is $\Ext^\hdot_A (\Bbbk Q_0, \Bbbk Q_0)$ with the cup product) if and only if $A$ is Koszul (see \cite[\S 2.2]{keller1} and \cite[Cor.~V.0.6]{conner}). 

Now let $A$ be Koszul and let $A^!=\Bbbk \overline Q / (I_2^{\perp})$ be its Koszul dual. As before, assign any arrow $a \in Q$ the bidegree $(0, 1)$ and any arrow $\bar a \in \overline Q$ the bidegree $(1, -1)$. Then by Remark \ref{remark:firstorderdeformationhh} we may consider filtered deformations of $A$ and A$_\infty$ deformations of $A^!$. 

Theorem \ref{theorem:keller} and Remark \ref{remark:firstorderdeformationhh} motivate the following result. 

\begin{proposition}\label{prop:minimalmodel}
Let $A = \Bbbk Q / I$ be a Koszul algebra and let $A^!=\Bbbk \overline Q / (I_2^{\perp})$ be its Koszul dual. View $A$ as an ungraded algebra and $A^!$ as a graded algebra. Then for any nontrivial filtered deformation $B = (A, \widetilde \mu)$ of $A$, the minimal model of the derived endomorphism algebra of $\Bbbk Q_0$ is a nontrivial A$_\infty$ deformation of $A^!$.
\end{proposition}

\begin{proof}
Let us first show that the A$_\infty$ structure on $\Ext^\hdot_{B} (\Bbbk Q_0, \Bbbk Q_0)$, obtained as the minimal model of the derived endomorphism algebra of $\Bbbk Q_0$, is indeed an A$_\infty$ deformation of $A^!$, i.e.\ the underlying graded algebra is isomorphic to $A^!$. 

Recall the $A$-$A$-bimodule Koszul resolution $K_\ldot(A)$ of $A$ where $K_0(A) = A \otimes A$, $K_{-1}(A) = A \otimes \Bbbk Q_1 \otimes A$ and 
$$
K_{-n}(A) = A \otimes \bigl( \textstyle\bigcap\limits_{i=0}^n \Bbbk Q_i \otimes I_2 \otimes \Bbbk Q_{n-2-i} \bigr) \otimes A \qquad \text{for $n \geq 2$}
$$
where we write $\otimes = \otimes_{\Bbbk Q_0}$. Its differential is given by 
$$
d(1 \otimes \textstyle\sum\limits_i (a^i_1\otimes \dotsb \otimes a^i_n) \otimes 1) = \sum\limits_i a^i_1 \otimes a^i_2 \otimes \dotsb \otimes a^i_n \otimes 1 - \sum\limits_i 1 \otimes a^i_1 \otimes \dotsb \otimes a^i_{n-1} \otimes a^i_n
$$
for any $\sum_i a^i_1\otimes \dotsb \otimes a^i_n \in \bigcap_{k=0}^n \Bbbk Q_k \otimes I_2 \otimes \Bbbk Q_{n-2-k}$ with $a^i_1, \dotsc, a^i_n \in Q_1$. Since $K_\ldot(A)$ is exact in negative degrees we may choose a homotopy 
$\rho \colon K_{-n}(A) \tikzto K_{-n-1}(A)$ so that $\rho$ is a left $A$-module morphism which preserves the path length. We may also assume that $\rho(a \otimes \sum_i (a^i_1\otimes \dotsb \otimes a^i_n) \otimes b) = 0$ if $b = 1$. 

Let $\Bar_\ldot(A)$ denote the normalised bar resolution of $A$, i.e.\ $\Bar_{-n}(A) = A \otimes \bar{A}^{\otimes n} \otimes A$ where $\bar{A} = A / (\Bbbk Q_0)$ is the quotient $\Bbbk Q_0$-$\Bbbk Q_0$-bimodule. Using the same construction as in \cite[\S 5.1]{barmeierwang1} we obtain a deformation retract of $A$-$A$-bimodules
\begin{align}\label{align:deformationtretract}
\begin{tikzpicture}[baseline=-2.6pt,description/.style={fill=white,inner sep=1pt,outer sep=0}]
\matrix (m) [matrix of math nodes, row sep=0em, text height=1.5ex, column sep=3em, text depth=0.25ex, ampersand replacement=\&, inner sep=2.5pt]
{
K_\ldot(A) \& \Bar_\ldot(A) \\
};
\path[->,line width=.4pt,font=\scriptsize, transform canvas={yshift=.4ex}]
(m-1-1) edge node[above=-.4ex] {$F_\ldot$} (m-1-2)
;
\path[->,line width=.4pt,font=\scriptsize, transform canvas={yshift=-.4ex}]
(m-1-2) edge node[below=-.4ex] {$G_\ldot$} (m-1-1)
;
\path[<-,line width=.4pt,font=\scriptsize, looseness=8, in=30, out=330]
(m-1-2.east) ++(0,-.7ex) edge node[right=-.4ex] {$h_\ldot$} ++(0,1.4ex)
;
\end{tikzpicture}
\end{align}
namely $G_n F_n = \id$ and $F_n G_n - \id = h_{n-1} d_n + d_{n+1} h_n$. In particular, $F_\ldot$ is the natural embedding  
$$
F_n(1 \otimes \textstyle\sum\limits_i (a^i_1\otimes \dotsb \otimes a^i_n) \otimes 1) = 1 \otimes \sum\limits_i (a^i_1\otimes \dotsb \otimes a^i_n) \otimes 1.
$$
Since $\rho$ preserves the path length so do the maps $G_\ldot$ and $h_\ldot$ by construction. Applying $\Hom_{A^{\mathrm{op}}}(\Bbbk Q_0 \otimes_A -, \Bbbk Q_0)$ to  \eqref{align:deformationtretract} we obtain a new deformation retract
\begin{align}\label{align:deformation1}
\begin{tikzpicture}[baseline=-2.6pt,description/.style={fill=white,inner sep=1pt,outer sep=0}]
\matrix (m) [matrix of math nodes, row sep=0em, text height=1.5ex, column sep=3em, text depth=0.25ex, ampersand replacement=\&, inner sep=2.5pt]
{
(A^!, 0) \& (\Hom(\bar A^{\otimes\hdot}, \Bbbk Q_0), \partial) \\
};
\path[<-,line width=.4pt,font=\scriptsize, transform canvas={yshift=.4ex}]
(m-1-1) edge node[above=-.4ex] {$F^\hdot$} (m-1-2)
;
\path[<-,line width=.4pt,font=\scriptsize, transform canvas={yshift=-.5ex}]
(m-1-2) edge node[below=-.4ex] {$G^\hdot$} (m-1-1)
;
\path[->,line width=.4pt,font=\scriptsize, looseness=8, in=30, out=330]
(m-1-2.east) ++(0,-.7ex) edge node[right=-.4ex] {$h^\hdot$} ++(0,1.4ex)
;
\end{tikzpicture}
\end{align}
Here we use the natural isomorphisms $\Hom_{A^{\mathrm{op}}}(\Bbbk Q_0 \otimes_A K_\ldot(A), \Bbbk Q_0) \simeq A^!$ and $\Hom_{A^{\mathrm{op}}}(\Bbbk Q_0 \otimes_A \Bar_n(A), \Bbbk Q_0)\simeq \Hom(\bar A^{\otimes\hdot}, \Bbbk Q_0)$. Note that $\Hom(\bar A^{\otimes\hdot}, \Bbbk Q_0)$ is a DG algebra whose cohomology is isomorphic to the graded algebra $A^!$,   see  \cite[\S7.1]{chenwang}. 

Let $B = (A, \widetilde \mu)$ be a filtered deformation of $A$. Clearly, the DG algebra  $\Hom(\bar B^{\otimes\hdot}, \Bbbk Q_0)$, which computes $\Ext_{B}^\hdot(\Bbbk Q_0, \Bbbk Q_0)$, has the same underlying graded space as  $\Hom(\bar A^{\otimes\hdot}, \Bbbk Q_0)$. The differential $\widetilde \partial$ in the former may be viewed as a perturbed differential by the perturbation $\delta = \widetilde \partial - \partial$. Note that $\Hom(\bar A^{\otimes\hdot}, \Bbbk Q_0)$ has an additional grading (filtration) induced by the path length, i.e.\ the dual of a path $p$ is of degree $-\lvert p \rvert$, so that $\Hom(\bar A^{\otimes\hdot}, \Bbbk Q_0)$ is in negative degrees.  Note that $\delta$ strictly increases this grading and it follows that for any $f \in \Hom(\bar A^{\otimes n}, \Bbbk Q_0)$ we have 
$(\delta h^\hdot)^i(f) = 0$ for $i > n$, i.e.\ $\delta$ is (locally) small. Applying the homological perturbation lemma (see e.g.\ \cite[\S1]{huebschmannkadeishvili}) to $\delta$ we have that the homotopy deformation retract \eqref{align:deformation2} induces a new one 
\begin{align}\label{align:deformation2}
\begin{tikzpicture}[baseline=-2.6pt,description/.style={fill=white,inner sep=1pt,outer sep=0}]
\matrix (m) [matrix of math nodes, row sep=0em, text height=1.5ex, column sep=3em, text depth=0.25ex, ampersand replacement=\&, inner sep=2.5pt]
{
(A^!, d_\infty) \& (\Hom(\bar A^{\otimes\hdot}, \Bbbk Q_0), \widetilde \partial) \\
};
\path[<-,line width=.4pt,font=\scriptsize, transform canvas={yshift=.4ex}]
(m-1-1) edge node[above=-.4ex] {$F^\hdot_\infty$} (m-1-2)
;
\path[<-,line width=.4pt,font=\scriptsize, transform canvas={yshift=-.5ex}]
(m-1-2) edge node[below=-.4ex] {$G^\hdot_\infty$} (m-1-1)
;
\path[->,line width=.4pt,font=\scriptsize, looseness=8, in=30, out=330]
(m-1-2.east) ++(0,-.7ex) edge node[right=-.4ex] {$h^\hdot_\infty$} ++(0,1.4ex)
;
\end{tikzpicture}
\end{align}
where 
$$
d_\infty = \sum_{n \geq 0} F^\hdot (\delta h^\hdot)^n \delta G^\hdot,\quad F_\infty^\hdot = \sum_{n \geq 0}  F^\hdot (\delta h^\hdot)^n,\quad G_\infty^\hdot =G^\hdot +  \sum_{n \geq 0}  h^\hdot ( \delta h^\hdot )^n  \delta G^\hdot.
$$
Since $\delta$ is locally small the above sums are well-defined. We claim that $\delta G^\hdot = 0$. Indeed, for any $f \in \Hom_{A^{\mathrm{op}}}(\Bbbk Q_0 \otimes_A K_n(A), \Bbbk Q_0)$ we have 
\begin{align*}
\delta G^\hdot(f) ( \bar a_1 \otimes \dotsb \otimes \bar a_{n+1}) = \sum_{i=1}^n (-1)^i f G_\ldot ( 1 \otimes  \bar a_1 \otimes \dotsb \otimes \widetilde a_{i,i+1} \otimes \dotsb \otimes  \bar a_{n+1}\otimes 1)
\end{align*}
where $\bar a_j \in \bar A$ for $1\leq j \leq n+1$ and $\widetilde a_{i,i+1} := \widetilde \mu(a_i \otimes a_{i+1}) - a_ia_{i+1}$. Since the path length of $\widetilde a_{i,i+1}$ is strictly bigger than $2$, the one of $p:=a_1 \otimes \dotsb \otimes \widetilde a_{i,i+1} \otimes \dotsb \otimes a_{n+1}$ is bigger than $n+1$ so is the one of $G_\ldot(p)$. It follows that the right-hand side of the equality vanishes since the path length of elements in $\Bbbk Q_0 \otimes_A K_n(A)$ is $n$. This proves the claim and shows that $d_\infty = 0$ and $G^\hdot_\infty = G^\hdot$. 

Applying the homotopy transfer theorem we obtain an A$_\infty$ algebra $(A^!, m_2, m_3, \dotsc)$ which is the minimal model $\Ext^\hdot_{B} (\Bbbk Q_0, \Bbbk Q_0)$. Using the path length  as above, we may show that  $m_2=F_\infty^\hdot\circ m_2'\circ (G^\hdot \otimes G^\hdot) = F^\hdot \circ m_2' \circ (G^\hdot \otimes G^\hdot)$ and  the latter coincides with the underlying product of $A^!$, since the product $m_2'$ in $\Hom(\bar A^{\otimes\hdot}, \Bbbk Q_0)$ does not depend on the product of $A$, see \cite[(7.1)]{chenwang}.

If  $\Ext^\hdot_{B} (\Bbbk Q_0, \Bbbk Q_0)$ is A$_\infty$-quasi-isomorphic to the underlying graded algebra $A^!$, then $B$ is Koszul so that  $B^! \simeq  A^!$. Thus, we have  $B \simeq B^{!!} \simeq  A^{!!} \simeq A$, i.e.\ $B$ is a trivial associative deformation of $A$.  
\end{proof}

\begin{remark}
Proposition \ref{prop:minimalmodel} should be well-known to experts as it may be viewed as a deformation-theoretical interpretation of Theorem \ref{theorem:keller}. Note that Braverman and Gaitsgory \cite{bravermangaitsgory} gave a deformation-theoretical construction of PBW deformations of quadratic Koszul algebras, which are filtered deformations in the ``opposite'' sense, using shorter instead of longer paths, and Fløystad and Vatne \cite[Thm.~2.1]{floystadvatne} gave a correspondence between PBW deformations of an $N$-Koszul algebra and certain A$_\infty$ deformations of its Koszul dual.
\end{remark}

\subsection{Associative deformations via reduction systems}
\label{subsection:reductionsystem}

By the results of \S\ref{subsection:duality}, A$_\infty$ deformations of a Koszul algebra correspond to (filtered) associative deformations of its Koszul dual. In \cite{barmeierwang1} we show how to describe the latter concretely as deformations of a reduction system (cf.\ \S\ref{subsection:reduction}). The main result of \cite{barmeierwang1} is the following.

\begin{theorem}[{\cite[Thm.~7.1]{barmeierwang1}}]
\label{theorem:equivalence}
Let $Q$ be any finite quiver, let $I \subset \Bbbk Q$ be any two-sided ideal and let $R$ be any reduction system satisfying the diamond condition for $I$. Then there is an equivalence of deformation problems between
\begin{enumerate}
\item deformations of the associative algebra $A = \Bbbk Q / I$
\item deformations of the ideal $I$
\item deformations of the reduction system $R$.
\end{enumerate}
\end{theorem}

This result can be obtained by ``replacing'' the bar resolution of $A$, used in the classical formulation of associative deformations in terms of the Hochschild cochain complex equipped with the Gerstenhaber bracket \cite{gerstenhaber}, by the Bardzell--Chouhy--Solotar resolution \cite{bardzell,chouhysolotar} associated to a reduction system and studying the deformation theory on the resulting smaller cochain complex. Theorem \ref{theorem:equivalence} gives a practicable description of the full formal deformation theory of $A = \Bbbk Q / I$. A special case, corresponding to the case of first-order deformations, is the following result.

\begin{corollary}[{\cite[Cor.~7.44]{barmeierwang1}}]
\label{corollary:firstorder}
Let $A = \Bbbk Q / I$ and $R$ be any reduction system satisfying the diamond condition for $I$. Then $\HH^2 (A, A)$ is isomorphic to the space of first-order deformations of $R$ modulo equivalence.
\end{corollary}

Here a {\it first-order deformation} (over $\Bbbk [t] / (t^2)$) of a reduction system $R = \{ (s, \varphi_s) \}_{s \in S}$ is given by
\[
\widetilde R = \{ (s, \varphi_s + \widetilde\varphi_s t) \}_{s \in S}
\]
where $\widetilde \varphi_s \in \Bbbk \mathrm{Irr}_S$ such that $\widetilde R$ is reduction-unique when viewed as a reduction system for $\Bbbk Q \otimes_\Bbbk \Bbbk [t] / (t^2)$. Here, $\widetilde \varphi_s$ can be viewed as the image of $s$ under a map $\widetilde \varphi \in \Hom_{\Bbbk Q_0^\e}(\Bbbk S, \Bbbk \mathrm{Irr}_S)$ sending $s \tikzmapsto \widetilde \varphi_s$. Note that $\Hom_{\Bbbk Q_0^\e} (\Bbbk S, \Bbbk \mathrm{Irr}_S)$ is the space of $2$-cochains in the cochain complex associated to the Bardzell--Chouhy--Solotar resolution associated to $R$ \cite{chouhysolotar} and $\widetilde \varphi$ defines a first-order deformation of $R$ if and only if it is a $2$-cocycle in this complex \cite[\S 7.A]{barmeierwang1}.

Two first-order deformations $\widetilde R$ and $\widetilde R' = \{ (s, \varphi_s + \widetilde\varphi_s' t) \}_{s \in S}$ of $R$ are called {\it equivalent} if there exists $\psi \in \Hom_{\Bbbk Q_0^\e}( \Bbbk Q_1 , \Bbbk \mathrm{Irr}_S) $ such that 
$\delta(\psi)_s = \widetilde \varphi_s - \widetilde \varphi_s'$ 
for any $s = s_1 \cdots s_m \in S$ with $s_i \in Q_1$. Here $\delta$ is the map
\begin{align}\label{align:gaugeequivalent}
\delta \colon \Hom_{\Bbbk Q_0^\e} (\Bbbk Q_1, \Bbbk \mathrm{Irr}_S) \tikzto \Hom_{\Bbbk Q_0^\e}(\Bbbk S, \Bbbk \mathrm{Irr}_S) 
\end{align}
defined by $\delta(\psi)(s) = T(\psi)(s - \varphi_s)$ 
and $T(\psi)\colon \Bbbk Q \tikzto \Bbbk \mathrm{Irr}_S$ is the $\Bbbk$-linear map
\begin{align*}
T(\psi)(a_1 a_2 \dotsb a_m) =\sum_{i=1}^m \sigma \pi( a_1 \dotsb a_{i-1} \psi(a_i) a_{i+1} \dotsb a_m)
\end{align*}
where $a_1, \dotsc, a_m$ are arrows in $Q$, $\pi \colon \Bbbk Q \tikzto \Bbbk Q/I$ is the natural projection and $\sigma \colon \Bbbk Q/I \tikztosim \Bbbk\mathrm{Irr}_S$ is the inverse of the restriction $\pi\vert_{\Bbbk \mathrm{Irr}_S}$. In other words, for any path $p$ in $Q$, $\sigma \pi (p)$ may be obtained by performing reductions (with respect to $R$) on $p$ until all elements are irreducible. Here $\delta$ coincides with the differential of the cochain complex associated to the Bardzell--Chouhy--Solotar resolution and two first-order deformations of $R$ are equivalent if and only if the corresponding $2$-cocycles $\widetilde \varphi$ and $\widetilde \varphi'$ are cohomologous \cite[\S 7.A]{barmeierwang1}.

Corollary \ref{corollary:firstorder} gives a straightforward method to compute $\HH^2 (A, A)$ and in \S\ref{section:conjecture} we will apply this to the algebra $\KK_m^n = \Bbbk \QQ_m^n / \II_m^n$ with the reduction system $\RR_m^n$.

\section{A$_\infty$ deformations and Stroppel's conjecture}
\label{section:conjecture}

We now apply the general theory in \S\ref{section:hochschild} to the extended Khovanov arc algebras, viewing $\K_m^n \simeq \Bbbk \Q_m^n / \I_m^n$ and its Koszul dual $\KK_m^n \simeq \Bbbk \QQ_m^n / \II_m^n$ as bigraded algebras by assigning any arrow $a \in \Q_m^n$ the bidegree $(1,-1)$ and any arrow $\bar a \in \QQ_m^n$ the bidegree $(0, 1)$.\footnote{In the notation of \S\ref{section:hochschild}, we are considering filtered associative deformations of $A = \KK_m^n$ (so that $Q = \QQ_m^n$) and A$_\infty$ deformations of $A^! = \K_m^n$.} Theorem \ref{theorem:keller} gives for each $q \in \mathbb Z$ an isomorphism
\begin{align}
\label{align:Koszuldualisomorphims}
\HH^\hdot_q (\K_m^n, \K_m^n) \simeq \HH^\hdot_q (\KK_m^n, \KK_m^n)
\end{align}
whence Stroppel's Conjecture \ref{conjecture:stroppel} is equivalent to
$$
\HH_{i-2}^2 (\KK_m^n, \KK_m^n) = 0 \qquad \text{if $i \neq 0$}.
$$
By \S\ref{subsection:reductionsystem} these cohomology groups may be computed as certain equivalence classes of first-order deformations of the reduction system $\RR_m^n$.

The following vanishing result holds for degree reasons. 

\begin{lemma}\label{lemma:vanishing}
If either $i$ is odd, $i <0$ or $i> 2mn-2$ then 
$$
\HH^2_{i-2} (\K_m^n, \K_m^n) \simeq \HH_{i-2}^2(\KK_m^n, \KK_m^n) = 0.
$$
\end{lemma}

\begin{proof}
Using the Koszul resolution of $\KK_m^n$ (see the proof of Proposition \ref{prop:minimalmodel}), the $2$-cocycles of $\HH_{i-2}^2(\KK_m^n, \KK_m^n)$ lie in $\Hom_{\Bbbk Q_0^\e}((\II_m^n)_2, \KK_m^n)^0_{i-2}= \Hom_{\Bbbk Q_0^\e}((\II_m^n)_2, (\KK_m^n)_i)$ (see Remark \ref{remark:firstorderdeformationhh}). By Proposition \ref{proposition:bipartite} parallel paths have the same parity, so if $i$ is odd then $\Hom_{\Bbbk Q_0^\e}((\II_m^n)_2, (\KK_m^n)_i) = 0$ whence $\HH_{i-2}^2(\KK_m^n, \KK_m^n)=0$. 

Since $(\KK_m^n)_i = 0$ if $i < 0$ or $i > 2mn$ by Remark \ref{remark:irreduciblegrading}, we obtain  $\HH_{i-2}^2(\KK_m^n, \KK_m^n)=0$ in this case. Let us consider $i = 2mn$. Note that the longest irreducible path  is of length $2mn$ and parallel to the vertex $e_\lambda$ where $\lambda$ is the highest weight and elements in $(\II_m^n)_2$ are not parallel to $e_\lambda$, i.e.\ there are no quadratic relations at $e_{\lambda}$, so that $\Hom_{\Bbbk Q_0^\e}((\II_m^n)_2, (\KK_m^n)_{2mn}) = 0$ whence $\HH_{2mn-2}^2(\KK_m^n, \KK_m^n)=0$.
\end{proof}

The next vanishing result holds by direct computation.

\begin{lemma}\label{lemma:vanishing2mn-4}
$\HH^2_{2mn-4} (\K_m^n, \K_m^n) \simeq \HH_{2mn-4}^2(\KK_m^n, \KK_m^n) = 0$.
\end{lemma}
\begin{proof}
We apply Corollary \ref{corollary:firstorder}. Note that an element $\widetilde \varphi$ in $\Hom_{\Bbbk Q_0^\e}(\Bbbk \SS_m^n, \KK_m^n)_{2mn-2}$ has the following general form
$$
\widetilde \varphi_{\bar y_0 \bar x_0} =\alpha \bar x_{1 \ldots mn-1} \bar y_{mn-1 \ldots 1}
$$
for $\alpha \in \Bbbk$ and $\widetilde \varphi_s = 0$ for all other $s\in \SS_m^n$. Denote $\widetilde R = \{ (s, \varphi_s + \widetilde\varphi_s t) \}_{s \in \SS_m^n}$. Observe that the irreducible paths parallel to each overlap are of length smaller than $2mn-2$, so all overlaps have to reduce to zero.  It follows that all overlaps are resolvable in $\widetilde R$, i.e.\ $\widetilde R$ is a first-order deformation of $\RR_m^n$. 

We claim that $\widetilde R$ is equivalent to the trivial deformation. Indeed, consider $\psi \in \Hom_{\Bbbk Q_0^\e}(\Bbbk (\QQ_m^n)_1, \KK_m^n)_{2mn-2}$ given by
$$
\psi_{\bar x_0} = (-1)^{n-m} \alpha \bar x_{0 \ldots mn-2} \bar y_{mn-2 \ldots 1}
$$
and $\psi_{\bar x} = 0 = \psi_{\bar y}$ for all the other arrows $\bar x$ and $\bar y$. Then we have 
\begin{align*}
\delta(\psi)_{\bar y_0 \bar x_0} = \bar y_0 \psi_{\bar x_0} &= (-1)^{n-m} \alpha\bar y_0 \bar x_{0 \ldots mn-2} \bar y_{mn-2 \ldots 1} \\
&= \alpha \bar x_{1 \ldots nm-1} \bar y_{nm-1 \ldots 1}
\end{align*}
where the last equality follows from Lemma \ref{lemma:higherrelations3}.
\end{proof}

The following theorem shows that, contrary to Conjecture \ref{conjecture:stroppel}, $\HH^2$ also has nonvanishing components $\HH^2_i$ for $i \neq 0$.

\begin{theorem}\label{thm:maintheorem}
Let $m, n \geq 2$. Then
\[
\dim_\Bbbk \HH^2_{2mn-6} (\K_m^n, \K_m^n) \simeq \dim_\Bbbk \HH_{2mn-6}^2(\KK_m^n, \KK_m^n) = 1.
\]
As a result, $\K_m^n$ is not intrinsically formal. 
\end{theorem}

In the rest of this section we give a proof of this theorem by using the results of \S\ref{subsection:reductionsystem} to calculate the Hochschild cohomology $\HH_{2mn-6}^2 (\KK_m^n, \KK_m^n)$ as equivalence classes of certain first-order deformations of the reduction system $\RR_m^n$. From now on we assume that $m \geq n$ since by \cite[Proof of Cor.~1.5]{maksmith} we have $\K_m^n \simeq (\K_n^m)^{\mathrm{op}}$, so that it follows from \cite[Prop.~6.4]{chenliwang} (cf.\ \cite[E.2.1.4]{loday}) that $\K_m^n$ and $\K_n^m$ have the same (bigraded) Hochschild cohomology and the same deformation theory.

\subsection{\texorpdfstring{The case $\K_2^2$}{The case K22}}
\label{subsection:K22}

We first give a proof of Theorem \ref{thm:maintheorem} for the case $(m, n) = (2, 2)$ which also serves as a blueprint for the general case $m, n \geq 2$. An illustration of the degree $0$ and $1$ arc diagrams of $\K_2^2$ and the vertices and arrows of the Koszul dual $\KK_2^2$ are given in Fig.~\ref{figure:22}.

\begin{figure}
\begin{tikzpicture}[x=.825em,y=1.1em,decoration={markings,mark=at position 0.99 with {\arrow[black]{Stealth[length=3pt]}}}] 
\begin{scope}
\node[font=\small] at (0,14) {$\Q_2^2$};
\node[shape=ellipse,minimum height=2em,minimum width=2.4em] (3-1) at (0,12) {};
\begin{scope}[shift={(-.405em,12)},x=.27em,y=.45em]
\UP{0} \UP{1} \DN{2} \DN{3} 
\RAY{0} \RAY{1} \RAY{2} \RAY{3}
\end{scope}
\node[shape=ellipse,minimum height=2em,minimum width=2.4em] (3-2) at (0,8) {};
\begin{scope}[shift={(-.405em,8)},x=.27em,y=.45em]
\UP{0} \DN{1} \UP{2} \DN{3}
\RAY{0} \CIRCLE{1} \RAY{3}
\end{scope}
\node[shape=ellipse,minimum height=2em,minimum width=2.4em] (3-3a) at (-4,4) {};
\begin{scope}[shift={(-3.705em,4)},x=.27em,y=.45em]
\DN{0} \UP{1} \UP{2} \DN{3} 
\CIRCLE{0} \RAY{2} \RAY{3} 
\end{scope}
\node[shape=ellipse,minimum height=2em,minimum width=2.4em] (3-3b) at (4,4) {};
\begin{scope}[shift={(2.895em,4)},x=.27em,y=.45em]
\UP{0} \DN{1} \DN{2} \UP{3}
\RAY{0} \RAY{1} \CIRCLE{2} 
\end{scope}
\node[shape=ellipse,minimum height=2em,minimum width=2.4em] (3-4a) at (0,0) {};
\begin{scope}[shift={(-.405em,0)},x=.27em,y=.45em]
\DN{0} \UP{1} \DN{2} \UP{3} 
\CIRCLE{0} \CIRCLE{2} 
\end{scope}
\node[shape=ellipse,minimum height=2em,minimum width=2.4em] (3-4b) at (0,-4) {};
\begin{scope}[shift={(-.405em,-4)},x=.27em,y=.45em]
\DN{0}  \DN{1} \UP{2} \UP{3}
\SCCIRCLE{0} \CIRCLE{1}
\end{scope}
\path[<-, line width=.5pt] (3-1) edge[transform canvas={xshift=-.5ex}] (3-2);
\path[->, line width=.5pt] (3-1) edge[transform canvas={xshift=.5ex}] (3-2);
\path[->, line width=.5pt] (3-4b) edge[transform canvas={xshift=-.5ex}] (3-4a);
\path[<-, line width=.5pt] (3-4b) edge[transform canvas={xshift=.5ex}] (3-4a);
\path[->, line width=.5pt] (3-4a) edge[transform canvas={shift={(225:.5ex)}}] (3-3a);
\path[<-, line width=.5pt] (3-4a) edge[transform canvas={shift={(45:.5ex)}}] (3-3a);
\path[->, line width=.5pt] (3-4a) edge[transform canvas={shift={(135:.5ex)}}] (3-3b);
\path[<-, line width=.5pt] (3-4a) edge[transform canvas={shift={(-45:.5ex)}}] (3-3b);
\path[->, line width=.5pt] (3-3b) edge[transform canvas={shift={(225:.5ex)}}] (3-2);
\path[<-, line width=.5pt] (3-3b) edge[transform canvas={shift={(45:.5ex)}}] (3-2);
\path[->, line width=.5pt] (3-3a) edge[transform canvas={shift={(135:.5ex)}}] (3-2);
\path[<-, line width=.5pt] (3-3a) edge[transform canvas={shift={(-45:.5ex)}}] (3-2);
\draw[<-,line width=.5pt,line cap=round] (3-2.180) ++(90:.5ex) arc[start angle=100, end angle=261, radius=16ex];
\draw[->,line width=.5pt,line cap=round] (3-2.180) ++(90:-.5ex) arc[start angle=100, end angle=261, radius=15ex];
\begin{scope}[shift={(-1.47em,10)},x=.27em,y=.45em]
\UP{0} \UP{1} \DN{2} \DN{3} 
\RAY{0} \RAYUP{1} \RAYUP{2} \CUP{1} \RAY{3} 
\end{scope}
\begin{scope}[shift={(.66em,10)},x=.27em,y=.45em]
\UP{0} \UP{1} \DN{2} \DN{3} 
\RAY{0} \RAYDN{1} \RAYDN{2} \CAP{1} \RAY{3} 
\end{scope}
\begin{scope}[shift={(-2.83em,7.19em)},x=.27em,y=.45em]
\UP{0} \DN{1} \UP{2} \DN{3} 
\RAYUP{0} \CUP{0} \CONNECT{1} \CAP{1} \RAYDN{2} \RAY{3} 
\end{scope}
\begin{scope}[shift={(-1.58em,5.55em)},x=.27em,y=.45em]
\UP{0} \DN{1} \UP{2} \DN{3} 
\RAYDN{0} \CAP{0} \CONNECT{1} \CUP{1} \RAYUP{2} \RAY{3}
\end{scope}
\begin{scope}[shift={(1.72em,7.65em)},x=.27em,y=.45em]
\UP{0} \DN{1} \UP{2} \DN{3} 
\RAY{0} \RAYUP{1} \CUP{1} \CONNECT{2} \CAP{2} \RAYDN{3} 
\end{scope}
\begin{scope}[shift={(.77em,5.55em)},x=.27em,y=.45em]
\UP{0} \DN{1} \UP{2} \DN{3}
\RAY{0} \RAYDN{1} \CAP{1} \CONNECT{2} \CUP{2} \RAYUP{3} 
\end{scope}
\begin{scope}[shift={(-1.58em,3.25em)},x=.27em,y=.45em]
\DN{0} \UP{1} \UP{2} \DN{3} 
\CIRCLE{0} \RAYDN{2} \CAP{2} \RAYDN{3} 
\end{scope}
\begin{scope}[shift={(-2.53em,1.15em)},x=.27em,y=.45em]
\DN{0} \UP{1} \UP{2} \DN{3} 
\CIRCLE{0} \RAYUP{2} \CUP{2} \RAYUP{3} 
\end{scope}
\begin{scope}[shift={(-8.03em,2.2em)},x=.27em,y=.45em]
\UP{0} \DN{1} \UP{2} \DN{3} 
\RRAYUP{0} \SCCUP{0} \CIRCLE{1} \RRAYUP{3}
\end{scope}
\begin{scope}[shift={(-6.2em,2.2em)},x=.27em,y=.45em]
\UP{0} \DN{1} \UP{2} \DN{3} 
\RRAYDN{0} \SCCAP{0} \RRAYDN{3} \CIRCLE{1}
\end{scope}
\begin{scope}[shift={(.77em,3.2em)},x=.27em,y=.45em]
\UP{0} \DN{1} \DN{2} \UP{3} 
\RAYUP{0} \CUP{0} \RAYUP{1} \CIRCLE{2} 
\end{scope}
\begin{scope}[shift={(1.72em,1.2em)},x=.27em,y=.45em]
\UP{0} \DN{1} \DN{2} \UP{3} 
\RAYDN{0} \CAP{0} \RAYDN{1} \CIRCLE{2}
\end{scope}
\begin{scope}[shift={(-1.47em,-2)},x=.27em,y=.45em]
\DN{0} \UP{1} \DN{2} \UP{3} 
\SCCUP{0}  \CUP{1} \CAP{0} \CAP{2}
\CONNECT{0}  \CONNECT{1} \CONNECT{2}  \CONNECT{3} 
\end{scope}
\begin{scope}[shift={(.66em,-2)},x=.27em,y=.45em]
\DN{0} \UP{1} \DN{2} \UP{3} 
\CUP{0} \CUP{2} \SCCAP{0} \CAP{1}
\CONNECT{0}  \CONNECT{1} \CONNECT{2}  \CONNECT{3} 
\end{scope}

\end{scope}

\begin{scope}[xshift=17em]
\node[font=\small] at (0,14) {$\QQ_2^2$};
\node[shape=ellipse,minimum height=2em,minimum width=2.4em] (3-1) at (0,12) {};
\begin{scope}[shift={(-.405em,12)},x=.27em,y=.45em]
\UP{0} \UP{1} \DN{2} \DN{3} 
\RAY{0} \RAY{1} \RAY{2} \RAY{3}
\end{scope}
\node[shape=ellipse,minimum height=2em,minimum width=2.4em] (3-2) at (0,8) {};
\begin{scope}[shift={(-.405em,8)},x=.27em,y=.45em]
\UP{0} \DN{1} \UP{2} \DN{3}
\RAY{0} \CIRCLE{1} \RAY{3}
\end{scope}
\node[shape=ellipse,minimum height=2em,minimum width=2.4em] (3-3a) at (-4,4) {};
\begin{scope}[shift={(-3.705em,4)},x=.27em,y=.45em]
\DN{0} \UP{1} \UP{2} \DN{3} 
\CIRCLE{0} \RAY{2} \RAY{3} 
\end{scope}
\node[shape=ellipse,minimum height=2em,minimum width=2.4em] (3-3b) at (4,4) {};
\begin{scope}[shift={(2.895em,4)},x=.27em,y=.45em]
\UP{0} \DN{1} \DN{2} \UP{3}
\RAY{0} \RAY{1} \CIRCLE{2} 
\end{scope}
\node[shape=ellipse,minimum height=2em,minimum width=2.4em] (3-4a) at (0,0) {};
\begin{scope}[shift={(-.405em,0)},x=.27em,y=.45em]
\DN{0} \UP{1} \DN{2} \UP{3} 
\CIRCLE{0} \CIRCLE{2} 
\end{scope}
\node[shape=ellipse,minimum height=2em,minimum width=2.4em] (3-4b) at (0,-4) {};
\begin{scope}[shift={(-.405em,-4)},x=.27em,y=.45em]
\DN{0}  \DN{1} \UP{2} \UP{3}
\SCCIRCLE{0} \CIRCLE{1}
\end{scope}
\path[->, line width=.5pt] (3-1) edge[transform canvas={xshift=-.5ex}] node[font=\scriptsize,left=-.5ex] {$\bar x_{11}$} (3-2);
\path[<-, line width=.5pt] (3-1) edge[transform canvas={xshift=.5ex}] node[font=\scriptsize,right=-.3ex] {$\bar y_{11}$} (3-2);
\path[<-, line width=.5pt] (3-4b) edge[transform canvas={xshift=-.5ex}]node[font=\scriptsize,left=-.5ex] {$\bar x_{32}$}  (3-4a);
\path[->, line width=.5pt] (3-4b) edge[transform canvas={xshift=.5ex}] node[font=\scriptsize,right=-.3ex] {$\bar y_{32}$} (3-4a);
\path[<-, line width=.5pt] (3-4a) edge[transform canvas={shift={(225:.5ex)}}] node[font=\scriptsize,left=-.3ex,pos=.4] {$\bar x_{22}$} (3-3a);
\path[->, line width=.5pt] (3-4a) edge[transform canvas={shift={(45:.5ex)}}] node[font=\scriptsize,right=-.2ex,pos=.6] {$\bar y_{22}$} (3-3a);
\path[<-, line width=.5pt] (3-4a) edge[transform canvas={shift={(135:.5ex)}}] node[font=\scriptsize,left=-.3ex,pos=.6] {$\bar x_{31}$}(3-3b);
\path[->, line width=.5pt] (3-4a) edge[transform canvas={shift={(-45:.5ex)}}] node[font=\scriptsize,right=-.2ex,pos=.4] {$\bar y_{31}$}(3-3b);
\path[<-, line width=.5pt] (3-3b) edge[transform canvas={shift={(225:.5ex)}}] node[font=\scriptsize, left=-.3ex,pos=.4] {$\bar x_{12}$}(3-2);
\path[->, line width=.5pt] (3-3b) edge[transform canvas={shift={(45:.5ex)}}] node[font=\scriptsize,right=-.1ex,pos=.6] {$\bar y_{12}$}(3-2);
\path[<-, line width=.5pt] (3-3a) edge[transform canvas={shift={(135:.5ex)}}] node[font=\scriptsize,left=-.2ex,pos=.6] {$\bar x_{21}$} (3-2);
\path[->, line width=.5pt] (3-3a) edge[transform canvas={shift={(-45:.5ex)}}] node[font=\scriptsize,right=-.1ex,pos=.4] {$\bar y_{21}$} (3-2);
\draw[->,line width=.5pt,line cap=round] (3-2.180) ++(90:.5ex) arc[start angle=100, end angle=261, radius=16ex];
\draw[<-,line width=.5pt,line cap=round] (3-2.180) ++(90:-.5ex) arc[start angle=100, end angle=261, radius=15ex];

\begin{scope}[shift={(-7.55em,2.2em)}]
\node[font=\scriptsize] at (0,0) {\strut$\bar x_2$};
\end{scope}
\begin{scope}[shift={(-5.9em,2.2em)}]
\node[font=\scriptsize] at (0,0) {\strut$\bar y_2$};
\end{scope}
\end{scope}
\end{tikzpicture}
\caption{The quiver $\Q_2^2$ for $\K_2^2$ with all arc diagrams of degrees $0$ (vertices) and $1$ (arrows) and the opposite quiver $\QQ_2^2$ for $\KK_2^2$}
\label{figure:22}
\end{figure}
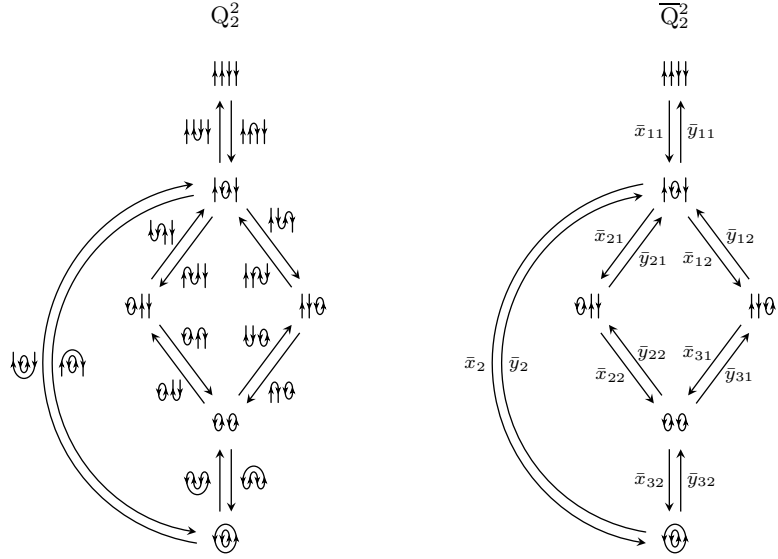

The reduction system $\RR_2^2$ (see \S\ref{subsection:irreduciblepaths}) for $\KK_2^2$ is given by the following set of pairs
\begin{flalign*}
&& (\bar y_2 \bar x_2&, 0)  & (\bar y_{11} \bar x_{11}&, - \bar x_{21} \bar y_{21} - \bar x_{12} \bar y_{12})  &  (\bar y_2 \bar x_{21}&, - \bar y_{32} \bar y_{22}) &  (\bar y_{22}\bar x_{22}&, - \bar x_{32} \bar y_{32}) && \\
&& (\bar y_2 \bar y_{11}&, 0) & (\bar y_{31} \bar y_{12}&, -\bar y_{22} \bar y_{21} - \bar x_{32} \bar y_2) & (\bar y_{21} \bar x_{12}&, -\bar x_{22} \bar y_{31}) & (\bar y_{12} \bar x_2&, - \bar x_{31} \bar x_{32}) && \\
&& (\bar x_{11} \bar x_2&, 0) & (\bar x_{12} \bar x_{31}&, -\bar x_{21} \bar x_{22} - \bar x_2 \bar y_{32}) & (\bar y_{12} \bar x_{21}&, -\bar x_{31} \bar y_{22}) & (\bar y_2 \bar x_{12}&, - \bar y_{32} \bar y_{31}) &&\\
&&  (\bar y_{21} \bar x_{21}&, 0)  & (\bar y_{12} \bar x_{12}&, 0)\quad   (\bar y_{32} \bar x_{32}, 0) & (\bar y_{31}\bar x_{31}&, - \bar x_{32} \bar y_{32}) & (\bar y_{21} \bar x_2&, - \bar x_{22} \bar x_{32})&&
\end{flalign*}
so that 
\begin{multline*}
\SS_2^2 = \{ \bar y_2 \bar x_2, \bar y_{11} \bar x_{11}, \bar y_2 \bar x_{21}, \bar y_{22} \bar x_{22}, \bar y_2 \bar y_{11}, \bar y_{31} \bar y_{12}, \bar y_{21} \bar x_{12}, \bar y_{12} \bar x_2, \\
\bar x_{11}\bar x_2,  \bar x_{12} \bar x_{31}, \bar y_{12} \bar x_{21}, \bar y_2 \bar x_{12}, \bar y_{21} \bar x_{21}, \bar y_{12} \bar x_{12}, \bar y_{32} \bar x_{32}, \bar y_{31}\bar x_{31}, \bar y_{21} \bar x_2
\}.
\end{multline*}
We thus have eight overlaps
\begin{alignat*}{4}
\bar y_2 \bar x_{12} \bar x_{31},&&\quad \bar y_{12} \bar x_{12} \bar x_{31},&&\quad \bar y_{21} \bar x_{12} \bar x_{31},&&\quad \bar y_{11} \bar x_{11} \bar x_2,& \\
\bar y_{31} \bar y_{12} \bar x_2,&&\quad \bar y_{31} \bar y_{12} \bar x_{12},&&\quad \bar y_{31} \bar y_{12} \bar x_{21},&&\quad \bar y_2 \bar y_{11} \bar x_{11}.&
\end{alignat*}

\begin{proposition}\label{proposition:k22}
$\dim_\Bbbk \HH_{2}^2(\K_2^2, \K_2^2) = \dim_\Bbbk \HH_{2}^2(\KK_2^2, \KK_2^2)  = 1$. 
\end{proposition}

\begin{proof}
Let $\{ (s, \varphi_s + \widetilde\varphi_s t) \}_{s \in \SS_2^2}$ be a first-order deformation of $\RR_2^2$ corresponding to a cocycle in  $\HH_{2}^2(\KK_2^2, \KK_2^2)$ as in \S\ref{subsection:reductionsystem}. Here, $\widetilde\varphi_s$ is the image of $s$ under the cochain $\widetilde\varphi \in \Hom_{\Bbbk Q_0^\e}(\Bbbk \SS_2^2, \KK_2^2)_2$. Note that $\widetilde \varphi_{s}$ has the following form 
\begin{equation}
\label{cochain22}
\begin{alignedat}{2}
\widetilde\varphi_{\bar y_{11} \bar x_{11}} &= \alpha_1 \bar x_{21} \bar x_{22} \bar y_{22} \bar y_{21} + \alpha_2 & \bar x_{21} \bar x_{22} \bar x_{32} &\bar y_2 + \alpha_3 \bar x_2 \bar y_{32} \bar y_{22} \bar y_{21} \\[.5em]
\widetilde\varphi_{\bar y_{21} \bar x_{21}} &= \alpha_4  \bar x_{22} \bar x_{32} \bar y_{32} \bar y_{22} &  \widetilde\varphi_{\bar y_{12} \bar x_{12}} &= \alpha_5  \bar x_{31} \bar x_{32} \bar y_{32} \bar y_{31} \\
\widetilde\varphi_{\bar x_{11} \bar x_2} &= \alpha_6 \bar x_{11} \bar x_{21} \bar x_{22} \bar x_{32} &  \widetilde\varphi_{\bar y_2 \bar y_{11}} &= \alpha_7 \bar y_{32} \bar y_{22} \bar y_{21} \bar y_{11} \\
\widetilde\varphi_{\bar y_{21} \bar x_{12}} &= \alpha_8 \bar x_{22} \bar x_{32} \bar y_{32} \bar y_{31} & \widetilde\varphi_{\bar y_{12} \bar x_{21}} &= \alpha_9 \bar x_{31} \bar x_{32} \bar y_{32} \bar y_{22} \\
\widetilde\varphi_{\bar x_{12} \bar x_{31}} &= \alpha_{10} \bar x_{21} \bar x_{22} \bar x_{32} \bar y_{32} &  \widetilde\varphi_{\bar y_{31} \bar y_{12}} &= \alpha_{11} \bar x_{32} \bar y_{32} \bar y_{22} \bar y_{21}
\end{alignedat}
\end{equation}
for some coefficients $\alpha_1, \dotsc, \alpha_{11} \in \Bbbk$ and $\widetilde \varphi_{s} = 0$ for all other $s \in \SS_2^2$. Note that none of the overlaps has a parallel irreducible path of length $\geq 5$, whence they are resolvable in the first-order deformation, i.e.\ any cochain $\widetilde\varphi$ is a cocycle.

Next we compute the coboundaries. An element $\psi \in \Hom_{\Bbbk Q_0^\e}(\Bbbk (\QQ_2^2)_1, \KK_2^2)_2$ is of the form 
\begin{align*}
\psi_{\bar x_{11}} &= \mu_1 \bar x_{11} \bar x_{21} \bar y_{21} + \mu_2 \bar x_{11} \bar x_{12} \bar y_{12} \quad & \psi_{\bar y_{11}} &= \nu_1 \bar x_{21} \bar y_{21} \bar y_{11} + \nu_2 \bar x_{12} \bar y_{12} \bar y_{11}\\
\psi_{\bar x_{21}} &= \mu_3 \bar x_{21} \bar x_{22} \bar y_{22} + \mu_4 \bar x_2 \bar y_{32}\bar y_{22} \quad & \psi_{\bar y_{21}} &= \nu_3 \bar x_{22} \bar y_{22} \bar y_{21} + \nu_4 \bar x_{22} \bar x_{32} \bar y_2\\
\psi_{\bar x_{12}} &= \mu_5 \bar x_{21} \bar x_{22} \bar y_{31} + \mu_6 \bar x_2 \bar y_{32} \bar y_{31} \quad & \psi_{\bar y_{12}} &= \nu_5 \bar x_{31} \bar y_{22} \bar y_{21} + \nu_6 \bar x_{31} \bar x_{32} \bar y_2\\
\psi_{\bar x_{22}} &= \mu_7 \bar x_{22} \bar x_{32} \bar y_{32} \quad & \psi_{\bar y_{22}} &= \nu_7 \bar x_{32} \bar y_{32} \bar y_{22}\\ 
\psi_{\bar x_{31}} &= \mu_8 \bar x_{31} \bar x_{32} \bar y_{32} \quad & \psi_{\bar y_{31}} &= \nu_8 \bar x_{32} \bar y_{32} \bar y_{31}\\
\psi_{\bar x_2} &= \mu_{9} \bar x_{21} \bar x_{22} \bar x_{32} \quad & \psi_{\bar y_2} &= \nu_{9} \bar y_{32} \bar y_{22} \bar y_{21}
\end{align*} 
and $\psi_{\bar x_{32}} = 0 = \psi_{\bar y_{32}}$, where $\mu_1, \dotsc,  \mu_{9}, \nu_1, \dotsc, \nu_{9} \in \Bbbk$. 
If $\delta(\psi)_s = \widetilde \varphi_s$ for each $s \in \SS_m^n$ then using \eqref{align:gaugeequivalent} we obtain the following eleven equations
\[
\begin{aligned}
&& \mathllap{-\mu_1 - \nu_1 -\mu_2-\nu_2 + \mu_3 + \nu_3 - \mu_5 - \nu_5} &= \alpha_1 \\
 -\nu_1 - \mu_2 + \nu_4 - \mu_5 - \nu_6 &= \alpha_2 \quad & -\mu_1 - \nu_2 + \mu_4 - \nu_5 - \mu_6 &= \alpha_3 \\
 -\nu_4 - \mu_4 &= \alpha_4 \quad & \nu_5 + \mu_5 - \nu_6 - \mu_6 &= \alpha_5 \\
 -\mu_1 + \mu_2 + \mu_{9} &= \alpha_6  \quad &  -\nu_1 + \nu_2 + \nu_{9} &= \alpha_7\\
 \nu_3 - \nu_4 - \mu_6 + \mu_7 + \nu_8 &= \alpha_8 \quad & \mu_3 - \mu_4 - \nu_6 + \nu_7 + \mu_8 &= \alpha_9\\ 
 -\mu_3 - \mu_5 + \mu_7 - \mu_8 + \mu_{9} &= \alpha_{10} \quad &  -\nu_3 - \nu_5 + \nu_7 - \nu_8 + \nu_{9} &= \alpha_{11}.
 \end{aligned}
\]
Let us check the first three equations. For this, we have
\begin{align*}
\delta(\psi)_{\bar y_{11} \bar x_{11}} ={} & \psi_{\bar y_{11}} \bar x_{11} + \bar y_{11} \psi_{\bar x_{11}} + \psi_{\bar x_{21}} \bar y_{21} + \bar x_{21} \psi_{\bar y_{21}} + \psi_{\bar x_{12}} \bar y_{12} + \bar x_{12} \psi_{\bar y_{12}}\\
={}&  \nu_1 \bar x_{21} \bar y_{21} \bar y_{11}\bar x_{11} + \nu_2 \bar x_{12} \bar y_{12} \bar y_{11}\bar x_{11} + \mu_1 \bar y_{11} \bar x_{11} \bar x_{21} \bar y_{21} + \mu_2 \bar y_{11}\bar x_{11} \bar x_{12} \bar y_{12}\\
& + \mu_3 \bar x_{21} \bar x_{22} \bar y_{22}\bar y_{21} + \mu_4 \bar x_2 \bar y_{32}\bar y_{22}\bar y_{21} +   \nu_3  \bar x_{21} \bar x_{22} \bar y_{22} \bar y_{21} + \nu_4  \bar x_{21}\bar x_{22} \bar x_{32} \bar y_2\\
& +\mu_5 \bar x_{21} \bar x_{22} \bar y_{31}\bar y_{12} + \mu_6 \bar x_2 \bar y_{32} \bar y_{31} \bar y_{12} + \nu_5 \bar x_{12}  \bar x_{31} \bar y_{22} \bar y_{21} + \nu_6\bar x_{12}  \bar x_{31} \bar x_{32} \bar y_2\\
={}&    \beta_1 \bar x_{21} \bar x_{22} \bar y_{22} \bar y_{21} +   \beta_2 \bar x_{21} \bar x_{22} \bar x_{32} \bar y_2 + \beta_3 \bar x_2 \bar y_{32} \bar y_{22} \bar y_{21}
\end{align*}
where $\beta_1 := -\mu_1 - \nu_1 -\mu_2-\nu_2 + \mu_3 + \nu_3 - \mu_5 - \nu_5$, $\beta_2 :=-\nu_1 - \mu_2 + \nu_4 - \mu_5 - \nu_6$ and $\beta_3:=-\mu_1 - \nu_2 + \mu_4 - \nu_5 - \mu_6$, and the third equality follows by performing reductions (with respect to $\RR_2^2$) on each summand. For instance, we have 
\begin{align*}
\bar x_{21} \bar y_{21} \bar y_{11} \bar x_{11} = -\bar x_{21} \bar y_{21} \bar x_{21} \bar y_{21} - \bar x_{21} \bar y_{21} \bar x_{12} \bar y_{12} &= \bar x_{21}\bar x_{22}\bar y_{31} \bar y_{12} \\
&= -\bar x_{21} \bar x_{22} \bar y_{22} \bar y_{21} - \bar x_{21} \bar x_{22} \bar x_{32} \bar y_2
\end{align*}
where the first equality follows by performing reductions on $ \bar y_{11}\bar x_{11}$, the second one on $ \bar y_{21} \bar x_{21}$ and $\bar y_{21} \bar x_{12}$, and the third one on $\bar y_{31} \bar y_{12}$. Then by $\delta(\psi)_{\bar y_{11} \bar x_{11}} = \varphi_{\bar y_{11} \bar x_{11}}$ we obtain the first three equations. Similarly we may verify the other equations. 

From the above eleven equations we observe that
\begin{align}
\label{align:alpha}
- \alpha_2 + \alpha_3 - \alpha_6  + \alpha_7 - \alpha_8 + \alpha_9 + \alpha_{10} - \alpha_{11} = 0.
\end{align}
Moreover, writing the coefficients as an $(18 \times 11)$-matrix, the rank of this matrix is seen to be $10$, i.e.\ \eqref{align:alpha} turns out to be the unique constraint for $\widetilde \varphi$ to be a coboundary. Thus, $\HH_{2}^2(\KK_2^2, \KK_2^2)$ is $1$-dimensional. By \eqref{align:alpha} an explicit nontrivial cocycle can be given by setting $\alpha_2 = 1$ and $\alpha_i = 0$ for $i \neq 2$. 
\end{proof}

\begin{corollary}
\label{corollary:Ainfinitydeformation22}
We have that $\dim_\Bbbk \bigl( \bigoplus_{i \neq 0, 2} \HH^2_{i-2} (\K_2^2, \K_2^2) \bigr) = 1$. Moreover, $\K_2^2$ admits a unique A$_\infty$ deformation (up to A$_\infty$ isomorphism) given on arrows by setting
\begin{align*}
m_4 (y_2 \otimes x_{32} \otimes x_{22} \otimes x_{21}) = x_{11} y_{11}
\end{align*} 
and $m_4 = 0$ when restricting to other arrows and setting $m_i = 0$ for $i \neq 2, 4$ for all arrows.
\end{corollary}

\begin{proof}
The first assertion follows from Proposition \ref{proposition:k22} and Lemmas \ref{lemma:vanishing} and \ref{lemma:vanishing2mn-4}. 

For the second assertion, there is a nontrivial cocycle $\widetilde \varphi$ in $\HH_{2}^2(\KK_2^2, \KK_2^2)$ given by $\alpha_2 = 1$ and $\alpha_i = 0$ for $i \neq 2$, as in the proof of Proposition \ref{proposition:k22}. Note that $\{ (s, \varphi_s + \widetilde\varphi_s) \}_{s \in \SS_2^2}$ is an actual deformation of $\RR_2^2$, i.e.\ it is reduction-finite and reduction-unique so that $\KK_2^2$ admits a filtered deformation by changing the relation $\bar y_{11} \bar x_{11}+ \bar x_{21} \bar y_{21} + \bar x_{12} \bar y_{12} = 0$ into 
$$
\bar y_{11} \bar x_{11}+ \bar x_{21} \bar y_{21} + \bar x_{12} \bar y_{12} = \bar x_2 \bar y_{32} \bar y_{22} \bar y_{21}.
$$
By Proposition \ref{prop:minimalmodel}, the minimal model of the derived endomorphism algebra of $\Bbbk (\QQ_m^n)_0$ is an A$_\infty$ deformation of $\K_2^2$. 
\end{proof}

\begin{figure}
\begin{tikzpicture}[x=.75em,y=.75em,decoration={markings,mark=at position 0.99 with {\arrow[black]{Stealth[length=4.8pt]}}}]
\draw[->, line width=.5pt] (4,-3) -- ++(2,0);
\draw[->, line width=.5pt] (11,-3) -- ++(2,0);
\draw[->, line width=.5pt] (18,-3) -- ++(2,0);
\draw[->, line width=.5pt] (25,-3) -- ++(2,0);
\node[font=\small] at (25.9,-2.5) {$\ast$};
\draw[->, line width=.5pt] (32,-3) -- ++(2,0);
\begin{scope} 
\node[font=\small,left] at (-.4,0) {$x_{21}$};
\UP{0} \DN{1} \UP{2} \DN{3}
\RAYUP{0} \CUP{0} \CONNECT{1} \CAP{1} \RAYDN{2} \RAY{3} 
\end{scope}
\begin{scope}[shift={(0,-2)}] 
\node[font=\small,left] at (-.4,0) {$x_{22}$};
\DN{0} \UP{1} \UP{2} \DN{3}
\CIRCLE{0} \RAYUP{2} \CUP{2} \RAYUP{3}
\end{scope}
\begin{scope}[shift={(0,-4)}] 
\node[font=\small,left] at (-.4,0) {$x_{32}$};
\DN{0} \UP{1} \DN{2} \UP{3}
\CCCUP{0}{1.5} \CONNECT{0} \CAP{0} \CONNECT{1} \CUP{1} \CONNECT{2} \CAP{2} \CONNECT{3}
\end{scope}
\begin{scope}[shift={(0,-7)}] 
\node[font=\small,left] at (-.4,0) {$y_2$};
\UP{0} \DN{1} \UP{2} \DN{3}
\RAYDN{0} \CCCAP{0}{1.5} \CIRCLE{1} \RAYDN{3}
\end{scope}
\begin{scope}[shift={(7,0)}]
\begin{scope} 
\UP{0} \DN{1} \UP{2} \DN{3}
\RAYUP{0} \CUP{0} \CONNECT{1} \CAP{1} \RAYDN{2} \RAY{3} 
\end{scope}
\begin{scope}[shift={(0,-2)}] 
\DN{0} \UP{1} \UP{2} \DN{3}
\CAP{0} \RAYDN{0} \RAYDN{1} \RAYUP{2} \CUP{2} \RAYUP{3}
\end{scope}
\begin{scope}[shift={(0,-4)}] 
\DN{0} \UP{1} \DN{2} \UP{3}
\CCCUP{0}{1.5} \RAYUP{0} \RAYUP{1} \CUP{1} \CONNECT{2} \CAP{2} \CONNECT{3}
\end{scope}
\begin{scope}[shift={(0,-7)}] 
\UP{0} \DN{1} \UP{2} \DN{3}
\RAYDN{0} \CCCAP{0}{1.5} \CIRCLE{1} \RAYDN{3}
\end{scope}
\end{scope}
\begin{scope}[shift={(14,0)}]
\begin{scope}[shift={(0,-1)}] 
\UP{0} \DN{1} \UP{2} \DN{3}
\RAY{0} \RAYDN{1} \CAP{1} \CONNECT{2} \CUP{2} \RAYUP{3}
\end{scope}
\begin{scope}[shift={(0,-3)}] 
\UP{0} \DN{1} \UP{2} \DN{3}
\RAYUP{0} \CCCUP{0}{1.5} \RAYUP{1} \CUP{1} \CONNECT{2} \CAP{2} \CONNECT{3}
\end{scope}
\begin{scope}[shift={(0,-6)}] 
\UP{0} \DN{1} \UP{2} \DN{3}
\RAYDN{0} \CCCAP{0}{1.5} \CIRCLE{1} \RAYDN{3}
\end{scope}
\end{scope}
\begin{scope}[shift={(21,0)}]
\begin{scope}[shift={(0,-1.5)}] 
\UP{0} \UP{1} \DN{2} \DN{3}
\RAYUP{0} \CCCUP{0}{1.5} \CIRCLE{1} \RAYUP{3}
\end{scope}
\begin{scope}[shift={(0,-4.5)}] 
\UP{0} \DN{1} \UP{2} \DN{3}
\RAYDN{0} \CCCAP{0}{1.5} \CIRCLE{1} \RAYDN{3}
\end{scope}
\end{scope}
\begin{scope}[shift={(28,0)}]
\begin{scope}[shift={(0,-2)}] 
\UP{0} \UP{1} \DN{2} \DN{3}
\RAY{0} \CIRCLE{1} \RAY{3}
\end{scope}
\begin{scope}[shift={(0,-4)}] 
\UP{0} \DN{1} \UP{2} \DN{3}
\RAY{0} \CIRCLE{1} \RAY{3}
\end{scope}
\end{scope}
\begin{scope}[shift={(35,0)}]
\node[font=\small] at (1.5,-5.5) {$x_{11} y_{11}$};
\begin{scope}[shift={(0,-3)}] 
\UP{0} \UP{1} \DN{2} \DN{3}
\RAY{0} \CIRCLE{1} \RAY{3}
\end{scope}
\end{scope}
\end{tikzpicture}
\caption{A diagrammatic illustration of the unique A$_\infty$ deformation of $\K_2^2$ given by $m_4 (y_2 x_{32} x_{22} x_{21}) = x_{11} y_{11}$ where the arrow marked with $\ast$ is of degree $-2$ and breaks the usual diagrammatic multiplication rule of $\K_2^2$}
\label{figure:diagramdeformation}
\end{figure}
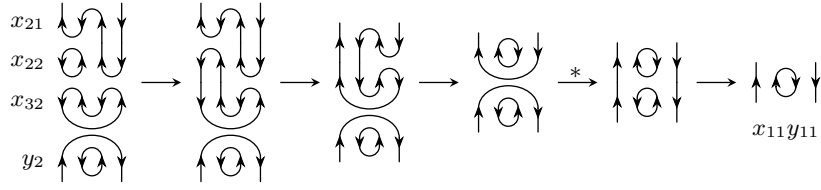

See Fig.~\ref{figure:diagramdeformation} for a diagrammatic interpretation of the A$_\infty$ deformation of $\K_2^2$ given in Corollary \ref{corollary:Ainfinitydeformation22}. (A diagrammatic description of the arrows in $\K_2^2$ is given in Fig.~\ref{figure:22}.) This A$_\infty$ deformation appears to correspond to using the usual surgery/reorientation rules for $\K_m^n$, but breaking the rule $\zeta \otimes \zeta \tikzmapsto 0$ exactly once by using $\zeta \otimes \zeta \tikzmapsto \zeta \otimes \zeta$ of degree $-2$ of the form
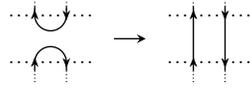
\[
\begin{tikzpicture}[x=.6em,y=.6em,decoration={markings,mark=at position 0.99 with {\arrow[black]{Stealth[length=3.8pt]}}}]
\begin{scope} 
\DDOTS{0} \UP{1} \DDOTS{2} \DN{3} \DDOTS{4}
\draw[line width=.5pt,line cap=round] (1,.5) -- (1,0) arc[start angle=-180, end angle=0, radius=1] -- ++(0,.5);
\draw[dash pattern=on 0pt off 1.3pt, line width=.5pt, line cap=round] (1,.35) -- ++(0,.7);
\draw[dash pattern=on 0pt off 1.3pt, line width=.5pt, line cap=round] (3,.35) -- ++(0,.7);
\end{scope}
\begin{scope}[shift={(0,-3)}] 
\DDOTS{0} \UP{1} \DDOTS{2} \DN{3} \DDOTS{4}
\draw[line width=.5pt,line cap=round] (1,-.5) -- (1,0) arc[start angle=180, end angle=0, radius=1] -- ++(0,-.5);
\draw[dash pattern=on 0pt off 1.3pt, line width=.5pt, line cap=round] (1,-.55) -- ++(0,-.7);
\draw[dash pattern=on 0pt off 1.3pt, line width=.5pt, line cap=round] (3,-.55) -- ++(0,-.7);
\end{scope}
\draw[->, line width=.5pt] (6,-1.5) -- ++(2,0);
\begin{scope}[shift={(10,0)}]
\begin{scope}[shift={(0,-3)}] 
\DDOTS{0} \UP{1} \DDOTS{2} \DN{3} \DDOTS{4}
\draw[dash pattern=on 0pt off 1.3pt, line width=.5pt, line cap=round] (1,-.55) -- ++(0,-.7);
\draw[dash pattern=on 0pt off 1.3pt, line width=.5pt, line cap=round] (3,-.55) -- ++(0,-.7);
\end{scope}
\begin{scope} 
\DDOTS{0} \UP{1} \DDOTS{2} \DN{3} \DDOTS{4}
\draw[line width=.5pt,line cap=round] (1,.5) -- (1,-3.5);
\draw[line width=.5pt,line cap=round] (3,.5) -- (3,-3.5);
\draw[dash pattern=on 0pt off 1.3pt, line width=.5pt, line cap=round] (1,.55) -- ++(0,.7);
\draw[dash pattern=on 0pt off 1.3pt, line width=.5pt, line cap=round] (3,.55) -- ++(0,.7);
\end{scope}
\end{scope}
\end{tikzpicture}
\]
(For higher $m, n$ also other rules appear to be broken, but we do not give a full description here.) It would also be interesting to compare this to a diagrammatic interpretation of the deformations of $\KK_m^n$ (cf.\ Remark \ref{remark:koszuldiagrammatic}).

\begin{remark}\label{remark:k22}
We also have $\HH_{0}^2(\K_2^2, \K_2^2) \neq 0$. For this, we show that there is a nontrivial cocycle given by $\widetilde \varphi \in \Hom_{\Bbbk Q_0^\e}(\Bbbk\SS_2^2, \KK_2^2)_0 = \Hom_{\Bbbk Q_0^\e}(\Bbbk\SS_2^2, (\KK_2^2)_{2})$ such that 
\begin{align*}
\widetilde\varphi_{\bar y_{11} \bar x_{11}} = \bar x_2 \bar y_2
\end{align*}
and $\widetilde\varphi_s = 0$ for all the other $s \in \SS_2^2$. The overlaps are resolvable in the first-order deformation $\{ (s, \varphi_s + \widetilde\varphi_s t) \}_{s \in \SS_2^2}$ as follows. The involved overlaps are 
$ \bar y_{11} \bar x_{11} \bar x_2$ and $\bar y_2 \bar y_{11} \bar x_{11}$. Let us verify the claim for $\bar y_{11} \bar x_{11} \bar x_2$. We have
\begin{align*}
 \bar y_{11} \bar x_{11} \bar x_2 &\tikzmapsto - \bar x_{21}\bar y_{21} \bar x_2 - \bar x_{12} \bar y_{12} \bar x_2 +\bar x_2 \bar y_2 \bar x_2 t \tikzmapsto 0\\
 \bar y_{11} \bar x_{11} \bar x_2 &\tikzmapsto 0
 \end{align*}
 where we use $\bar y_2 \bar x_2 = 0$ and $\bar x_{11} \bar x_2 = 0$.  Let us show that $\widetilde\varphi$ is not a coboundary. Note that an element $\psi \in \Hom_{\Bbbk Q_0^\e}(\Bbbk (\QQ_2^2)_1, (\KK_2^2)_1)$ has the following general form 
 $\psi(a) =\beta_a a$ for each arrow $a \in \QQ_2^2$ where $\beta_a \in \Bbbk$. Since $\bar x_2 \bar y_2$ does not appear in $\RR_2^2$ it follows that $\delta(\psi)_{\bar y_{11} \bar x_{11}}$ cannot equal $\bar x_2 \bar y_2.$ This shows that $\K_2^2$ also admits a nontrivial (graded) associative deformation.
\end{remark}

\subsection{The general case $\K_m^n$}

The case $(m,n) = (2,2)$ was proved in Proposition \ref{proposition:k22}. We thus assume that $(m, n) \neq (2, 2)$. Although $\KK_m^n$ is generally much more complicated than $\KK_2^2$ (cf.\ Fig.~\ref{figure:graphs}), we may exploit the fact that the long irreducible paths in $\KK_m^n$ have a structure quite similar to those of $\KK_2^2$. (By Lemmas \ref{lemma:vanishing} and \ref{lemma:vanishing2mn-4}, $\HH^2_{2mn-6}$ appearing in Theorem \ref{thm:maintheorem} is the top nonvanishing component of $\HH^2$ with respect to the lower index corresponding to the Adams grading.) The computation of $\HH^2_{2mn-6} (\KK_m^n, \KK_m^n)$ for general $(m, n)$ then parallels the proof for $(m, n) = (2, 2)$ and is for the most part an exercise in labelling the arrows appearing in long paths and keeping track of the indices which is why we relegate the technical computations to the appendix. The arrows appearing in the computation are illustrated in Fig.~\ref{figure:counterexample}.

Note that any element $\widetilde \varphi  \in \Hom_{\Bbbk Q_0^\e}(\Bbbk\SS_m^n, \KK_m^n)_{2mn-6}$ (cf.\ \eqref{cochain22}) is of the form 
\begin{align}\label{align:generalformawidetildephimn}
\begin{aligned}
\widetilde \varphi_{\bar Y_0 \bar X_0} &= \alpha_1 \bar x_{1 \ldots nm-2} \bar y_{nm-2 \ldots 1} + \alpha_2\bar x_{1 \ldots nm-1} \bar y_{nm-1 \ldots n+2} \bar y_{n+1 \ldots 3}^2 \\
&\quad\ + \alpha_3 \bar x_{3 \ldots n+1}^2 \bar x_{n+2 \ldots nm-1} \bar y_{nm-1 \ldots 1} \\
\widetilde \varphi_{\bar Y_1 \bar X_1} &= \alpha_4 \bar x_{2 \ldots nm-1} \bar y_{nm-1 \ldots 2} \\
\widetilde \varphi_{\bar Y_3 \bar X_3} &= \alpha_5 \bar x_{2 \ldots n}^1 \bar x_{n+1 \ldots nm-1} \bar y_{nm-1 \ldots n+1} \bar y_{n \ldots 2}^1 \\
\widetilde \varphi_{\bar X_0 \bar X_7} &= \alpha_6 \bar x_{0 \ldots nm-1} \bar y_{nm-1 \ldots n+2} \bar y_{n+1 \ldots 4}^2 \\
\widetilde \varphi_{\bar Y_7 \bar Y_0} &= \alpha_7 \bar x_{4 \ldots n+1}^2 \bar x_{n+2 \ldots nm-1} \bar y_{nm-1 \ldots 0} \\
\widetilde \varphi_{\bar Y_1 \bar X_3} &= \alpha_8 \bar x_{2 \ldots nm-1} \bar y_{nm-1 \ldots n+1} \bar y_{n \ldots 2}^1 \\
\widetilde \varphi_{\bar Y_2 \bar X_1} &= \alpha_9\bar x_{2 \ldots n}^1 \bar x_{n+1 \ldots nm-1} \bar y_{nm-1 \ldots 2} \\
\widetilde \varphi_{\bar X_3 \bar X_5} &= \alpha_{10} \bar x_{1 \ldots nm-1} \bar y_{nm-1 \ldots n+1} \bar y_{n \ldots 3}^1\\
\widetilde \varphi_{\bar Y_5 \bar Y_3} &= \alpha_{11} \bar x_{3 \ldots n}^1 \bar x_{n+1 \ldots nm-1} \bar y_{nm-1 \ldots 1}
\end{aligned}
\end{align}
and $\widetilde \varphi_s = 0$ for all the other $s \in \SS_m^n$, where $\alpha_1, \dotsc, \alpha_{11} \in \Bbbk$. The arrows appearing in \eqref{align:generalformawidetildephimn} are illustrated in Fig.~\ref{figure:counterexample}. The space $\Hom_{\Bbbk Q_0^\e} (\Bbbk\SS_m^n, \KK_m^n)_{2mn-6}$ is thus of dimension $11$. 

Let $w$ be any overlap of $\SS_m^n$ starting from $e_\lambda$ and ending at $e_\mu$ for some weights $\lambda, \mu$. Observe that $\lvert \lambda \rvert +  \lvert \mu \rvert \leq 2mn-4$ (for instance, the equality only holds for $\bar y_0 \bar x_0 \bar x_3^2$ and $\bar y_3^2 \bar y_0 \bar x_0$). By Remark \ref{remark:irreduciblegrading}, there are no irreducible paths of length $\geq 2mn-3$ parallel to $w$. So all overlaps will reduce to zero after using $\widetilde \varphi_s$, whence $\widetilde R = \{ (s, \varphi_s + \widetilde\varphi_s t) \}_{s \in \SS_m^n}$ is vacuously reduction-unique --- even when evaluating $t$ to any constant in $\Bbbk$.

\begin{proposition}\label{proposition:coboundary}
The element $\widetilde \varphi$ in \eqref{align:generalformawidetildephimn} induces a trivial first-order deformation of $R$ if and only if the coefficients $\alpha_1, \dotsc, \alpha_{11}$ satisfy
\begin{align}\label{align:constraint}
-(-1)^n\alpha_2 + (-1)^n \alpha_3 - (-1)^n \alpha_6 + (-1)^n  \alpha_7 - \alpha_8 + \alpha_9 + \alpha_{10} - \alpha_{11} = 0.
\end{align}
\end{proposition}

The proof, which relies on some long and technical computations, can be found in the appendix. Proposition \ref{proposition:coboundary} directly yields our main result Theorem \ref{thm:maintheorem}. 
    
\begin{proof}[Proof of Theorem \ref{thm:maintheorem}]
It follows from Proposition \ref{proposition:coboundary} that the coboundary space is of dimension $10$. Thus, $\HH_{2mn-6}^2(\KK_m^n, \KK_m^n)$ is $1$-dimensional.
\end{proof}
 
\begin{corollary}\label{corollary:Ainfinitydeformation}
For $m, n\geq 2$, the algebra $\K_m^n$ admits a unique A$_\infty$ deformation such that the only nontrivial higher product restricted to arrows is given by 
\begin{align*}
m_{2mn-4}(y_1 \otimes \dotsb \otimes y_{nm-1} \otimes x_{nm-1} \otimes \dotsb \otimes x_{n+2} \otimes x_{n+1}^2 \otimes \dotsb \otimes x_3^2) = x_0 y_0
\end{align*}
where the arrows in this formula are illustrated in Fig.~\ref{figure:counterexample}.
\end{corollary}

\begin{proof}
Consider  $\widetilde \varphi$ in \eqref{align:generalformawidetildephimn} and take $\alpha_2 = 1$ and $\alpha_i = 0$ for $i \neq 2$. By Proposition \ref{proposition:coboundary} this $\widetilde \varphi$ gives a nontrivial cocycle of $\HH_{2mn-6}^2(\KK_m^n, \KK_m^n)$. As in the proof of Corollary \ref{corollary:Ainfinitydeformation22} this induces a filtered deformation of $\KK_m^n$, namely the algebra 
$$
B:=\Bbbk \QQ_m^n / (s-\varphi_s - \widetilde \varphi_s)_{s \in \SS_m^n}
$$
is a filtered deformation of $\KK_m^n$. By Proposition \ref{prop:minimalmodel} we obtain that the minimal model $\Ext_B^\hdot (\Bbbk Q_0, \Bbbk Q_0)$ is an A$_\infty$ deformation of $\K_m^n$. 
\end{proof}

See Fig.~\ref{figure:diagramdeformation} for a diagrammatic interpretation of this deformation for the case $m = n = 2$.

\begin{remark}[The case $m\geq n = 1$]
If $m \geq n =1$ then by \cite[Lem.~4.21]{seidelthomas} one has
$$
\HH_{i-2}^2(\K_m^1, \K_m^1) \simeq \HH_{i-2}^2(\KK_m^1, \KK_m^1) = 0 \qquad \text{for $i > 0$}
$$
i.e.\ Stroppel's Conjecture \ref{conjecture:stroppel} holds. Using the language of reduction systems, this can be seen as follows. The reduction system $\RR_m^1$ for $\KK_m^1$ (Definition \ref{definition:reductionsystem}) is
\[
\{(\bar y_k \bar x_k, -\bar x_{k+1} \bar y_{k+1}) \}_{0 \leq k \leq m-2} \cup \{ (\bar y_{m-1} \bar x_{m-1}, 0) \}.
\]
By Lemma \ref{lemma:vanishing} it suffices to prove  $\HH_{2i-2}^2(\KK_m^1, \KK_m^1) = 0$ for $1\leq i \leq m-1$.    For any fixed $i$, an element $\widetilde \varphi\in \Hom_{\Bbbk Q_0^\e}(\Bbbk\SS_m^1, \KK_m^1)_{2i-2}$ has the following general form 
$$
\widetilde \varphi_{\bar y_k \bar x_k} = 
\begin{cases}
\alpha_k \bar x_{k+1 \ldots k+i} \bar y_{k+i \ldots k+1} & \text{if } 0\leq k < m-i \\
0 & \text{otherwise}.
\end{cases}
$$
Since there are no overlaps, any such $\widetilde\varphi$ is a cocycle. We claim that it is also a coboundary. Indeed, consider $\psi \in \Hom_{\Bbbk Q_0^\e} (\Bbbk (\QQ_m^1)_1, \KK_m^1)_{2i-2}$ given by 
$$
\psi_{\bar x_k} = \beta_k \bar x_{k \ldots k+i-1} \bar y_{k+i-1 \ldots k+1}
$$
for each $0 \leq k< m-i$, where $\beta_k := \sum_{j=0}^{m-i-1-k} (-1)^{(i+1)j} \alpha_{k+j}$. Then we have 
\[
\delta (\psi)_{\bar y_k \bar x_k} = \bar y_k \psi_{\bar x_k} + \psi_{\bar x_{k+1}} \bar y_{k+1} = ((-1)^i \beta_k + \beta_{k+1}) \bar x_{k+1 \ldots k+i} \bar y_{k+i \ldots k+1}.
\]
Noting that $(-1)^i \beta_k + \beta_{k+1} = \alpha_k$ this shows that $\delta (\psi) = \widetilde \varphi$.
\end{remark}

\begin{figure} 
\begin{tikzpicture}[x=2em,y=3.4em,decoration={markings,mark=at position 0.99 with {\arrow[black]{Stealth[length=3pt]}}}]
\begin{scope}[scale=0.9]
\begin{scope}[shift={(0,0)}]
\foreach \x in {0,1,...,4}
\node[shape=ellipse,minimum height=2.4em,minimum width=2.8em] (a\x) at (0, -\x) {};
\path[-, line width=.5pt, line cap=round] (a0) edge node[left=-.4ex, yshift=-.35em,font=\tiny,pos=.1] {$\bar Y_0=\bar y_0$} (a1);
\path[-, line width=.5pt, line cap=round] (a1) edge node[left=-.4ex, yshift=-.35em,font=\tiny,pos=.1] {$\bar Y_1=\bar y_1$} (a2);
\path[-, line width=.5pt, line cap=round] (a2) edge node[left=-.4ex, yshift=-.35em,font=\tiny,pos=.1] {$\bar Y_2=\bar y_2$} (a3);
\path[-, line width=.5pt, line cap=round] (a3) edge node[left=-.4ex, yshift=-.35em,font=\tiny,pos=.1] {$\bar y_3$} (a4);
\foreach \x in {1,2,...,4}
{
\node[shape=ellipse,minimum height=2.4em,minimum width=2.8em](b\x) at (2, -1-\x) {};
}
\path[-, line width=.5pt, line cap=round] (b1) edge node[left=-.4ex, xshift=-.16em, yshift=-0em,font=\tiny,pos=.1] {$\bar Y_3$}(a1);
\path[-, line width=.5pt, line cap=round] (b2) edge node[right=.4ex, xshift=-1.4em, yshift=1.2em,font=\tiny,pos=.1] {$\bar y_4^3$} node[left=-.4ex, xshift=-.16em, yshift=-0em,font=\tiny,pos=.1] {$\bar Y_4$}(a2);
\path[-, line width=.5pt, line cap=round] (b3) edge node[right=.4ex, xshift=-1.4em, yshift=1.2em,font=\tiny,pos=.1] {$\bar y_3^{1'}$}(a3);
\path[-, line width=.5pt, line cap=round] (b4) edge node[right=.4ex, xshift=-1.4em, yshift=1.2em,font=\tiny,pos=.1]{}(a4);

\path[-, line width=.5pt, line cap=round] (b1) edge node[right=-.4ex, yshift=-.35em,font=\tiny,pos=.1] {$\bar Y_5=\bar y_2^1$} (b2);
\path[-, line width=.5pt, line cap=round] (b2) edge node[right=-.4ex, yshift=-.35em,font=\tiny,pos=.1] {$\bar y_3^1$} (b3);
\path[-, line width=.5pt, line cap=round] (b3) edge node[right=-.4ex, yshift=-.35em,font=\tiny,pos=.1] {$\bar y_4^1$} (b4);
\foreach \x in {2,3,4}
{
\node[shape=ellipse,minimum height=2.4em,minimum width=2.8em](c\x) at (4, -2-\x) {};
}
\path[-, line width=.5pt, line cap=round] (c2) edge node[right=.4ex, xshift=-.95em, yshift=.55em,font=\tiny,pos=.1] {$\bar y_3^{2'}$} (b2);
\path[-, line width=.5pt, line cap=round] (c3) edge node[right=.4ex, xshift=-1em, yshift=.8em,font=\tiny,pos=.1] {$\bar y_4^{2'}$} (b3);
\path[-, line width=.5pt, line cap=round] (c4) edge node[left=-.4ex, yshift=-.35em,font=\tiny,pos=.1] {} (b4);

\path[-, line width=.5pt, line cap=round] (c2) edge node[right=-.4ex, yshift=-.35em,font=\tiny,pos=.1] {$\bar y_4^2$} (c3);
\path[-, line width=.5pt, line cap=round] (c3) edge node[right=-.4ex, yshift=-.35em,font=\tiny,pos=.1] {$\bar y_5^2$} (c4);
\foreach \x in {3,4}
\node[shape=ellipse,minimum height=2.4em,minimum width=2.8em](d\x) at (6, -3-\x) {};
\path[-, line width=.5pt, line cap=round] (c3) edge node[right=.4ex, xshift=.55em, yshift=-.65em,font=\tiny,pos=.1] {$\bar y_5^{3'}$} (d3);
\path[-, line width=.5pt, line cap=round] (c4) edge node[left=-.4ex,  yshift=-.4em,font=\tiny,pos=.1] {} (d4);
\path[-, line width=.5pt, line cap=round] (d3) edge node[right=-.4ex, yshift=-.35em,font=\tiny,pos=.1] {$\bar y_6^3$} (d4);
\foreach \x in {1,2,...,4} {%
\node[shape=ellipse,minimum height=2.4em,minimum width=2.8em](e\x) at (12, -2-\x) {};
}
\foreach \x in {3,4}
\path[-, line width=.5pt, line cap=round] (e\x) edge node[left=-.4ex, yshift=-.35em,font=\tiny,pos=.1] {} (b\x);
\path[-, line width=.5pt, line cap=round] (e1) edge node[right=.4ex, xshift=-7em, yshift=1.5em,font=\tiny,pos=.1] {$\bar Y_6$} (b1);
\path[-, line width=.5pt, line cap=round] (e2) edge node[right=.4ex, xshift=-7em, yshift=1.8em,font=\tiny,pos=.1] {$\bar y_3^{4'}$} (b2);

\foreach \x in {2, 3, 4} {
\node[shape=ellipse,minimum height=2.4em,minimum width=2.8em](f\x) at (15, -3-\x) {};
}
\path[dotted, line width=.3pt] (f2) edge node[right=.4ex, xshift=-7.5em, yshift=1.55em,font=\tiny,pos=.1] {$\bar y_{4}^{7'}$} (c2);
\path[dotted, line width=.3pt] (f3) edge node[left=-.4ex, yshift=1.35em,font=\tiny,pos=.1] {} (c3);
\path[dotted, line width=.3pt] (f4) edge node[left=-.4ex, yshift=-.35em,font=\tiny,pos=.1] {} (c4);
\path[dotted, line width=.3pt] (f2) edge node[right=.4ex, yshift=-.35em,font=\tiny,pos=.1] {$\bar y_5^7$} (f3);
\path[dotted, line width=.3pt] (f3) edge node[right=.4ex, yshift=-.35em,font=\tiny,pos=.1] {$\bar y_6^7$} (f4);
\path[-, line width=.5pt, line cap=round] (e1) edge node[right=-.4ex, yshift=-.35em,font=\tiny,pos=.1] {$\bar y_3^4$} (e2);
\path[-, line width=.5pt, line cap=round] (e2) edge node[right=-.4ex, yshift=-.35em,font=\tiny,pos=.1] {$\bar y_4^4$} (e3);
\path[-, line width=.5pt, line cap=round] (e3) edge node[right=-.4ex, yshift=-.35em,font=\tiny,pos=.1] {$\bar y_5^4$} (e4);
\draw[-,line width=.5pt, line cap=round] (c2) .. controls (5, -2).. (a1);
\draw[-,line width=.5pt, line cap=round] (b2) .. controls (7, -4).. (d3);
\begin{scope}[shift={(4.5,-1.6)}]
\node[font=\scriptsize] at (0,0) {$\bar Y_7=\bar y_3^2$};
\end{scope}
\begin{scope}[shift={(6.5,-3.85)}]
\node[font=\scriptsize] at (0,0) {\strut$\bar y_5^3$};
\end{scope}

\begin{scope}[shift={(-.5em,0)},x=.27em,y=.45em]
\UP{0} \DN{1}
\MRAY{0} \MRAY{1}
\node[font=\scriptsize] at (3,0) {$=$};
\UP{5}
\node[font=\scriptsize] at (6.5,0) {.\hspace{-.3pt}.\hspace{-.3pt}.};
\node[font=\scriptsize] at (10.5,0) {.\hspace{-.3pt}.\hspace{-.3pt}.};
\UP{8} \DN{9} \DN{12}
\RAY{5} \RAY{8} \RAY{9} \RAY{12}
\node[font=\small] at (6.5,1.8) {\rotatebox{-90}{$\{$}};
\node[font=\scriptsize] at (6.5,3.1) {\strut$_{n}$};
\node[font=\small] at (10.5,1.8) {\rotatebox{-90}{$\{$}};
\node[font=\scriptsize] at (10.5,3.1) {\strut$_{m}$};
\end{scope}
\begin{scope}[shift={(-.5em,-1)},x=.27em,y=.45em]
\UP{0} \DN{1} \UP{2} \DN{3}
\MRAY{0} \CIRCLE{1} \MRAY{3}
\node[font=\scriptsize] at (5,0) {$=$};
\UP{7} \UP{10} \DN{11} \UP{12} \DN{13} \DN{16}
\node[font=\scriptsize] at (8.5,0) {.\hspace{-.3pt}.\hspace{-.3pt}.};
\node[font=\scriptsize] at (14.5,0) {.\hspace{-.3pt}.\hspace{-.3pt}.};
\RAY{7} \RAY{10} \CIRCLE{11} \RAY{13} \RAY{16}
\node[font=\small] at (8.5,1.8) {\rotatebox{-90}{$\{$}};
\node[font=\scriptsize] at (8.5,3.1) {\strut$_{n-1}$};
\node[font=\small] at (14.5,1.8) {\rotatebox{-90}{$\{$}};
\node[font=\scriptsize] at (14.5,3.1) {\strut$_{m-1}$};
\end{scope}
\begin{scope}[shift={(-.5em,-2)},x=.27em,y=.45em]
\UP{0} \DN{1} \UP{2} \UP{3} \DN{4}
\MRAY{0} \CIRCLE{1} \RAY{3} \MRAY{4}
\end{scope}
\begin{scope}[shift={(-.6em, -3)},x=.27em,y=.45em]
\UP{0} \DN{1} \UP{2} \UP{3} \UP{4} \DN{5}
\MRAY{0} \CIRCLE{1} \RAY{3} \RAY{4} \MRAY{5}
\end{scope}
\begin{scope}[shift={(-.5em, -4)},x=.27em,y=.45em]
\UP{0} \DN{1} \UP{2} \UP{3} \UP{4} \UP{5} \DN{6}
\MRAY{0} \CIRCLE{1} \RAY{3} \RAY{4} \RAY{5}  \MRAY{6}
\end{scope}
\begin{scope}[shift={(1.8,-2)},x=.27em,y=.45em]
\UP{0} \DN{1} \DN{2} \UP{3} \DN{4}
\MRAY{0} \RAY{1} \CIRCLE{2} \MRAY{4}
\end{scope}
\begin{scope}[shift={(1.78,-3)},x=.27em,y=.45em]
\UP{0} \DN{1} \UP{2} \DN{3} \UP{4} \DN{5	}
\MRAY{0}  \CIRCLE{1} \CIRCLE{3}  \MRAY{5}
\end{scope}
\begin{scope}[shift={(1.70,-4)},x=.27em,y=.45em]
\UP{0} \DN{1} \UP{2} \UP{3} \DN{4} \UP{5} \DN{6}
\MRAY{0} \CIRCLE{1} \RAY{3} \CIRCLE{4} \MRAY{6}
\end{scope}
\begin{scope}[shift={(1.65,-5)},x=.27em,y=.45em]
\UP{0} \DN{1} \UP{2} \UP{3} \UP{4} \DN{5} \UP{6} \DN{7}
\MRAY{0} \CIRCLE{1} \RAY{3} \RAY{4}  \CIRCLE{5} \MRAY{7}
\end{scope}
\begin{scope}[shift={(3.8,-4)},x=.27em,y=.45em]
\UP{0} \DN{1} \DN{2} \UP{3} \UP{4} \DN{5}
\MRAY{0} \SCCIRCLE{1}  \CIRCLE{2} \MRAY{5}
\end{scope}
\begin{scope}[shift={(3.8,-5)},x=.27em,y=.45em]
\UP{0} \DN{1} \UP{2} \DN{3} \UP{4} \UP{5} \DN{6}
\MRAY{0} \CIRCLE{1}  \CIRCLE{3} \RAY{5} \MRAY{6}
\end{scope}
\begin{scope}[shift={(3.8,-6)},x=.27em,y=.45em]
\UP{0} \DN{1} \UP{2} \UP{3} \DN{4} \UP{5} \UP{6} \DN{7}
\MRAY{0} \CIRCLE{1}  \RAY{3} \CIRCLE{4} \RAY{6} \MRAY{7}
\end{scope}
\begin{scope}[shift={(5.6,-6)},x=.27em,y=.45em]
\UP{0} \DN{1} \DN{2} \UP{3} \UP{4} \UP{5} \UP{6} 
\MRAY{0} \SCCIRCLE{1}   \CIRCLE{2} \RAY{5} \MRAY{6}
\end{scope}
\begin{scope}[shift={(5.6,-7)},x=.27em,y=.45em]
\UP{0} \DN{1} \UP{2} \DN{3} \UP{4} \UP{5} \UP{6}  \DN{7}
\MRAY{0} \CIRCLE{1}   \CIRCLE{3} \RAY{5} \RAY{6} \MRAY{7}
\end{scope}
\begin{scope}[shift={(11.5,-3)},x=.27em,y=.45em]
\UP{0} \DN{1} \DN{2} \DN{3} \UP{4} \DN{5} 
\MRAY{0} \RAY{1} \RAY{2} \CIRCLE{3} \MRAY{5}
\end{scope}
\begin{scope}[shift={(11.5,-4)},x=.27em,y=.45em]
\UP{0} \DN{1} \UP{2} \DN{3} \DN{4} \UP{5}  \DN{6}
\MRAY{0} \CIRCLE{1} \RAY{3} \CIRCLE{4} \MRAY{6}
\end{scope}
\begin{scope}[shift={(11.5,-5)},x=.27em,y=.45em]
\UP{0} \DN{1} \UP{2} \UP{3} \DN{4} \DN{5}  \UP{6} \DN{7}
\MRAY{0} \CIRCLE{1} \RAY{3} \RAY{4} \CIRCLE{5} \MRAY{7}
\end{scope}
\begin{scope}[shift={(11.5,-6)},x=.27em,y=.45em]
\UP{0} \DN{1} \UP{2} \UP{3} \UP{4}  \DN{5} \DN{6}  \UP{7} \DN{8}
\MRAY{0} \CIRCLE{1} \RAY{3} \RAY{4} \RAY{5} \CIRCLE{6} \MRAY{8}
\end{scope}
\begin{scope}[shift={(14.5,-5)},x=.27em,y=.45em]
\UP{0} \DN{1} \DN{2} \UP{3} \DN{4} \UP{5} \DN{6}
\MRAY{0} \RAY{1} \CIRCLE{2} \CIRCLE{4} \MRAY{6}
\end{scope}
\begin{scope}[shift={(14.5,-6)},x=.27em,y=.45em]
\UP{0} \DN{1} \UP{2} \DN{3} \UP{4} \DN{5} \UP{6} \DN{7}
\MRAY{0} \CIRCLE{1} \CIRCLE{3} \CIRCLE{5} \MRAY{7}
\end{scope}
\begin{scope}[shift={(14.5,-7)},x=.27em,y=.45em]
\UP{0} \DN{1} \UP{2} \UP{3} \DN{4} \UP{5} \DN{6} \UP{7} \DN{8}
\MRAY{0} \CIRCLE{1} \RAY{3} \CIRCLE{4} \CIRCLE{6} \MRAY{8}
\end{scope}

\end{scope}

\begin{scope}[shift={(0, -6)}]
\foreach \x in {0,1,...,5}
\node[shape=ellipse,minimum height=2.4em,minimum width=2.8em] (a\x) at (0, -\x) {};
\path[-, line width=.5pt, line cap=round] (a0) edge node[left=-.4ex, yshift=-.35em,font=\tiny,pos=.1] {$\bar y_{n-1}$} (a1);
\path[-, line width=.5pt, line cap=round] (a1) edge node[left=-.4ex, yshift=-.35em,font=\tiny,pos=.1] {$\bar y_{n}$} (a2);
\path[-, line width=.5pt, line cap=round] (a2) edge node[left=-.4ex, yshift=-.35em,font=\tiny,pos=.1] {$\bar y_{n+1}$} (a3);
\path[-, line width=.5pt, line cap=round] (a3) edge node[left=-.4ex, yshift=-.35em,font=\tiny,pos=.1] {$\bar y_{n+2}$} (a4);
\path[-, line width=.5pt, line cap=round] (a4) edge node[left=-.4ex, yshift=-.35em,font=\tiny,pos=.1] {$\bar y_{n+3}$} (a5);
\begin{scope}[shift={(9,-6)}]
\node[font=\scriptsize] at (0,0) {\strut$\bar y_{n+5}^6$};
\end{scope}
\begin{scope}[shift={(7.3,-3.35)}]
\node[font=\scriptsize] at (0,0) {\strut$\bar y_{n+3}^5$};
\end{scope}
\node[shape=ellipse,minimum height=2.4em,minimum width=2.8em](b0) at (2, -1) {};
\node[shape=ellipse,minimum height=2.4em,minimum width=2.8em](b1) at (2, -5) {};
\node[shape=ellipse,minimum height=2.4em,minimum width=2.8em](b2) at (2, -6) {};
\node[shape=ellipse,minimum height=2.4em,minimum width=2.8em](c0) at (4, -2) {};
\node[shape=ellipse,minimum height=2.4em,minimum width=2.8em](c1) at (4, -6) {};
\node[shape=ellipse,minimum height=2.4em,minimum width=2.8em](c2) at (4, -7) {};
\node[shape=ellipse,minimum height=2.4em,minimum width=2.8em](d0) at (6, -3) {};
\node[shape=ellipse,minimum height=2.4em,minimum width=2.8em](d1) at (6, -7) {};
\node[shape=ellipse,minimum height=2.4em,minimum width=2.8em](d2) at (6, -8) {};
\node[shape=ellipse,minimum height=2.4em,minimum width=2.8em](e0) at (12, -2) {};
\node[shape=ellipse,minimum height=2.4em,minimum width=2.8em](e1) at (12, -3) {};
\node[shape=ellipse,minimum height=2.4em,minimum width=2.8em](f0) at (7,-4) {};
\node[shape=ellipse,minimum height=2.4em,minimum width=2.8em](g0) at (9,-5) {};
\node[shape=ellipse,minimum height=2.4em,minimum width=2.8em](a-1) at (0,1) {};
\node[shape=ellipse,minimum height=2.4em,minimum width=2.8em](b-1) at (2,0) {};
\node[shape=ellipse,minimum height=2.4em,minimum width=2.8em](c-1) at (4,-1) {};
\node[shape=ellipse,minimum height=2.4em,minimum width=2.8em](d-1) at (6,-2) {};
\node[shape=ellipse,minimum height=2.4em,minimum width=2.8em](e-1) at (12,-1) {};
\node[shape=ellipse,minimum height=2.4em,minimum width=2.8em](f1) at (15,-2) {};
\node[shape=ellipse,minimum height=2.4em,minimum width=2.8em](f2) at (15,-3) {};
\path[dotted, line width=.3pt] (f1) edge node[left=-.4ex, yshift=-.35em,font=\tiny,pos=.1] {} (c-1);
\path[dotted, line width=.3pt] (f2) edge node[left=-.4ex, yshift=-.35em,font=\tiny,pos=.1] {} (c0);
\path[dotted, line width=.3pt] (f2) edge node[right=.4ex, xshift=-5em, yshift=-1.5em,font=\tiny,pos=.1] {$\bar y_{n+2}^7$} (f0);
\path[dotted, line width=.3pt] (f1) edge node[right=.3ex, yshift=-.35em,font=\tiny,pos=.1] {$\bar  y_{n+1}^7$ } (f2);
\path[-, line width=.5pt, line cap=round] (a-1) edge node[left=-.4ex, yshift=-.35em,font=\tiny,pos=.1] {$\bar y_{n-2}$} (a0);
\path[-, line width=.5pt, line cap=round] (b-1) edge node[right=-.4ex, yshift=-.35em,font=\tiny,pos=.1] {$\bar y_{n-1}^1$} (b0);
\path[-, line width=.5pt, line cap=round] (c-1) edge node[right=-.4ex, yshift=-.35em,font=\tiny,pos=.1] {$\bar y_n^2$} (c0);
\path[-, line width=.5pt, line cap=round] (d-1) edge node[right=-.4ex, yshift=-.35em,font=\tiny,pos=.1] {$\bar y_{n+1}^3$} (d0);
\path[-, line width=.5pt, line cap=round] (e-1) edge node[right=-.4ex, yshift=-.35em,font=\tiny,pos=.1] {$\bar y_{n}^4$} (e0);
\path[-, line width=.5pt, line cap=round] (a-1) edge node[left=-.4ex, yshift=-.35em,font=\tiny,pos=.1] {} (b-1);
\path[-, line width=.5pt, line cap=round] (b-1) edge node[left=-.4ex, yshift=-.35em,font=\tiny,pos=.1] {} (c-1);
\path[-, line width=.5pt, line cap=round] (c-1) edge node[left=-.4ex, yshift=-.35em,font=\tiny,pos=.1] {} (d-1);
\path[-, line width=.5pt, line cap=round] (b-1) edge node[left=-.4ex, yshift=-.35em,font=\tiny,pos=.1] {} (e-1);
\path[-, line width=.5pt, line cap=round] (b0) edge node[left=-.4ex, yshift=-.35em,font=\tiny,pos=.1] {} (a0);
\path[-, line width=.5pt, line cap=round] (b0) edge node[left=-.4ex, yshift=-.35em,font=\tiny,pos=.1] {} (c0);
\path[-, line width=.5pt, line cap=round] (c0) edge node[left=-.4ex, yshift=-.35em,font=\tiny,pos=.1] {} (d0);
\path[-, line width=.5pt, line cap=round] (c0) edge node[right=-.4ex, xshift=-3.5em, yshift=-1.8em,font=\tiny,pos=.1] {$\bar y_{n+1}^2$} (a3);
\path[-, line width=.5pt, line cap=round] (d0) edge node[right=-.4ex, xshift=-5.9em, yshift=-2.2em,font=\tiny,pos=.1] {$\bar y_{n+2}^3$} (a4);
\path[-, line width=.5pt, line cap=round] (e1) edge node[left=-.4ex, yshift=-.35em,font=\tiny,pos=.1] {} (a2);
\path[-, line width=.5pt, line cap=round] (e0) edge node[left=-.4ex, yshift=-.35em,font=\tiny,pos=.1] {} (b0);
\path[-, line width=.5pt, line cap=round] (e0) edge node[right=-.4ex, yshift=-.35em,font=\tiny,pos=.1] {$\bar y_{n+1}^4$} (e1);
\path[-, line width=.5pt, line cap=round] (b1) edge node[right=-.4ex, yshift=-.35em,font=\tiny,pos=.1] {$\bar y_{n+4}^4$} (b2);
\path[-, line width=.5pt, line cap=round] (c1) edge node[right=-.4ex, yshift=-.35em,font=\tiny,pos=.1] {$\bar y_{n+5}^5$} (c2);
\path[-, line width=.5pt, line cap=round] (d1) edge node[right=-.4ex, yshift=-.35em,font=\tiny,pos=.1] {$\bar y_{n+6}^6$} (d2);
\path[-, line width=.5pt, line cap=round] (b1) edge node[left=-.4ex, yshift=-.35em,font=\tiny,pos=.1] {} (a4);
\path[-, line width=.5pt, line cap=round] (b1) edge node[left=-.4ex, yshift=-.35em,font=\tiny,pos=.1] {} (c1);
\path[-, line width=.5pt, line cap=round] (c1) edge node[left=-.4ex, yshift=-.35em,font=\tiny,pos=.1] {} (d1);
\path[-, line width=.5pt, line cap=round] (a5) edge node[left=-.4ex, yshift=-.35em,font=\tiny,pos=.1] {} (b2);
\path[-, line width=.5pt, line cap=round] (b2) edge node[left=-.4ex, yshift=-.35em,font=\tiny,pos=.1] {} (c2);
\path[-, line width=.5pt, line cap=round] (c2) edge node[left=-.4ex, yshift=-.35em,font=\tiny,pos=.1] {} (d2);
\path[-, line width=.5pt, line cap=round] (f0) edge node[right=-.4ex, xshift=-2.55em, yshift=-1.55em,font=\tiny,pos=.1] {$\bar y_{n+3}^4$} (b1);
\path[-, line width=.5pt, line cap=round] (f0) edge node[left=-.4ex, xshift=3.5em, yshift=1.6em,font=\tiny,pos=.1] {$\bar y_{n+2}^4$} (e1);
\path[-, line width=.5pt, line cap=round] (f0) edge node[left=-.4ex, xshift=-.55em, yshift=-.15em,font=\tiny,pos=.1] {$\bar y_{n+4}^6$} (a3);
\path[-, line width=.5pt, line cap=round] (f0) edge node[left=-.4ex, xshift=.7em, yshift=-.7em,font=\tiny,pos=.1] {$\bar y_{n+3}^{5'}$} (g0);
\path[-, line width=.5pt, line cap=round] (g0) edge node[right=-.4ex, xshift=-2.5em, yshift=-1.35em,font=\tiny,pos=.1] {$\bar y_{n+4}^5$} (c1);
\path[-, line width=.5pt, line cap=round] (b0) edge node[right=-.4ex, xshift=-.8em,yshift=-.75em,font=\tiny,pos=.1] {$\bar y_n^1$} (a2);
\draw[-,line width=.5pt, line cap=round] (d1) .. controls (11.3, -4.7).. (f0);
\draw[-,line width=.5pt, line cap=round] (g0) .. controls (10.5, -4.3).. (a2);
\draw[dotted, line width=.6pt] (a-1) -- ++(0,.7);
\draw[dotted, line width=.6pt] (b-1) -- ++(0,.7);
\draw[dotted, line width=.6pt] (c-1) -- ++(0,.7);
\draw[dotted, line width=.6pt] (d-1) -- ++(0,.7);
\draw[dotted, line width=.6pt] (e-1) -- ++(0,.7);
\draw[dotted, line width=.6pt] (f1) -- ++(0,.7);
\draw[dotted, line width=.6pt] (a5) -- ++(0,-.7);
\draw[dotted, line width=.6pt] (b2) -- ++(0,-.7);
\draw[dotted, line width=.6pt] (c2) -- ++(0,-.7);
\draw[dotted, line width=.6pt] (d2) -- ++(0,-.7);
\begin{scope}[shift={(-.5, 1)},x=.27em,y=.45em]
\UP{0} \UP{1} \DN{2} \UP{3} \UP{4} \DN{5}  
\RAY{0} \RAY{1} \CIRCLE{2}  \MRAY{4} \MRAY{5}
\end{scope}
\begin{scope}[shift={(-.5, 0)},x=.27em,y=.45em]
\UP{0} \DN{1} \UP{2} \UP{3} \DN{4} 
\RAY{0} \CIRCLE{1}  \MRAY{3} \MRAY{4}
\end{scope}
\begin{scope}[shift={(-.5, -1)},x=.27em,y=.45em]
\DN{0} \UP{1} \UP{2} \DN{3} 
\CIRCLE{0}  \MRAY{2} \MRAY{3}
\end{scope}
\begin{scope}[shift={(-.5, -2)},x=.27em,y=.45em]
\DN{0} \UP{1} \UP{2} \DN{3} \UP{4} \DN{5}
\CIRCLE{0}  \MRAY{2} \CIRCLE{3} \MRAY{5}
\end{scope}
\begin{scope}[shift={(-.5, -3)},x=.27em,y=.45em]
\DN{0} \UP{1} \UP{2} \DN{3} \UP{4} \UP{5} \DN{6}
\CIRCLE{0}  \MRAY{2} \CIRCLE{3} \RAY{5} \MRAY{6}
\end{scope}
\begin{scope}[shift={(-.5, -4)},x=.27em,y=.45em]
\DN{0} \UP{1} \UP{2} \DN{3} \UP{4} \UP{5} \UP{6} \DN{7}
\CIRCLE{0}  \MRAY{2} \CIRCLE{3} \RAY{5} \RAY{6} \MRAY{7}
\end{scope}
\begin{scope}[shift={(-.5, -5)},x=.27em,y=.45em]
\DN{0} \UP{1} \UP{2} \DN{3} \UP{4} \UP{5} \UP{6} \UP{7} \DN{8}
\CIRCLE{0}  \MRAY{2} \CIRCLE{3} \RAY{5} \RAY{6} \RAY{7} \MRAY{8}
\end{scope}
\begin{scope}[shift={(1.5, 0)},x=.27em,y=.45em]
\UP{0} \UP{1} \DN{2} \UP{3} \UP{4} \DN{5} \UP{6} \DN{7} 
\RAY{0} \RAY{1} \CIRCLE{2}  \MRAY{4} \CIRCLE{5} \MRAY{7}
\end{scope}
\begin{scope}[shift={(1.5, -1)},x=.27em,y=.45em]
\UP{0} \DN{1} \UP{2} \UP{3} \DN{4} \UP{5} \DN{6}
\RAY{0} \CIRCLE{1}  \MRAY{3} \CIRCLE{4} \MRAY{6}
\end{scope}
\begin{scope}[shift={(1.5, -5)},x=.27em,y=.45em]
\DN{0} \UP{1} \UP{2} \DN{3} \UP{4} \UP{5} \DN{6} \UP{7} \DN{8}
\CIRCLE{0}  \MRAY{2} \CIRCLE{3} \RAY{5} \CIRCLE{6} \MRAY{8}
\end{scope}
\begin{scope}[shift={(1.5, -6)},x=.27em,y=.45em]
\DN{0} \UP{1} \UP{2} \DN{3} \UP{4} \UP{5} \UP{6} \DN{7} \UP{8} \DN{9}
\CIRCLE{0}  \MRAY{2} \CIRCLE{3} \RAY{5} \RAY{6} \CIRCLE{7} \MRAY{9}
\end{scope}
\begin{scope}[shift={(3.5, -1)},x=.27em,y=.45em]
\UP{0} \UP{1}  \DN{2} \UP{3} \UP{4} \DN{5} \UP{6} \UP{7} \DN{8} 
\RAY{0} \RAY{1} \CIRCLE{2}  \MRAY{4} \CIRCLE{5} \RAY{7} \MRAY{8}
\end{scope}
\begin{scope}[shift={(3.5, -2)},x=.27em,y=.45em]
\UP{0}  \DN{1} \UP{2} \UP{3} \DN{4} \UP{5} \UP{6} \DN{7} 
\RAY{0} \CIRCLE{1}  \MRAY{3} \CIRCLE{4} \RAY{6} \MRAY{7}
\end{scope}
\begin{scope}[shift={(3.5, -6)},x=.27em,y=.45em]
\DN{0} \UP{1} \UP{2} \DN{3} \UP{4} \DN{5} \UP{6} \UP{7} \DN{8} 
\CIRCLE{0}  \MRAY{2} \CIRCLE{3} \CIRCLE{5} \RAY{7} \MRAY{8}
\end{scope}
\begin{scope}[shift={(3.5, -7)},x=.27em,y=.45em]
\DN{0} \UP{1} \UP{2} \DN{3} \UP{4} \UP{5} \DN{6} \UP{7} \UP{8} \DN{9}
\CIRCLE{0}  \MRAY{2} \CIRCLE{3} \RAY{5} \CIRCLE{6} \RAY{8} \MRAY{9}
\end{scope}
\begin{scope}[shift={(5.5, -2)},x=.27em,y=.45em]
\UP{0} \UP{1}  \DN{2} \UP{3} \UP{4} \DN{5} \UP{6} \UP{7} \UP{8} \DN{9} 
\RAY{0} \RAY{1} \CIRCLE{2}  \MRAY{4} \CIRCLE{5} \RAY{7} \RAY{8} \MRAY{9}
\end{scope}
\begin{scope}[shift={(5.5, -3)},x=.27em,y=.45em]
\UP{0}  \DN{1} \UP{2} \UP{3} \DN{4} \UP{5} \UP{6} \UP{7} \DN{8} 
\RAY{0} \CIRCLE{1}  \MRAY{3} \CIRCLE{4} \RAY{6} \RAY{7} \MRAY{8}
\end{scope}
\begin{scope}[shift={(5.5, -7)},x=.27em,y=.45em]
\DN{0}  \UP{1} \UP{2} \DN{3} \DN{4} \UP{5} \UP{6} \UP{7} \DN{8} 
 \CIRCLE{0}  \MRAY{2} \SCCIRCLE{3} \CIRCLE{4} \RAY{7} \MRAY{8}
\end{scope}
\begin{scope}[shift={(5.5, -8)},x=.27em,y=.45em]
\DN{0} \UP{1} \UP{2} \DN{3} \UP{4} \DN{5} \UP{6} \UP{7} \UP{8} \DN{9}
\CIRCLE{0}  \MRAY{2} \CIRCLE{3} \CIRCLE{5} \RAY{7} \RAY{8} \MRAY{9}
\end{scope}
\begin{scope}[shift={(11.5, -1)},x=.27em,y=.45em]
\UP{0} \UP{1} \DN{2} \UP{3} \UP{4} \DN{5} \DN{6} \UP{7} \DN{8}
\RAY{0} \RAY{1} \CIRCLE{2}  \MRAY{4} \RAY{5} \CIRCLE{6} \MRAY{8}
\end{scope}
\begin{scope}[shift={(11.5, -2)},x=.27em,y=.45em]
\UP{0} \DN{1} \UP{2} \UP{3} \DN{4} \DN{5} \UP{6} \DN{7}
\RAY{0} \CIRCLE{1}  \MRAY{3} \RAY{4} \CIRCLE{5} \MRAY{7}
\end{scope}
\begin{scope}[shift={(11.5, -3)},x=.27em,y=.45em]
 \DN{0} \UP{1} \UP{2} \DN{3} \DN{4} \UP{5} \DN{6}
 \CIRCLE{0}  \MRAY{2} \RAY{3} \CIRCLE{4} \MRAY{6}
\end{scope}
\begin{scope}[shift={(6.5,-4)},x=.27em,y=.45em]
 \DN{0} \UP{1} \UP{2} \DN{3} \UP{4} \DN{5} \UP{6} \DN{7}
 \CIRCLE{0}  \MRAY{2}  \CIRCLE{3}  \CIRCLE{5} \MRAY{7}
\end{scope}
\begin{scope}[shift={(8.5,-5)},x=.27em,y=.45em]
 \DN{0} \UP{1} \UP{2} \DN{3} \DN{4} \UP{5} \UP{6} \DN{7}
 \CIRCLE{0}  \MRAY{2}  \SCCIRCLE{3}  \CIRCLE{4} \MRAY{7}
\end{scope}
\begin{scope}[shift={(14.5,-2)},x=.27em,y=.45em]
 \UP{0} \UP{1} \DN{2}\UP{3} \UP{4} \DN{5} \UP{6} \DN{7} \UP{8} \DN{9}
 \RAY{0} \RAY{1} \CIRCLE{2} \MRAY{4} \CIRCLE{5} \CIRCLE{7} \MRAY{9}
\end{scope}
\begin{scope}[shift={(14.5,-3)},x=.27em,y=.45em]
\UP{0} \DN{1} \UP{2} \UP{3} \DN{4} \UP{5} \DN{6} \UP{7} \DN{8}
\RAY{0} \CIRCLE{1} \MRAY{3} \CIRCLE{4} \CIRCLE{6} \MRAY{8}
\end{scope}
\end{scope}

\begin{scope}[shift={(0, -12)}]
\foreach \x in {0,1,...,5}
\node[shape=ellipse,minimum height=2.4em,minimum width=2.8em] (a\x) at (0, -\x) {};
\path[-, line width=.5pt, line cap=round] (a0) edge node[left=-.4ex, yshift=-.35em,font=\tiny,pos=.1] {$\bar y_{2n-2}$} (a1);
\path[-, line width=.5pt, line cap=round] (a1) edge node[left=-.4ex, yshift=-.35em,font=\tiny,pos=.1] {$\bar y_{2n-1}$} (a2);
\path[-, line width=.5pt, line cap=round] (a2) edge node[left=-.4ex, yshift=-.35em,font=\tiny,pos=.1] {$\bar y_{2n}$} (a3);
\path[-, line width=.5pt, line cap=round] (a3) edge node[left=-.4ex, yshift=-.35em,font=\tiny,pos=.1] {$\bar y_{2n+1}$} (a4);
\path[-, line width=.5pt, line cap=round] (a4) edge node[left=-.4ex, yshift=-.35em,font=\tiny,pos=.1] {$\bar y_{2n+2}$} (a5);
\node[shape=ellipse,minimum height=2.4em,minimum width=2.8em](b0) at (2, -1) {};
\node[shape=ellipse,minimum height=2.4em,minimum width=2.8em](b1) at (2, -2) {};
\node[shape=ellipse,minimum height=2.4em,minimum width=2.8em](c0) at (4, -2) {};
\node[shape=ellipse,minimum height=2.4em,minimum width=2.8em](c1) at (4, -3) {};
\node[shape=ellipse,minimum height=2.4em,minimum width=2.8em](d0) at (6, -3) {};
\node[shape=ellipse,minimum height=2.4em,minimum width=2.8em](d1) at (6, -4) {};
\path[-, line width=.5pt, line cap=round] (a0) edge node[left=-.4ex, yshift=-.35em,font=\tiny,pos=.1] {} (b0);
\path[-, line width=.5pt, line cap=round] (b0) edge node[left=-.4ex, yshift=-.35em,font=\tiny,pos=.1] {} (c0);
\path[-, line width=.5pt, line cap=round] (d0) edge node[left=-.4ex, yshift=-.35em,font=\tiny,pos=.1] {} (c0);
\path[-, line width=.5pt, line cap=round] (b1) edge node[left=-.4ex, yshift=-.35em,font=\tiny,pos=.1] {} (a1);
\path[-, line width=.5pt, line cap=round] (b1) edge node[left=-.4ex, yshift=-.35em,font=\tiny,pos=.1] {} (c1);
\path[-, line width=.5pt, line cap=round] (c1) edge node[left=-.4ex, yshift=-.35em,font=\tiny,pos=.1] {} (d1);
\path[-, line width=.5pt, line cap=round] (b0) edge node[right=-.55ex, yshift=-.38em,font=\tiny,pos=.1] {$\bar y_{2n-1}^4$} (b1);
\path[-, line width=.5pt, line cap=round] (c0) edge node[right=-.4ex, yshift=-.35em,font=\tiny,pos=.1] {$\bar y_{2n}^5$} (c1);
\path[-, line width=.5pt, line cap=round] (d0) edge node[right=-.4ex, yshift=-.35em,font=\tiny,pos=.1] {$\bar y_{2n+1}^6$} (d1);
\path[-, line width=.5pt, line cap=round] (a3) edge node[right=.4ex, xshift=.3em, yshift=.15em,font=\tiny,pos=.1] {$\bar y_{2n}^4$} (b1);
\path[-, line width=.5pt, line cap=round] (a4) edge node[right=.4ex, xshift=1.35em, yshift=.2em,font=\tiny,pos=.1] {$\bar y_{2n+1}^5$} (c1);
\path[-, line width=.5pt, line cap=round] (a5) edge node[right=.4ex, xshift=1.95em, yshift=.15em,font=\tiny,pos=.1] {$\bar y_{2n+2}^6$} (d1);
\begin{scope}[shift={(-.5,0)},x=.27em,y=.45em]
 \DN{0} \UP{1} \UP{2} \DN{3} \UP{4} \UP{5} \DN{6} 
  \CIRCLE{0}  \RAY{2}  \CIRCLE{3}  \MRAY{5} \MRAY{6}
\end{scope}
\begin{scope}[shift={(-.5,-1)},x=.27em,y=.45em]
 \DN{0} \UP{1} \DN{2} \UP{3} \UP{4} \DN{5}
  \CIRCLE{0}  \CIRCLE{2}   \MRAY{4} \MRAY{5}
\end{scope}
\begin{scope}[shift={(-.5,-2)},x=.27em,y=.45em]
 \DN{0} \DN{1} \UP{2} \UP{3} \UP{4} \DN{5}
  \SCCIRCLE{0}  \CIRCLE{1}   \MRAY{4} \MRAY{5}
\end{scope}
\begin{scope}[shift={(-.5,-3)},x=.27em,y=.45em]
 \DN{0} \DN{1} \UP{2} \UP{3} \UP{4} \DN{5} \UP{6} \DN{7}
  \SCCIRCLE{0}  \CIRCLE{1}   \MRAY{4} \CIRCLE{5} \MRAY{7}
\end{scope}
\begin{scope}[shift={(-.7,-4)},x=.27em,y=.45em]
 \DN{0} \DN{1} \UP{2} \UP{3} \UP{4} \DN{5} \UP{6} \UP{7} \DN{8}
  \SCCIRCLE{0}  \CIRCLE{1}   \MRAY{4} \CIRCLE{5} \RAY{7} \MRAY{8}
\end{scope}
\begin{scope}[shift={(-.7,-5)},x=.27em,y=.45em]
 \DN{0} \DN{1} \UP{2} \UP{3} \UP{4} \DN{5} \UP{6} \UP{7} \UP{8} \DN{9}
  \SCCIRCLE{0}  \CIRCLE{1}   \MRAY{4} \CIRCLE{5} \RAY{7} \RAY{8} \MRAY{9}
\end{scope}
\begin{scope}[shift={(1.5,-1)},x=.27em,y=.45em]
 \DN{0} \UP{1} \UP{2} \DN{3} \UP{4} \UP{5} \DN{6} \UP{7} \DN{8}
  \CIRCLE{0}  \RAY{2}  \CIRCLE{3}   \MRAY{5} \CIRCLE{6} \MRAY{8}
\end{scope}
\begin{scope}[shift={(1.5,-2)},x=.27em,y=.45em]
 \DN{0} \UP{1} \DN{2} \UP{3} \UP{4} \DN{5} \UP{6} \DN{7} 
  \CIRCLE{0}    \CIRCLE{2}   \MRAY{4} \CIRCLE{5} \MRAY{7}
\end{scope}
\begin{scope}[shift={(3.5,-2)},x=.27em,y=.45em]
 \DN{0} \UP{1} \UP{2} \DN{3} \UP{4} \UP{5} \DN{6} \UP{7} \UP{8} \DN{9}
  \CIRCLE{0}  \RAY{2}  \CIRCLE{3}   \MRAY{5} \CIRCLE{6} \RAY{8} \MRAY{9}
\end{scope}
\begin{scope}[shift={(3.5,-3)},x=.27em,y=.45em]
 \DN{0} \UP{1} \DN{2} \UP{3} \UP{4} \DN{5} \UP{6} \UP{7} \DN{8}
  \CIRCLE{0}    \CIRCLE{2}   \MRAY{4} \CIRCLE{5} \RAY{7} \MRAY{8}
\end{scope}
\begin{scope}[shift={(5.5,-3)},x=.27em,y=.45em]
 \DN{0} \UP{1} \UP{2} \DN{3} \UP{4} \UP{5} \DN{6} \UP{7} \UP{8} \UP{9} \DN{10}
  \CIRCLE{0}  \RAY{2}  \CIRCLE{3}   \MRAY{5} \CIRCLE{6} \RAY{8} \RAY{9} \MRAY{10}
\end{scope}
\begin{scope}[shift={(5.5,-4)},x=.27em,y=.45em]
 \DN{0} \UP{1} \DN{2} \UP{3} \UP{4} \DN{5} \UP{6} \UP{7} \UP{8} \DN{9}
  \CIRCLE{0}    \CIRCLE{2}   \MRAY{4} \CIRCLE{5} \RAY{7} \RAY{8} \MRAY{9}
\end{scope}
\end{scope}

\begin{scope}[scale = 1.05, shift={(12.5, -10)}]
\foreach \x in {0,1,...,6}
\node[shape=ellipse,minimum height=2.4em,minimum width=2.8em] (a\x) at (0, -\x) {};
\path[-, line width=.5pt, line cap=round] (a0) edge node[left=-.4ex, yshift=-.35em,font=\tiny,pos=.1] {\strut$\bar y_{n(m-1)-1}$} (a1);
\path[-, line width=.5pt, line cap=round] (a1) edge node[left=-.4ex, yshift=-.35em,font=\tiny,pos=.1] {$\bar y_{n(m-1)}$} (a2);
\path[-, line width=.5pt, line cap=round] (a2) edge node[left=-.4ex, yshift=-.35em,font=\tiny,pos=.1] {$\bar y_{n(m-1)+1}$} (a3);
\draw[dotted, line width=.6pt] (a3) -- ++(a4);
\path[-, line width=.5pt, line cap=round] (a4) edge node[left=-.4ex, yshift=-.35em,font=\tiny,pos=.1] {$\bar y_{nm-2}$} (a5);
\path[-, line width=.5pt, line cap=round] (a5) edge node[left=-.4ex, yshift=-.35em,font=\tiny,pos=.1] {$\bar y_{nm-1}$} (a6);
\node[shape=ellipse,minimum height=2.4em,minimum width=2.8em](b0) at (2, -1) {};
\node[shape=ellipse,minimum height=2.4em,minimum width=2.8em](b1) at (2, -2) {};
\node[shape=ellipse,minimum height=2.4em,minimum width=2.8em](b2) at (2, -3) {};
\node[shape=ellipse,minimum height=2.4em,minimum width=2.8em](b3) at (2, -4) {};
\path[-, line width=.5pt, line cap=round] (a0) edge node[right=.4ex, xshift=-.1em, yshift=-.1em, font=\tiny,pos=.1] {$\bar y_{n(m-1)-1}^0$} (b0);
\path[-, line width=.5pt, line cap=round] (b0) edge node[right=3ex, yshift=-.6em,font=\tiny,pos=.1] {} (a2);
\path[-, line width=.5pt, line cap=round] (b0) edge node[right=.4ex, yshift=-.25em,font=\tiny,pos=.1] {\strut$\bar y_{n(m-1)}^0$} (b1);
\path[-, line width=.5pt, line cap=round] (b1) edge node[left=-.4ex, yshift=-.35em,font=\tiny,pos=.1] {} (a3);
\path[-, line width=.5pt, line cap=round] (b2) edge node[right=.1ex, xshift=-1.6em, yshift=-1.4em, font=\tiny,pos=.1] {$\bar y_{nm-3}'$} (a4);
\path[-, line width=.5pt, line cap=round] (b3) edge node[right=.1ex, xshift=-1.2em, yshift=-1em,font=\tiny,pos=.1] {$\bar y_{nm-2}'$} (a5);
\draw[dotted, line width=.6pt] (b1) -- (b2);
\path[-, line width=.5pt, line cap=round] (b2) edge node[right=2ex, xshift=-.5em, yshift=-.5em,font=\tiny,pos=.1] {$\bar y_{nm-3}^0$} (b3);
\draw[dotted, line width=.6pt] (a0) -- ++(0,.7);
\begin{scope}[shift={(-1,0)},x=.27em,y=.4em]
 \DN{0} \DN{1} \DDOT{2} \DN{3} \DN{4} \UP{5} \DN{6} \UP{7} \UP{8} \DDOT{9} \UP{10} \DN{11} 
  \MRAY{0}  \SCIRCLES{1}{4.5}  \SCIRCLES{3}{2.5}  \CIRCLE{4} \CIRCLE{6} \RAY{11}
\end{scope}
\begin{scope}[shift={(-1,-1)},x=.27em,y=.4em]
 \DN{0} \DN{1} \DDOT{2} \DN{3} \DN{4} \UP{5} \UP{6}\DDOT{7} \UP{8} \DN{9} 
  \MRAY{0}  \SCIRCLES{1}{3.5}  \SCCIRCLE{3}  \CIRCLE{4}  \RAY{9}
\end{scope}
\begin{scope}[shift={(-1,-2)},x=.27em,y=.4em]
 \DN{0} \DN{1} \DDOT{2} \DN{3} \DN{4} \UP{5} \UP{6}\DDOT{7} \UP{8} \DN{9} \UP{10}
  \MRAY{0}  \SCIRCLES{1}{3.5}  \SCCIRCLE{3}  \CIRCLE{4}  \CIRCLE{9}
\end{scope}
\begin{scope}[shift={(-1.2,-3)},x=.27em,y=.4em]
 \DN{0} \DN{1} \DDOT{2} \DN{3} \DN{4} \UP{5} \UP{6}\DDOT{7} \UP{8} \DN{9} \UP{10}
  \MRAY{0}  \SCIRCLES{1}{4.5}  \SCCIRCLE{3}  \CIRCLE{4}  \CIRCLE{8}
\end{scope}
\begin{scope}[shift={(-1.3,-4)},x=.27em,y=.4em]
 \DN{0} \DN{1} \DDOT{2} \DN{3} \DN{4} \DN{5} \UP{6} \UP{7} \DN{8} \UP{9} \UP{10} \DDOT{11} \UP{12} 
  \MRAY{0}  \SCIRCLES{1}{5.5}  \SCIRCLES{3}{3.5}  \SCCIRCLE{4}\CIRCLE{5} \CIRCLE{8}
\end{scope}
\begin{scope}[shift={(-1.2,-5)},x=.27em,y=.4em]
 \DN{0} \DN{1} \DDOT{2} \DN{3} \DN{4} \UP{5} \DN{6} \UP{7} \UP{8} \DDOT{9} \UP{10} 
  \MRAY{0}  \SCIRCLES{1}{4.5}  \SCIRCLES{3}{2.5}  \CIRCLE{4}\CIRCLE{6} 
\end{scope}
\begin{scope}[shift={(-1,-6)},x=.27em,y=.4em]
 \DN{0} \DN{1} \DDOT{2} \DN{3} \DN{4} \UP{5} \UP{6}\DDOT{7} \UP{8} 
  \MRAY{0}  \SCIRCLES{1}{3.5}  \SCCIRCLE{3}  \CIRCLE{4} 
\end{scope}
\begin{scope}[shift={(1.7,-.95)},x=.27em,y=.4em]
 \DN{0} \DN{1} \DDOT{2} \DN{3}  \DN{4} \UP{5} \DN{6} \UP{7} \UP{8}  \DDOT{9}  \UP{10} \DN{11} \UP{12}
  \MRAY{0}  \SCIRCLES{1}{4.5} \SCIRCLES{3}{2.5}   \CIRCLE{4} \CIRCLE{6}\CIRCLE{11} 
\end{scope}
\begin{scope}[shift={(1.7,-2.1)},x=.27em,y=.4em]
 \DN{0} \DN{1} \DDOT{2} \DN{3}  \DN{4} \UP{5} \DN{6} \UP{7} \UP{8}  \DDOT{9}  \DN{10} \UP{11} \UP{12}
  \MRAY{0}  \SCIRCLES{1}{5.5} \SCIRCLES{3}{2.5}   \CIRCLE{4} \CIRCLE{6}\CIRCLE{10} 
\end{scope}
\begin{scope}[shift={(1.7,-3)},x=.27em,y=.4em]
 \DN{0} \DN{1} \DDOT{2}\DN{3} \DN{4}  \UP{5} \DN{6} \UP{7} \DN{8} \UP{9} \UP{10} \DDOT{11}  \UP{12}
  \MRAY{0}  \SCIRCLES{1}{5.5} \SCIRCLES{3}{3.5}   \CIRCLE{4} \CIRCLE{6}\CIRCLE{8} 
\end{scope}
\begin{scope}[shift={(1.75,-4.2)},x=.27em,y=.4em]
 \DN{0} \DN{1} \DDOT{2} \DN{3}  \DN{4} \UP{5} \DN{6} \DN{7} \UP{8} \UP{9} \UP{10} \DDOT{11}  \UP{12}
  \MRAY{0}  \SCIRCLES{1}{5.5} \SCIRCLES{3}{3.5} \CIRCLE{4} \SCCIRCLE{6}\CIRCLE{7} 
\end{scope}
\end{scope}

\end{scope}

\end{tikzpicture}
\caption{A part of $\QQ_m^n$ used in the proof of Theorem \ref{thm:maintheorem} where for simplicity the pairs of opposite arrows are drawn by an unoriented edge and only the label of the ascending arrow is given}
\label{figure:counterexample}
\end{figure}
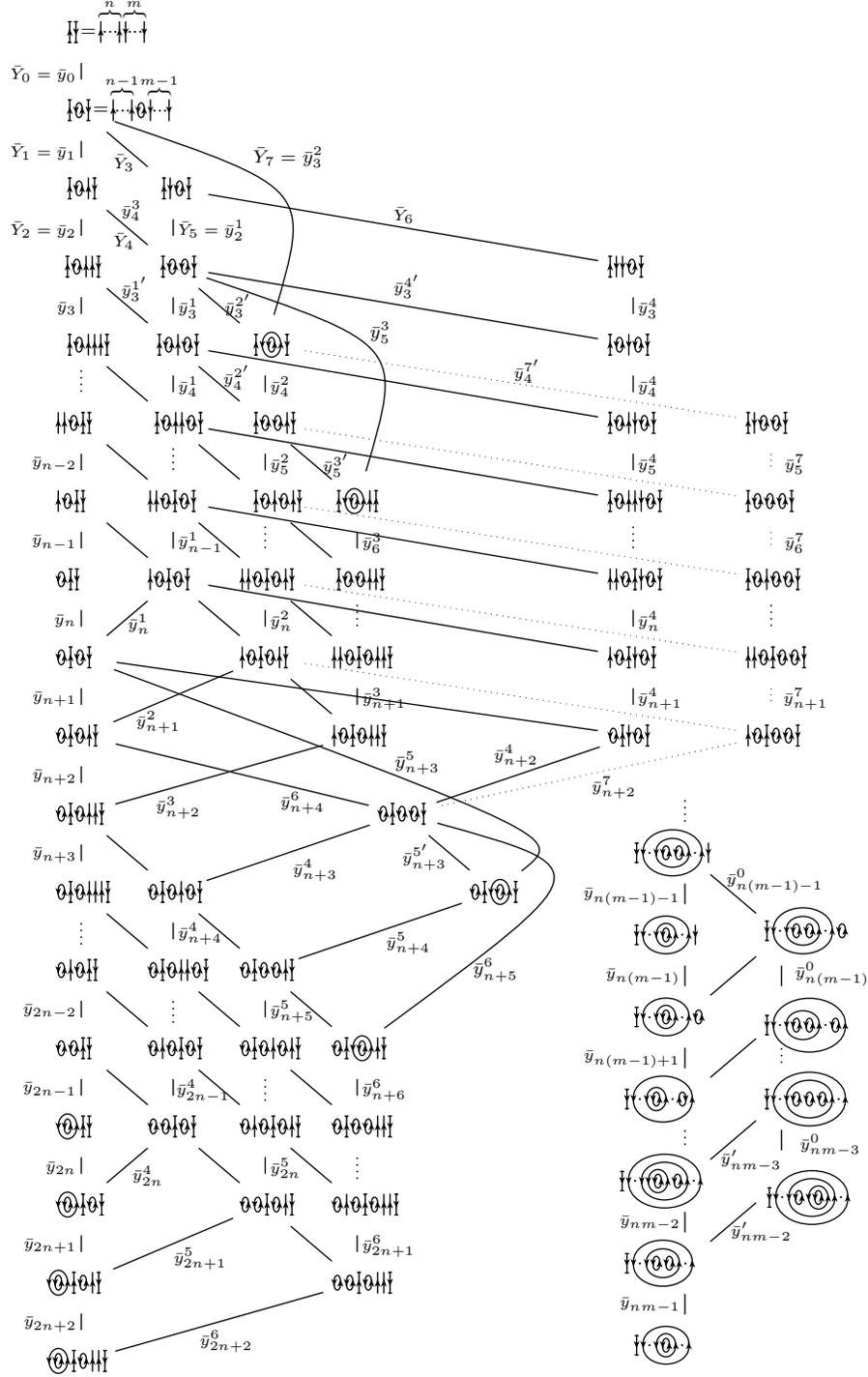

\begin{remark}
\label{remark:fukayaseidel}
For $\Bbbk = \mathbb C$ it follows from \cite{maksmith} that there is a quasi-equivalence between the DG category $\mathrm{perf}_{\mathrm{dg}} (\K_m^n)$ and the Fukaya--Seidel category $\mathcal{FS} (X_m^n, \pi_m^n)$ of a symplectic Lefschetz fibration
\[
\pi_m^n \colon X_m^n \tikzto \mathbb C
\]
where $X_m^n = \mathrm{Hilb}^m (A_{m+n-1}) \smallsetminus D$ is the affine complement of an ample divisor $D$ on the Hilbert scheme of $m$ points of the Milnor fibre of the surface singularity of type $A_{m+n-1}$. (See \cite{seidelsmith,manolescu,abouzaidsmith1,abouzaidsmith2} for a similar collection of results for the classical Khovanov arc algebras.) The A$_\infty$ deformations of $\K_m^n$ can thus be viewed as A$_\infty$ deformations of the Fukaya--Seidel category $\mathcal{FS} (X_m^n, \pi_m^n)$. That is, our main algebraic result (Theorem \ref{thm:maintheorem}) shows that $\mathcal{FS} (X_m^n, \pi_m^n)$ is not intrinsically formal. In light of Seidel's ICM 2002 address \cite{seidel1}, the algebraic deformations may correspond geometrically to partial compactifications of the affine variety $X_m^n = \mathrm{Hilb}^m (A_{m+n-1}) \smallsetminus D \subset \mathrm{Hilb}^m (A_{m+n-1})$. We leave a geometric interpretation of these deformations to future work.
\end{remark}

\section*{Acknowledgements}

We express our gratitude to Catharina Stroppel for introducing us to the extended Khovanov arc algebras and for her suggestion to apply the combinatorial method developed in \cite{barmeierwang1} to study their deformations. We would also like to thank Bernhard Keller, Steffen Koenig, Julian K\"ulshammer, Yanbo Li, Mikhail Khovanov, Ivan Smith, Jim Stasheff and Ben Webster for many helpful comments and discussions and the anonymous referee for their suggestions for improvements.

Part of this work was carried out during the Junior Trimester Program ``New Trends in Representation Theory'' at the Hausdorff Research Institute for Mathematics in Bonn, funded by the Deutsche Forschungsgemeinschaft (DFG, German Research Foundation) under Germany's Excellence Strategy -- EXC-2047/1 -- 390685813 whose support we gratefully acknowledge. The first-named author also warmly thanks the Institute of Algebra and Number Theory at the University of Stuttgart for the hospitality while working on this article. The second-named author was also supported by the Alexander von Humboldt Stiftung.

\appendix 
\section{Proof of Proposition \ref{proposition:coboundary}}
\begingroup
\allowdisplaybreaks

In this appendix we give a proof of Proposition \ref{proposition:coboundary} which is the key ingredient in the proof of our main result Theorem \ref{thm:maintheorem} for the case $m \geq n \geq 2$ and $(m, n) \neq (2,2)$.

An element $\psi \in \Hom_{\Bbbk Q_0^\e}(\Bbbk (\QQ_m^n)_1, (\KK_m^n)_{2mn-5})$ has the following general form 
\begin{align*}
\begin{split}
\psi_{\bar X_0} &= \mu_1 \bar x_{0 \dotsc nm-3}\bar y_{nm-3 \dotsc 1} \\
&\quad\ + \mu_2 \bar x_{0 \ldots  n(m-1)-2} \bar x_{n(m-1)-1 \ldots nm-3}^0  \bar y_{nm-3 \ldots n(m-1)-1}^0  \bar y_{n(m-1)-2 \ldots 1} \\
&\quad\ + \mu_3 \bar x_{0 \ldots nm-2} \bar y_{nm-2 \ldots n+2}  \bar y_{n+1 \ldots 3}^2 \\
&\quad\ + \mu_4 \bar x_{0 \ldots nm-1} \bar y_{nm-1 \ldots 2n+3} \bar y_{2n+2 \ldots n+4}^6  \bar y_{n+1 \ldots 3}^2
\end{split}
\\
\begin{split}
\psi_{\bar X_1} &= \mu_5 \bar x_{1 \ldots nm-2} \bar y_{nm-2 \ldots 2} + \mu_6 \bar x_{1 \ldots nm-1} \bar y_{nm-1 \ldots n+3} \bar y_{n+2 \ldots 4}^3 \\
&\quad\ + \mu_7 \bar x_{3 \ldots n+1}^2 \bar x_{n+2 \ldots nm-1} \bar y_{nm-1 \ldots 2}
\end{split}
\\
\psi_{\bar X_2} &= \mu_8 \bar x_{2 \ldots nm-1} \bar y_{nm-1 \ldots 3} \\
\begin{split}
\psi_{\bar X_3} &= \mu_9 \bar x_{1 \ldots nm-2} \bar y_{nm-2 \ldots n+1} \bar y_{n \ldots 2}^1 \\
&\quad\ + \mu_{10} \bar x_{1 \ldots nm-1} \bar y_{nm-1 \ldots 2n+2} \bar y_{2n+1 \ldots n+3}^5 \bar y_{n \ldots 2}^1 \\
&\quad\ + \mu_{11} \bar x_{3 \ldots n+1}^2  \bar x_{n+2 \ldots nm-1} \bar y_{nm-1 \ldots n+1} \bar y_{n \ldots 2}^1
\end{split}
\\
\psi_{\bar X_4} &= \mu_{12} \bar x_{2 \ldots nm-1} \bar y_{nm-1 \ldots n+1} \bar y_{n \ldots 3}^1 \\
\psi_{\bar X_5} &= \mu_{13} \bar x_{2 \ldots n}^1  \bar x_{n+1 \ldots nm-1} \bar y_{nm-1 \ldots n+1} \bar y_{n \ldots 3}^1 \\
\psi_{\bar X_6} &= \mu_{14} \bar x_{2 \ldots n}^1  x_{n+1 \ldots nm-1} \bar y_{nm-1 \ldots 2n+1} \bar y_{2n \ldots 3}^4 \\
\psi_{\bar X_7} &= \mu_{15} \bar x_{1 \ldots nm-1} \bar y_{nm-1 \ldots n+2} \bar y_{n+1 \ldots 4}^2
\end{align*}
where
$\psi_{\bar Y_i}$ is involutive to $\psi_{\bar X_i}$ with coefficients $\nu_1, \dotsc, \nu_{15}$. Note that $\psi_{\bar X_i} = 0 = \psi_{\bar Y_i} = 0$ for all other arrows. It follows that the dimension of the vector space $\Hom_{\Bbbk Q_0^\e} (\Bbbk (\QQ_m^n)_1, (\KK_m^n)_{2mn-5})$ is $30$. If $n = 2$, the terms with coefficient $\mu_4$ and $\nu_4$ do not exist, i.e.\ $\Hom_{\Bbbk Q_0^\e}(\Bbbk (\QQ_m^2)_1, (\KK_m^2)_{4m-5})$ is of dimension $28$. 

Recall that any element $\widetilde \varphi \in \Hom_{\Bbbk Q_0^\e} (\Bbbk\SS_m^n, \KK_m^n)_{2mn-4}$ has the form as in \eqref{align:generalformawidetildephimn}. Also recall the differential $\delta$ from \eqref{align:gaugeequivalent}.

\begin{lemma}\label{lemma:appendix}
$\delta(\psi) = \widetilde \varphi$ if and only if the following eleven equations hold
\begin{align*}
\alpha_1 &= -(-1)^{n+m} (\mu_1+\nu_1) - (-1)^{n} (\mu_2 + \nu_2) + (\mu_5+\nu_5) - (-1)^{n} (\mu_9+ \nu_9) \\
\alpha_2 &= (-1)^{n-m} \mu_3 - (-1)^{n} \mu_6 + (-1)^{n} \mu_{10} - \nu_1 + \nu_7 - (-1)^{n} \nu_{11} \\ 
\alpha_3 &= (-1)^{n-m} \nu_3 - (-1)^{n} \nu_6 + (-1)^{n} \nu_{10} -\mu_1 + \mu_7 - (-1)^{n} \mu_{11} \\
\alpha_4 &= -(\mu_7+\nu_7) + (\mu_8 + \nu_8) \\
\alpha_5 &= (-1)^{m} (\mu_9 + \nu_9) - (-1)^{n} (\mu_{11} + \nu_{11}) + (\mu_{14} + \nu_{14}) \\
\alpha_6 &= -\mu_1 - (-1)^{n-m} \mu_3 + \mu_{15} \\
\alpha_7 &= -\nu_1 - (-1)^{n-m} \nu_3 + \nu_{15} \\
\alpha_8 &= -\mu_{11} + \mu_{12} + (-1)^m \nu_5 - (-1)^{n} \nu_7 +(-1)^{n} \nu_{13} \\
\alpha_9 &= -\nu_{11} + \nu_{12} + (-1)^m \mu_5 - (-1)^{n} \mu_7+ (-1)^{n} \mu_{13} \\
\alpha_{10} &= -(-1)^{m} \mu_5 - \mu_6 + \mu_{10} + \mu_{12} - (-1)^{n} \mu_{13} + (-1)^n \mu_{15} \\
\alpha_{11} &= -(-1)^{m} \nu_5 - \nu_6 + \nu_{10} + \nu_{12} - (-1)^{n} \nu_{13} + (-1)^n \nu_{15}.
\end{align*}
\end{lemma}
 
This lemma immediately yields Proposition \ref{proposition:coboundary}. 

\begin{proof}[Proof of Proposition \ref{proposition:coboundary}]
From Lemma \ref{lemma:appendix} it follows that if $\widetilde \varphi$ is a coboundary then equation \eqref{align:constraint} holds. It not difficult to see that \eqref{align:constraint} is the only constraint since the rank of the coefficient matrix is $10$.
\end{proof}

The remaining part of this appendix is thus devoted to the proof of Lemma \ref{lemma:appendix}. Note that the equations involving $\alpha_3, \alpha_7, \alpha_9, \alpha_{11}$ may be obtained from those involving $\alpha_2, \alpha_6, \alpha_{8}, \alpha_{10}$ by exchanging $\mu_i$ and $\nu_i$, and the others involving $\alpha_1, \alpha_4, \alpha_5$ are invariant under this operation. (This is the case because the reduction system $\RR_m^n$ is invariant under the involution in Remark \ref{remark:involutionmapkosuzldaul}.)

The proof of Lemma \ref{lemma:appendix} is divided into several computations which are performed by ``chasing the diagram'' in Fig.~\ref{figure:counterexample}, i.e.\ by repeatedly using the anticommutativity and Plücker-type relations for $\KK_m^n$ given in Proposition \ref{proposition:relationdual} for the paths appearing in the components of $\delta (\psi)$. In this process, some of the terms appearing from applying the Plücker-type relations vanish due to the following observation.

\begin{observation}\label{keyobservation}
Let $p=\bar y_{\lambda_1}^{\lambda_2} \bar y_{\lambda_2}^{\lambda_3} \dotsb \bar y_{\lambda_k}^{\lambda_{k+1}}$ be any ascending path in $\QQ_m^n$ such that $\lambda_1$ is the lowest weight (i.e.\ $\lvert \lambda_1 \rvert = 0$) and $\lvert \lambda_{k+1} \rvert = k$. Assume that $\bar y_{\lambda_{i-1}}^{\lambda_i} \bar y_{\lambda_i}^{\lambda_{i+1}}$ lies in Fig.~\ref{figure:squares} (c) for some $1< i \leq k$ such that $\lambda_{i}$ is the right vertex in the figure. Denote by $\lambda_i'$ the left vertex. Then
$$
p = -\bar y_{\lambda_1}^{\lambda_2} \bar y_{\lambda_2}^{\lambda_3} \dotsb \bar y_{\lambda_{i-1}}^{\lambda_i'} \bar y_{\lambda_i'}^{\lambda_{i+1}} \dotsb \bar y_{\lambda_k}^{\lambda_{k+1}}.
$$
A similar statement holds for descending paths.
\end{observation}

\begin{proof}
We have $\bar y_{\lambda_{i-1}}^{\lambda_i} \bar y_{\lambda_i}^{\lambda_{i+1}} = - \bar y_{\lambda_{i-1}}^{\lambda_i'} \bar y_{\lambda_i'}^{\lambda_{i+1}} - \bar x_{\kappa}^{\lambda_{i-1}} \bar y_\kappa^{\lambda_{i+1}}$ by the Plücker-type relation given in Proposition \ref{proposition:relationdual}, where $\kappa$ is the bottom vertex in Fig.~\ref{figure:squares} (c) whence we have
$$
p = -\bar y_{\lambda_1}^{\lambda_2} \bar y_{\lambda_2}^{\lambda_3} \dotsb \bar y_{\lambda_{i-1}}^{\lambda_i'} \bar y_{\lambda_i'}^{\lambda_{i+1}} \dotsb \bar y_{\lambda_k}^{\lambda_{k+1}} - \bar y_{\lambda_1}^{\lambda_2} \bar y_{\lambda_2}^{\lambda_3} \dotsb \bar x_{\kappa}^{\lambda_{i-1}} \bar y_\kappa^{\lambda_{i+1}} \dotsb \bar y_{\lambda_k}^{\lambda_{k+1}}.
$$
The desired equality then follows since $\bar y_{\lambda_1}^{\lambda_2} \bar y_{\lambda_2}^{\lambda_3} \dotsb \bar y_{\lambda_{i-2}}^{\lambda_{i-1}} \bar x_{\kappa}^{\lambda_{i-1}} =0$ by Remark \ref{remark:irreduciblegrading}. 
\end{proof}

\addtocontents{toc}{\protect\setcounter{tocdepth}{1}}
\subsection{The term \texorpdfstring{$\delta(\psi)_{\bar Y_0\bar X_0}$}{}}

The following lemma proves $\frac{3}{11}$ of Lemma \ref{lemma:appendix}. The computations for the remaining $\frac{8}{11}$ are very similar and are completed in the following subsections.

\begin{lemma}
$\delta(\psi)_{\bar Y_0\bar X_0} = \widetilde \varphi_{\bar Y_0\bar X_0}$ if and only if the first three equations in Lemma \ref{lemma:appendix} involving $\alpha_1, \alpha_2, \alpha_3$ hold.
\end{lemma}

\begin{proof}
Since $\bar Y_0 \bar X_0 + \bar X_1 \bar Y_1 + \bar X_3 \bar Y_3 = 0$, we have
\[
\delta(\psi)_{\bar Y_0\bar X_0}= \psi_{\bar Y_0} \bar X_0 + \bar Y_0 \psi_{\bar X_0} + \psi_{\bar X_1} \bar Y_1 + \bar X_1 \psi_{\bar Y_1} + \psi_{\bar X_3} \bar Y_3 + \bar X_3 \psi_{\bar Y_3}
\]
where $\psi_{\bar Y_0} \bar X_0$ is the product of $\psi_{\bar Y_0}$ and $\bar X_0$ in $\KK_m^n$ and similarly for the other terms. Note that $\psi_{\bar Y_0} \bar X_0, \bar X_1 \psi_{\bar Y_1}, \bar X_3 \psi_{\bar Y_3}$ may be obtained from $\bar Y_0 \psi_{\bar X_0}, \psi_{\bar X_1} \bar Y_1, \psi_{\bar X_3} \bar Y_3$, respectively, by taking the involution in Remark \ref{remark:involutionmapkosuzldaul} (and exchanging $\mu_i$ and $\nu_i$). Therefore, it suffices to compute the latter ones. We first claim 
\begin{equation}
\label{claim1}
\begin{split}
\bar Y_0 \psi_{\bar X_0} &= (-(-1)^{n+m}\mu_1-(-1)^{n} \mu_2) \bar x_{1 \ldots nm-2} \bar y_{nm-2 \ldots 1} \\
&\quad\ - \mu_1 \bar x_{3 \ldots n+2}^2 \bar x_{n+2 \ldots nm-1} \bar y_{nm-1 \ldots 1} \\
&\quad\ + (-1)^{n+m} \mu_3 \bar x_{1 \ldots nm-1} \bar y_{nm-1 \ldots n+2} \bar y_{n+1 \ldots 3}^2.
\end{split}
\end{equation}
To see this, note that $\bar Y_0 = \bar y_0$ (see Fig.~\ref{figure:counterexample}), whence by definition of $\psi_{\bar X_0}$ we have 
\begin{align*}
\begin{split}
\bar Y_0 \psi_{\bar X_0} &= \mu_1 \bar y_0\bar x_{0 \ldots nm-3}\bar y_{nm-3 \ldots 1} \\
&\quad\ + \mu_2 \bar y_0 \bar x_{0 \ldots n(m-1)-2} \bar x_{n(m-1)-1 \ldots nm-3}^0 \bar y_{nm-3 \ldots n(m-1)-1}^0 \bar y_{n(m-1)-2 \ldots 1} \\
&\quad\ + \mu_3 \bar y_0 \bar x_{0 \ldots nm-2} \bar y_{nm-2 \ldots n+2} \bar y_{n+1 \ldots 3}^2 \\
&\quad\ + \mu_4 \bar y_0 \bar x_{0 \ldots nm-1} \bar y_{nm-1 \ldots 2n+3} \bar y_{2n+2 \ldots n+4}^6 \bar y_{n+1 \ldots 3}^2.
\end{split}
\end{align*}
Applying Lemma \ref{lemma:higherrelations4} to the second summand and Lemma \ref{lemma:higherrelations3} (with $i = 0$) to the other three summands we obtain \eqref{claim1}. 

We also claim that
\begin{equation}
\label{claim2}
\begin{split}
\psi_{\bar X_1} \bar Y_1 &= \mu_5   \bar x_{1 \ldots nm-2} \bar y_{nm-2 \ldots 1} \\
&\quad\ + (-1)^{n-1} \mu_6 \bar x_{1 \ldots nm-1} \bar y_{nm-1 \ldots n+2} \bar y_{n+1 \ldots 3}^2 \\
&\quad\ + \mu_7 \bar x_{3 \ldots n+1}^2 \bar x_{n+2 \ldots nm-1} \bar y_{nm-1 \ldots 1}.
\end{split}
\end{equation}
Indeed, note that $\bar Y_1 = \bar y_1$ so that by definition of $\psi_{\bar X_1}$ we have
\begin{align*}
\psi_{\bar X_1} \bar Y_1 &= \mu_5 \bar x_{1 \ldots nm-2} \bar y_{nm-2 \ldots 1} + \mu_6 \bar x_{1 \ldots nm-1} \bar y_{nm-1 \ldots n+3} \bar y_{n+2 \ldots 4}^3\bar y_1 \\
&\quad\ + \mu_7 \bar x_{3 \ldots n+1}^2 \bar x_{n+2 \ldots nm-1} \bar y_{nm-1 \ldots 1}.
\end{align*}
It suffices to perform reductions on the second summand. For this, we have
\begin{align*}
\mu_6 \bar x_{1 \ldots nm-1} \bar y_{nm-1 \ldots n+3} \bar y_{n+2 \ldots 4}^3\bar y_1
&= \mu_6 \bar x_{1 \ldots nm-1} \bar y_{nm-1 \ldots n+3} \bar y_{n+2 \ldots 6}^3  \bar y_5^{3'} \bar y_4^2 \bar y_3^2\\
&= (-1)^{n-1} \mu_6 \bar x_{1 \ldots nm-1} \bar y_{nm-1 \ldots n+2} \bar y_{n+1 \ldots 3}^2
\end{align*}
where the first equality uses the relation $\bar y_5^3 \bar y_4^3 \bar y_1 = \bar y_5^{3'} \bar y_4^2 \bar y_3^2$ in Lemma \ref{lemma:cubicrelations}, the second one uses the Plücker-type relation $\bar y_6^3 \bar y_5^{3'}=-\bar y_6^{3'} \bar y_5^2-\bar x'\bar y''$ and the anticommutativity relations  $\bar y_{i+1}^3\bar y_{i}^{3'}=-\bar y_{i+1}^{3'} \bar y_i^2$ for $6 \leq i \leq n+1$ (denote $\bar y_{n+2}^{3'}=\bar y_{n+2}$). Here, Observation \ref{keyobservation} shows that the summand given by $\bar x'\bar y''$ vanishes. This gives \eqref{claim2}. 

Now we claim
\begin{equation}
\label{claim3}
\begin{split}
\psi_{\bar X_3} \bar Y_3 &= (-1)^{n-1} \mu_9  \bar x_{1 \ldots nm-2} \bar y_{nm-2 \ldots 1} \\
&\quad\ + (-1)^n \mu_{10} \bar x_{1 \ldots nm-1} \bar y_{nm-1 \ldots n+2} \bar y_{n+1 \ldots 3}^2 \\
&\quad\ -(-1)^{n} \mu_{11} \bar x_{3 \ldots n+1}^2 \bar x_{n+2 \ldots nm-1} \bar y_{nm-1 \ldots 1}.
\end{split}
\end{equation}
Indeed, we first note that
\begin{align}\label{align:equalityY3}
\bar y^1_{n \ldots 2} \bar Y_3 = (-1)^{n-1} \bar y_{n \ldots 1} + (-1)^{n-1} \bar x_{n+1} \bar y_{n+1 \ldots 3}^2
\end{align}
where we use the Plücker-type relation $\bar y_2^1 \bar Y_3 = -\bar y_4^3 \bar y_1 - \bar x_3^{2'} \bar y_3^2$ and  the anticommutativity relations $\bar y_i^1 \bar y_{i-1}^{1'} = - \bar y_{i}^{1'} \bar y_{i-1}$ and $\bar y_i^1 \bar x_{i}^{2'} = - \bar x_{i+1}^{2'} \bar y_{i+1}^2$ for $3 \leq i \leq n$ (where $\bar y_2^{1'} = \bar y_4^3$, $\bar y_n^{1'}=\bar y_n$ and $\bar x_{n+1}^{2'} = \bar x_{n+1}$). Then we have
\begin{align*}
\begin{split}
\psi_{\bar X_3} \bar Y_3 &= \mu_9 \bar x_{1 \ldots nm-2} \bar y_{nm-2 \ldots n+1} \bar y_{n \ldots 2}^1  \bar Y_3 \\
&\quad\ + \mu_{10} \bar x_{1 \ldots nm-1} \bar y_{nm-1 \ldots 2n+2} \bar y_{2n+1 \ldots n+3}^5 \bar y_{n \ldots 2}^1 \bar Y_3 \\
&\quad\ + \mu_{11} \bar x_{3 \ldots n+1}^2 \bar x_{n+2 \ldots nm-1} \bar y_{nm-1 \ldots n+1} \bar y_{n \ldots 2}^1 \bar Y_3
\end{split}
\\
\begin{split}
&= (-1)^{n-1} \mu_9 \bar x_{1 \ldots nm-2} \bar y_{nm-2 \ldots 1} \\
&\quad\ + (-1)^n \mu_{10} \bar x_{1 \ldots nm-1} \bar y_{nm-1 \ldots n+2}\bar y_{n+1 \ldots 3}^2 \\
&\quad\ + (-1)^{n-1} \mu_{11} \bar x_{3 \ldots n+1}^2 \bar x_{n+2 \ldots nm-1} \bar y_{nm-1 \ldots 1}
\end{split}
\end{align*} 
where in the second equality we first use \eqref{align:equalityY3} and then for the summands with coefficients $\mu_9$ and $\mu_{11}$ we use Lemma \ref{lemma:higherrelations1} \ref{lemma:higherrelations30}, and for the two summands with coefficient $\mu_{10}$ we use $\bar y_{n+3}^5 \bar y_{n} = 0$ and respectively 
\begin{align*}
\bar y_{nm-1 \ldots 2n+2} \bar y_{2n+1 \ldots n+3}^5 \bar x_{n+1} &= -\bar y_{nm-1 \ldots 2n+2}\bar y_{2n+1 \ldots n+4}^5 \bar y_{n+3}^{5'} \bar y_{n+4}^6 \\
&= - \bar y_{nm-1 \ldots n+2}.
\end{align*}
Here in the first equality, we use $\bar y_{n+3}^5 \bar x_{n+1} = \bar y_{n+3}^{5'} \bar y_{n+4}^6$ which follows from the anticommutativity relation across a square (cf.\ the eighth diagram in Remark \ref{remark:commutativity} viewed as a diagram in $\QQ_m^n$) which may be verified by using the Plücker-type relation $\bar y_{n+4}^5 \bar y_{n+3}^{5'} = -\bar y_{n+4}^{5'} \bar y_{n+3}^4 -\bar x_{n+5}^{6'} \bar y_{n+5}^6$ and then the anticommutativity relations $\bar y_{n+1+j}^5 \bar y_{n+j}^{5'}= -\bar y_{n+1+j}^{5'} \bar y_{n+j}^4$ for $4 \leq j \leq n+1$ (where $\bar y_{2n+1}^{5'} = \bar y_{2n+1}$). Here, the extra term $\bar x_{n+5}^{6'} \bar y_{n+5}^6$ in the Plücker-type relation vanishes by Observation \ref{keyobservation}. This proves \eqref{claim3}. 

Combining \eqref{claim1}, \eqref{claim2} and \eqref{claim3} and comparing the coefficients of each summand in $\delta(\psi)_{\bar Y_0\bar X_0} = \widetilde \varphi_{\bar Y_0\bar X_0}$ completes the proof. 
\end{proof}

\subsection{The term \texorpdfstring{$\delta(\psi)_{\bar Y_1 \bar X_1}$}{}}

\begin{lemma}
$\delta(\psi)_{\bar Y_1 \bar X_1} = \widetilde \varphi_{\bar Y_1 \bar X_1}$ if and only if the equation in Lemma \ref{lemma:appendix} involving $\alpha_4$ holds.
\end{lemma} 

\begin{proof}
We have $\delta(\psi)_{\bar Y_1 \bar X_1} = \psi_{\bar Y_1} X_1 + \bar Y_1 \psi_{\bar X_1} + \psi_{\bar X_2} \bar Y_2 + \bar X_2 \psi_{\bar Y_2},$ where $\psi_{\bar Y_1} X_1, \bar X_2 \psi_{\bar Y_2}$ are involutive to $\bar Y_1 \psi_{\bar X_1}, \psi_{\bar X_2} \bar Y_2$ by exchanging $\mu_i$ to $\nu_i$. Since $\bar Y_1 = \bar y_1$ we have
\begin{align*}
\begin{split}
\bar Y_1 \psi_{\bar X_1} &= \mu_5 \bar y_1  \bar x_{1 \ldots nm-2} \bar y_{nm-2 \ldots 2}+\mu_6 \bar y_1 \bar x_{1 \ldots nm-1} \bar y_{nm-1 \ldots n+3} \bar y_{n+2 \ldots 4}^3 \\
&\quad\ + \mu_7 \bar y_1 \bar x_{3 \ldots n+1}^2 \bar x_{n+2 \ldots nm-1} \bar y_{nm-1 \ldots 2}
\end{split}
\\
&= -\mu_7 \bar x_{2 \ldots nm-1} \bar y_{nm-1 \ldots 2}
\end{align*}
where by Lemma \ref{lemma:higherrelations1} \ref{lemma:higherrelations30} the first two summands vanish and the third one uses
\begin{align}\label{align:y1x3}
\bar y_1 \bar x_{3 \ldots n+1}^2  \bar x_{n+2 \ldots nm-1} = - \bar x_4^3 \bar x_3^{2'} \bar x_{4 \ldots n+1}^2 \bar x_{n+2 \ldots nm-1}  = - \bar x_{2 \ldots nm-1}.
\end{align}
Here, the first equality uses an anticommutativity relation across a square and the second equality uses the Plücker-type relation $\bar x_3^{2'} \bar x_4^2 = - \bar x_3^1 \bar x_4^{2'} - \bar x_5^3 \bar y_5^{3'}$ and then the anticommutativity relations ($2n-5$ times). Again, the summand given by $\bar x_5^{3'} \bar y_5^3$ vanishes by Observation \ref{keyobservation}.

Note that $\bar Y_2 = \bar y_2$. Then $\psi_{\bar X_2} \bar Y_2 ={}  \mu_8 \bar x_{2 \ldots nm-1} \bar y_{nm-1 \ldots 2}$ is already irreducible. Comparing the coefficients we complete the proof. 
\end{proof}
 
\subsection{The term \texorpdfstring{$\delta(\psi)_{\bar Y_3 \bar X_3}$}{} }
 
\begin{lemma}
$\delta(\psi)_{\bar Y_3 \bar X_3} = \widetilde \varphi_{\bar Y_3 \bar X_3}$ if and only if the equation in Lemma \ref{lemma:appendix} involving $\alpha_5$ holds.
\end{lemma}

\begin{proof}
We have $\delta(\psi)_{\bar Y_3 \bar X_3} = \psi_{\bar Y_3} \bar X_3 + \bar Y_3 \psi_{\bar X_3} + \psi_{\bar X_6} \bar Y_6 + \bar X_6 \psi_{\bar Y_6}$, where $\psi_{\bar Y_3} \bar X_3 , \bar X_6 \psi_{\bar Y_6}$ are involutive to $ \bar Y_3 \psi_{\bar X_3}, \psi_{\bar X_6} \bar Y_6$ by exchanging $\mu_i$ to $\nu_i$. We claim that 
\begin{align*}
\bar Y_3 \psi_{\bar X_3} = ((-1)^m \mu_9 - (-1)^{n}  \mu_{11}) \bar x_{2 \ldots n}^1  \bar x_{n+1 \ldots nm-1} \bar y_{nm-1 \ldots n+1}  \bar y_{n \ldots 2}^1.
\end{align*}
Indeed, we have 
\begin{align*}
\begin{split}
\bar Y_3 \psi_{\bar X_3} &= \mu_9 \bar Y_3 \bar x_{1 \ldots nm-2} \bar y_{nm-2 \ldots n+1} \bar y_{n \ldots 2}^1 \\
&\quad\ + \mu_{10} \bar Y_3 \bar x_{1 \ldots nm-1} \bar y_{nm-1 \ldots 2n+2} \bar y_{2n+1 \ldots n+3}^5 \bar y_{n \ldots 2}^1 \\
&\quad\ + \mu_{11} \bar Y_3 \bar x_{3 \ldots n+1}^2 \bar x_{n+2 \ldots nm-1} \bar y_{nm-1 \ldots n+1} \bar y_{n \ldots 2}^1
\end{split}
\\
\begin{split}
&= (-1)^{m} \mu_9 \bar x_{2 \ldots n}^1 \bar x_{n+1 \ldots nm-1} \bar y_{nm-1 \ldots n+1} \bar y_{n \ldots 2}^1 \\
&\quad\ + (-1)^{n-1} \mu_{11} \bar x_{2 \ldots n}^1 \bar x_{n+1 \ldots nm-1} \bar y_{nm-1 \ldots n+1} \bar y_{n \ldots 2}^1
\end{split}
\end{align*}
where in the second equality we use $\bar Y_3 \bar x_{1 \ldots nm-1} = 0$ by Remark \ref{remark:irreduciblegrading} to obtain that the summand with $\mu_{10}$ vanishes, and we also use the following relations
\begin{align*}
\begin{split}
\bar Y_3 \bar x_{1 \ldots nm-2} &= (-1)^{n-1} \bar x_{2 \ldots n}^1  \bar y_n \bar x_{n \ldots nm-2} \\
&= (-1)^{m}  \bar x_{2 \ldots n}^1 \bar x_{n+1 \ldots nm-1}
\end{split}
\\
\begin{split}
\bar Y_3 \bar x_{3 \ldots n+1}^2 \bar x_{n+2 \ldots nm-1} &= -\bar x_2^1 \bar x_3^{2'}\bar x_{4 \ldots n+1}^2  \bar x_{n+2 \ldots nm-1} \\
&= (-1)^{n-1} \bar x_{2 \ldots n}^1  \bar x_{n+1 \ldots nm-1}
\end{split}
\end{align*}
which may be verified by using the anticommutativity and Plücker-type relations (and applying Observation \ref{keyobservation}). We have 
\begin{align*}
\psi_{\bar X_6} \bar Y_6 &= \mu_{14} \bar x_{2 \ldots n}^1 x_{n+1 \ldots nm-1} \bar y_{nm-1 \ldots 2n+1} \bar y_{2n \ldots 3}^4  \bar Y_6 \\
&= \mu_{14}  \bar x_{2 \ldots n}^1 \bar x_{n+1 \ldots nm-1} \bar y_{nm-1 \ldots n+1}  \bar y_{n \ldots 2}^1  \end{align*}
where similarly the second equality uses the anticommutativity relations ($2n-3$ times) and the Plücker-type relation once (also using Observation \ref{keyobservation}). 
\end{proof}

\subsection{The term \texorpdfstring{$\delta(\psi)_{\bar X_0 X_7}$}{} } 
  
  \begin{lemma}\label{lemma:A6}
$\delta(\psi)_{\bar X_0 X_7} = \widetilde \varphi_{\bar X_0 X_7}$ if and only if the equation in Lemma \ref{lemma:appendix} involving $\alpha_6$ holds.
\end{lemma} 
\begin{proof}
Note that  $\delta(\psi)_{\bar X_0 X_7} = \psi_{\bar X_0} X_7 + \bar X_0 \psi_{\bar X_7}$. We claim that 
\begin{align*}
\psi_{\bar X_0} X_7 = (- \mu_1-(-1)^{n-m} \mu_3) \bar x_{0 \ldots nm-1} \bar y_{nm-1 \ldots n+2} \bar y_{n+1 \ldots 4}^2.
 \end{align*}
Indeed, since  $X_7 = \bar x_3^2$ we have
\begin{align*}
\begin{split}
\psi_{\bar X_0} X_7 &= \mu_1 \bar x_{0 \ldots nm-3}\bar y_{nm-3 \ldots 1} \bar x_3^2 \\
&\quad\ + \mu_2 \bar x_{0 \ldots n(m-1)-2} \bar x_{n(m-1)-1 \ldots nm-3}^0 \bar y_{nm-3 \ldots n(m-1)-1}^0 \bar y_{n(m-1)-2 \ldots 1} \bar x_3^2 \\
&\quad\ + \mu_3 \bar x_{0 \ldots nm-2} \bar y_{nm-2 \ldots n+2}  \bar y_{n+1 \ldots 3}^2  \bar x_3^2 \\
&\quad\ + \mu_4 \bar x_{0 \ldots nm-1} \bar y_{nm-1 \ldots 2n+3} \bar y_{2n+2 \ldots n+4}^6 \bar y_{n+1 \ldots 3}^2 \bar x_3^2
\end{split}
\\
&= (- \mu_1- (-1)^{n-m }\mu_3) \bar x_{0 \ldots nm-1} \bar y_{nm-1 \ldots n+2}\bar y_{n+1 \ldots 4}^2 
 \end{align*}
where the summands with coefficients $\mu_3$ and $\mu_4$ are computed in Lemma \ref{lemma:A9} below. Let us compute the first two summands with coefficients $\mu_1$ and $\mu_2$. For this, we apply the relation $\bar y_{n-1 \ldots 1}\bar x_3^2 = -\bar x_n \bar x_{n+1} \bar y_{n+1 \ldots 4}^2$ (which uses the anticommutativity relations $2n-3$ times) to the two summands and obtain 
\begin{multline}\label{align:A5}
{-\mu_1} \bar x_{0 \ldots nm-3}\bar y_{nm-3 \ldots n} \bar x_n \bar x_{n+1} \bar y_{n+1 \ldots 4}^2 \\
+ (-1)^{n} \mu_2 \bar x_{0 \ldots nm-2} \bar y_{nm-2 \ldots n+1}  \bar x_{n+1} \bar y_{n+1 \ldots 4}^2
\end{multline}
where the second summand also uses the involutive version of Lemma \ref{lemma:higherrelations4}. The second summand of \eqref{align:A5} vanishes by Lemma \ref{lemma:higherrelations1} \ref{lemma:higherrelations30} and the first one equals 
\begin{multline*}
{-\mu_1} \bar x_{0 \ldots nm-3}\bar y_{nm-3 \ldots n} \bar x_n \bar x_{n+1} \bar y_{n+1 \ldots 4}^2 \\
\begin{split}
&= -(-1)^{n+m}\mu_1\bar x_{0 \ldots nm-2} \bar y_{nm-2 \ldots n+1} \bar x_{n+1}\bar y_{n+1 \ldots 4}^2 \\
&\quad\ + \mu_1 \bar x_{0 \ldots nm-1} \bar y_{nm-1 \ldots 2n+2} \bar y_{2n+1 \ldots n+3}^5 \bar x_{n+1} \bar y_{n+1 \ldots 4}^2 \\
&= -\mu_1 \bar x_{0 \ldots nm-1} \bar y_{nm-1 \ldots n+2}\bar y_{n+1 \ldots 4}^2
\end{split}
\end{multline*}
where the first equality follows from Lemma \ref{lemma:higherrelations3},  in the second equality, the first summand  vanishes  by Lemma \ref{lemma:higherrelations1} \ref{lemma:higherrelations30} and for the second summand we use the anticommutativity relations ($2n-4$ times) and the Plücker-type relation (involving $\bar y_{n+4}^5 \bar y_{n+3}^{5'}$) once (applying Observation \ref{keyobservation} to see the vanishing of the extra term). This proves the claim. 

Since $\bar X_0 = \bar x_0$ we have that $\bar X_0 \psi_{\bar X_7} = \mu_{15}\bar x_{0 \ldots nm-1} \bar y_{nm-1 \ldots n+2} y_{n+1 \ldots 4}^2$ is already irreducible. We complete the proof by comparing the coefficients.
 \end{proof}
 
\subsection{The term \texorpdfstring{$\delta(\psi)_{\bar Y_1 \bar X_3}$}{} } 

\begin{lemma}
$\delta(\psi)_{\bar Y_1 \bar X_3} = \widetilde \varphi_{\bar Y_1 \bar X_3}$ if and only if the equation in Lemma \ref{lemma:appendix} involving $\alpha_8$ holds.
\end{lemma} 

\begin{proof}
We have  $\delta(\psi)_{\bar Y_1 \bar X_3} = \psi_{\bar Y_1} \bar X_3 + \bar Y_1 \psi_{\bar X_3} + \psi_{\bar X_4} \bar Y_5 + \bar X_4 \psi_{\bar Y_5}$. We claim that 
\begin{align*}
\psi_{\bar Y_1} \bar X_3 = ((-1)^{m}\nu_5 -(-1)^{n} \nu_7) \bar x_{2 \ldots nm-1} \bar y_{nm-1 \ldots n+1} \bar y_{n \ldots 2}^1.
\end{align*}
Indeed, we have
\begin{align*}
\begin{split}
\psi_{\bar Y_1} \bar X_3 &= \nu_5 \bar x_{2 \ldots nm-2} \bar y_{nm-2 \ldots 1} \bar X_3 + \nu_6 \bar x_{4 \ldots n+2}^3 \bar x_{n+3 \ldots nm-1} \bar y_{nm-1 \ldots 1} \bar X_3 \\
&\quad\ + \nu_7 \bar x_{2 \ldots nm-1} \bar y_{nm-1 \ldots n+2} \bar y_{n+1 \ldots 3}^2 \bar X_3
\end{split}
\\
\begin{split}
&= -(-1)^{n} \nu_5 \bar x_{2 \ldots nm-2} \bar y_{nm-2 \ldots n} \bar x_n \bar y_{n \ldots 2}^1 \\
&\quad\ -(-1)^{n} \nu_7 \bar x_{2 \ldots nm-1} \bar y_{nm-1 \ldots n+1} y_{n \ldots 2}^1
\end{split}
\\
&= ((-1)^{m}\nu_5 -(-1)^{n} \nu_7) \bar x_{2 \ldots nm-1} \bar y_{nm-1 \ldots n+1} \bar y_{n \ldots 2}^1
\end{align*}
where in the second equality, the summand with coefficient $\nu_6$ vanishes by Remark \ref{remark:irreduciblegrading}, the summand with $\nu_5$ uses the anticommutativity relations ($n-1$ times), and the summand with $\nu_7$ uses the anticommutativity relations ($n-4$ times) and the Plücker-type relation $\bar x_3^{2'} \bar x_4^2 = - \bar x_3^1 \bar x_4^{2'} - \bar x_5^3 \bar y_5^{3'}$ once. The term $\bar x_5^3 \bar y_5^{3'}$ vanishes by Observation \ref{keyobservation}. The third equality follows from Lemma \ref{lemma:higherrelations3}. 
 
Since $\bar Y_1 = \bar y_1$, we have
\begin{align*}
\begin{split}
\bar Y_1 \psi_{\bar X_3} &= \mu_9  \bar y_1 \bar x_{1 \ldots nm-2} \bar y_{nm-2 \ldots n+1} \bar y_{n \ldots 2}^1 \\
&\quad\ + \mu_{10} \bar y_1 \bar x_{1 \ldots nm-1} \bar y_{nm-1 \ldots 2n+2} \bar y_{2n+1 \ldots n+3}^5 \bar y_{n \ldots 2}^1 \\
&\quad\ + \mu_{11} \bar y_1 \bar x_{3 \ldots n+1}^2 \bar x_{n+2 \ldots nm-1} \bar y_{nm-1 \ldots n+1} \bar y_{n \ldots 2}^1
\end{split}
\\
&= -\mu_{11} \bar x_{2 \ldots nm-1} \bar y_{nm-1 \ldots n+1} \bar y_{n \ldots 2}^1
\end{align*}
where the first two summands vanish by Lemma \ref{lemma:higherrelations1} \ref{lemma:higherrelations30} and the third one follows from  \eqref{align:y1x3}. Since $\bar X_4 = \bar x_4^3$ we  have
\begin{align*}
\bar X_4 \psi_{\bar Y_5} &= \nu_{13}  \bar x_4^3  \bar x_{3 \ldots n}^1 \bar x_{n+1 \ldots nm-1} \bar y_{nm-1 \ldots n+1} \bar y_{n \ldots 2}^1 \\
&= (-1)^{n} \nu_{13} \bar x_{2 \ldots nm-1} \bar y_{nm-1 \ldots n+1} \bar y_{n \ldots 2}^1
\end{align*}
where the second equality uses the anticommutativity relations ($n-2$ times). 

Since $\bar Y_5 = \bar y_2^1$, the term $\psi_{\bar X_4} \bar Y_5 =  \mu_{12} \bar x_{2 \ldots nm-1} \bar y_{nm-1 \ldots n+1} \bar y_{n \ldots 2}^1$
is already irreducible.  We complete the proof by comparing the coefficient of each summand. 
\end{proof}

\subsection{The term \texorpdfstring{$\delta(\psi)_{\bar X_3 \bar X_5}$}{}} 
 
\begin{lemma}
$\delta(\psi)_{\bar X_3 \bar X_5} = \widetilde \varphi_{\bar X_3 \bar X_5}$ if and only if the equation in Lemma \ref{lemma:appendix} involving $\alpha_{10}$ holds.
\end{lemma} 

\begin{proof}
We have $\delta(\psi)_{\bar X_3 \bar X_5} = \psi_{\bar X_3} \bar X_5 + \bar X_3 \psi_{\bar X_5} + \psi_{\bar X_1} \bar X_4 + \bar X_1 \psi_{\bar X_4} + \psi_{\bar X_7} \bar y_{3}^{2'}$. Then we claim that 
\[
\psi_{\bar X_3} \bar X_5 = \mu_{10} \bar x_{1 \ldots nm-1} \bar y_{nm-1 \ldots n+1} \bar y_{n \ldots 3}^1.
\]
Indeed, since $\bar X_5 = \bar x_2^1$ we have
\begin{equation}\label{align:mu9}
\begin{aligned}
\psi_{\bar X_3} \bar X_5 &= \mu_9 \bar x_{1 \ldots nm-2} \bar y_{nm-2 \ldots n+1} \bar y_{n \ldots 2}^1\bar x_2^1 \\
&\quad\ + \mu_{10} \bar x_{1 \ldots nm-1} \bar y_{nm-1 \ldots 2n+2}  \bar y_{2n+1 \ldots n+3}^5 \bar y_{n \ldots 2}^1 \bar x_2^1 \\
&\quad\ + \mu_{11} \bar x_{3 \ldots n+1}^2 \bar x_{n+2 \ldots nm-1} \bar y_{nm-1 \ldots n+1} \bar y_{n \ldots 2}^1\bar x_2^1 \\
&= \mu_{10}  \bar x_{1 \ldots nm-1} \bar y_{nm-1 \ldots n+1} \bar y_{n \ldots 3}^1
\end{aligned}
\end{equation}
where the third summand vanishes since $\bar y_{nm-1 \ldots n+1} \bar y_{n \ldots 2}^1\bar x_2^1 = 0$ by Remark \ref{remark:irreduciblegrading}. 

We first show that the first summand involving $\mu_9$ in \eqref{align:mu9} vanishes. Note that $\bar y_2^1 \bar x_2^1 = - \bar x_3^1 \bar y_3^1 - \bar x_3^{2'} \bar y_3^{2'} - \bar x'' \bar y''$ where $\bar x''$ is  the other longer arrow (which is missing in Fig.~\ref{figure:counterexample}) with the same starting vertex as $\bar x_3^1$.  So we have 
\begin{multline*}
\mu_9 \bar x_{1 \ldots nm-2} \bar y_{nm-2 \ldots n+1} \bar y_{n \ldots 2}^1\bar x_2^1 \\
= -\mu_9 \bar x_{1 \ldots nm-2} (\bar y_{nm-2 \ldots n+1} \bar y_{n \ldots 3}^1\bar x_3^1 \bar y_3^1 +  \bar y_{nm-2 \ldots n+1} \bar y_{n \ldots 3}^1\bar x_3^{2'} \bar y_3^{2'}) = 0
\end{multline*}
where the first equality uses $\bar y_{nm-2 \ldots n+1} \bar y_{n \ldots 3}^1\bar x'' = 0$ by Remark \ref{remark:irreduciblegrading}. In the second equality, the first summand vanishes by using  $\bar y_i^1 \bar x_i^1 = -\bar x_{i+1}^1 \bar y_{i+1}^1$ for $3 \leq i \leq n-1$ and $\bar y_n^1 \bar x_n^1 = 0$ and the second summand vanishes since  we have 
$$ \bar y_{nm-2 \ldots n+1} \bar y_{n \ldots 3}^1\bar x_3^{2'} = (-1)^n \bar y_{nm-2 \ldots n+1} \bar x_{n+1} \bar y_{n+1 \ldots 4}^2  = 0.$$ Here, the first equality uses the anticommutativity relations and the second one follows from Lemma \ref{lemma:higherrelations1} \ref{lemma:higherrelations30}.

We now compute the summand with $\mu_{10}$ in \eqref{align:mu9}. We have
\begin{multline*}
\mu_{10} \bar x_{1 \ldots nm-1} \bar y_{nm-1 \ldots 2n+2}  \bar y_{2n+1 \ldots n+3}^5 \bar y_{n \ldots 2}^1 \bar x_2^1 \\
\begin{split}
&= (-1)^{n-1}\mu_{10}  \bar x_{1 \ldots nm-1}\bar y_{nm-1 \ldots 2n+2} \bar y_{2n+1 \ldots n+3}^5 \bar x_{n+1} \bar y_{n+1 \ldots 4}^2 \bar y_3^{2'} \\
&= \mu_{10}  \bar x_{1 \ldots nm-1} \bar y_{nm-1 \ldots n+1} \bar y_{n \ldots 3}^1
\end{split}
\end{multline*} 
where the first equality is similar to the above computation, and the second equality uses the Plücker-type relations twice (involving $\bar y_4^2 \bar y_3^{2'}$ and $\bar y_{n+4}^2 \bar y_{n+3}^{2'}$) and  the anticommutativity relations ($3n-7$ times). This proves the claim. 
 
Using  \eqref{align:equalityY3} and Lemma \ref{lemma:higherrelations1} \ref{lemma:higherrelations30} we obtain 
\begin{align*}
\bar X_3 \psi_{\bar X_5} &= \mu_{13}  \bar X_3 \bar x_{2 \ldots n}^1 \bar x_{n+1 \ldots nm-1} \bar y_{nm-1 \ldots n+1}  \bar y_{n \ldots 3}^1 \\
&= (-1)^{n-1} \mu_{13} \bar x_{1 \ldots nm-1} \bar y_{nm-1 \ldots n+1} \bar y_{n \ldots 3}^1.
\end{align*}
Since $\bar X_4 = \bar x_4^3$ we have 
\begin{align*}
\begin{split}
\psi_{\bar X_1} \bar X_4 &= \mu_5 \bar x_{1 \ldots nm-2} \bar y_{nm-2 \ldots 2} \bar x_4^3 + \mu_6 \bar x_{1 \ldots nm-1} \bar y_{nm-1 \ldots n+3} \bar y_{n+2 \ldots 4}^3\bar x_4^3 \\
&\quad\ + \mu_7 \bar x_{3 \ldots n+1}^2 \bar x_{n+2 \ldots nm-1} \bar y_{nm-1 \ldots 2}\bar x_4^3
\end{split}
\\
\begin{split}
&= (-1)^n \mu_5 \bar x_{1 \ldots nm-2} \bar y_{nm-2 \ldots n} \bar x_n \bar y_{n \ldots 3}^1 \\
&\quad\ - \mu_6 \bar x_{1 \ldots nm-1} \bar y_{nm-1 \ldots n+1} \bar y_{n \ldots 3}^1
\end{split}
\\
&= ((-1)^{m-1}\mu_5 - \mu_6) \bar x_{1 \ldots nm-1} \bar y_{nm-1 \ldots n+1} \bar y_{n \ldots 3}^1
\end{align*}
where the summand with $\mu_7$ vanishes since $\bar y_{nm-1 \ldots 2}\bar x_4^3=0$  by Remark \ref{remark:irreduciblegrading}. 
 
To obtain the summand with coefficient $\mu_5$ we use the anticommutativity relations and Lemma \ref{lemma:higherrelations3}. To obtain the summand with $\mu_6$, we apply the relation 
$\bar y_4^3 \bar x_4^3 = -\bar x_{3}^{2'} \bar y_{3}^{2'}  -\bar x_3^{4'} \bar y_3^{4'}- \bar x_5^3 \bar y_5^3$ and obtain 
\begin{multline*}
\mu_6 \bar x_{1 \ldots nm-1} \bar y_{nm-1 \ldots n+3} \bar y_{n+2 \ldots 4}^3\bar x_4^3 \\
\begin{split}
&= -\mu_6 \bar x_{1 \ldots nm-1}(\bar y_{nm-1 \ldots n+3} \bar y_{n+2 \ldots 5}^3 \bar x_{3}^{2'} \bar y_{3}^{2'} + \bar y_{nm-1 \ldots n+3} \bar y_{n+2 \ldots 5}^3\bar x_3^{4'} \bar y_3^{4'}) \\
&= -\mu_6\bar x_{1 \ldots nm-1} \bar y_{nm-1 \ldots n+1} \bar y_{n \ldots 3}^1
\end{split}
\end{multline*}
where the first equality uses  $\bar y_{nm-1 \ldots n+3} \bar y_{n+2 \ldots 5}^3 \bar x_5^3 = 0$ by Remark \ref{remark:irreduciblegrading} and in the second equality we note that $\bar y_5^3, \bar x_3^{4'}$ lie in a square so that  
 $\bar y_5^3 \bar x_3^{4'} = - \bar x' \bar y''$ for some $\bar x', \bar y''$, then $\bar y_{nm-1 \ldots n+3} \bar y_{n+2 \ldots 5}^3 \bar x' = 0$ by Remark \ref{remark:irreduciblegrading} and thus the second summand vanishes. For the first summand we use the anticommutativity relations ($2n-5$ times) and the Plücker-type relation once (involving $\bar y_4^2 \bar y_3^{2'}$). 
 
 Similar to the second equality in  \eqref{align:y1x3} we have  
 \begin{align*}
  \psi_{\bar X_7} \bar y_{3}^{2'} ={} &  \mu_{15} \bar x_{1 \ldots nm-1} \bar y_{nm-1 \ldots n+2} y_{n+1 \ldots 4}^2\bar y_{3}^{2'}\\
  ={}&(-1)^{n} \mu_{15} \bar x_{1 \ldots nm-1} \bar y_{nm-1 \ldots n+1} \bar y_{n \ldots 3}^1.
 \end{align*}
Note that $ \bar X_1 \psi_{\bar X_4} =  \mu_{12}\bar x_{1 \ldots nm-1} \bar y_{nm-1 \ldots n+1} \bar y_{n \ldots 3}^1$ is already irreducible. 
\end{proof}

The following lemma was used in Lemma \ref{lemma:A6}.

\begin{lemma}\label{lemma:A9}
We have 
\begin{align*}
\mu_3 \bar y_{nm-2 \ldots n+2}  \bar y_{n+1 \ldots 3}^2  \bar x_3^2 &= -(-1)^{n-m} \mu_3  \bar y_{nm-1 \ldots n+2}\bar y_{n+1 \ldots 4}^2  \\
\mu_4 \bar y_{nm-1 \ldots 2n+3} \bar y_{2n+2 \ldots n+4}^6 \bar y_{n+1 \ldots 3}^2 \bar x_3^2 &=0.
\end{align*}
\end{lemma}

\begin{proof}
Let us prove the first identity. Note that we have
\begin{multline*}
\mu_3\bar y_{nm-2 \ldots n+2}  \bar y_{n+1 \ldots 3}^2  \bar x_3^2 \\
\begin{split}
&= \mu_3 \bar y_{nm-2 \ldots n+2} (\bar y_{n+1 \ldots 5}^2 \bar x_5^2\bar y_5^2 \bar y_4^2 + \bar y_{n+1 \ldots 5}^2 \bar x_5^{3'} \bar y_5^{3'}\bar y_4^2 - \bar y_{n+1 \ldots 4}^2\bar x_4^{7'} \bar y_4^{7'}) \\
&= - \mu_3 \bar y_{nm-2 \ldots 2n} \bar x_{2n} \bar y_{2n \ldots n+2} \bar y_{n+1 \ldots 4}^2
\end{split}
\end{multline*}
where the first equality uses $\bar y_3^2 \bar x_3^2 = - \bar x_4^2 \bar y_4^2 - \bar x_4^{7'} \bar y_4^{7'}$ and then $\bar y_4^2\bar x_4^2 = -\bar x_5^2 \bar y_5^2 - \bar x_5^{3'} \bar y_5^{3'}$. Let us explain the second equality: the first summand vanishes by using  $\bar y_i^2 \bar x_i^2 = - \bar x_{i+1}^2 \bar y_{i+1}^2$ for $5\leq i \leq n$ and $\bar y_{n+1}^2 \bar x_{n+1}^2  = 0$; the second summand vanishes since by the anticommutativity relations  $\bar y_i^2 \bar x_i^{3'} = - \bar x_{i+1}^{3'}\bar y_{i+1}^3$ for $5\leq i \leq n+1$ (denote $\bar x_{n+2}^{3'} = \bar x_{n+2}$) we have 
\[
\bar y_{nm-2 \ldots n+2}  \bar y_{n+1 \ldots 5}^2 \bar x_5^{3'}  =-(-1)^n \bar y_{nm-2 \ldots n+2}  \bar x_{n+2} \bar y_{n+2 \ldots 6}^3 = 0
\] 
where the second equality follows from Lemma \ref{lemma:higherrelations1} \ref{lemma:higherrelations30}, and for the third summand we use the anticommutativity relations ($4n-7$ times) and the Plücker-type relation $\bar y_5^7 \bar y_4^{7'} =-\bar y_5^{7'} \bar y_4^2 - \dotsb$ (Observation \ref{keyobservation} is applied to kill the extra term here). Then the desired equality follows from Lemma \ref{lemma:higherrelations3}. 

Let us prove the second identity. Note that we have
\begin{multline}
\label{align:A10}
\mu_4 \bar y_{nm-1 \ldots 2n+3} \bar y_{2n+2 \ldots n+4}^6 \bar y_{n+1 \ldots 3}^2 \bar x_3^2 \\
\begin{split}
&= (-1)^{n+1}\mu_4 \bar y_{nm-1 \ldots 2n+3} \bar y_{2n+2 \ldots n+4}^6 (\bar x_{n+2}\bar y_{n+2 \ldots 6}^3 \bar y_5^{3'} \bar y_4^2 +\bar x_{n+4}^6 \bar y_{n+2 \ldots 5}^7\bar y_4^{7'} )\\
&= \mu_4 \bar y_{nm-1 \ldots 2n+3} \bar y_{2n+2 \ldots n+4}^6 (\bar x_{n+2} \bar y_{n+2} \bar y_{n+1 \ldots 4}^2 - \bar x_{n+4}^6 \bar y_{n+4}^6 \bar y_{n+1 \ldots 4}^2 )
\end{split}
\end{multline}
where the first equality follows from a similar computation as above, the second one uses the relations  $\bar y_i^3 \bar y_{i-1}^{3'} = - \bar y_i^{3'} \bar y_{i-1}^2$ and $\bar y_j^7 \bar y_{j-1}^{7'} = - \bar y_j^{7'} \bar y_{j-1}^2$ for $6 \leq i  \leq n+2$ and $ 5 \leq j \leq n+2$  (denote $y_{n+2}^{3'}= \bar y_{n+2}$ and $y_{n+2}^{7'}=y_{n+4}^6$). 

The first summand on the right-hand side of the second equality in \eqref{align:A10} equals
\begin{multline*}
\mu_4 \bar y_{nm-1 \ldots 2n+3} \bar y_{2n+2 \ldots n+4}^6 \bar x_{n+2} \bar y_{n+2} \bar y_{n+1 \ldots 4}^2 \\
= \mu_4 \bar y_{nm-1 \ldots 2n+3} \bar y_{2n+2 \ldots n+6}^6  \bar y_{n+5}^{6'} \bar y_{n+4}^{5'} \bar y_{n+3}^{4'} \bar y_{n+2} \bar y_{n+1 \ldots 4}^2
\end{multline*}
where we first use $\bar y_{n+4}^6 \bar x_{n+2} = - \bar x_{n+3}^4\bar y_{n+3}^{4'} $ and then  $\bar y_{n+5}^6 \bar x_{n+3}^4 = -\bar y_{n+5}^{6'} \bar y_{n+4}^{5'}.$

Let us compute the second summand in \eqref{align:A10}. Note that $\bar y_{n+4}^6 \bar x_{n+4}^6 = -\bar x_{n+3}^{5'} \bar y_{n+3}^{5'} -\bar x_{n+5}^6 \bar y_{n+5}^6 - \bar x' \bar y'$ where $\bar x'$ is the  arrow (which is missing in Fig.~\ref{figure:counterexample}) with the same starting vertex as $\bar x_{n+5}^6$. Then we have
\begin{multline}
\label{last}
\mu_4 \bar y_{nm-1 \ldots 2n+3} \bar y_{2n+2 \ldots n+4}^6\bar x_{n+4}^6 \bar y_{n+4}^6 \bar y_{n+1 \ldots 4}^2\\
\begin{aligned}
&= -\mu_4 \bar y_{nm-1 \ldots 2n+3} \bar y_{2n+2 \ldots n+5}^6(\bar x_{n+3}^{5'} \bar y_{n+3}^{5'} +\bar x_{n+5}^6 \bar y_{n+5}^6 + \bar x' \bar y') \bar y_{n+4}^6 \bar y_{n+1 \ldots 4}^2\\
&= - \mu_4 \bar y_{nm-1 \ldots 2n+3} \bar y_{2n+2 \ldots n+5}^6 \bar x_{n+3}^{5'} \bar y_{n+3}^{5'}  \bar y_{n+4}^6 \bar y_{n+1 \ldots 4}^2\\
&= \mu_4 \bar y_{nm-1 \ldots 2n+3} \bar y_{2n+2 \ldots n+6}^6  \bar y_{n+5}^{6'} \bar y_{n+4}^{5'} \bar y_{n+3}^{4'} \bar y_{n+2} \bar y_{n+1 \ldots 4}^2.
\end{aligned}
\end{multline}
To see the second equality of \eqref{last} follows since the second summand vanishes using $\bar y_{n+5}^6 \bar x_{n+5}^6 = 0$ and the third summand vanishes, since $\bar y_{n+5}^6 \bar x'$ lies in a square so that $\bar y_{n+5}^6 \bar x' = - \bar x'' \bar y'''$ for some arrows $\bar x'', \bar y'''$. We then have 
$$
\bar y_{nm-1 \ldots 2n+3} \bar y_{2n+2 \ldots n+5}^6\bar x' = - \bar y_{nm-1 \ldots 2n+3} \bar y_{2n+2 \ldots n+6}^6\bar x'' y'''= 0
$$
where the second equality follows from Remark \ref{remark:irreduciblegrading}.  
For the third equality of \eqref{last} we use $\bar y_{n+5}^6\bar x_{n+3}^{5'} = - \bar y_{n+5}^{6'} \bar y_{n+4}^{5}$,  $\bar y_{n+4}^{5} \bar y_{n+3}^{5'}  = - \bar y_{n+4}^{5'} \bar y_{n+3}^4 - \dotsb$ and then $ \bar y_{n+3}^4 \bar y_{n+4}^{6} = -\bar y_{n+3}^{4'} \bar y_{n+2}$ (the extra term $\dotsb$ vanishes due to Observation \ref{keyobservation}).

Substituting the above equalities into \eqref{align:A10} we obtain the desired equality. 
\end{proof}

\endgroup


\begin{thebibliography}{AKO08}

\bibitem[AS16]{abouzaidsmith1}
M.~Abouzaid, I.~Smith,
The symplectic arc algebra is formal,
Duke Math.\ J.\ 165 (2016) 985--1060.

\bibitem[AS19]{abouzaidsmith2}
M.~Abouzaid, I.~Smith,
Khovanov homology from Floer cohomology,
J.\ Amer.\ Math.\ Soc.\ 32 (2019) 1--79.

\bibitem[Bar97]{bardzell}
M.\,J.~Bardzell,
The alternating syzygy behavior of monomial algebras,
J.\ Algebra 188 (1997) 69--89.

\bibitem[BW]{barmeierwang1}
S.~Barmeier, Z.~Wang,
Deformations of path algebras of quivers with relations, forthcoming in Astérisque,
\href{https://arxiv.org/abs/2002.10001}{arXiv:2002.10001} (2020).

\bibitem[BGS96]{beilinsonginzburgsoergel}
A.~Beilinson, V.~Ginzburg, W.~Soergel,
Koszul duality patterns in representation theory,
J.\ Amer.\ Math.\ Soc.\ 9 (1996) 473--527.

\bibitem[Ber78]{bergman}
G.\,M.~Bergman,
The Diamond Lemma for ring theory,
Adv.\ Math.\ 29 (1978) 178--218.

\bibitem[Bra02]{braden}
T.~Braden,
Perverse sheaves on Grassmannians,
Canad.\ J.\ Math.\ 54 (2002) 493--532.

\bibitem[BG96]{bravermangaitsgory}
A.~Braverman, D.~Gaitsgory,
Poincaré--Birkhoff--Witt theorem for quadratic algebras of Koszul type,
J.\ Algebra 181 (1996) 315--328.

\bibitem[BS11a]{brundanstroppel1}
J.~Brundan, C.~Stroppel,
Highest weight categories arising from Khovanov's diagram algebra I: Cellularity,
Moscow Math.\ J.\ 11 (2011) 685--722.

\bibitem[BS10]{brundanstroppel2}
J.~Brundan, C.~Stroppel,
Highest weight categories arising from Khovanov's diagram algebra II: Koszulity,
Transform.\ Groups 15 (2010) 1--45.

\bibitem[BS11b]{brundanstroppel3}
J.~Brundan, C.~Stroppel,
Highest weight categories arising from Khovanov's diagram algebra III: Category $\mathcal O$,
Represent.\ Theory 15 (2011) 170--243.

\bibitem[BS12]{brundanstroppel4}
J.~Brundan, C.~Stroppel,
Highest weight categories arising from Khovanov's diagram algebra IV: The general linear supergroup,
J.\ Eur.\ Math.\ Soc.\ (JEMS) 14 (2012) 373--419.

\bibitem[Buc03]{buchweitz}
R.-O.~Buchweitz, Hochschild cohomology of Koszul algebras, 
talk at the Conference on Representation Theory (Canberra 2003).

\bibitem[CK14]{chenkhovanov}
Y.~Chen, M.~Khovanov,
An invariant of tangle cobordisms via subquotients of arc rings,
Fund.\ Math.\ 225 (2014) 23--44.

\bibitem[CYZ16]{chenyangzhou}
X.~Chen, S.~Yang, G.~Zhou, 
Batalin--Vilkovisky algebras and the noncommutative Poincar\'e duality of Koszul Calabi--Yau algebras, 
J.\ Pure Appl.\ Algebra 220 (2016) 2500--2532.

\bibitem[CLW]{chenliwang}
X.-W.~Chen, H.~Li, Z.~Wang, 
Leavitt path algebras, $B_\infty$-algebras and Keller's conjecture for singular Hochschild cohomology, 
\href{https://arxiv.org/abs/2007.06895}{arXiv:2007.06895} (2020).

\bibitem[CW21]{chenwang}
X.-W.~Chen, Z.~Wang, 
The dg Leavitt algebra, singular Yoneda category and singularity category, with an appendix by Bernhard Keller and Yu Wang,
Adv.\ Math.\ 440 (2024) 109541.

\bibitem[CS15]{chouhysolotar}
S.~Chouhy, A.~Solotar,
Projective resolutions of associative algebras and ambiguities,
J.\ Algebra 432 (2015) 22--61.

\bibitem[Con11]{conner}
A.\,B.~Conner, 
A$_\infty$-structures, generalized Koszul properties, and combinatorial topology,
PhD thesis, available at \url{https://hdl.handle.net/1794/11559}, University of Oregon, 2011.

\bibitem[FV06]{floystadvatne}
G.~Fløystad, J.\,E.~Vatne,
PBW-deformations of $N$-Koszul algebras,
J.\ Algebra 302 (2006) 116--155.

\bibitem[Ger64]{gerstenhaber}
M.~Gerstenhaber,
On the deformation of rings and algebras,
Ann.\ Math.\ 79 (1964) 59--103.

\bibitem[HK91]{huebschmannkadeishvili}
J.~Huebschmann, T.~Kadeishvili, 
Small models for chain algebras, 
Math.\ Z.\ 207 (1991) 245--280.

\bibitem[Kad82]{kadeishvili}
T.\,V.~Kadeishvili,
The algebraic structure in the homology of an $A(\infty)$-algebra (Russian),
Soobshch.\ Akad.\ Nauk Gruzin.\ SSR 108 (1982) 249--252.

\bibitem[Kad88]{kadeishvili2}
T.\,V.~Kadeishvili,
The structure of the $A(\infty)$-algebra, and the Hochschild and Harrison cohomologies, 
Trudy Tbiliss.\ Mat.\ Inst.\ Razmadze Akad.\ Nauk Gruzin.\ SSR 91 (1988) 19--27.

\bibitem[KL79]{kazhdanlusztig}
D.~Kazhdan, G.~Lusztig,
Representations of Coxeter groups and Hecke algebras,
Invent.\ Math.\ 53 (1979) 165--184.

\bibitem[Kel01]{keller1}
B.~Keller,
Introduction to A-infinity algebras and modules, 
Homol.\ Homotopy Appl.\ 3 (2001) 1--35.

\bibitem[Kel03]{keller3}
B.~Keller,
Derived invariance of higher structures on the Hochschild complex, available at \url{https://webusers.imj-prg.fr/~bernhard.keller/publ/dih.pdf} (2003)

\bibitem[Kel21]{keller4}
B.~Keller,
A remark on Hochschild cohomology and Koszul duality,
Advances in Representation Theory of Algebras (ARTA VII) (Mexico City 2018), 131--136, 
American Mathematical Society, Providence, RI, 2021.

\bibitem[Kho00]{khovanov}
M.~Khovanov,
A categorification of the Jones polynomial,
Duke Math.\ J.\ 101 (2000) 359--426.

\bibitem[KS12]{klamtstroppel}
A.~Klamt, C.~Stroppel,
On the Ext algebras of parabolic Verma modules and $A_\infty$-structures,
J.\ Pure Appl.\ Algebra 216 (2012) 323--336.

\bibitem[LS81]{lascouxschuetzenberger}
A.~Lascoux, M.-P.~Schützenberger,
Polynômes de Kazhdan et Lusztig pour les Grassmanniennes,
Astérisque 87--88 (1981) 249--266.

\bibitem[Lod98]{loday}
J.-L.~Loday, 
Cyclic homology, 
Springer-Verlag, Berlin, 1992. 

\bibitem[MS22]{maksmith}
C.\,Y.~Mak, I.~Smith,
Fukaya--Seidel categories of Hilbert schemes and parabolic category $\mathcal O$,
J.\ Eur.\ Math.\ Soc.\ (JEMS) 24 (2022) 3215--3332.

\bibitem[Man06]{manolescu}
C.~Manolescu,
Nilpotent slices, Hilbert schemes, and the Jones polynomial,
Duke Math.\ J.\ 132 (2006) 311--369.

\bibitem[Pos11]{positselski}
L.~Positselski,
Two kinds of derived categories, Koszul duality, and comodule--contramodule correspondence,
Mem.\ Amer.\ Math.\ Soc.\ 212 (2011) no.~996.

\bibitem[Pri70]{priddy}
S.\,B.~Priddy, 
Koszul resolutions, 
Trans.\ Amer.\ Math.\ Soc.\ 152 (1970) 39--60.

\bibitem[Roz]{rozansky}
L.~Rozansky,
A categorification of the stable $\mathrm{SU} (2)$ Witten--Reshetikhin--Turaev invariant of links in $\mathbb S^2 \times \mathbb S^1$,
\href{https://arxiv.org/abs/1011.1958}{arXiv:1011.1958} (2010).

\bibitem[Sei02]{seidel1}
P.~Seidel,
Fukaya categories and deformations,
Proceedings of the International Congress of Mathematicians (Beijing 2002), Vol.\ II, 351--360,
Higher Education Press, Beijing, 2002.

\bibitem[Sei08]{seidel2}
P.~Seidel,
Fukaya categories and Picard--Lefschetz theory,
Zurich Lectures in Advanced Mathematics,
European Mathematical Society, Zürich, 2008.

\bibitem[SS06]{seidelsmith}
P.~Seidel, I.~Smith,
A link invariant from the symplectic geometry of nilpotent slices,
Duke Math.\ J.\ 134 (2006) 453--514.

\bibitem[ST01]{seidelthomas}
P.~Seidel, R.~Thomas,
Braid group actions on derived categories of coherent sheaves,
Duke Math.\ J.\ 108 (2001) 37--108.

\bibitem[Str09]{stroppel1}
C.~Stroppel,
Parabolic category $\mathcal O$, perverse sheaves on Grassmannians, Springer fibres and Khovanov homology,
Compos.\ Math.\ 145 (2009) 954--992.

\bibitem[Str10]{stroppel2}
C.~Stroppel,
Schur--Weyl dualities and link homologies,
Proceedings of the International Congress of Mathematicians (Hyderabad 2010), Vol.~III, 1344--1365,
Hindustan Book Agency, New Delhi, 2010.

\bibitem[Str23]{stroppel3}
C.~Stroppel,
Categorification: tangle invariants and TQFTs, 
Proceedings of the International Congress of Mathematicians (Virtual ICM 2022), Vol.~II, 1312--1353,
EMS Press, Berlin, 2023.

\bibitem[Vyb07]{vybornov}
M.~Vybornov,
Perverse sheaves, Koszul IC-modules, and the quiver for the category $\mathcal O$,
Invent.\ Math.\ 167 (2007) 19--46.

\bibitem[Web16]{webster}
B.~Webster,
Tensor product algebras, Grassmannians and Khovanov homology, 
Physics and mathematics of link homology (Montreal 2013), 23--58, 
American Mathematical Society, Providence, RI, 2016.

\end{thebibliography}
\end{document}